\documentclass{article}

\usepackage[arrow,matrix,curve]{xy}
\usepackage{tikz}
\usetikzlibrary{matrix}
\usepackage{ushort}
\usetikzlibrary{decorations.markings}
\pgfdeclarelayer{edgelayer}
\pgfdeclarelayer{nodelayer}
\pgfsetlayers{edgelayer,nodelayer,main}

\tikzstyle{none}=[inner sep=0pt]
\definecolor{hexcolor0x5235ff}{rgb}{0.322,0.208,1.000}

\tikzstyle{oval}=[oval,fill=hexcolor0x5235ff,draw=Black]
\tikzstyle{$2$}=[circle,fill=White,draw=Black]
\usepackage{amscd}

\usepackage{verbatim}
\usepackage{mathtools}
\usepackage{amssymb}
\usepackage{amsmath}
\usepackage{bm}
\usepackage{stmaryrd}
\usepackage{hyperref}
\usepackage{graphicx}
\usepackage{amsmath}
\usepackage{upgreek}
\usepackage{mathrsfs}
\numberwithin{equation}{section}
\newtheorem{theorem}[equation]{Theorem}

\newtheorem{conjecture}{Conjecture}
\newtheorem{construction}[equation]{Construction}
\newtheorem{convention}{Convention}
\newtheorem{corollary}[equation]{Corollary}

\newtheorem{definition}[equation]{Definition}
\newtheorem{example}[equation]{Example}

\newtheorem{lemma}[equation]{Lemma}

\newtheorem{proposition}[equation]{Proposition}
\newtheorem{remark}[equation]{Remark}

\newtheorem{question}{Question}

\newenvironment{proof}[1][Proof]{\textbf{#1.} }{\ \rule{0.5em}{0.5em}}

\newcommand {\id}{\mathsf{id}}

\newcommand{\Imm}{\mbox{{\rm Im}}}
\newcommand{\Tor}{\mathrm{Tor}}
\newcommand{\supp}{\mathrm{supp}}
\newcommand{\Coker}{\mathrm{coKer}}
\newcommand{\Sym}{\mathrm{Sym}}
\newcommand{\rHom}{\mathrm{Hom}}
\newcommand {\Ker}{\mathrm{Ker}}
\newcommand {\tr}{\mathrm{tr}}

\newcommand{\dbar}{\bar \partial}

\newcommand{\so}{\mathfrak{so}}

\newcommand{\gl}{\mathfrak{gl}}

\renewcommand{\O}{\mathcal{O}}

\newcommand{\F}{\mathcal{F}}
\newcommand{\SO}{\mathrm{SO}}
\newcommand{\OGr}{\mathrm{OGr}}
\newcommand{\GL}{\mathrm{GL}}

\renewcommand{\P}{\mathbf{P}}

\newcommand{\Spin}{\mathrm{Spin}}

\newcommand{\M}{\mathpzc{M}}

\newcommand{\sd}{\partial}
\newcommand{\Pf}{\mathrm{Pf}}

\newcommand{\g}{\mathfrak{g}}

\newcommand{\Ext}{\mathrm{Ext}}
\newcommand{\Hom}{\mathrm{Hom}}

\newcommand{\halpha}{\alpha}
\newcommand{\hbeta}{\beta}
\newcommand{\hgamma}{\gamma}
\newcommand{\hdelta}{\delta}

\newcommand{\ralpha}{\upalpha}
\newcommand{\rbeta}{\upbeta}
\newcommand{\rgamma}{\upgamma}
\newcommand{\rdelta}{\updelta}
\newcommand{\repsilon}{\upepsilon}
\newcommand{\rvarepsilon}{\upvarepsilon}
\newcommand{\rmu}{\upmu}
\newcommand{\rnu}{\upnu}

\newcommand{\reflection}{\mathsf{\sigma}}
\newcommand{\shift}{\mathsf{\tau}}

\newcommand{\setP}{\mathrm{P}}
\newcommand{\setA}{\mathrm{A}}
\newcommand{\setB}{\mathrm{B}}
\newcommand{\setC}{\mathrm{C}}
\newcommand{\setK}{\mathrm{K}}
\newcommand{\setL}{\mathrm{L}}
\newcommand{\setN}{\mathrm{N}}
\newcommand{\setQ}{\mathrm{Q}}
\newcommand{\setE}{\mathrm{E}}
\newcommand{\setS}{\mathrm{S}}
\newcommand{\setU}{\mathrm{U}}
\newcommand{\setX}{\mathrm{X}}
\newcommand{\setY}{\mathrm{Y}}
\newcommand{\setCL}{\mathrm{CL}}
\newcommand{\sethE}{\hat{\mathrm{E}}}

\newcommand{\Aut}{Aut}

\newcommand{\inclusiono}{\mathsf{i}}
\newcommand{\inclusiont}{\mathsf{k}}
\newcommand{\inclusionth}{\mathsf{u}}
\newcommand{\projectiono}{\mathsf{pr}}

\newcommand{\C}{\mathbb{C}}
\newcommand{\R}{\mathbb{R}}

\newcommand{\Z}{\mathbb{Z}}
\newcommand{\Q}{\mathbb{Q}}

\newcommand{\Maps}{\mathpzc{Maps}}

\newcommand{\codim}{\mathrm{codim}}
\newcommand{\Cech}{\check{\mathrm{C}}\mathrm{ech}}

\newcommand{\s}{\bm{s}}

\newcommand{\Lie}{\mathrm{Lie}}

\newcommand{\rH}{\mathrm{H}}
\newcommand{\rG}{\mathrm{G}}
\newcommand{\bT}{\mathbf{T}}
\newcommand{\rk}{\mathrm{rk}}

\newcommand{\cZ}{\mathpzc{Z}}
\newcommand{\Arc}{\mathpzc{Arc}}

\newcommand{\cH}{\mathpzc{H}}

\DeclareMathAlphabet{\mathpzc}{OT1}{pzc}{m}{it}
\newcommand{\MI}{\mathcal{I}}
\newcommand{\rCap}{\mathrm{Cap}}
\newcommand{\rCircle}{\mathrm{Circle}}
\newcommand{\Th}{\mathsf{th}}
\newcommand{\Thl}{\mathsf{Th}}

\newcommand{\cL}{\mathcal{L}}
\newcommand{\rM}{\mathrm{M}}
\newcommand{\rL}{\mathrm{L}}
\newcommand{\fm}{\mathfrak{m}}
\newcommand{\fp}{\mathfrak{p}}

\newcommand{\rReg}{\mathrm{Reg}}
\newcommand{\rCore}{\mathrm{Core}}

\newcommand{\rA}{\mathrm{A}}
\newcommand{\rB}{\mathrm{B}}

\newcommand{\alt}{\mathrm{Alt}}
\newcommand{\Span}[1]{\mathrm{span}\left\langle{#1}\right\rangle}
\newcommand{\LSpan}[1]{\mathrm{span}_{\cL}\left\langle{#1}\right\rangle}

\newcommand{\Ann}{\mathrm{Ann}}
\newcommand{\Rad}{\mathrm{Rad}}
\newcommand{\bx}{\overline{x}}
\newcommand{\bl}{\overline{\lambda}}
\newcommand{\itwo}{\frac{\infty}{2}}

\newcommand{\fr}{\mathfrak{r}}

\newcommand{\functionl}{\mathsf{l}}
\newcommand{\functionf}{\mathsf{f}}

\newcommand{\fq}{\mathfrak{q}}
\newcommand{\fri}{\mathfrak{i}}
\newcommand{\fj}{\mathfrak{j}}
\newcommand{\fa}{\mathfrak{a}}
\newcommand{\fb}{\mathfrak{b}}
\newcommand{\ff}{\mathfrak{f}}

\newcommand{\fe}{\mathfrak{e}}

\newcommand{\fz}{\mathfrak{z}}
\newcommand{\fy}{\mathfrak{y}}
\newcommand{\fc}{\mathfrak{c}}
\newcommand{\res}{\mathrm{res}}
\newcommand{\ot}{\leftarrow}

\newcommand{\Mrg}{I}
\newcommand{\Frg}{\mathrm{Fock}}
\newcommand{\pl}{\underline{\mathsf{p}}}

\newcommand{\hA}{A}
\newcommand{\sA}{\mathring{A}}
\newcommand{\jA}{J}
\newcommand{\tjA}{J}
\newcommand{\bjA}{J}
\newcommand{\hRo}{R}
\newcommand{\hRt}{S}
\newcommand{\hRth}{Q}
\newcommand{\BV}{BV}

\newcommand{\moduleT}{T}
\newcommand{\smoduleT}{\mathring{T}}
\newcommand{\fT}{\mathpzc{T}}
\newcommand{\sfT}{\mathring{\mathpzc{T}}}
\newcommand{\sV}{\mathring{V}}
\newcommand{\moduleS}{S}

\newcommand{\grdif}{\delta}
\newcommand{\spinor}{S_+}
\newcommand{\sspinor}{\mathring{S}_+}

\newcommand{\sdspinor}{\mathring{S}_-}
\newcommand{\sP}{\mathring{P}}
\newcommand{\sD}{\mathring{D}}

\newcommand{\sLambda}{\mathring{\Lambda}}
\newcommand{\vectorrep}{\mathring{V}}

\newcommand{\cone}{\mathcal{C}}
\newcommand{\tspace}{\mathcal{X}}
\newcommand{\rmk}{\mathsf{n}}
\newcommand{\hd}{\mathrm{hd}}
\newcommand{\hcodim}{\mathrm{hcodim}}

\newcommand{\sfa}{\mathsf{a}}
\newcommand{\sfb}{\mathsf{b}}
\newcommand{\sfd}{\mathsf{d}}
\newcommand{\sfm}{\mathsf{m}}
\newcommand{\sfp}{\mathsf{p}}
\newcommand{\sfq}{\mathsf{q}}

\newcommand{\sfu}{\mathsf{u}}
\newcommand{\sfv}{\mathsf{v}}
\newcommand{\sfr}{\mathsf{r}}
\newcommand{\sfs}{\mathsf{w}}
\newcommand{\sfres}{\mathsf{res}}
\newcommand{\sfi}{\mathsf{i}}
\newcommand{\sfj}{\mathsf{j}}
\newcommand{\sfk}{\mathsf{k}}
\newcommand{\sff}{\mathsf{f}}
\newcommand{\sfg}{\mathsf{g}}

\newcommand{\sfU}{\mathsf{Y}}

\newcommand{\lu}{\underline{u}}
\newcommand{\uu}{\overline{u}}

\begin{document}
\title{Local algebra and string theory.} 
\author{M.V. Movshev\\
Mathematics Department, Stony Brook University,\\
 Stony Brook NY, 11794-3651, USA\\
\texttt{mmovshev@math.sunysb.edu}}
\date{\today}
\maketitle
\begin{abstract}
The $\beta\gamma$ system on the cone of pure spinors is an integral part of the string theory developed by N. Berkovits. 

This $\beta\gamma$ system offer a number of questions for pure mathematicians: what is a precise definition of the space of states of the theory? Is there a mathematical explanations for various dualities (or pairings) predicted by physicists? Can a formula for partition function be written? 
To help me to answer these questions I use local algebra in an essential way. 

\end{abstract}
\tableofcontents
\section{Introduction}
The notion of a $\beta\gamma$ system is not very familiar to mathematicians. Let us go over some of its basic aspect before we turn to detailed mathematical analysis of $\beta\gamma$ system, whose target is the cone of pure spinors. The interested reader can consult \cite{Nekrasovbetagamma} for a more comprehensive account. 

Let $\Lambda=\bigoplus_k \Lambda^k$ stands for the exterior algebra on some vector space or a bundle. I fix a smooth complex manifold $\tspace$ of real dimension $2n$, which will be the target of the $\beta\gamma$ system.

 Let $T_{\R}^*
 $ be the real cotangent bundle of $\tspace$. 
 Complexification of $T_{\R}^*$ splits into the direct sum $T_{\C}^*=T^*+\overline{T}^*$, which leads to  the $(p,q)$-decomposition 
 $\Lambda^k T^*_{\C}=\bigoplus_{p+q=k}\Lambda^p T^*\otimes \Lambda^q \overline{T}^*, k=0,\dots,2n$.
 The de Rham differential splits accordingly: $d=\sd+\dbar$.
 
Let $\Sigma$ be a complex one-dimensional manifold, considered as a $C^{\infty}$-surface. The fields of the $\beta\gamma$ system are: 
\begin{enumerate}
\item a smooth map $\beta:\Sigma\rightarrow \tspace$,
\item a smooth section $\gamma$ of the pullback $\beta^*T^*\otimes T^*_{\Sigma}$. 
\end{enumerate}
The anti holomorphic part $\dbar\beta$ of the differential $D\beta$ is a section of $\beta^*T^*\otimes \overline{T}^*_{\Sigma}$.
The pairing $\langle\dbar\beta,\gamma\rangle$ is a $(1,1)$-form on $\Sigma$. The action of the theory 
\[S(\beta,\gamma)=\int_{\Sigma}\langle\dbar\beta,\gamma\rangle\]
makes sense when, for example, $\gamma$ has a compact support.

Let $G$ be a group of holomorphic symmetries of $\tspace$.
Elements of the space $\Maps_{hol}(\Sigma,G)$ of holomorphic maps are symmetries of $S$. The functional $S$ is also invariant with respect to (infinitesimal) holomorphic reparametrizations of $\Sigma$.

Points of the space of solutions $\M_{\Sigma}$
of the Euler-Lagrange equations 
\[\begin{split}
&\dbar \beta=0\\
&\dbar \gamma=0
\end{split}\]
are a holomorphic map $\beta:\Sigma\rightarrow \tspace$ and a holomorphic section $\gamma$ of $\beta^*T^*\otimes T^*_{\Sigma}$.

The space $\M_{\C^{\times}}$, or the phase space, is equipped with a "$pdq$"-type form $\alpha$. It value on $\delta{\beta}$ at $(\beta,\gamma)\in \M_{\C^{\times}}$ is equal to 
\[\alpha(\delta\beta)=\oint \langle \gamma, \delta\beta \rangle.\] 

Exterior variational derivative of $\alpha$ is a closed two-form $(\delta\beta,\delta\gamma)=\oint \langle  \delta\beta,\delta\gamma\rangle.$ 

By the Noether theorem,  the semidirect product 
\begin{equation}\label{E:Lliealg}
Vect_{poly}(\C^{\times})\ltimes \Maps_{poly}(\C^{\times},\g),\quad \g=\mathrm{Lie}(G)
\end{equation}
of the Lie algebras is acting on the phase space $\M_{\C^{\times}}$. It is also should be acting projectively on the space of quantum states.

In the field theory, the space $\M_{\C^{\times}}$ plays the role of the cotangent bundle to 
\[\cZ_{\C^{\times}}=\{\beta:\C^{\times}\rightarrow \tspace|\dbar\beta=0\}\subset \M_{\C^{\times}}.\]
\subsection{Quantization}

Before I discuss quantization of the $\beta\gamma$-system, I need to remind the reader some of its general aspects.
Let $(Y,g_{ij})$ be a finite-dimensional Riemannian manifold, thought as a configuration space of a quantum mechanical system with constraints. The Hilbert space obtained by quantizing the phase space $T_{\R,Y}^*$ is the $L^2$ space. It is the closure of the space of the complex compactly supported smooth functions $C_{comp}^{\infty}(Y)$ with respect to the inner product
\begin{equation}\label{E:innerhilb}
(a,b)=\int_Ya\bar{b}\, dvol_g,\quad a,b\in C^{\infty}(Y).
\end{equation}
 The space $C_{comp}^{\infty}(Y)$ is a module over the algebra of differential operators on $Y$ (a D-module for short). There are other D-modules which could be candidates in quantum Hilbert spaces.
 An important example of a D-module is the space of distributions $D_U$ with the support on a submanifold 
$U\subset Y.$
 It is a common believe however, that distributions never form a space with an inner product. The reason - the integrand in (\ref{E:innerhilb}) for distributions $a,b$ is ill-defined. Still, there is an alternative way to define the  pairing.
Let $U_{+},U_{-}$ be two closed submanifolds of complementary dimension that intersect transversally at points $x_1,\dots,x_k$. $I(Y,U_{\pm})$ is the space of conormal distributions in the sense of H\"{o}rmander \cite{Hoermander}. There is a natural map $I(Y,U_{+})\otimes I(Y,U_{-})\overset{\sfm}{\rightarrow} I(Y,\{x_1,\dots,x_n\})$-the exterior product of distributions (\cite{Strohmaier}, Corollary 3, p. 113). I can use this map to define a pairing between spaces $I(Y,U_{+})\otimes I(Y,U_{-})$ by the formula \[\langle a,b\rangle=\int_Y \sfm(a,b)dx.\]
In the presence of a symmetry 
\[\reflection:Y\rightarrow Y\text{ such that }\reflection(U_{\pm})=U_{\mp},\] the pairing can be turned into an inner product \[( a,b)=\langle a,\reflection^*\overline{b}\rangle.\]

The theory of distributions of S. Sobolev and L. Schwartz has cohomological generalization in M.Sato's works on hyperfunctions (see e.g. \cite{Schapira} for the account). Hyperfunctions are defined on a real analytic space $Y$ embedded into its Grauert tube $Y^{\C}$. The elements of local cohomology $H^{\dim Y}_Y(Y^{\C},\O)$ of the sheaf of analytic functions $\O$ on $Y^{\C}$ with support on $Y$ are the Sato's hyperfunctions.
A typical example of a hyperfunction is Cauchy principal value $f(x)\to \mathrm{p.v.}\int \frac{f(x)}{x}dx$.
 In the same vein, one can define complex-analytic hyperfunctions on a complex manifold $Z$ with support on a complex submanifold 
\begin{equation}\label{E:complpair}
Z'\subset Z.
\end{equation}
The basic example of $Z'$  is a divisor  of a rational function $g:\C\to \C$. It define a  functional on the space of analytic functions by the formula
\[f\to \oint f(z)g(z)dz\]
The integration is taken over a circle of sufficiently large radius. The cohomological nature of hyperfunctions is already seen in this simple example. 
The integral is invariant under substitution $g(z)\to g(z)+h(z)$ where $h(z)$ is a polynomial. Thus the space of hyperfunctions is in fact the first cohomology of the Cousin complex $\O[\C]\hookrightarrow \O[\C\backslash Z']$. $\O[\C\backslash Z']$ is the space of rational functions with poles at $Z'$. 

The space of hyperfunctions is a D-module. If $\reflection$ is a holomorphic involution such that $\reflection(Z')\cap Z'$ is discrete and transversal, then the pairing on local cohomology is the composition
\[H^{i}_{Z'}(Z,\O)\otimes H^{\dim Z-i}_{Z'}(Z,\O)\overset{\sfm}{\longrightarrow} H^{\dim Z}_{\reflection(Z')\cap Z'}(Z,\O)\overset{\sfres}\longrightarrow \C,\quad i=\dim Z/2.\]
The map $\sfm$ is a cup product (see e.g. \cite{Iversen} Section II.10 ), $\sfres$ is the Poincar\'{e} residue. This construction
will be implemented in this paper. 
\subsection{Regularization}
 Let us return to $\beta\gamma$ systems. I would like to specialize the target to a conical set.
To be more precise, let $x^i, i=1,\dots,n$ be the coordinates of a vector $e$ in the standard basis for $\C^n$. Homogeneous algebraic equations 
\begin{equation}\label{E:manifoldeq}
r^p(x^1,\dots,x^n), p=1,\dots,k
\end{equation}
define a possibly nonsmooth algebraic cone $\tspace$ in $\C^n$. Though  $\dbar$ is undefined  on the whole   $\tspace$, the spaces $\cZ_{\Sigma}$ do make sense as 
\begin{equation}\label{E:holsing}
\cZ_{\Sigma}=\{\beta:\Sigma\to \C^n| \dbar \beta=0, \beta(z)\in \tspace\, \forall z\in \Sigma\}.
\end{equation}
It is a common practice in quantum field theory to approximate an infinite-dimensional $\cZ_{\Sigma}$ by more manageable finite-dimensional spaces. This procedure is called regularization.

The finite-dimensional approximations $\cZ_{\C^{\times}}(N,N')\subset \cZ_{\C^{\times}}, N\leq N'$ consists of  the maps 
\begin{equation}\label{E:coordFourier}
e(z)=\sum_{s=1}^nx^{s}(z)e_{s}=\sum_{N\leq k\leq N'}\sum_{i=1}^nx^{s^k}z^ke_{s}
\end{equation} from $\C^{\times}$ to $\tspace$ ($s^k$ is a multi-index 
) which are Laurent polynomials in $z$.
The algebra of polynomial functions on $\cZ(N,N')$ is the quotient of $\C[x^{s^k}], s=1,\dots,n, k=N,\dots,N'$. The ideal is generated by relations $R^{p^k}(x)$ whose generating function satisfies
\[\sum_{k=\deg(r^p)N}^{\deg(r^p)N'}R^{p^k}(x)z^k=r^p(x^1(z),\dots,x^n(z)).\]
In our application, the regularized pair (\ref{E:complpair}) is 
\[ \cZ(0,N')\subset \cZ(N,N').\]
The cone $\tspace$ has a singularity at the origin. This singularity propagates in $\cZ(N,N')$ and makes the latter highly singular. This makes the definition of conormal distributions $I(Z,Z')$ problematic. At this point the algebraic version of hyperfunctions $H_{Z'}^{\codim Z'}(Z,\O)$ comes to rescue because this sheaf-theoretic definition is less sensitive to singularities of the space. 

Technically it is more convenient to replace the space of analytic maps (\ref{E:holsing}) by the space of polynomial maps
\[\cZ^{poly}_{\C^{\times}}:=\bigcup_{N,N'}\cZ_{\C^{\times}}(N,N'), \]
\[\cZ^{poly}_{\C}:=\bigcup_{N'}\cZ_{\C}(0,N').\]
The symmetry group $\cZ^{poly}_{\C^{\times}}$ always contains the  product $\C^{\times}\times\C^{\times}$. The first factor 
corresponds to dilations of the cone $\cZ^{poly}_{\C^{\times}}$. The second factor, which will 
be denoted by $\bT$, corresponds to the loop rotation. Let $w=(a,u)$ be a $\C^{\times}\times\bT$ weight. In the following $V^w$ stands for weight $w$ subspace in representations $V$ of $\C^{\times}\times\bT$.
\begin{conjecture}\label{CON:main}
Let us assume that the algebra of global algebraic functions $\C[\tspace]$ is Gorenstein (
\ref{D:gorensteinserre}). Let $\P$ stands for projectivization of the cone. Suppose that $X=\P(\tspace)$ is smooth Fano such that 
$\mathrm{ind}X$\footnote{Index $\mathrm{ind}X$ of a projective Fano manifold $X$ is a maximal integer such that 
 $\mathcal{K}=\mathcal{L}^{\otimes \mathrm{ind}X}$ for the canonical $\mathcal{K}$ and some line bundle $\mathcal{L}$. $\mathrm{coind}\,X=\dim X- \mathrm{ind}X$. } is large (at least $\mathrm{ind}X>1$). Then
\begin{enumerate} 
\item \label{I:CONmain1} $\C[\cZ(N,N')]$ is Gorenstein for any $N<N'$.
\item \label{I:CONmain2} One can define a limit 
of weight spaces 
\begin{equation}\label{E:loccohini}
H_{\cZ^{poly}_{\C}}^{i+\itwo}(\cZ^{poly}_{\C^{\times}},\O)^w:= \underset{\underset{N}{\longrightarrow}}{\lim}\underset{\underset{N'}{\longleftarrow}}{\lim}H_{\cZ(0,N')}^{i+\codim \cZ(0,N')}(\cZ(N,N'),\O)^w,
\end{equation}
 when $N\to -\infty, N'\to \infty$. 
 $H_{\cZ^{poly}_{\C}}^{i+\itwo}(\cZ^{poly}_{\C^{\times}},\O):=\bigoplus_w H_{\cZ^{poly}_{\C}}^{i+\itwo}(\cZ^{poly}_{\C^{\times}},\O)^w$

\item \label{CON:main4}There is a nondegenerate pairing
\begin{equation}\label{E:pairinggen}
 H_{\cZ^{poly}_{\C}}^{i+\itwo}(\cZ^{poly}_{\C^{\times},},\O)\otimes H_{\cZ^{poly}_{\C}}^{k-i+\itwo}(\cZ^{poly}_{\C^{\times},},\O)\to \C
\end{equation}
for some $k$ depending on $\tspace$. In all known examples $k=\mathrm{coind}(\P(\tspace))+1$. 

\item \label{I:mainrep} The groups $H_{\cZ^{poly}_{\C}}^{i+\itwo}(\cZ^{poly}_{\C^{\times},},\O)$ are representations with weights bounded from below of a central extension of (\ref{E:Lliealg}).
The pairing (\ref{E:pairinggen}) is compatible with the symmetries. 
\item \label{I:CONmain45} 

I conjecture that  $\dim H_{\cZ(0,N')}^{i+\codim \cZ(0,N')}(\cZ(N,N'),\O)^w\leq C_w$. The constant $C_w$ doesn't depend on $N$ and $N'$.
\end{enumerate}
\end{conjecture}
There is a conjecture closely related to item (\ref{I:CONmain45}).
\begin{conjecture}\label{C:extweightbound}
I continue using assumptions of Conjecture \ref{CON:main}. 
 Fix a weight $w=(a,u)$.
I conjecture that \[\dim_{\C} \Ext^l_{\O[\cZ(N,N')]}(\C,\C)^{w}\leq C_{w,l}\] where $C_{w,l}$ doesn't depend on $N,N'$.
\end{conjecture}

\begin{remark}
It is worthwhile to indicate where the maps $H_{\cZ(0,N')}^{i}(\cZ(N,N'),\O)$ in (\ref{E:loccohini}) can come from. There is a pairing between cohomology $H_{\cZ(0,N')}^{i}(\cZ(N,N'),\O)$ and $H_{\cZ(N,-1)}^{i}(\cZ(N,N'),\O)$ when $\cZ(N,N')$ is Gorenstein. The maps $H_{\cZ(0,N')}^{i}(\cZ(N,N'),\O)\to H_{\cZ(0,N'-1)}^{i}(\cZ(N,N'-1),\O)$ that form the inductive part of the limit, are the restriction map. The "wrong way" maps are $H_{\cZ(0,N')}^{i}(\cZ(N,N'),\O)\to H_{\cZ(0,N')}^{i+\codim}(\cZ(N+1,N'),\O)$. They form the direct part of the limit, are the adjoint to the restriction maps $H_{\cZ(N+1,-1)}^{i}(\cZ(N+1,N'),\O)\to H_{\cZ(N,-1)}^{i}(\cZ(N,N'),\O)$. In my treatment I will use a different but equivalent construction of these maps.
\end{remark}
Local cohomology groups (\ref{E:loccohini}) can be computed by taking sub-complex of sections 
$\Gamma_{\cZ_{\C}(N')}(I^{\bullet})$ with the support on $\cZ_{\C}(N')$ in 
an injective resolution $I^{\bullet}$ of the sheaf $\O$. The total complex of the Koszul bicomplex \[\fT_I^k(N,N'):=\bigoplus_{i-j=k}B_{j}(\Gamma_{\cZ_{\C}(N')}(I^{i}),\{x^{s^k}\})\] conjecturally has its own pairing in cohomology $BV^i(N,N'):=H^i(\fT_I(N,N'))$.

\begin{conjecture}\label{CON:main2}
Let $n$ be as in (\ref{E:manifoldeq}).
\begin{enumerate}
\item \label{I:CONmain22} One can define a limit of linear spaces $BV^{i+\itwo}=\bigoplus_w BV^{i+\itwo,w}$
\[BV^{i+\itwo,w} :=\underset{\underset{N}{\longrightarrow}}{\lim}\underset{\underset{N'}{\longleftarrow}}{\lim} BV^{i+\codim \cZ(0,N')+Nn}(N,N')^w.\]
\item \label{I:CONmain23} There is a nondegenerate pairing
 \begin{equation}\label{E:fpar33}
 BV^{i+\itwo} \otimes BV^{k-i+\itwo} \to \C
 \end{equation}
 with $k$ the same as in Conjecture \ref{CON:main}.
\end{enumerate}
\end{conjecture}

In the following sections, I will prove 
Conjecture \ref{CON:main}(with omission of the Lie algebra action ) and Conjecture \ref{CON:main2} for the cone over the isotropic Grassmannian, which I will define presently.

\subsection{The cone of pure spinors $\cone$}
In this section, I will briefly remind the reader the definition of the space of pure spinors. Let $\vectorrep$ be a ten-dimensional complex vector space, equipped with a nondegenerate complex-linear inner product $(\cdot,\cdot)$. The central object of study of this paper is an affine cone $\cone$ over $\OGr^+(5,10)$- a connected component of the Grassmannian of five-dimensional isotropic subspaces. I refer the reader to \cite{Cartan},\cite{ChevalleySpinors} and \cite{CortiReid} for details on isotropic grassmannians.

The cone of pure spinors 
\begin{equation}\label{E:purecone}\cone\subset \sspinor \end{equation}
 over $\OGr^+(5,10)$ is an algebraic variety over $\C$ in 16-dimensional linear space $\sspinor$ with coordinates 
 \begin{equation}\label{E:spinorcoord}
 \begin{split}
& \lambda,p_i,w_{ij},\\
&1\leq i,j\leq 5, w_{ij}=-w_{ji}.
 \end{split}
 \end{equation} $\cone$ is defined by equations 
\begin{equation}\label{E:relexpl0}
\begin{split}
&\lambda p_{i}-\Pf_{i}(w)=0\quad {i}=1,\dots,5,\\
&pw=0.\\
\end{split}
\end{equation}
$\Pf_{i}(w), 1\leq i\leq 5$ are the principal Pfaffians of $w$.
\begin{definition}\label{E:sAdef}
$\sA$ is the graded algebra generated by (\ref{E:spinorcoord}) that are subject to relations (\ref{E:relexpl0}) 
\end{definition}

\subsection{Formulation of the results}\label{SS:mainresult}
In this section I collected main theorems which are proved in this paper. In the end of the section the reader will find a brief synopsis of the remaining parts. 

Denote 
\begin{equation}\label{E:zdef}\begin{split}
&\cZ^{poly}_{\C^{\times},\cone}\text{ by }\cZ_{\C^{\times}},\quad \cZ^{poly}_{\C,\cone}\text{ by }\cZ_{\C},\\
& \cZ^{poly}_{\C^{\times},\cone}(N,N')\text{ by }\cZ(N,N').
\end{split}\end{equation}

 \begin{theorem}\label{T:theoremmain}
 Conjecture \ref{CON:main} (with the omission of the affine Lie algebra action) hold true for $\tspace=\cone$.
 \end{theorem}
 \begin{proof}
  Item \ref{I:CONmain1} follows  from Proposition \ref{P:GorensteinCL}).
Item  \ref{I:CONmain2} follows from Definition \ref{E:semiinfdefsum}. 
 Item \ref{CON:main4} follows from Proposition \ref{P:shiftcong2}.
 Item \ref{I:mainrep} is item \ref{I:infinityweightpositivity} from Corollary \ref{C:limitchardim}, .
Item \ref{I:CONmain45} is item \ref{I:infinityweightfinitness}  from Corollary \ref{C:limitchardim} 
 . The constant $k$ is equal to $3$. All the structures  that appear are compatible with the $\Spin(10)$-action.
 \end{proof}
\begin{theorem}\label{T:boundExt}
Conjecture \ref{C:extweightbound} hold true for  for $\tspace=\cone$. 
 \end{theorem}
 \begin{proof}
 Follows from Proposition \ref{P:boundextweight}.
 \end{proof}

\begin{theorem}\label{T:fafteoremMain}
 Items \ref{I:CONmain22},\ref{I:CONmain23} of Conjecture \ref{CON:main2} hold true for $\cZ_{\C^{\times}}(N,N')$. The linear spaces are equipped with $\Spin(10)$-action.
 \end{theorem}
 \begin{proof}
See (\ref{E:semiinfinitecohdef1}), \ref{E:pairingferm} in Section \ref{S:bvdef}.
\end{proof}

One of the observation of this paper is that there is a nontrivial Thom class $\Th\in \Ext^8_P(\hA_N^{N'},\hA_{N-1}^{N'})$ for Gorenstein algebras $\hA_N^{N'}=\O[\cZ(N,N')]$
. $P$ is the graded polynomial algebra on the same set of generators as $\hA_N^{N'}$. 
Product with $\Th$ defines the maps
\begin{equation}\label{E:thini1}
\Th: H_{\cZ(0,N')}^{i}(\cZ(N,N'),\O)\to H_{\cZ(0,N')}^{i+8}(\cZ(N+1,N'),\O).
\end{equation}
Together with the pullbacks
\[H_{\cZ(0,N')}^{i}(\cZ(N,N'),\O)\to H_{\cZ(0,N'-1)}^{i}(\cZ(N,N'-1),\O)\]
these maps turn $\{H_{\cZ(0,N')}^{i}(\cZ(N,N'),\O)\}$ into a bidirect system (\ref{E:loccohini}).

\paragraph{Virtual characters}
 $\C^{\times}\times\bT\times\Spin(10)$ is a subgroup in the groups of symmetries of $H^{i+\itwo}$. Let $\widetilde{\bT}^5$ be the maximal torus of $\Spin(10)$. Denote 
 \begin{equation}\label{E:symmstries0}
\Aut:= \C^{\times}\times\bT\times\widetilde{\bT}^5.
\end{equation}
 $\widetilde{\bT}^5$ is the subgroup of diagonal matrices $\widetilde{\bT}^5\subset \widetilde{\GL}(5)\subset \Spin(10)$ ( " $\widetilde{ }$ " stands for two-sheeted cover).
 
The virtual character function $Z(t,q):=Z(t,q,1)$ 
 \[Z(t,q,z):=\sum_{i=0}^3(-1)^i\chi_{H^{i+\itwo}}(t,q,z),(t,q,z)\in \Aut,\] 
 is understood  as an element in $\Z((t))[[q]]\cap \Q(t)[[q]]$.
It satisfies a pair of equations
\begin{equation}\label{E:antifield}
Z(t,q)=-t^{-8}Z(1/t,q), \text{ field-antifield symmetry },
\end{equation}
\begin{equation}\label{E:stardual}
Z(t,q)=-q^2t^{-4}Z(q/t,q)\text{ $*$-conjugation symmetry}.
\end{equation}
Equations were written for the first time in \cite{AABN} (equations 4.40 and 4.41). 
I have an independent  verifications (\ref{E:stardual})  in Corollaries  \ref{C:equationZ} and \ref{E:ZprimeN}. Equation (\ref{E:hilbertini}
\ref{E:explicitform}) contain a formula for a function ${Z_{\fa}}_N^{N'}$ whose limit $N\to-\infty, N'\to \infty$ is $Z$. First few terms of $q$-expansion are given in (\ref{E:qexpansion}).

\paragraph{The space of states}
The space of states, as I define it in \ref{D:defloccoh}, might look rather mysterious. I give an elementary description of $H^{3+\itwo}$ in Proposition \ref{P:tensorproduct1}. 

The structure of the groups $H^{i+\itwo},i=1,2$ is obscure. Experimentations with {\it Macaulay2} show that $H_{\cZ(0,N')}^{i+\codim \cZ(0,N')}(\cZ(N,N'),\O),i=1,2$ are nontrivial for small values of $N,N'$.

{\it Missing state phenomenon} (see \cite{AABN} Section 3) in the space of states is an unresolved problem for the original formulation of this $\beta\gamma$ system.
 Theorems \ref{T:theoremmain}, \ref{T:fafteoremMain} and equation (\ref{E:qexpansion}) indicate that our formalism if free from this drawback.

\paragraph{The problem of denominators}
 The linear spaces $H^{i+\itwo}$ are modules over polynomials algebra in coefficients of the Laurent series $\lambda(z)=\sum\lambda^kz^k, w_{ij}(z)=\sum w_{ij}^kz^k$ and $p_i(z)=\sum p_{i}^kz^k$. In particular, it is a module over $P=\C[\lambda^0,w_{ij}^0,p_{i}^0]$ which I identify with $\C[\lambda,w_{ij},p_{i}]$. I can use these generators to define affine charts for $Spec P\backslash \{0\}=\sspinor \backslash \{0\}$: 
\begin{equation}
 \begin{split}
& U_0=Spec (\lambda)^{-1}P,\\
& U_i=Spec (p_i)^{-1}P,\\
& U_{ij}=Spec (w_{ij})^{-1}P.\\
\end{split}
\end{equation}
As usual, $(a_i)^{-1}R$ stands for localization of a ring $R$. By using the covering $\mathfrak{U} =\{U_0,U_i,U_{ij}\}$, I can define the \v{C}ech complex
$\Cech_e^{\bullet}(\mathfrak{U},H^{i+\itwo})$. The following problem arises in the theory of the $b$-ghost
(see Section 2.4 in \cite{BerkovitsNekrasovMultiloop}).
There is a map \[\sfr:H^{k+\itwo}\to \Cech^{0}(\mathfrak{U},H^{k+\itwo})\]
which is the diagonal localizations
\[H^{k+\itwo}\to (\lambda)^{-1}H^{k+\itwo}\oplus\bigoplus (p_i)^{-1}H^{k+\itwo}\oplus\bigoplus(w_{ij})^{-1}H^{k+\itwo}.\]
Fix $k$. The problem of denominators (of cochains in \v{C}ech complex) roughly is: Is it true that the map $\sfr$ defines an isomorphism of $H^{k+\itwo}$ with $\rH^0(H^{k+\itwo})$?
And also is it true that $\rH^m(H^{k+\itwo})=0$, for $m=1,2,3,4$? Here in the formulas, $\rH^m$ stands for \v{C}ech cohomology of the sheaf associated with the $P$-module $H^{k+\itwo}$. It is believed that the answer on this question is affirmative. 

I prove a weaker version of this conjecture. The groups $\rH^m(H^{k+\itwo})$ form the second page of some spectral sequence $\rH^m(H^{k+\itwo}) \Rightarrow H^{k+m}C $ (Proposition \ref{E:cechspectral}). 
It turns out that the natural map $H^{k+\itwo}\to H^{k}C$ is an isomorphism for $k=0,\dots,7$ (Proposition \ref{E:isomorphism1}) and $H^{k}C=\{0\},k=4,\dots,7$, which supports the conjecture. For the complete proof of the conjecture, it suffice to establish degeneration the spectral sequence in the second page. 

\paragraph{The structure of the paper}
Results of this work rely heavily on the local structure of $\cZ(N,N')$. The relevant analysis is done in Section \ref{S:localstructure}, which is independent from the rest of the paper. There I prove the theorem (Proposition \ref{P:GorensteinCL}) that $\cZ(N,N')$ is Gorenstein. Note that Gorenstein property is probably the only thing that is truly responsible for all the structures appearing in this paper. 

In Section \ref{S:CM}, I study what I call the Thom class $\Th$ (\ref{E:classt}). This class will enable me to define "wrong direction" maps.
My definition of $H^{i+\itwo}$ as a  limit of a bidirect system heavily relies on $\Th$. 
The section also contains some simple $\Tor$ functor construction for computation of local cohomology. 

I devote Section \ref{S:spaceofstates} to systematic study of groups $H_{\cZ(0,N')}^{i}(\cZ(N,N'),\O)$. This includes verification of the nondegeneracy of the pairing (\ref{E:pairinggen})
 for $H_{\cZ(0,N')}^{i}(\cZ(N,N'),\O)$ and the proof of the functional equation (\ref{E:stardual}). I introduce complexes $\moduleT_{\bullet}(\fa)$, $\moduleS_{\bullet}(\fa)$ and $\Frg^{\fa}_{\bullet}[\hdelta,\hdelta']$
 for computations of $H_{\cZ(0,N')}^{i}(\cZ(N,N'),\O)$. 

 I study $N$ and $N'$ dependencies in $H_{\cZ(0,N')}^{i}(\cZ(N,N'),\O)$, which  will be used for gluing these groups into $H^{i+\itwo}$. 

The limiting  procedure used to define $H^{i+\itwo}$ is rather delicate. Easy examples shows that nondegeneracy of the pairing between  $H^i_{\fa}[\hdelta,\hdelta']$ $H^j_{\fb}[\hdelta,\hdelta']$ for all $\hdelta,\hdelta'$ doesn't automatically imply it , as I mistakenly thought in the first draft of this paper,   for $H^{i+\itwo}$.  Section \ref{S:lowerbound} contains theorems which  enable me to overcome this difficulty.

In Section \ref{S:limiting}, I finally introduce $H^{i+\itwo}$ and prove Theorem \ref{T:theoremmain}.
 The definition of $H^{i+\itwo}$ involves several limiting procedures. I verify that the result doesn't depend on the order of limits. I also give an elementary construction of $H^{3+\itwo}$.

 Section \ref{S:faf} is devoted to verification of Theorem \ref{T:fafteoremMain}.

 I establish the denominator conjecture, discussed above, in Section \ref{S:localization}.
 
 Section \ref{S:formalmaps} contains some speculations about the relation of this work to the earlier works on $\beta\gamma$-systems \cite{GMSch},\cite{GMSch},\cite{Nekrasovbetagamma}.

 The paper contains several appendices. For the reader's convenience, Appendix \ref{S:local} contains basic facts about local cohomology, that are frequently used in the main text. Appendix \ref{A:hibi} contains short review of the theory of Hibi algebras that is essential for Section \ref{S:localstructure}. Some standard $D$-modules, that are used throughout the paper are introduced in Appendix \ref{S:Dmodulenoteconstr}. I moved a number of Lemmas whose proofs are not very illuminating into Appendix \ref{S:techlemmas}.
\paragraph{What was left out of scope}
To keep the manuscript within reasonable size limits, I put off discussion of some topics with a hope to return to them in future publications. Here are some of them.

The first topic is the structure of the vertex operator algebra on $H^{i+\itwo}$ and the action of the affine $\hat\so_{10}$ and Virasoro algebra on it. The definition of the action is delicate: $\hat\so_{10}$ doesn't act on $\O[\cZ(N,N')]$ nor on $H_{\cZ(0,N')}^{i}(\cZ(N,N'),\O)$. It emerges on $H^{i+\itwo}$ in the limit.

The second topic is the Cousin complex. 
 In representation theory Cousin complex appears under the name of the BGG resolution. Such BGG-like construction would certainly be relevant in my setup, since $\cZ_{\space}(N,N')$ is equipped with a suitable filtration. Even though the Kempf's algorithm \cite{Kempf} that defines the Cousin complex is compatible with my limiting procedure, the outcome of the algorithm resists at present to a simple characterization. 

The third topic is the extension of the results to more general spaces. The Gorenstein property of $\O[\cZ_{\tspace}(N,N')], \tspace=\cone$ is fundamental in all of my constructions. For cones over grassmannians, this property was established in \cite{SottileSturmfels}. For cones over isotropic symplectic grassmannians, it was proved in \cite{Ruffo}. For cones over the spaces of full flags in simply-laced case, it was verified in \cite{BF}. 
In all these cases, the class $\Th$ (\ref{E:thini1}) is defined, which enables us to carry out many of the constructions from this paper. The question is how far one can push the theory. An issue arising on this way is discussed in Remark \ref{R:mainremark}. It indicates that an upgrade of my method to a more systematic technique is desirable.

Fix a smooth Fano variety $X$ with $ind(X)\geq 2$. Let $L^{-\otimes ind(X)}\cong K_{X}$, where $K_{X}$ is the canonical line bundle over $X$.  $L$ can be used to defines an  affine cone  $\tspace$ over $X$.  I wonder if the spaces $\cZ_{\tspace}(N,N')$ are Gorenstein?

\paragraph{General conventions}
 
 For an algebraic variety $X$ over $\C$ $\O[X]$ will stand for the algebra of regular algebraic functions on $X$. 
 
 The tensor product $\underset{R}\otimes$ stands for the tensor product of $R$ modules and $\otimes$ for product of $\C$-vector spaces.
\paragraph{Acknowledgment}  
The author would like thank N.Berkovits, M. Finkelberg and A.S. Schwarz  for their  interest and  helpful conversations about this work.
 
\section{The local structure of $\cZ(N,N')$}\label{S:localstructure}
The space $\cZ(N,N')$ is singular. The goal of this section is to prove that its singularities are mild. This will be done by showing that $\O[\cZ(N,N')]$ Gorenstein, which will prove item \ref{I:CONmain1} of Conjecture \ref{CON:main}.
I will be establish this in Section \ref{S:singularities} by degenerating $\cZ(N,N')$ into a toric variety $\cH(N,N')$ of a special kind. Here $\cH$-stands for Hibi. He was the first who systematically studied algebras of homogeneous functions on such spaces. A Hibi algebra (see Appendix \ref{A:hibi} for a brief review) has combinatorial description in terms of a lattice (in the sense of Birkhoff). Gorenstein property of the algebra can be extracted directly from the underlying lattice. I do this for $\O[\cH(N,N')]$ in Section \ref{S:Gorenstein}. The lattice corresponding to $\O[\cH(N,N')]$ is described in Section \ref{S:alghat}.
In this section I continue to use notations from the Section \ref{SS:mainresult}.
\setcounter{subsection}{0}

\subsection{Basic definitions}\label{S:alghat}
\paragraph{Spinors and the cone $\cone$}It will be useful to have a more invariant description of the cone $\cone$ (\ref{E:purecone}).
The set 
\[\setE=\{(0),(ij),(k)|1\leq i<j\leq 5,1\leq k \leq 5\}\]
labels a weight basis for an irreducible spinor representation $\sspinor$ of the complex group $\Spin(10)$ with respect to the maximal torus $\bT^5$
Let $\vectorrep$ the fundamental representations of the complex $\SO(10)$ with a $\bT^5$-weight basis labelled by the set of symbols 
\[\rG=\{1,\dots,5, 1^*,\dots,5^*\}.\]
See \cite{MovStr} for details.
The direct sum of $\sspinor$ with its dual $\sdspinor$ comprise an irreducible representation of the Clifford algebra $Cl=Cl(\vectorrep,(\cdot,\cdot))$:
$Cl\otimes(\sspinor+\sdspinor)\rightarrow \sspinor+\sdspinor$.
Restriction to $\vectorrep\subset Cl$ defines maps $\vectorrep\otimes \sspinor\rightarrow \sdspinor$, $\vectorrep\otimes \sdspinor\rightarrow \sspinor$.
Their adjoints are the 
$\Gamma$-maps 
\[\begin{split}
&\Sym^2(\sspinor)\rightarrow \vectorrep,\quad \Sym^2(\sdspinor)\rightarrow \vectorrep.
\end{split}\]
In $\bT^5$-weight bases $\{\theta_{\rbeta},\rbeta\in \setE\}$ for $\sspinor$ and $\{v_{\s}, \s\in \rG\}$ for $\vectorrep$,
the $\Gamma$-matrices $\Gamma_{\ralpha\rbeta}^{\s}$ are defined by the formula $\Gamma(\theta_{\ralpha},\theta_{\rbeta})=\Gamma_{\ralpha\rbeta}^{\s}v_{\s}.$
 Summation over repeated indices will always be assumed. 
 Let us introduce a more uniform notations for coordinates on the space of spinors: $\theta=\lambda^{(0)}\theta_{0}+\lambda^{(ij)}\theta_{ij}+\lambda^{(i)}\theta_{i}\in \sspinor$.
Equations
\begin{equation}\label{E:pure}
\begin{split}
&\Gamma^{\s}_{\ralpha\rbeta}\lambda^{\ralpha}\lambda^{\rbeta}=0 ,\s \in \rG, \lambda^{\rbeta}\theta_{\rbeta}\in \sspinor\\
\end{split}
\end{equation}
after the identification 
 \[\lambda=\lambda^{(0)}, w_{ij}=\lambda^{(ij)}, p_i=\lambda^{(i)}\]
coincide with (\ref{E:relexpl0}).
More invariantly than in (\ref{E:sAdef}), I define $\sA$ as 
\[\sA=\C[\lambda^{\rbeta}]/(\Gamma^{\s}_{\ralpha\rbeta}\lambda^{\ralpha}\lambda^{\rbeta}).\]

In this description, the $\C^{\times}\times \Spin(10)$-symmetries of the cone $\cone$ becomes manifest. The factor $\C^{\times}$ stands for dilations of the cone.

\paragraph{Defining equations for $\cZ(N,N')$}The scheme $\cZ(N,N')$ (\ref{E:zdef}) 
 is an affine algebraic manifold. Occasionally, following \cite{SottileSturmfels}, I will call it {\it Quantum Isotropic Grassmannian}.
The algebra of polynomial functions $\hA_{N}^{N'}=\O[\cZ(N,N')]$ (\cite{MovStr}) is generated by the variables \[\{\lambda^{\rbeta^l}| N\leq l\leq N', \rbeta\in \setE\}\] 
where $\rbeta^l$ is a multi-index. Defining relations of $\hA_N^{N'}$ are 
\begin{equation}\label{E:equationsaf0}
\Gamma^{\bm{s}^k}=\sum_{l+l'=k}\Gamma^{\bm{s}}_{\ralpha\rbeta}\lambda^{\ralpha^l}\lambda^{\rbeta^{l'}}, \quad N \leq l,l' \leq N', 2N \leq k\leq 2N' \quad \bm{s}\in \rG,
\end{equation}
\begin{equation}\label{E:relgenfunction}
\sum_{2N\leq l\leq 2N'} \Gamma^{\bm{s}^k}z^k=\Gamma^{\bm{s}}_{\ralpha\rbeta}\lambda^{\ralpha}(z)\lambda^{\rbeta}(z), \quad
\lambda^{\rbeta}(z)=\sum_{N\leq l\leq N'}\lambda^{\rbeta^l}z^l.
\end{equation}
Coordinates $\lambda^{\rbeta^k}$ (\ref{E:relgenfunction}) are the decorated Fourier coefficients $x^{i^k}$ (\ref{E:coordFourier})
\begin{remark}\label{R:symmetries}
From this definition it is obvious that $\C^{\times}\times\Spin(10)$ is a groups of symmetries of the algebra $\hA_N^{N'}$.
\end{remark}
\paragraph{The fundamental Hasse diagram}
 The following diagram $\sethE$ is fundamental for Quantum Isotropic Grassmannians:
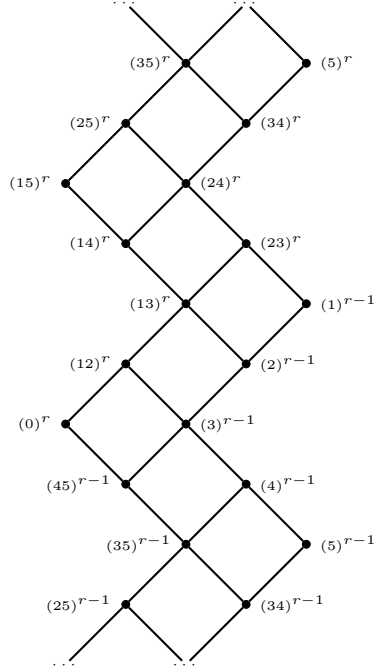
\begin{figure}[h]
\centering
\begin{tikzpicture}[ scale=0.8]
	\begin{pgfonlayer}{nodelayer}
		\node [circle, fill=black, draw, style=none,label=180:\tiny{$(25)^{r-1}$}] (0) at (-6, -5) {.}; 
		\node [circle, fill=black, draw, style=none,label=180:\tiny{$(35)^{r-1}$}] (1) at (-5, -4) {.};
 \node [circle, fill=black, draw, style=none,label=0:\tiny{$(34)^{r-1}$}] (2) at (-4, -5) {.};
 \node [circle, fill=black, draw, style=none,label=0:\tiny{$(5)^{r-1}$}] (3) at (-3, -4) {.};
 \node [circle, fill=black, draw, style=none,label=0:\tiny{$(4)^{r-1}$}] (4) at (-4, -3) {.};
 \node [circle, fill=black, draw, style=none,label=180:\tiny{$(45)^{r-1}$}] (5) at (-6, -3) {.}; 
 \node [circle, fill=black, draw, style=none,label=0:\tiny{$(3)^{r-1}$}] (6) at (-5, -2) {.};
 \node [circle, fill=black, draw, style=none,label=180:\tiny{$(12)^{r}$}] (7) at (-6, -1) {.};
 \node [circle, fill=black, draw, style=none,label=0:\tiny{$(2)^{r-1}$}] (8) at (-4, -1) {.};
 \node [circle, fill=black, draw, style=none,label=0:\tiny{$(1)^{r-1}$}] (9) at (-3, 0) {.};
 \node [circle, fill=black, draw, style=none,label=180:\tiny{$(13)^{r}$}] (10) at (-5, 0) {.};
 \node [circle, fill=black, draw, style=none,label=180:\tiny{$(0)^{r}$}] (11) at (-7, -2) {.};
 \node [circle, fill=black, draw, style=none,label=0:\tiny{$(23)^{r}$}] (12) at (-4, 1) {.};
 \node [circle, fill=black, draw, style=none,label=180:\tiny{$(14)^{r}$}] (13) at (-6, 1) {.};
 \node [circle, fill=black, draw, style=none,label=0:\tiny{$(24)^{r}$}] (14) at (-5, 2) {.};
 \node [circle, fill=black, draw, style=none,label=180:\tiny{$(15)^{r}$}] (15) at (-7, 2) {.};
 \node [circle, fill=black, draw, style=none,label=180:\tiny{$(25)^{r}$}] (16) at (-6, 3) {.};
 \node [circle, fill=black, draw, style=none,label=0:\tiny{$(34)^{r}$}] (17) at (-4, 3) {.};
 \node [circle, fill=black, draw, style=none,label=180:\tiny{$(35)^{r}$}] (18) at (-5, 4) {.};
 \node [circle, fill=black, draw, style=none,label=0:\tiny{$(5)^{r}$}] (19) at (-3, 4) {.};
		\node [style=none] (20) at (-7, -6) {\tiny $\cdots$};
		\node [style=none] (21) at (-5, -6) {\tiny $\cdots$};
		\node [style=none] (22) at (-6, 5) {\tiny $\cdots$};
		\node [style=none] (23) at (-4, 5) {\tiny $\cdots$};
	\end{pgfonlayer}
	\begin{pgfonlayer}{edgelayer}
		\draw [thick] (0) to (1);
		\draw [thick] (2) to (1);
		\draw [thick] (2) to (3);
		\draw [thick] (3) to (4);
		\draw [thick] (1) to (5);
		\draw [thick] (5) to (6);
		\draw [thick] (4) to (6);
		\draw [thick] (6) to (7);
		\draw [thick] (6) to (8);
		\draw [thick] (1) to (4);
		\draw [thick] (8) to (9);
		\draw [thick] (7) to (10);
		\draw [thick] (8) to (10);
		\draw [thick] (5) to (11);
		\draw [thick] (11) to (7);
		\draw [thick] (10) to (13);
		\draw [thick] (13) to (15);
		\draw [thick] (9) to (12);
		\draw [thick] (10) to (12);
		\draw [thick] (12) to (14);
		\draw [thick] (13) to (14);
		\draw [thick] (15) to (16);
		\draw [thick] (14) to (17);
		\draw [thick] (16) to (18);
		\draw [thick] (16) to (14);
		\draw [thick] (18) to (17);
		\draw [thick] (17) to (19);
		\draw [thick] (20) to (0);
		\draw [thick] (21) to (2);
		\draw [thick] (0) to (21);
		\draw [thick] (18) to (22);
		\draw [thick] (18) to (23);
		\draw [thick] (19) to (23);
	\end{pgfonlayer}
\end{tikzpicture}

\begin{equation}
\label{P:kxccvgw13}
\end{equation}
\end{figure}
\newpage 
Its vertices are decorated elements of the set $\setE$. I embed $\setE\to\sethE,\rbeta\to \rbeta^0$ as a full sub-diagram.
I will denote elements $\rbeta^r\in \hat \setE$ by $\halpha$. Then
\begin{equation}\label{E:udef}
u(\halpha):=r \text{ for }\halpha=\rbeta^r.
\end{equation}
It will be convenient to work with a slightly more general family of algebras than $\hA_N^{N'}$.
The Hasse graph $\sethE$ defines a {\it poset} (partially ordered set) structure on the set of its vertices (see Appendix \ref{A:hibi} for details). 
Suppose $\setL\subset \sethE$. I set
\[\begin{split}
&\C[\setL]=\mathbb{C}[\lambda^{\halpha}],\quad \Lambda[\setL]=\Lambda[\xi^{\halpha}],\quad \halpha\in \setL.
\end{split}\]
Fix $\setK\subset \setL$.
\[\fri(\setK)=\left\{\sum_{\halpha\in \setK} \lambda^{\halpha}\C[\setL]\right\}\subset \C[\setL]\]
is an ideal of $\C[\setL]$.

\paragraph{Algebraic variety $\cZ(\hdelta,\hdelta')$}

A {\it segment} or an {\it interval} $[\hdelta,\hdelta']$ is a subset $\{\halpha\in \sethE|\hdelta\leq \halpha\leq \hdelta'\}$. The set of generators of $\hA_N^{N'}$ is $\{\lambda^{\halpha}|\halpha\in [{(0)^N},(1)^{N'}]\}$.

Occasionally I will denote $\sethE$ by $[-\infty,\infty]$.

\begin{definition}\label{D:algAdelta}
Fix $\setX\subset [{(0)^N},(1)^{N'}]$ 
 for some $N,N'$. The algebra 
 $\hA[\setX]$
 is a quotient of $\C[{(0)^N},(1)^{N'}]$.
 The ideal of relations $\fr[\setX]$ is generated by $\Gamma^{\hat{\s}}$ (\ref{E:equationsaf0}) and $\{\lambda^{\halpha}|\halpha \notin \setX\}$.
 Assignment $\setX\Rightarrow \hA[\setX]$ is a contravariant functor from the category of finite subsets of $\sethE$ (morphisms are inclusions) to the category of algebras. When $\setX$ is the interval $[\hdelta,\hdelta']$, the algebra $A[\setX]$ will be denoted by ${\hA[\hdelta,\hdelta']}$.
By definition $A_N^{N'}:=\hA[{(0)^N},(1)^{N'}]$. The algebra ${\hA[\hdelta,\hdelta']}$ was denoted by $A_{\hdelta}^{\hdelta'}$ in \cite{MovStr}.
 Let $\setK$ be a subset of $\setX$. By abuse of notations, I will denote the ideal $\left\{\sum_{\halpha\in \setK} \lambda^{\halpha}\hA[\setX]\right\}\subset \hA[\setX]$ by $\fri[\setK]$.
\end{definition}
I define 
\[\cZ[\hdelta,\hdelta']:=Spec\, {\hA[\hdelta,\hdelta']}\]
\begin{equation}\label{E:spinordef}
\spinor[\hdelta,\hdelta']:= \Span{\theta_{\rbeta^k}=\theta_{\rbeta}z^k|\rbeta^k\in [\hdelta,\hdelta']}
\end{equation}
The shift operator $ \shift$ 
\begin{equation}\label{E:shift} \shift (\rbeta^r)=\rbeta^{r+1}\quad \rbeta^r\in \hat \setE
\end{equation}
is an automorphism of the graded lattice $\sethE$ (see Definition \ref{D:gradedlattice} and Example \ref{E:Egrposet}). 
Note that the map of the lattice (\ref{E:shift}) define homomorphism:
\begin{equation}\label{E:tisomorphism}
 \shift(\lambda^{\rbeta^n})=\lambda^{\rbeta^{n+1}}\text{ induces an isomorphism } \shift:{\hA[\hdelta,\hdelta']}\to \hA[ \shift(\hdelta), \shift(\hdelta')].
\end{equation}
The algebra $\hA[-\infty,\infty]:=\underset{\underset{\hdelta,\hdelta'}{\longleftarrow}}\lim {\hA[\hdelta,\hdelta']}$ admits an involution $\reflection$. This involutions defines an anti-involution of the poset 
\begin{equation}\label{E:involution}
\reflection:\sethE\to \sethE.
\end{equation}
 Geometrically $\reflection$ is a central symmetry about a center $O$ which is a bisector of the segment $|(25)^0,(34)^0|$ (\ref{P:kxccvgw13}) The map $\reflection$ is induced by isomorphisms 
\begin{equation}\label{E:mapsintr}
\reflection:{\hA[\hdelta,\hdelta']}\to \hA[\reflection(\hdelta'),\reflection(\hdelta)].
\end{equation} (see Appendix A in \cite{MovStr}).

\begin{convention}\label{C:mainconvention}{\rm
We will encounter many constructions that depend on the interval from $\sethE$, e.g $\hA[\hdelta,\hdelta']$, $\spinor[\hdelta,\hdelta']$, $\C[\hdelta,\hdelta']$. To avoid clutter in notations, I will often drop dependence on the interval in the case when the interval can be extracted from the context. In this case $\C[\hdelta,\hdelta']$ will be denoted by $P$
}\end{convention}

\paragraph{Algebras with the straightening law}
\begin{definition}
{\rm
Let $A,P$ be the algebras based on $[\hdelta,\hdelta']$.
A monomial $\lambda^{\halpha_1}\cdots\lambda^{\halpha_n}$ in $P$ is said to be standard if $\hdelta' \leq \halpha_1\leq \cdots\leq \halpha_n\leq\hdelta $ with respect to the partial order in $[\hdelta,\hdelta']\subset \sethE$. An element $x\in {\hA}$ is said to be a standard monomial if one of its pre-images in $P$ is standard.
}
\end{definition}

\begin{proposition}\label{E:degreenull}(See \cite{MovStr})
\renewcommand{\theenumi}{\roman{enumi}}
Fix $[\hdelta,\hdelta']$ on which algebra $A$ is based.
\begin{enumerate}
\item Each of the defining equations of ${\hA}$ contains a unique monomial $\lambda^{\halpha}\lambda^{\hbeta}$, which is called a {\it clutter}, such that $\halpha$ and $\hbeta$ are not comparable.
\item More precisely equations have the form 
\begin{equation}\label{E:reluniv}
\lambda^{\halpha}\lambda^{\hbeta}\pm\lambda^{\halpha\vee\hbeta}\lambda^{\halpha\wedge\hbeta}-\sum_{\hgamma< \halpha\wedge\hbeta, \hdelta > \halpha\vee\hbeta} \pm\lambda^{\hgamma}\lambda^{\hdelta}=0
\end{equation}
where $\lambda^{\halpha}\lambda^{\hbeta}$ is a clutter.
\item Standard monomials define a basis for ${\hA}$. 
\item $\hA$ is a Cohen-Macaulay algebra.
\end{enumerate}
\end{proposition}
\begin{definition}{\rm
An algebra with a set of generators labeled by the elements of a lattice $F$ with relations (\ref{E:reluniv}) is said to be an algebra with {\it straightening laws
} (see e.g. \cite{EisenbudStr}) if standard monomials form a basis. Thus ${\hA[\hdelta,\hdelta']}$ and in particular $\hA_{N}^{N'}$ are algebras with straightening laws
.
}
\end{definition}
I will occasionally use algebras based on semi-intervals:
\begin{equation}\label{E:semiinterval}
\hA(\hdelta,\hdelta']:={\hA[\hdelta,\hdelta']}/(\lambda^{\hdelta}),\quad \hA[\hdelta,\hdelta'):={\hA[\hdelta,\hdelta']}/(\lambda^{\hdelta'}).
\end{equation}
\begin{proposition}\label{R:stbasis}
The 
 standard monomials defined in terms of the order on semi-interval form a basis for $\hA$ based on $(\hdelta,\hdelta']$ or $[\hdelta,\hdelta')$.
\end{proposition}
\begin{proof}
Let us work out the case of $(\hdelta,\hdelta']$
Elements $\hdelta,\hdelta'$ are comparable with all the elements of $[\hdelta,\hdelta']$. Thus the Gr\"{o}bner basis for the defining ideal $\fc=\fr[\hdelta,\hdelta']+(\lambda^{\hdelta})$ such that $ P/\fc\cong \hA$ besides (\ref{E:reluniv}) contains only $\lambda^{\hdelta}$ (see more on application of Gr\"{o}bner basis technique to in \cite{MovStr}). I conclude from this that the 
 standard monomials defined in terms of the order on $(\hdelta,\hdelta']$ form a basis for $\hA$. 
 
 The arguments for $[\hdelta,\hdelta')$ are similar.
\end{proof}

\subsection{The toric degeneration of $\cZ[\hdelta,\hdelta']$}\label{S:singularities}
The algebra of functions on the toric degeneration announced in the title of this section will be constructed combinatorially as a contraction of the algebra $\hA$ based on $[\hdelta,\hdelta']$. More precisely, I will define an algebra $\hA_{h}$ which is isomorphic to $\hA\otimes \C[[h]]$ as $\C[[h]]$-module. I will define \[Hibi[\hdelta,\hdelta']:=\hA_{h}[\hdelta,\hdelta']/(h).\] The formal family $\hA_{h}$ has the property that 
\[\hA_{h}\underset{\C[[h]]}{\otimes} \C((h))\cong \hA\otimes \C((h)).\]

In this section, I will use some notation of Appendix \ref{A:hibi}.

For construction of $\hA_h$, I use some ideas from \cite{SottileSturmfels}. They define a similar deformation in the context of ordinary Grassmannians $\mathrm{Gr}(n,m+n)$. As in my case, their algebra is controlled by a Hasse diagram. Sottile and Sturmfels's idea was to modify defining relations of the algebra by using a special function associated with the diagram.

\paragraph{The function $\functionl$}To define $\functionl:\sethE\rightarrow \mathbb{Z}_{>0}$, I embed the graph $\sethE$ into $\Z^2\subset \R^2$.
For this, I use the straightforward identification of the physical surface(e.g. a computer screen or a sheet of a paper ) on which the diagram \ref{P:kxccvgw13} is drawn with the plane $\R^2$. To be more precise,
I can characterize this embedding by a map 
\begin{equation}\label{E:fdef} 
\functionf:\sethE\rightarrow \R^2.
\end{equation} To describe it, I fix a sublattice $\setS\subset \Z^2$ generated by vectors $u=(2,2),v=(-2,2)$. I characterize $\functionf$ by the set of conditions:
\begin{enumerate}
\item $\functionf(\halpha)
\in (0,-1)+\setS$.
\item \label{I:proot} If $\halpha \gtrdot \hbeta$, then $\functionf(\halpha)-\functionf(\hbeta)\in \{u,v\}$;
\item $\functionf((34)^{0})=(2,0), \functionf((35)^{0})=(0,2),\functionf((25)^{0})=(-2,0),\functionf((24)^{0})=(0,-2)$.
\end{enumerate}
\[\functionl(\halpha)=||\functionf(\halpha)||^2\text{ where }||(x,y)||^2=x^2+y^2.\]

For finite subset $\setN\subset \sethE$ define
\begin{equation}\label{E:coredef}
\rCore(\setN):=\{\halpha\in \setN|\exists \hbeta\neq \halpha, \rho(\halpha)=\rho(\hbeta)\}.
\end{equation}
The function $\rho=\rho_{\sethE}$ is defined in (\ref{E:LdefE})
{\it Capacity} of $\setN$ is 
\begin{equation}\label{E:capdif}
\rCap \setN:=|\rCore(\setN)|.
\end{equation}
The next proposition underlies many results of this paper. 
 \begin{proposition}\label{P:functionL}
\begin{enumerate}
\item $\forall \halpha\in \sethE$ $\functionl(\halpha)>0$ 
\item 
Pick any rectangle in the diagram (\ref{P:kxccvgw13}):

\begin{figure}[h]
\centering
\begin{tikzpicture}[ scale=1]
	\begin{pgfonlayer}{nodelayer}
		\node [circle, fill=black, draw, style=none,label=90:\tiny{$\hgamma$}] (0) at (0, 1) {.};
		\node [circle, fill=black, draw, style=none,label=180:\tiny{$\halpha$}] (1) at (-1, 0) {.};
		\node [circle, fill=black, draw, style=none,label=-90:\tiny{$\hdelta$}] (2) at (0, -1) {.};
		\node [circle, fill=black, draw, style=none,label=0:\tiny{$\hbeta$}] (3) at (1, 0) {.};
		\node [style=none] (4) at (-0.5, 0.5) {$\cdots$};
		\node [style=none] (5) at (0.5, 0.5) {$\cdots$};
		\node [style=none] (6) at (0.5, -0.5) {$\cdots$};
		\node [style=none] (7) at (-0.5, -0.5) {$\cdots$};
	\end{pgfonlayer}
	\begin{pgfonlayer}{edgelayer}
		\draw (0) to (4);
		\draw (4) to (1);
		\draw (1) to (7);
		\draw (7) to (2);
		\draw (6) to (2);
		\draw (0) to (5);
		\draw (5) to (3);
		\draw (3) to (6);
	\end{pgfonlayer}
\end{tikzpicture}
\begin{equation}
\label{E:diagram}
\end{equation}
\end{figure}
Then 
\begin{equation}\label{E:Leq}
\functionl(\halpha)+\functionl(\hbeta)=\functionl(\hgamma)+\functionl(\hdelta).
\end{equation}
\item Pick $\halpha$,$\hbeta$ as in (\ref{E:diagram}). Suppose $\exists t\in \Z^2, \hgamma', \hdelta'\in \sethE$ such that 
\[\begin{split}
\functionf(\hgamma')=(\functionf(\halpha)+\functionf(\hbeta))/2+t,\functionf(\hdelta')=(\functionf(\halpha)+\functionf(\hbeta))/2-t.
\end{split}\] 
If 
\begin{equation}\label{E:inequlity}
\functionl(\halpha)+\functionl(\hbeta)\geq \functionl(\hgamma')+\functionl(\hdelta'),
\end{equation}
then 
\begin{equation}\label{E:ininterval}
\hgamma',\hdelta'\in [\halpha\wedge \hbeta,\halpha\vee \hbeta].
\end{equation}
Moreover, equality holds iff $\hgamma',\hdelta'$ coincides with opposite corners of the rectangle $[\halpha\wedge \hbeta,\halpha\vee \hbeta]$.
\end{enumerate}
\end{proposition}
\begin{proof}
\begin{enumerate}
\item $\functionl(\halpha)>0$ because $\functionf^{-1}(0)=\emptyset$.
\item Set $g=\functionf(\hdelta)$ $a=\functionf(\halpha)-\functionf(\hdelta),b=\functionf(\hbeta)-\functionf(\hdelta)$. By construction $\functionf(\hgamma)-\functionf(\hdelta)=a+b$. Equation (\ref{E:Leq}) is equivalent to
\[||g||^2+||g+a+b||^2-||g+a||^2-||g+b||^2=2(a,b)=0.\]

The equality holds true because vectors $a$ and $b$ are proportional to $u$ and $v$ (\ref{I:proot}) which are orthogonal by construction.
\item Suppose $(\halpha,\hbeta)$ is a clutter. The quantity $||(\functionf(\halpha)-\functionf(\hbeta))/2||^2=Q(c)$ depends only on the capacity $c=\rCap[\halpha\wedge \hbeta,\halpha\vee \hbeta]$ (\ref{E:capdif}). Under my assumptions capacity ranges from one to three and \[Q(1)=4,Q(2)=10,Q(3)=20.\] The center of symmetry of $\functionf([\halpha\wedge \hbeta,\halpha\vee \hbeta])$ is located at $(\functionf(\halpha)+\functionf(\hbeta))/2$. The square of the radius $R^2(c)$ of the circumscribed circle $\rCircle[\halpha\wedge \hbeta,\halpha\vee \hbeta]$ about $\functionf([\halpha\wedge \hbeta,\halpha\vee \hbeta])$ as a function of capacity $c$ is equal to \[R(1)=4,R(2)=10,R(3)=20.\] It implies that that $Q(c)=R^2(c)$.
I set $\functionf(\halpha)=m$, $\functionf(\hbeta)=n$. From the identity
\[\begin{split}&||(m+n)/2+t||^2+||(m+n)/2-t||^2-||m||^2-||n||^2=\\
&2(||t||^2-||(m-n)/2||^2)
\end{split}\]
I infer that if the inequality (\ref{E:inequlity}) holds, then the vector
 $(m+n)/2+t$ lies inside $\rCircle[\halpha\wedge \hbeta,\halpha\vee \hbeta]$ or on its boundary. Then (\ref{E:ininterval}) holds. When the inequality becomes equality, the vector $(m+n)/2+t$ lies on the circle $\rCircle[\halpha\wedge \hbeta,\halpha\vee \hbeta]$ and $\hgamma',\hdelta'$ the pair of opposite corners of $[\halpha\wedge \hbeta,\halpha\vee \hbeta]$.

\end{enumerate}
\end{proof}

\begin{remark}
Algebra ${\hA[\hdelta,\hdelta']}$ has a symmetry group containing $\Aut$ (\ref{E:symmstries0})
The action of $g=(t,q,z)=(t,q,z_1,\dots,z_5)\in \Aut
$ is 
\begin{equation}\label{E:weights}
\begin{split}
&\Pi(g)\lambda^{\halpha}=\hat{v}_{\halpha}(t,q,z)\lambda^{\halpha} =tq^{u(\halpha)}v_{\rbeta}(z)\lambda^{\halpha},\\
&\text{{\rm or in more details: }}\\
&\Pi(g)\lambda^{(0)^l} =tq^l{\det}^{-\frac{1}{2}}(z)\lambda^{(0)^l},\\
&\Pi(g)\lambda^{(ij)^l}=tq^l{\det}^{-\frac{1}{2}}(z)z_iz_j\lambda^{(ij)^l},\\
&\Pi(g)\lambda^{(k)^l}=tq^l{\det}^{\frac{1}{2}}(z)z^{-1}_k\lambda^{(k)^l}.
\end{split}
\end{equation}
See (\ref{E:udef}) for $u(\halpha)$.
\end{remark}

The following proposition gives an additional insight into the structure of the defining relations of ${\hA}$.
\begin{proposition}
Let $\hA$ be based on $[\hdelta,\hdelta']$. For any clutter $\halpha,\hbeta\in [\hdelta,\hdelta']\subset \sethE$, there is a unique defining relation (\ref{E:reluniv}) of ${\hA}$ that can be written in the form 
\[\lambda^{{\halpha}}\lambda^{{\hbeta}}\pm\lambda^{{\halpha}\vee{\hbeta}}\lambda^{{\halpha}\wedge{\hbeta}}=
\sum_{\functionf({\hgamma})=(\functionf({\halpha})+\functionf({\hbeta})/2+t, \functionf({\hdelta}) = \functionf({\halpha})+\functionf({\hbeta})-t} \pm\lambda^{{\hgamma}}\lambda^{{\hdelta}}\]
with $\functionf$ as in (\ref{E:fdef}). The summation is taken over $t\in \Z^2$ such that the corresponding $\hgamma',\hdelta'$ satisfy 
\begin{equation}\label{E:tinequality}
\functionl(\halpha)+\functionl(\hbeta)< \functionl(\hgamma')+\functionl(\hdelta').
\end{equation}
\end{proposition}
\begin{proof}

$\mathbf{T}\times \bT^5\subset \bT\times \SO(10)$ contains two one-parametric subgroups 
\[x:\C^{\times}\rightarrow \mathbf{T}\times \bT^5 \quad x(z) =(1,1,z^2,z^4,z^2,1),\]
\[y:\C^{\times}\rightarrow \mathbf{T}\times \bT^5 \quad y(w) =(w^{-16},1,w^2,w^4,w^6,w^8).\]
One-parametric subgroups $x(z)$ and $y(w)$ designed in a such a way that \[\Pi(x(z))\lambda^{\halpha}=z^{\functionf_1(\halpha)}\lambda^{\halpha},\Pi(y(w))\lambda^{\halpha}=w^{\functionf_2(\halpha)}\lambda^{\halpha}.\]
$\functionf_1(\halpha),\functionf_2(\halpha)$ are components of the map $\functionf(\halpha)$. Defining relations $\Gamma^{\bm{s}^k}$ (\ref{E:equationsaf0}
) are $\mathbf{T}\times \widetilde{\bT}^5$-weight vectors. Thus all monomials scale by the same factor under $\Pi(x(z)y(w))$-action. The weight of the clutter $\lambda^{\halpha}\lambda^{\hbeta}$ is $z^{\functionf_1(\halpha)+\functionf_1(\hbeta)}w^{\functionf_2(\halpha)+\functionf_2(\hbeta)}$. So $\functionf(\halpha)+\functionf(\hbeta)=\functionf(\hgamma')+\functionf(\hdelta')$. Let $t$ be such that 
\[\functionf(\hgamma')=(\functionf(\halpha)+\functionf(\hbeta))/2+t,\]
\[\functionf(\hdelta')=(\functionf(\halpha)+\functionf(\hbeta))/2-t.\]
I know that $\Gamma^{\bm{s}^k}$ contains a unique clutter $\lambda^{\halpha}\lambda^{\hbeta}$. It corresponds to a rectangle 
$[\halpha\wedge \hbeta,\halpha\vee \hbeta]$ in the diagram (\ref{P:kxccvgw13}). By Proposition \ref{P:functionL}, the inequality (\ref{E:tinequality}) is violated precisely at $\hgamma',\hdelta' \in [\halpha\wedge \hbeta,\halpha\vee \hbeta]$. Straitened relation forbid such $\hgamma',\hdelta' $.
\end{proof}

\paragraph{The family $\hA_h$} Let $\hA$ be based on $\hA[\hdelta,\hdelta']$.

I use the function $\functionl$ to define the flat family of algebras $\hA_h$. It is a quotient of $\C[[h]]$ by the ideal $\fr$ generated by 
\begin{equation}\label{E:relconvfam}
\begin{split}
&\lambda^{{\halpha}}\lambda^{{\hbeta}}\pm\lambda^{{\halpha}\vee{\hbeta}}\lambda^{{\halpha}\wedge{\hbeta}}=\\
&\sum_{\functionf({\hgamma})=\functionf({\halpha})+\functionf({\hbeta})/2+t, \functionf({\hdelta}) = \functionf({\halpha})+\functionf({\hbeta})-t} \pm h^{\functionl(\hgamma)+\functionl(\hdelta)-\functionl(\halpha)-\functionl(\hbeta)}\lambda^{{\hgamma}}\lambda^{{\hdelta}}.
\end{split}
\end{equation}
It is worth pointing that by Proposition \ref{P:functionL}, the exponents of $h$ in the relation are positive. Relations (\ref{E:relconvfam}) define a straightening law on $\hA_h$, whose proof is the same as in the case when $h=1$ (see Proposition \ref{E:degreenull}). This means that standard monomials define a $\C[[h]]$-basis for $\hA_h$. As a result, $\hA_h$ is flat over $\C[[h]]$.
After reduction $\mod h$, relations (\ref{E:relconvfam}) become
\begin{equation}\label{E:relunivred}
\lambda^{{\halpha}}\lambda^{{\hbeta}}\pm\lambda^{{\halpha}\vee{\hbeta}}\lambda^{{\halpha}\wedge{\hbeta}}=0.
\end{equation}
These equations define a toric variety announced in the title of the section.
The algebra $Hibi[\hdelta,\hdelta']$ is the quotient of $P$ by (\ref{E:relunivred}). 
Introduce a notation:
\[\cH[\hdelta,\hdelta']:=Spec\, Hibi[\hdelta,\hdelta'].\]
As the notation suggests, $Hibi$ is an example of a Hibi algebra. These algebras were studied in \cite{THibi}. See also Appendix \ref{A:hibi} for background information on Hibi algebras.

Triviality of the deformation $\hA_h$ on $Spec\, \C((h))$ is verified in the next proposition.
\begin{proposition}
 There is an isomorphism 
\[\sfu:\hA\underset{\C}{\otimes}\C((h))\rightarrow \hA_h\underset{\C[[h]]}{\otimes}\C((h)).\]
\end{proposition}
\begin{proof}
The isomorphism is defined by the formula $\sfu(\lambda^{\halpha})=h^{\functionl(\halpha)}\lambda^{\halpha}$.
\end{proof}
\begin{remark}
The substitution $\lambda^{(0)^n}\rightarrow \sqrt{-1}\lambda^{(0)^n}$, $\lambda^{(j)^n}\rightarrow \sqrt{-1}\lambda^{(j)^n}, j=1,\dots,5$ transforms relations (\ref{E:relunivred}) to the Hibi form (\ref{E:Hibi}). To see this, the reader should look at the defining relations for $ A$ where terms (\ref{E:relunivred}) are written explicitly as in \cite{MovStr} (Eqs. 56).
\end{remark}

\subsection{Classification of $\cZ[\hdelta,\hdelta']$ with Gorenstein singularities}\label{S:Gorenstein}

Classification of Gorenstein $\cZ[\hdelta,\hdelta']$ can be formulated in terms of the geometry of ends of the interval $[\hdelta,\hdelta']\subset \sethE$.
\begin{proposition}\label{P:GorensteinCL}
The algebra $\hA$ based on $[\hdelta,\hdelta']$ is Gorenstein iff the neighborhoods of the ends of the Hasse diagram of $[\hdelta,\hdelta']$ don't look like those on the following picture.
\begin{figure}[h]
\centering
\begin{tikzpicture}[scale=0.6]
	\begin{pgfonlayer}{nodelayer}
		\node [circle, fill=black, draw, style=none] (0) at (-3, 2) {.};
		\node [circle, fill=black, draw, style=none] (1) at (-3.5, 1.5) {.};
		\node [circle, fill=black, draw, style=none] (2) at (-2.5, 1.5) {.};
		\node [circle, fill=black, draw, style=none] (3) at (-3, 1) {.};
		\node [circle, fill=black, draw, style=none] (4) at (-2.5, 0.5) {.};
		\node [circle, fill=black, draw, style=none] (5) at (-2, 1) {.};
		\node [circle, fill=black, draw, style=none] (6) at (-2, 0) {.};
		\node [circle, fill=black, draw, style=none] (7) at (-1.5, 0.5) {.};
		\node [style=none] (8) at (-3, 0) { $\cdots$};
		\node [circle, fill=black, draw, style=none] (9) at (2, 2) {.};
		\node [circle, fill=black, draw, style=none] (10) at (1.5, 1.5) {.};
		\node [circle, fill=black, draw, style=none] (11) at (1, 1) {.};
		\node [circle, fill=black, draw, style=none] (12) at (0.5, 0.5) {.};
		\node [circle, fill=black, draw, style=none] (13) at (1, 0) {.};
		\node [circle, fill=black, draw, style=none] (14) at (1.5, 0.5) {.};
		\node [circle, fill=black, draw, style=none] (15) at (2, 1) {.};
		\node [circle, fill=black, draw, style=none] (16) at (2.5, 1.5) {.};
		\node [style=none] (17) at (2, 0) { $\cdots$};
		\node [circle, fill=black, draw, style=none] (18) at (-2, -1.75) {.};
		\node [circle, fill=black, draw, style=none] (19) at (-1.5, -2.25) {.};
		\node [circle, fill=black, draw, style=none] (20) at (-2, -2.75) {.};
		\node [circle, fill=black, draw, style=none] (21) at (-2.5, -2.25) {.};
		\node [circle, fill=black, draw, style=none] (22) at (-2.5, -3.25) {.};
		\node [circle, fill=black, draw, style=none] (23) at (-3, -2.75) {.};
		\node [circle, fill=black, draw, style=none] (24) at (-3, -3.75) {.};
		\node [circle, fill=black, draw, style=none] (25) at (-3.5, -3.25) {.};
		\node [style=none] (26) at (-3, -1.75) { $\cdots$};
		\node [circle, fill=black, draw, style=none] (27) at (1, -1.75) {.};
		\node [circle, fill=black, draw, style=none] (28) at (0.5, -2.25) {.};
		\node [circle, fill=black, draw, style=none] (29) at (1, -2.75) {.};
		\node [circle, fill=black, draw, style=none] (30) at (1.5, -2.25) {.};
		\node [circle, fill=black, draw, style=none] (31) at (1.5, -3.25) {.};
		\node [circle, fill=black, draw, style=none] (32) at (2, -2.75) {.};
		\node [circle, fill=black, draw, style=none] (33) at (2, -3.75) {.};
		\node [circle, fill=black, draw, style=none] (34) at (2.5, -3.25) {.};
		\node [style=none] (35) at (2, -1.75) { $\cdots$};
		\node [style=none] (36) at (1.5, -0.5) { $\cdots$};
		\node [style=none] (37) at (-2.5, -0.5) { $\cdots$};
		\node [style=none] (38) at (-2.5, -1.25) { $\cdots$};
		\node [style=none] (39) at (1.5, -1.25) { $\cdots$};
		\node [style=none] (40) at (-0.75, 2) {\tiny Terminal segments};
		\node [style=none] (41) at (-0.75, -3.75) {\tiny Initial segments};
		\node [circle, fill=black, draw, style=none] (42) at (7.5, -1.25) {.};
		\node [circle, fill=black, draw, style=none] (43) at (6.5, -1.25) {.};
		\node [circle, fill=black, draw, style=none] (44) at (7.5, -0.25) {.};
		\node [style=none] (45) at (7.5, 0.25) { $\cdots$};
		\node [circle, fill=black, draw, style=none] (46) at (7, -1.75) {.};
		\node [circle, fill=black, draw, style=none] (47) at (8, -0.75) {.};
		\node [circle, fill=black, draw, style=none] (48) at (7, -0.75) {.};
		\node [circle, fill=black, draw, style=none] (49) at (7.5, 0.75) {.};
		\node [style=none] (50) at (7, -2.25) {$\cdots$};
		\node [circle, fill=black, draw, style=none] (51) at (7, -2.75) {.};
		\node [circle, fill=black, draw, style=none] (52) at (4.25, -2.75) {.};
		\node [circle, fill=black, draw, style=none] (53) at (3.75, -1.25) {.};
		\node [style=none] (54) at (3.75, 0.25) { $\cdots$};
		\node [circle, fill=black, draw, style=none] (55) at (4.75, -1.25) {.};
		\node [circle, fill=black, draw, style=none] (56) at (4.25, -0.75) {.};
		\node [circle, fill=black, draw, style=none] (57) at (3.25, -0.75) {.};
		\node [style=none] (58) at (4.25, -2.25) { $\cdots$};
		\node [circle, fill=black, draw, style=none] (59) at (4.25, -1.75) {.};
		\node [circle, fill=black, draw, style=none] (60) at (3.75, -0.25) {.};
		\node [circle, fill=black, draw, style=none] (61) at (3.75, 0.75) {.};
		\node [style=none] (62) at (5.75, 0.5) {\tiny Exceptional cases};
	\end{pgfonlayer}
	\begin{pgfonlayer}{edgelayer}
		\draw [thick] (0.center) to (2.center);
		\draw [thick] (2.center) to (5.center);
		\draw [thick] (5.center) to (7.center);
		\draw [thick] (0.center) to (1.center);
		\draw [thick] (2.center) to (3.center);
		\draw [thick] (5.center) to (4.center);
		\draw [thick] (7.center) to (6.center);
		\draw [thick] (1.center) to (3.center);
		\draw [thick] (3.center) to (4.center);
		\draw [thick] (4.center) to (6.center);
		\draw [thick] (9.center) to (16.center);
		\draw [thick] (16.center) to (15.center);
		\draw [thick] (15.center) to (14.center);
		\draw [thick] (14.center) to (13.center);
		\draw [thick] (12.center) to (13.center);
		\draw [thick] (12.center) to (11.center);
		\draw [thick] (11.center) to (10.center);
		\draw [thick] (10.center) to (9.center);
		\draw [thick] (10.center) to (15.center);
		\draw [thick] (11.center) to (14.center);
		\draw [thick] (18.center) to (19.center);
		\draw [thick] (19.center) to (20.center);
		\draw [thick] (20.center) to (22.center);
		\draw [thick] (22.center) to (24.center);
		\draw [thick] (25.center) to (24.center);
		\draw [thick] (25.center) to (23.center);
		\draw [thick] (23.center) to (22.center);
		\draw [thick] (23.center) to (21.center);
		\draw [thick] (21.center) to (20.center);
		\draw [thick] (21.center) to (18.center);
		\draw [thick] (27.center) to (28.center);
		\draw [thick] (28.center) to (29.center);
		\draw [thick] (29.center) to (31.center);
		\draw [thick] (31.center) to (33.center);
		\draw [thick] (33.center) to (34.center);
		\draw [thick] (32.center) to (34.center);
		\draw [thick] (30.center) to (32.center);
		\draw [thick] (27.center) to (30.center);
		\draw [thick] (30.center) to (29.center);
		\draw [thick] (32.center) to (31.center);
		\draw [thick] (4.center) to (8.center);
		\draw [thick] (14.center) to (17.center);
		\draw [thick] (26.center) to (21.center);
		\draw [thick] (35.center) to (30.center);
		\draw [thick] (8.center) to (37.center);
		\draw [thick] (6.center) to (37.center);
		\draw [thick] (13.center) to (36.center);
		\draw [thick] (17.center) to (36.center);
		\draw [thick] (38.center) to (26.center);
		\draw [thick] (38.center) to (18.center);
		\draw [thick] (39.center) to (27.center);
		\draw [thick] (39.center) to (35.center);
		\draw [thick] (49.center) to (45.center);
		\draw [thick] (45.center) to (44.center);
		\draw [thick] (44.center) to (47.center);
		\draw [thick] (47.center) to (42.center);
		\draw [thick] (42.center) to (46.center);
		\draw [thick] (43.center) to (48.center);
		\draw [thick] (48.center) to (44.center);
		\draw [thick] (48.center) to (42.center);
		\draw [thick] (43.center) to (46.center);
		\draw [thick] (46.center) to (50.center);
		\draw [thick] (50.center) to (51.center);
		\draw [thick] (57.center) to (53.center);
		\draw [thick] (53.center) to (59.center);
		\draw [thick] (59.center) to (58.center);
		\draw [thick] (58.center) to (52.center);
		\draw [thick] (59.center) to (55.center);
		\draw [thick] (56.center) to (55.center);
		\draw [thick] (60.center) to (56.center);
		\draw [thick] (60.center) to (57.center);
		\draw [thick] (56.center) to (53.center);
		\draw [thick] (61.center) to (54.center);
		\draw [thick] (54.center) to (60.center);
	\end{pgfonlayer}
	
\end{tikzpicture}\label{pict:two}\end{figure}
\end{proposition}
\paragraph{Commutative algebra background} Before giving the proof of the proposition, I need to remind the reader some definition from commutative algebra (see \cite{BrunsandHerzog} and \cite{Matsumura} for details).

I will be using a definition of a Cohen-Macaulay and Gorenstein algebra that I borrowed from \cite{Stanley}. It is designed for the graded algebras. It has an advantage that it makes construction of a dualizing module more explicit.

\begin{definition}\label{D:gorensteinserre}{\em
\begin{equation}\label{E:presentation}
R=\C[x_1,\dots,x_n]/\fr
\end{equation} is a graded Cohen-Macaulay algebra of dimension $\dim$ iff the minimal graded free $P=\C[x_1,\dots,x_n]$-resolution 
\[
0\ot R \stackrel{\sfd_{0}}{\ot}F_0\stackrel{\sfd_{1}}{\ot} F_1\ot \cdots \ot F_{n-d-1}\stackrel{\sfd_{n-d}}{\ot}F_{n-\dim}\ot 0
\]
satisfies $H^i(\rHom_P(F_{\bullet},P)=\{0\}, i\neq n-\dim$. The $P$-module $\omega_{\hA}=H^{n-\dim}(\rHom_P(F_{\bullet},P)$ is in fact an $A$-module. It is called the canonical module associated with the presentation (\ref{E:presentation}).
 $F_{\bullet}$ 
is self-dual iff $R$ is Gorenstein. 
}
\end{definition}

The reader may wish to check with ( \cite{Serre2} and \cite{Bass} Proposition 5.1) for the proof of the equivalence to the standard definition (\cite{BrunsandHerzog} and \cite{Matsumura} ). 

There is yet another characterization of Gorenstein algebras due to Stanley.
Let $V=\bigoplus_{i\geq i_0} V_i$. The Poincar\'{e} series $V(t)$ is the formal function $\sum_{i\geq i_0}\dim V_it^i$. If $R$ is a graded algebra $R=\bigoplus_{i\geq 0}R_i$, then $R(t)$ is the Poincare series of the underlying vector space.
\begin{proposition} \cite{Stanley}\label{P:stanley}
Let $R$ be a graded algebra. Suppose $R$ is a Cohen-Macaulay integral domain of Krull dimension $d=\dim R$. Then $R$ is Gorenstein iff for some $p\in \Z$, 
\begin{equation}\label{E:goreneq}
R(1/t)=(-1)^dt^pR(t).
\end{equation}
\end{proposition}

\paragraph{The proof of Proposition \ref{P:GorensteinCL} and more}
To determine under what conditions on $\hdelta$, and $\hdelta'$ the algebra $\hA$ based on $[\hdelta,\hdelta']$ is Gorenstein, I will use the toric degeneration from Section \ref{S:singularities}. I will reduce the problem to the similar problem about Hibi algebra $Hibi$. For algebras based on distributive latices (this applies to us) the problem was completely solved in \cite{THibi}. The answer can be read off from the Hasse diagram $[\hdelta,\hdelta']$. By Remark \ref{R:hibi}, I have to analyze maximal chains (see Appendix \ref{A:hibi}) in the join-irreducible subposet $\rB[\hdelta,\hdelta']$ of $[\hdelta,\hdelta']$ (the letter $\rB$ is after Birkhoff). 

A poset $\setX$ is {\it equidimensional} if its maximal chains have equal length.
A lattice $\setX$ is {\it Gorenstein} if in the subposet $B(\setX)\subset \setX$ of join-irreducible elements is equidimensional.

The action $\shift$ (\ref{E:shift}) on $\sethE$ indices a periodic structure on $\rB[-\infty,\infty]$. $\rB[-\infty,\infty]$ is a graded poset (see Example \ref{E:Bgraded}). With the help of $\shift$-periodicity and direct inspection I verify that 
$\rho_{\rB}^{-1}([a,b])\subset \rB[-\infty,\infty]$ is an equidimensional poset ( see Fig 0 for $\rho_{\rB}^{-1}([-3,1])$). 
 Suppose that $a$ and $b$ are such that \[\emptyset\neq \rho_{\rB}^{-1}([a,b])\subset \rB[\hdelta,\hdelta']\] is maximal. The geometry of $\sethE$ dictates that $\rB[\hdelta,\hdelta']\backslash \rho_{\rB}^{-1}([a,b])$ is a union $\rB(\hdelta)\cup \rB(\hdelta')$ of two posets-{\it the ends}. They satisfy $\rB(\hdelta)<\rB(\hdelta')$.

If $\rB[\hdelta,\hdelta']$ is not Gorenstein, then at least one of the ends $\rB(\hdelta)$ or $ \rB(\hdelta')$ fails to be an equidimensional poset.

 The structure of $ \rB(\hdelta')$, the terminal part, depends only on $\hdelta'$. The list of Hasse diagrams $ \rB(\hdelta')$ 
 as functions of $\hdelta'$
 is given on Fig. 1-8. 
 
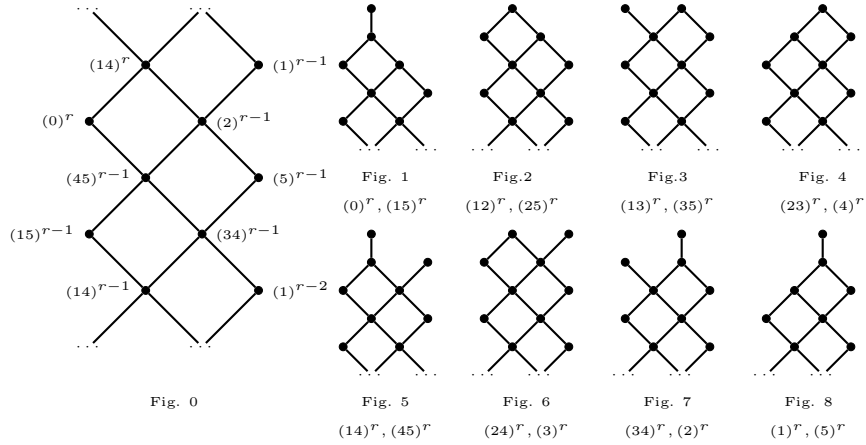
\begin{figure}[h]
\centering
\begin{tikzpicture}[level distance=50 mm, scale=0.75]
	\begin{pgfonlayer}{nodelayer}
		
		\node [circle, fill=black, draw, style=none,label=0:\tiny{$(1)^{r-1}$} ] (0) at (-4, 5) {.};

		\node [circle, fill=black, draw, style=none,label=0:\tiny{$(2)^{r-1}$} ] (1) at (-5, 4) {.};

		\node [circle, fill=black, draw, style=none,label=180:\tiny{$(14)^{r}$} ] (2) at (-6, 5) {.};
		
		\node [circle, fill=black, draw, style=none,label=180:\tiny{$(0)^{r}$} ] (3) at (-7, 4) {.};

		\node [circle, fill=black, draw, style=none,label=180:\tiny{$(45)^{r-1}$} ] (4) at (-6, 3) {.};

		\node [circle, fill=black, draw, style=none,label=0:\tiny{$(5)^{r-1}$} ] (5) at (-4, 3) {.};

		\node [circle, fill=black, draw, style=none,label=0:\tiny{$(34)^{r-1}$} ] (6) at (-5, 2) {.};

		\node [circle, fill=black, draw, style=none,label=180:\tiny{$(15)^{r-1}$} ] (7) at (-7, 2) {.};

		\node [circle, fill=black, draw, style=none,label=180:\tiny{$(14)^{r-1}$} ] (8) at (-6, 1) {.};

		\node [circle, fill=black, draw, style=none,label=0:\tiny{$(1)^{r-2}$} ] (9) at (-4, 1) {.};
		
		\node [style=none] (10) at (-5, 6) {\tiny $\cdots$};
		\node [style=none] (11) at (-7, 6) {\tiny $\cdots$};
		\node [style=none] (12) at (-7, 0) {\tiny $\cdots$};
		\node [style=none] (13) at (-5, 0) {\tiny $\cdots$};
		\node [style=none] (14) at (-5.5, -1) {\tiny Fig. 0};
		\node [style=none] (15) at (-2, 3.5) {\tiny$\cdots$};
		\node [circle, fill=black, draw, style=none] (16) at (-1.5, 5) {.};
		\node [style=none] (17) at (-1, 3.5) {\tiny $\cdots$};
		\node [circle, fill=black, draw, style=none] (18) at (-2, 5.5) {.};
		\node [circle, fill=black, draw, style=none] (19) at (-2.5, 5) {.};
		\node [circle, fill=black, draw, style=none] (20) at (-2.5, 4) {.};
		\node [circle, fill=black, draw, style=none] (21) at (-1.5, 4) {.};
		\node [circle, fill=black, draw, style=none] (22) at (-1, 4.5) {.};
		\node [circle, fill=black, draw, style=none] (23) at (-2, 4.5) {.};
		\node [circle, fill=black, draw, style=none] (24) at (-2, 6) {.};
		\node [circle, fill=black, draw, style=none] (25) at (0.5, 6) {.};
		\node [circle, fill=black, draw, style=none] (26) at (1, 5.5) {.};
		\node [circle, fill=black, draw, style=none] (27) at (0.5, 5) {.};
		\node [circle, fill=black, draw, style=none] (28) at (1.5, 5) {.};
		\node [circle, fill=black, draw, style=none] (29) at (1.5, 4) {.};
		\node [circle, fill=black, draw, style=none] (30) at (0.5, 4) {.};
		\node [circle, fill=black, draw, style=none] (31) at (1, 4.5) {.};
		\node [circle, fill=black, draw, style=none] (32) at (0, 5.5) {.};
		\node [circle, fill=black, draw, style=none] (33) at (0, 4.5) {.};
		\node [style=none] (34) at (1, 3.5) {\tiny $\cdots$};
		\node [style=none] (35) at (0, 3.5) {\tiny $\cdots$};
		\node [circle, fill=black, draw, style=none] (36) at (2.5, 4) {.};
		\node [circle, fill=black, draw, style=none] (37) at (3, 4.5) {.};
		\node [circle, fill=black, draw, style=none] (38) at (3, 5.5) {.};
		\node [circle, fill=black, draw, style=none] (39) at (2.5, 6) {.};
		\node [circle, fill=black, draw, style=none] (40) at (3.5, 5) {.};
		\node [circle, fill=black, draw, style=none] (41) at (2.5, 5) {.};
		\node [circle, fill=black, draw, style=none] (42) at (3.5, 4) {.};
		\node [circle, fill=black, draw, style=none] (43) at (4, 5.5) {.};
		\node [circle, fill=black, draw, style=none] (44) at (4, 4.5) {.};
		\node [circle, fill=black, draw, style=none] (45) at (3.5, 6) {.};
		\node [style=none] (46) at (3, 3.5) {\tiny $\cdots$};
		\node [style=none] (47) at (4, 3.5) {\tiny $\cdots$};
		\node [circle, fill=black, draw, style=none] (48) at (5.5, 5.5) {.};
		\node [circle, fill=black, draw, style=none] (49) at (6.5, 5.5) {.};
		\node [circle, fill=black, draw, style=none] (50) at (6, 6) {.};
		\node [style=none] (51) at (6.5, 3.5) {\tiny $\cdots$};
		\node [circle, fill=black, draw, style=none] (52) at (5, 5) {.};
		\node [circle, fill=black, draw, style=none] (53) at (5, 4) {.};
		\node [circle, fill=black, draw, style=none] (54) at (6, 5) {.};
		\node [circle, fill=black, draw, style=none] (55) at (5.5, 4.5) {.};
		\node [style=none] (56) at (5.5, 3.5) {\tiny $\cdots$};
		\node [circle, fill=black, draw, style=none] (57) at (6, 4) {.};
		\node [circle, fill=black, draw, style=none] (58) at (6.5, 4.5) {.};
		\node [circle, fill=black, draw, style=none] (59) at (-2, 2) {.};
		\node [style=none] (60) at (-2, -0.5) {\tiny $\cdots$};
		\node [circle, fill=black, draw, style=none] (61) at (-2.5, 1) {.};
		\node [circle, fill=black, draw, style=none] (62) at (-2, 1.5) {.};
		\node [circle, fill=black, draw, style=none] (63) at (-1.5, 0) {.};
		\node [circle, fill=black, draw, style=none] (64) at (-2, 0.5) {.};
		\node [circle, fill=black, draw, style=none] (65) at (-2.5, 0) {.};
		\node [style=none] (66) at (-1, -0.5) {\tiny $\cdots$};
		\node [circle, fill=black, draw, style=none] (67) at (-1.5, 1) {.};
		\node [circle, fill=black, draw, style=none] (68) at (-1, 0.5) {.};
		\node [circle, fill=black, draw, style=none] (69) at (-1, 1.5) {.};
		\node [circle, fill=black, draw, style=none] (70) at (1.5, 1) {.};
		\node [circle, fill=black, draw, style=none] (71) at (0, 1.5) {.};
		\node [circle, fill=black, draw, style=none] (72) at (1, 0.5) {.};
		\node [circle, fill=black, draw, style=none] (73) at (0.5, 1) {.};
		\node [circle, fill=black, draw, style=none] (74) at (1, 1.5) {.};
		\node [circle, fill=black, draw, style=none] (75) at (0.5, 0) {.};
		\node [circle, fill=black, draw, style=none] (76) at (0.5, 2) {.};
		\node [style=none] (77) at (0, -0.5) {\tiny $\cdots$};
		\node [style=none] (78) at (1, -0.5) {\tiny $\cdots$};
		\node [circle, fill=black, draw, style=none] (79) at (1.5, 0) {.};
		\node [circle, fill=black, draw, style=none] (80) at (0, 0.5) {.};
		\node [circle, fill=black, draw, style=none] (81) at (1.5, 2) {.};
		\node [style=none] (82) at (2.5, -0.5) {\tiny $\cdots$};
		\node [circle, fill=black, draw, style=none] (83) at (3.5, 0.5) {.};
		\node [circle, fill=black, draw, style=none] (84) at (2.5, 0.5) {.};
		\node [circle, fill=black, draw, style=none] (85) at (4, 0) {.};
		\node [circle, fill=black, draw, style=none] (86) at (2.5, 1.5) {.};
		\node [circle, fill=black, draw, style=none] (87) at (3, 0) {.};
		\node [circle, fill=black, draw, style=none] (88) at (4, 1) {.};
		\node [circle, fill=black, draw, style=none] (89) at (3.5, 1.5) {.};
		\node [circle, fill=black, draw, style=none] (90) at (3, 1) {.};
		\node [style=none] (91) at (3.5, -0.5) {\tiny $\cdots$};
		\node [circle, fill=black, draw, style=none] (92) at (3.5, 2) {.};
		\node [circle, fill=black, draw, style=none] (93) at (5.5, 1) {.};
		\node [circle, fill=black, draw, style=none] (94) at (6, 2) {.};
		\node [style=none] (95) at (5, -0.5) {\tiny $\cdots$};
		\node [circle, fill=black, draw, style=none] (96) at (5, 0.5) {.};
		\node [circle, fill=black, draw, style=none] (97) at (6.5, 1) {.};
		\node [style=none] (98) at (6, -0.5) {\tiny $\cdots$};
		\node [circle, fill=black, draw, style=none] (99) at (6, 1.5) {.};
		\node [circle, fill=black, draw, style=none] (100) at (6.5, 0) {.};
		\node [circle, fill=black, draw, style=none] (101) at (5.5, 0) {.};
		\node [circle, fill=black, draw, style=none] (102) at (6, 0.5) {.};
		\node [style=none] (103) at (-1.75, 2.5) {\tiny $(0)^r,(15)^r$};
		\node [style=none] (104) at (0.5, 2.5) {\tiny $(12)^r,(25)^r$};
		\node [style=none] (105) at (3.25, 2.5) {\tiny $(13)^r,(35)^r$};
		\node [style=none] (106) at (6, 2.5) {\tiny $(23)^r,(4)^r$};
		\node [style=none] (107) at (-1.75, -1.5) {\tiny $(14)^r,(45)^r$};
		\node [style=none] (108) at (0.75, -1.5) {\tiny $(24)^r,(3)^r$};
		\node [style=none] (109) at (3.25, -1.5) {\tiny $(34)^r,(2)^r$};
		\node [style=none] (110) at (5.75, -1.5) {\tiny $(1)^r,(5)^r$};
		\node [style=none] (111) at (-1.75, 3) {\tiny Fig. 1};
		\node [style=none] (112) at (0.5, 3) {\tiny Fig.2};
		\node [style=none] (113) at (5.75, -1) { \tiny Fig. 8};
		\node [style=none] (114) at (3.25, -1) {\tiny Fig. 7};
		\node [style=none] (115) at (0.75, -1) {\tiny Fig. 6};
		\node [style=none] (116) at (-1.75, -1) {\tiny Fig. 5};
		\node [style=none] (117) at (6, 3) {\tiny Fig. 4};
		\node [style=none] (118) at (3.25, 3) {\tiny Fig.3};
	\end{pgfonlayer}
	\begin{pgfonlayer}{edgelayer}
		\draw [thick] (2) to (3);
		\draw [thick] (3) to (4);
		\draw [thick] (2) to (1);
		\draw [thick] (1) to (5);
		\draw [thick] (5) to (6);
		\draw [thick] (4) to (6);
		\draw [thick] (4) to (1);
		\draw [thick] (1) to (0);
		\draw [thick] (4) to (7);
		\draw [thick] (7) to (8);
		\draw [thick] (6) to (8);
		\draw [thick] (6) to (9);
		\draw [thick] (11) to (2);
		\draw [thick] (10) to (0);
		\draw [thick] (10) to (2);
		\draw [thick] (8) to (12);
		\draw [thick] (8) to (13);
		\draw [thick] (9) to (13);
		\draw [thick] (18) to (19);
		\draw [thick] (19) to (23);
		\draw [thick] (18) to (16);
		\draw [thick] (16) to (22);
		\draw [thick] (22) to (21);
		\draw [thick] (23) to (21);
		\draw [thick] (23) to (16);
		\draw [thick] (23) to (20);
		\draw [thick] (20) to (15);
		\draw [thick] (21) to (15);
		\draw [thick] (21) to (17);
		\draw [thick] (24) to (18);
		\draw [thick] (25) to (32);
		\draw [thick] (32) to (27);
		\draw [thick] (25) to (26);
		\draw [thick] (26) to (28);
		\draw [thick] (28) to (31);
		\draw [thick] (27) to (31);
		\draw [thick] (27) to (26);
		\draw [thick] (27) to (33);
		\draw [thick] (33) to (30);
		\draw [thick] (31) to (30);
		\draw [thick] (31) to (29);
		\draw [thick] (30) to (35);
		\draw [thick] (30) to (34);
		\draw [thick] (29) to (34);
		\draw [thick] (39) to (38);
		\draw [thick] (45) to (43);
		\draw [thick] (43) to (40);
		\draw [thick] (38) to (40);
		\draw [thick] (38) to (45);
		\draw [thick] (38) to (41);
		\draw [thick] (41) to (37);
		\draw [thick] (40) to (37);
		\draw [thick] (40) to (44);
		\draw [thick] (37) to (36);
		\draw [thick] (37) to (42);
		\draw [thick] (44) to (42);
		\draw [thick] (36) to (46);
		\draw [thick] (42) to (47);
		\draw [thick] (42) to (46);
		\draw [thick] (50) to (49);
		\draw [thick] (49) to (54);
		\draw [thick] (48) to (54);
		\draw [thick] (48) to (50);
		\draw [thick] (48) to (52);
		\draw [thick] (52) to (55);
		\draw [thick] (54) to (55);
		\draw [thick] (54) to (58);
		\draw [thick] (55) to (53);
		\draw [thick] (55) to (57);
		\draw [thick] (58) to (57);
		\draw [thick] (53) to (56);
		\draw [thick] (57) to (51);
		\draw [thick] (57) to (56);
		\draw [thick] (62) to (61);
		\draw [thick] (61) to (64);
		\draw [thick] (62) to (67);
		\draw [thick] (67) to (68);
		\draw [thick] (68) to (63);
		\draw [thick] (64) to (63);
		\draw [thick] (64) to (67);
		\draw [thick] (64) to (65);
		\draw [thick] (65) to (60);
		\draw [thick] (63) to (60);
		\draw [thick] (63) to (66);
		\draw [thick] (59) to (62);
		\draw [thick] (67) to (69);
		\draw [thick] (76) to (71);
		\draw [thick] (71) to (73);
		\draw [thick] (76) to (74);
		\draw [thick] (74) to (70);
		\draw [thick] (70) to (72);
		\draw [thick] (73) to (72);
		\draw [thick] (73) to (74);
		\draw [thick] (73) to (80);
		\draw [thick] (80) to (75);
		\draw [thick] (72) to (75);
		\draw [thick] (72) to (79);
		\draw [thick] (75) to (77);
		\draw [thick] (75) to (78);
		\draw [thick] (79) to (78);
		\draw [thick] (74) to (81);
		\draw [thick] (86) to (90);
		\draw [thick] (89) to (88);
		\draw [thick] (88) to (83);
		\draw [thick] (90) to (83);
		\draw [thick] (90) to (89);
		\draw [thick] (90) to (84);
		\draw [thick] (84) to (87);
		\draw [thick] (83) to (87);
		\draw [thick] (83) to (85);
		\draw [thick] (87) to (82);
		\draw [thick] (87) to (91);
		\draw [thick] (85) to (91);
		\draw [thick] (92) to (89);
		\draw [thick] (99) to (97);
		\draw [thick] (97) to (102);
		\draw [thick] (93) to (102);
		\draw [thick] (93) to (99);
		\draw [thick] (93) to (96);
		\draw [thick] (96) to (101);
		\draw [thick] (102) to (101);
		\draw [thick] (102) to (100);
		\draw [thick] (101) to (95);
		\draw [thick] (101) to (98);
		\draw [thick] (100) to (98);
		\draw [thick] (94) to (99);
	\end{pgfonlayer}
\end{tikzpicture}
\end{figure}

In complete analogy with the terminal part the list of Hasse diagrams $ \rB(\hdelta)$ of initial parts as a function of $\hdelta$
is given on Fig. 9-16. 

\newpage

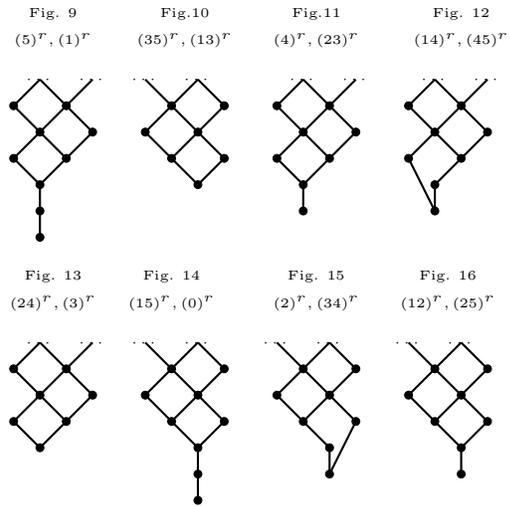
\begin{figure}[h]
\centering
\begin{tikzpicture}[level distance=50 mm, scale=0.7]
	\begin{pgfonlayer}{nodelayer}
		\node [circle, fill=black, draw, style=none] (0) at (-6.25, -1.25) {.};
		\node [circle, fill=black, draw, style=none] (1) at (-6.25, -0.25) {.};
		\node [style=none] (2) at (-6.75, 0.25) {\tiny $\cdots$};
		\node [circle, fill=black, draw, style=none] (3) at (-5.75, -0.75) {.};
		\node [circle, fill=black, draw, style=none] (4) at (-6.75, -0.75) {.};
		\node [circle, fill=black, draw, style=none] (5) at (-5.25, -0.25) {.};
		\node [circle, fill=black, draw, style=none] (6) at (-5.25, -1.25) {.};
		\node [style=none] (7) at (-5.75, 0.25) {\tiny $\cdots$};
		\node [circle, fill=black, draw, style=none] (8) at (-3.75, -0.25) {.};
		\node [circle, fill=black, draw, style=none] (9) at (-4.25, -0.75) {.};
		\node [circle, fill=black, draw, style=none] (10) at (-2.75, -0.25) {.};
		\node [circle, fill=black, draw, style=none] (11) at (-3.25, -1.75) {.};
		\node [style=none] (12) at (-3.25, 0.25) {\tiny $\cdots$};
		\node [circle, fill=black, draw, style=none] (13) at (-2.75, -1.25) {.};
		\node [circle, fill=black, draw, style=none] (14) at (-3.75, -1.25) {.};
		\node [circle, fill=black, draw, style=none] (15) at (-3.25, -0.75) {.};
		\node [style=none] (16) at (-11, 6) {\tiny $(5)^r,(1)^r$};
		\node [style=none] (17) at (-8.5, 6) {\tiny $(35)^r,(13)^r$};
		\node [style=none] (18) at (-6, 6) {\tiny $(4)^r,(23)^r$};
		\node [style=none] (19) at (-3.25, 6) {\tiny $(14)^r,(45)^r$};
		\node [style=none] (20) at (-11, 1) {\tiny $(24)^r,(3)^r$};
		\node [style=none] (21) at (-8.75, 1) {\tiny $(15)^r,(0)^r$};
		\node [style=none] (22) at (-6, 1) {\tiny $(2)^r,(34)^r$};
		\node [style=none] (23) at (-3.5, 1) {\tiny $(12)^r,(25)^r$};
		\node [style=none] (24) at (-11, 6.5) {\tiny Fig. 9};
		\node [style=none] (25) at (-8.5, 6.5) {\tiny Fig.10};
		\node [style=none] (26) at (-3.5, 1.5) {\tiny Fig. 16};
		\node [style=none] (27) at (-6, 1.5) {\tiny Fig. 15};
		\node [style=none] (28) at (-8.75, 1.5) {\tiny Fig. 14};
		\node [style=none] (29) at (-11, 1.5) {\tiny Fig. 13};
		\node [style=none] (30) at (-3.25, 6.5) {\tiny Fig. 12};
		\node [style=none] (31) at (-6, 6.5) {\tiny Fig.11};
		\node [circle, fill=black, draw, style=none] (32) at (-10.25, 4.25) {.};
		\node [circle, fill=black, draw, style=none] (33) at (-10.75, 4.75) {.};
		\node [circle, fill=black, draw, style=none] (34) at (-11.25, 3.25) {.};
		\node [circle, fill=black, draw, style=none] (35) at (-11.75, 4.75) {.};
		\node [circle, fill=black, draw, style=none] (36) at (-11.75, 3.75) {.};
		\node [circle, fill=black, draw, style=none] (37) at (-10.75, 3.75) {.};
		\node [circle, fill=black, draw, style=none] (38) at (-11.25, 4.25) {.};
		\node [style=none] (39) at (-11.25, 5.25) {\tiny $\cdots$};
		\node [circle, fill=black, draw, style=none] (40) at (-11.25, 2.75) {.};
		\node [style=none] (41) at (-10.25, 5.25) {\tiny $\cdots$};
		\node [circle, fill=black, draw, style=none] (42) at (-11.25, 2.25) {.};
		\node [circle, fill=black, draw, style=none] (43) at (-6.75, 4.75) {.};
		\node [style=none] (44) at (-5.25, 5.25) {\tiny $\cdots$};
		\node [circle, fill=black, draw, style=none] (45) at (-6.25, 3.25) {.};
		\node [circle, fill=black, draw, style=none] (46) at (-6.25, 4.25) {.};
		\node [circle, fill=black, draw, style=none] (47) at (-5.75, 4.75) {.};
		\node [style=none] (48) at (-6.25, 5.25) {\tiny $\cdots$};
		\node [circle, fill=black, draw, style=none] (49) at (-6.75, 3.75) {.};
		\node [circle, fill=black, draw, style=none] (50) at (-5.25, 4.25) {.};
		\node [circle, fill=black, draw, style=none] (51) at (-5.75, 3.75) {.};
		\node [circle, fill=black, draw, style=none] (52) at (-11.75, -0.25) {.};
		\node [circle, fill=black, draw, style=none] (53) at (-10.75, -1.25) {.};
		\node [circle, fill=black, draw, style=none] (54) at (-10.25, -0.75) {.};
		\node [style=none] (55) at (-10.25, 0.25) {\tiny $\cdots$};
		\node [circle, fill=black, draw, style=none] (56) at (-11.25, -0.75) {.};
		\node [circle, fill=black, draw, style=none] (57) at (-10.75, -0.25) {.};
		\node [circle, fill=black, draw, style=none] (58) at (-11.25, -1.75) {.};
		\node [style=none] (59) at (-11.25, 0.25) {\tiny $\cdots$};
		\node [circle, fill=black, draw, style=none] (60) at (-11.75, -1.25) {.};
		\node [circle, fill=black, draw, style=none] (61) at (-5.75, -1.75) {.};
		\node [circle, fill=black, draw, style=none] (62) at (-5.75, -2.25) {.};
		\node [style=none] (63) at (-9.25, 5.25) {\tiny $\cdots$};
		\node [circle, fill=black, draw, style=none] (64) at (-8.25, 3.25) {.};
		\node [circle, fill=black, draw, style=none] (65) at (-7.75, 4.75) {.};
		\node [circle, fill=black, draw, style=none] (66) at (-9.25, 4.25) {.};
		\node [circle, fill=black, draw, style=none] (67) at (-8.75, 4.75) {.};
		\node [circle, fill=black, draw, style=none] (68) at (-8.75, 3.75) {.};
		\node [style=none] (69) at (-8.25, 5.25) {\tiny $\cdots$};
		\node [circle, fill=black, draw, style=none] (70) at (-8.25, 4.25) {.};
		\node [circle, fill=black, draw, style=none] (71) at (-7.75, 3.75) {.};
		\node [circle, fill=black, draw, style=none] (72) at (-3.25, -2.25) {.};
		\node [style=none] (73) at (-4.25, 0.25) {\tiny $\cdots$};
		\node [circle, fill=black, draw, style=none] (74) at (-6.25, 2.75) {.};
		\node [style=none] (75) at (-2.75, 5.25) {\tiny $\cdots$};
		\node [circle, fill=black, draw, style=none] (76) at (-3.75, 2.75) {.};
		\node [circle, fill=black, draw, style=none] (77) at (-3.75, 3.25) {.};
		\node [style=none] (78) at (-3.75, 5.25) {\tiny $\cdots$};
		\node [circle, fill=black, draw, style=none] (79) at (-3.25, 4.75) {.};
		\node [circle, fill=black, draw, style=none] (80) at (-3.25, 3.75) {.};
		\node [circle, fill=black, draw, style=none] (81) at (-2.75, 4.25) {.};
		\node [circle, fill=black, draw, style=none] (82) at (-4.25, 4.75) {.};
		\node [circle, fill=black, draw, style=none] (83) at (-3.75, 4.25) {.};
		\node [circle, fill=black, draw, style=none] (84) at (-4.25, 3.75) {.};
		\node [circle, fill=black, draw, style=none] (85) at (-8.25, -0.75) {.};
		\node [style=none] (86) at (-9.25, 0.25) {\tiny $\cdots$};
		\node [circle, fill=black, draw, style=none] (87) at (-8.75, -0.25) {.};
		\node [style=none] (88) at (-8.25, 0.25) {\tiny $\cdots$};
		\node [circle, fill=black, draw, style=none] (89) at (-8.25, -2.25) {.};
		\node [circle, fill=black, draw, style=none] (90) at (-7.75, -1.25) {.};
		\node [circle, fill=black, draw, style=none] (91) at (-9.25, -0.75) {.};
		\node [circle, fill=black, draw, style=none] (92) at (-7.75, -0.25) {.};
		\node [circle, fill=black, draw, style=none] (93) at (-8.25, -1.75) {.};
		\node [circle, fill=black, draw, style=none] (94) at (-8.75, -1.25) {.};
		\node [circle, fill=black, draw, style=none] (95) at (-8.25, -2.75) {.};
		\node [style=none] (96) at (3.75, -2.5) {};
	\end{pgfonlayer}
	\begin{pgfonlayer}{edgelayer}
		\draw [thick] (2.center) to (1.center);
		\draw [thick] (7.center) to (5.center);
		\draw [thick] (5.center) to (3.center);
		\draw [thick] (1.center) to (3.center);
		\draw [thick] (1.center) to (7.center);
		\draw [thick] (1.center) to (4.center);
		\draw [thick] (4.center) to (0.center);
		\draw [thick] (3.center) to (0.center);
		\draw [thick] (3.center) to (6.center);
		\draw [thick] (12.center) to (10.center);
		\draw [thick] (10.center) to (15.center);
		\draw [thick] (8.center) to (15.center);
		\draw [thick] (8.center) to (12.center);
		\draw [thick] (8.center) to (9.center);
		\draw [thick] (9.center) to (14.center);
		\draw [thick] (15.center) to (14.center);
		\draw [thick] (15.center) to (13.center);
		\draw [thick] (14.center) to (11.center);
		\draw [thick] (13.center) to (11.center);
		\draw [thick] (39.center) to (35.center);
		\draw [thick] (35.center) to (38.center);
		\draw [thick] (39.center) to (33.center);
		\draw [thick] (33.center) to (32.center);
		\draw [thick] (32.center) to (37.center);
		\draw [thick] (38.center) to (37.center);
		\draw [thick] (38.center) to (33.center);
		\draw [thick] (38.center) to (36.center);
		\draw [thick] (36.center) to (34.center);
		\draw [thick] (37.center) to (34.center);
		\draw [thick] (34.center) to (40.center);
		\draw [thick] (40.center) to (42.center);
		\draw [thick] (33.center) to (41.center);
		\draw [thick] (48.center) to (43.center);
		\draw [thick] (43.center) to (46.center);
		\draw [thick] (48.center) to (47.center);
		\draw [thick] (47.center) to (50.center);
		\draw [thick] (50.center) to (51.center);
		\draw [thick] (46.center) to (51.center);
		\draw [thick] (46.center) to (47.center);
		\draw [thick] (46.center) to (49.center);
		\draw [thick] (49.center) to (45.center);
		\draw [thick] (51.center) to (45.center);
		\draw [thick] (47.center) to (44.center);
		\draw [thick] (59.center) to (52.center);
		\draw [thick] (52.center) to (56.center);
		\draw [thick] (59.center) to (57.center);
		\draw [thick] (57.center) to (54.center);
		\draw [thick] (54.center) to (53.center);
		\draw [thick] (56.center) to (53.center);
		\draw [thick] (56.center) to (57.center);
		\draw [thick] (56.center) to (60.center);
		\draw [thick] (60.center) to (58.center);
		\draw [thick] (53.center) to (58.center);
		\draw [thick] (57.center) to (55.center);
		\draw [thick] (6.center) to (62.center);
		\draw [thick] (0.center) to (61.center);
		\draw [thick] (61.center) to (62.center);
		\draw [thick] (69.center) to (65.center);
		\draw [thick] (65.center) to (70.center);
		\draw [thick] (69.center) to (67.center);
		\draw [thick] (67.center) to (66.center);
		\draw [thick] (66.center) to (68.center);
		\draw [thick] (70.center) to (68.center);
		\draw [thick] (70.center) to (67.center);
		\draw [thick] (70.center) to (71.center);
		\draw [thick] (68.center) to (64.center);
		\draw [thick] (67.center) to (63.center);
		\draw [thick] (11.center) to (72.center);
		\draw [thick] (73.center) to (8.center);
		\draw [thick] (45.center) to (74.center);
		\draw [thick] (78.center) to (82.center);
		\draw [thick] (82.center) to (83.center);
		\draw [thick] (78.center) to (79.center);
		\draw [thick] (79.center) to (81.center);
		\draw [thick] (81.center) to (80.center);
		\draw [thick] (83.center) to (80.center);
		\draw [thick] (83.center) to (79.center);
		\draw [thick] (83.center) to (84.center);
		\draw [thick] (80.center) to (77.center);
		\draw [thick] (79.center) to (75.center);
		\draw [thick] (84.center) to (76.center);
		\draw [thick] (76.center) to (77.center);
		\draw [thick] (88.center) to (92.center);
		\draw [thick] (92.center) to (85.center);
		\draw [thick] (87.center) to (85.center);
		\draw [thick] (87.center) to (88.center);
		\draw [thick] (87.center) to (91.center);
		\draw [thick] (91.center) to (94.center);
		\draw [thick] (85.center) to (94.center);
		\draw [thick] (85.center) to (90.center);
		\draw [thick] (94.center) to (93.center);
		\draw [thick] (90.center) to (93.center);
		\draw [thick] (93.center) to (89.center);
		\draw [thick] (86.center) to (87.center);
		\draw [thick] (95.center) to (89.center);
		\draw [thick] (71.center) to (64.center);
	\end{pgfonlayer}
\end{tikzpicture}\end{figure}

Note that $\rB[\hdelta,\hdelta']$ is non-Gorenstein iff it is having ends as in Fig.5 or 7 or 12 or 15.

\newpage
I can define the ends $\setU_{left}(\hdelta)$ and $\setU_{right}(\hdelta')$ of $[\hdelta,\hdelta']$ the same way as for $\rB[\hdelta,\hdelta']$. Indeed, $\sethE$ is a graded poset
(see Example \ref{E:Egrposet}). 
 Suppose that $a$ and $b$ are such that \[\emptyset\neq \rho_{\sethE}^{-1}([a,b])\subset [\hdelta,\hdelta']\] is maximal. Then 
 $[\hdelta,\hdelta']\backslash \rho_{\sethE}^{-1}([a,b])$ is a union $\setU_{left}(\hdelta)\cup \setU_{right}(\hdelta')$ of the ends which satisfy $\setU_{left}(\hdelta)<\setU_{right}(\hdelta')$.

It is a matter of finite check to see that $\setU_{left}(\hdelta)$ and $\setU_{right}(\hdelta')$ completely determine $\rB(\hdelta)$ and $\rB(\hdelta')$ respectively. $\setU_{right}(\hdelta')$ with $\hdelta'$ as on Fig. 5 and 7 are precisely the terminal segments on the diagram (\ref{pict:two}). Similarly, $\setU_{left}(\hdelta)$ with $\hdelta$ as on Fig. 12 and 15 are precisely the initial segments on the diagram (\ref{pict:two}).
$\blacksquare$

Introduce a pair of totally ordered subsets of $\sethE$:
\begin{equation}\label{E:Mdef}
\begin{split}
&\rM_1^+=\{\cdots \lessdot(35)^{r-1}\lessdot(4)^{r-1}\lessdot (3)^{r-1}\lessdot(12)^r\lessdot (13)^r\lessdot(23)^r\lessdot(24)^{r}\lessdot(25)^r\lessdot(35)^{r}\lessdot\cdots \quad r\in \Z\}\\
&\rM_1^-=\{\cdots \lessdot(24)^r\lessdot(34)^r\lessdot(35)^r\lessdot(45)^r\lessdot(3)^r\lessdot(2)^r\lessdot(13)^{r+1}\lessdot(14)^{r+1}\lessdot(24)^{r+1}\lessdot\cdots \quad r\in \Z\}
\end{split}
\end{equation}
For definition of $\rM^{\pm}_i$ see Lemma \ref{L:multstatM}.
I can sum up the above discussion in the following theorem.
\begin{theorem}
The interval $[\hdelta,\hdelta']$ is Gorenstein iff capacity $\rCap$ (\ref{E:capdif}) satisfies
\begin{equation}\label{E:purity}
\begin{split}
&\rCap[\hdelta,\hdelta']=0\text{ or }\rCap[\hdelta,\hdelta']=1 \text{, or }\\
&\rCap[\hdelta,\hdelta']\geq 3 \text{ and the end points satisfy :}\\
&\hdelta\in \sethE\backslash \rM_2^+=\rM_1^+\sqcup \rM_3^+ ,\\
&\rM_2^+=\{(14)^r,(45)^r,(34)^r,(2)^r|r\in\Z\}\text{ and }\\
&\hdelta'\in \sethE\backslash \rM_2^-=\rM_1^-\sqcup \rM_3^-,\\
&\rM_2^-=\{(12)^r,(25)^r,(23)^r,(4)^r |r\in \Z\}.
\end{split}
\end{equation}
\end{theorem}
The length (\ref{E:lenght}) of the maximal chain in $ [\hdelta,\hdelta']$ in terminology of Appendix \ref{A:hibi} is the rank 
\[\rk(\hdelta,\hdelta') \text{ of } [\hdelta,\hdelta'].\]
 I deduce from Remark \ref{R:hibi} a pair of corollaries.
\begin{corollary}\label{C:Hgorenstein}
The algebra $Hibi[\hdelta,\hdelta']$ is Gorenstein iff $\hdelta,\hdelta'$ satisfy (\ref{E:purity}).
\end{corollary}

I leave verification to the reader that $\rho$ (\ref{E:LdefE}) 
satisfies
\begin{equation}\label{E:htandsw}
\rho(\hdelta')-\rho(\hdelta)=\rk(\hdelta,\hdelta'),
\end{equation}
\begin{equation}\label{E:htands}
\rho \reflection (\halpha)=-\rho(\halpha)+10
\end{equation}
where $\reflection$ as in (\ref{E:involution}).
 \begin{corollary}
$\dim Spec\, Hibi[\hdelta,\hdelta']
=\rk(\hdelta,\hdelta')+1$. 
\end{corollary}

The next proposition is the central result of this section.
\begin{proposition}\label{C:dimalg}
\begin{enumerate}
\item The algebra $\hA[\hdelta,\hdelta']$ is Gorenstein iff $\hdelta,\hdelta'$ satisfy (\ref{E:purity}).
\item $\dim \hA[\hdelta,\hdelta']=\rk(\hdelta,\hdelta')+1$.
\end{enumerate}
\end{proposition}
\begin{proof}
By Corollary 4.2 \cite{EisenbudStr} $A$ is a Cohen-Macaulay algebra. This implies that $\hA_{h}\underset{\C[[h]]}{\otimes}\C((h))$ has the Cohen-Macaulay property. From the line of equalities \[\dim_{\C}\hA_{i}=\dim_{\C((h))}{\hA_{hi}\underset{\C[[h]]}{\otimes}\C((h))} =\mathrm{rk}_{\C[[h]]}\hA_{hi}=\dim_{\C}Hibi_i, \] I conclude that $\hA(t)=Hibi(t)$. 

\begin{lemma}\label{P:integraldomain}
Algebra $A$ based on any interval $[\hdelta,\hdelta']$ is an integral domain.
\end{lemma}
\begin{proof}
The algebra $Hibi=\hA_{h}/(h)$ is an integral domain \cite{THibi}. Since $\hA_{h}$ is $\C[[h]]$-flat, $\hA_{h}$ is an integral domain. Thus $\hA_{h}\underset{\C[[h]]}{\otimes}\C((h))\cong \hA\underset{\C[[h]]}{\otimes}\C((h))$ and $\hA$ are integral domains.
\end{proof}

By Corollary \ref{C:Hgorenstein}, $Hibi$ is Gorenstein. By Proposition \ref{P:stanley}, $Hibi(t)$ satisfies (\ref{E:goreneq}). Thus $\hA(t)$ satisfies the same equation. By the same proposition $\hA$ is Gorenstein.
 The dimension formula follows from the comparison of degrees of Hilbert polynomials of $Hibi$ and $\hA$
\end{proof}
\begin{example}
\begin{enumerate}
\item Suppose $N,N'\geq0$. From Corollary \ref{C:dimalg} follows that $\dim_{Krull} \hA_{N}^{N'}=11+8(N'-N)$.
\item The algebra $\hA_{N}^{N'}=\hA[(0)^N,(1)^{N'}]$ is Gorenstein $\forall N<N'$. This verifies Conjecture \ref{CON:main} item \ref{I:CONmain1} for $\tspace=\cone$. 
\end{enumerate}
\end{example}
\paragraph{Filtration by Gorenstein sub-schemes}
Maps between local cohomology $H^i_{X\cap Z}(X,\O)$ and $H^i_{ Z}(Y,\O)$ of a pair of schemes $X\subset Y$ often become more tractable if $\codim\, X=1$. For a pair $\cZ[\hdelta',\hbeta]\supset \cZ[\hdelta,\hbeta]$ of Gorenstein schemes corresponding to a pair of intervals $[\hdelta',\hbeta]\supset [\hdelta,\hbeta]$, I will now construct a sequence of Gorenstein schemes $\cZ_i$ \[\cZ[\hdelta',\hbeta]=\cZ_0\supset \cZ_1\supset\cdots\supset \cZ_n= \cZ[\hdelta,\hbeta]\] such that $\dim \cZ_{i-1}-\dim \cZ_{i}=1$.
 \begin{proposition}\label{P:seqint}
Let $[\hdelta,\hbeta]\supset [\hdelta',\hbeta]$ be intervals satisfying (\ref{E:purity}) such that $\hdelta'\in \rM_1^+$. Then there is a sequence of intervals 
 \begin{equation}\label{E:intervalsneed}
 [\hdelta,\hbeta]=[\hdelta_{1},\hbeta]\supset [\hdelta_{2},\hbeta]\supset\cdots \supset [\hdelta_{n},\hbeta]=[\hdelta',\hbeta]
 \end{equation}
 $\hdelta_{i}\in \rM_1^+,i>1$ such that they satisfy (\ref{E:purity}) and
\begin{equation}\label{E:lessdot}
\hdelta_{i}\lessdot \hdelta_{i+1},\text{ and in particular }\rho( \hdelta_{i+1})=\rho(\hdelta_{i})+1,
\end{equation}
and 
$\hdelta_{1}=\hdelta$, $\hdelta_{n}=\hdelta'$. 
\end{proposition}
\begin{proof}
Suppose that the capacities (\ref{E:capdif}) satisfy
$\rCap[\hdelta,\hbeta]\geq 3$,
$\rCap[\hdelta',\hbeta]\geq 3$. 

By the assumption $\hdelta'\in \rM_1^+$.
Conditions (\ref{E:purity}) imply that $\hdelta\notin \rM_2^+
\Rightarrow \hdelta\in \sethE\backslash \rM_2^+= \rM_1^+\sqcup \rM_3^{+}=\rM_1^+\sqcup \rM_3^{-}$(Lemma \ref{L:multstatM} item (\ref{I:sixmultstatM})). 

If $\hdelta\in\rM_1^+$ I define \[\{\hdelta_i\}:=[\hdelta,\hdelta']\cap \rM_1^+\] with the order induced from $<$.
Property (\ref{E:lessdot}) of $\{\hdelta_i\}$ follows from the same property of $\rM_1^+$, which obviously holds. By construction $\hdelta_i\ \notin \rM_2^+$, $\rCap[\hdelta_i,\hbeta]\geq \rCap[\hdelta',\hbeta]\geq 3$. Hence condition (\ref{E:purity}) is satisfied.

If 
$\hdelta\in \rM_3^{-}$,
 then, by Lemma \ref{L:multstatM}, item (\ref{I:m3}), there is a unique $\hdelta''\in \rM_1^+$ such that 
 $\hdelta\lessdot \hdelta''$. 
 I define the sequence $\{\hdelta_i\}$ to be 
 $\{\hdelta\}\cup([\hdelta'',\hdelta']\cap \rM_1^+) $.
 (\ref{E:lessdot}) is automatically enforced as before.

Suppose that $\rCap[\hdelta,\hbeta]\geq 3$ and
$\rCap[\hdelta',\hbeta]\leq 1$. I still can construct $\{\hdelta_i\}$. Let us choose the greatest $i$ such that $\rCap[\hdelta_i,\hbeta]=2$. The list of intervals with $\rCap[\hdelta,\hdelta']=2$ is given in (\ref{E:twointervalhigh}), (\ref{E:twointervalslow}). All the intervals $[\hdelta,\hdelta']$ from (\ref{E:twointervalhigh}) have $\hdelta'\in \rM_2^{-}$ and from (\ref{E:twointervalslow}) have $\hdelta\in \rM_2^{+}$. Then $\hdelta_i\in \rM_1^{+}\cap \rM_2^{+}=\emptyset$ or $\hbeta\in \rM_2^{-}$ (impossible by assumptions). Thus $\rCap[\hdelta_i,\hbeta]=2$ is ruled out in our setup.

If $\rCap[\hdelta,\hbeta]\leq 1$, then any maximal totally ordered subset of 
$[\hdelta,\hdelta']$ can be used as $\{\hdelta_i\}$.
\end{proof}
 \begin{corollary}
 Let $[\hdelta,\hbeta]\supset [\hdelta',\hbeta]$ be intervals satisfying (\ref{E:purity}) such that 
 $\hdelta'\in \rM_1^+$. 
 Then there is a sequence of Gorenstein algebras and surjective maps
\[\hA[\hdelta' ,\hbeta]=\hA[\hdelta_{n},\hbeta]\overset{\sfp_{n-1}}{\longleftarrow} \hA[\hdelta_{n-1},\hbeta]\overset{\sfp_{n-2}}{\longleftarrow} \cdots \overset{\sfp_1}{\longleftarrow} \hA[\hdelta_{1},\hbeta]={\hA[\hdelta,\hbeta]}\]
 such that $\hdelta_{i}\lessdot \hdelta_{i+1}$,
 $\dim_{Krull}\, \hA[\hdelta_i,\hbeta]=\dim_{Krull}\,\hA[\hdelta_1,\hbeta]-i+1$ and $\lambda^{\hdelta_{i}}\in \Ker\, \sfp_i$.
\end{corollary}

 \section{Digression to Cohen-Macaulay algebra}\label{S:CM}
In this section I collected some useful facts about Cohen-Macaulay and Gorenstein algebras. Though the subsequent sections rely strongly on the results presented here, the reader might want to skip this technical section at the first reading. 
\subsection{Definition of the Thom class $\Th$}
Let 
\begin{equation}\label{E:surjective}
\sfp:\hRo\to \hRt
\end{equation} be a surjective homomorphism of graded Cohen-Macaulay algebras. I assume that $\hRo$ and $\hRt$ are generated by elements in the first graded component as all the algebras that will appear in this paper. 
$\hRo$ and $\hRt$ are modules over the polynomial algebra $P$, whose first component $P_1$ is isomorphic to $\hRo_1$. More precisely, $\hRo=P/\fr_{\hRo}, {\hRt}=P/\fr_{\hRt}$. By Definition \ref{D:gorensteinserre}, the canonical modules $\omega_{\hRo}$ and $\omega_{\hRt}$ can be constructed by using the minimal free resolutions over $P$
\[\hRo \leftarrow F_0(\hRo)\leftarrow\cdots \leftarrow F_{s-\dim(\hRo)}(\hRo)\leftarrow0,\]
\begin{equation}\label{E:resolutionB}
\begin{split}
&\hRt\leftarrow F_0(\hRt)\leftarrow\cdots \leftarrow F_{s-\dim(\hRo)+ \codim}(\hRt)\leftarrow0, \\
& \codim(\hRo,\hRt):=\dim(\hRo)-\dim(\hRt).
\end{split}
\end{equation}
Here 
$\dim$ stands for the Krull dimension and 
$s=\dim_{\C} {\hRo}_1$.
I denote by $F^*$ the dual of a free finite rank $P$-module $F$. 
The cohomology of the dual complexes $F^{*i}(\hRo)$ and $F^{*i}(\hRt)$ are nonzero only in degrees $s-\dim(\hRo),s-\dim(\hRt)$ and coincide with the canonical modules $\omega_{\hRo}$ and $\omega_{\hRt}$. 
By construction $F^i(\omega_{\hRo})=F^{*(s-\dim(\hRo)-i)}(\hRo)$, $F^i(\omega_{\hRt})=F^{*(s-\dim(\hRt)-i)}(\hRt)$ are resolutions of the canonical modules.

The map of algebras $\sfp$ defines the map of $P$-modules. It induces the map of the free resolutions 
\begin{equation}\label{E:resmapind}
F^{\bullet}(\hRo)\overset{\sfp^{\bullet}}\to F^{\bullet}(\hRt)
\end{equation}
and the adjoint map between the dual complexes
\begin{equation}\label{E:classconstruction}
F^{\bullet}(\omega_{\hRt})[\codim]\overset{\sfp^{*\bullet}}\to F^{\bullet}(\omega_{\hRo}),
\end{equation} 
which I interpret as the element 
\begin{equation}\label{E:classt}
\Th_{\sfp}\in \Ext_{P}^{\codim}(\omega_{\hRt},\omega_{\hRo}).
\end{equation} 
In (\ref{E:classt}) $\hRt,\hRo$ appear in the opposite order than in (\ref{E:surjective}). This is why I call $\Th_{\sfp}$ a map in the "wrong direction".

\subsection{The map $\rmk$}\label{S:kos}
 The map $\rmk$ that will be described here, is another example of a "wrong direction" homomorphism. 
 
 Fix direct sum decomposition 
\[U+U'=V\]
and a generator $\rmk$ of $\Tor_{\dim U'}^{\C[U']}(\C,\C)\cong \Lambda^{\dim U'}(U')^{*}$. Let $M$ be a $\C[U]$-module. I will be interested in a map
 
\begin{equation}\label{E:kosmap}
 \rmk:\Tor_j^{\C[U]}(M,\C)\to \Tor_{j+\codim\, U}^{\C[V]}(M,\C)
\end{equation}

defined with a help of the K\"{u}nneth decomposition \[\Tor_k^{\C[V]}(M,\C)=\bigoplus_{k=i+j}\Tor_i^{\C[U]}(M,\C)\otimes \Tor_j^{\C[U']}(\C,\C).\]
The map is a composition of the embedding $\Tor_j^{\C[U]}(M,\C)\subset \Tor_j^{\C[V]}(M,\C)$ with the multiplication on $\rmk$-the generator of $\Tor^{\C[U']}_{\dim\, U'}(\C,\C)$. 
Construction extends in the standard way from the category of modules $Mod_{\C[U]}$ to the derived category $D(Mod_{\C[U]})$. 

\subsection{The general properties of $\Th$}
Besides the general properties of the Thom class, this section contains results of some elementary computations with $\Th$.

\begin{proposition}\label{P:composition}
Let ${\hRo}\overset{\sfp}\to {\hRt} \overset{\sfq}\to {\hRth}$ be surjective homomorphisms of Cohen-Macaulay graded algebras. The elements
$\Th_{\sfq}\in \Ext^{\codim({\hRth},\hRt)}(\omega_{\hRth},\omega_{\hRt}),\Th_{\sfp}\in \Ext^{\codim({\hRt},{\hRo})}(\omega_{\hRt},\omega_{\hRo}), \Th_{\sfq\circ \sfp}\in \Ext^{\codim({\hRth},{\hRo})}(\omega_{\hRth},\omega_{\hRo})$ satisfy
\[\Th_{\sfq\circ \sfp}=\Th_{\sfp}\cup \Th_{\sfq}.\]
\end{proposition}
\begin{proof}
Follows from functorial properties of dualization.
\end{proof}

The following theorem from \cite{Twenty} describes relation between canonical modules (Definition \ref{D:gorensteinserre}) for a pair Cohen-Macaulay graded
 rings. In the following, $\overline{x}M$ will stand for $\sum_{i=1}^m x_iM$ where $\overline{x}=(x_1,\dots,x_n)\subset {\hRo}$ is a finite sequence and $M$ is an $\hRo$-module. 
\begin{theorem}\label{T:Twenty}(\cite{Twenty} 11.35)
Let $\hRo$ be a Cohen-Macaulay local ring and $\overline x$ a regular sequence in $\hRo$. If $\omega_{\hRo}$ is a canonical module for $\hRo$, then $\hRo/(\overline{x})$ is Cohen-Macaulay and
\[\omega_{\hRo}/\overline{x}\omega_{\hRo}\cong \omega_{\hRo/(\overline{x})}.\]
\end{theorem}
The theorem obviously remain valid for graded rings.

The next two proposition describes $\Th_{\sfp}$ in some elementary situations. In the first proposition the homomorphism $\sfp$ is induced by the restriction on a hypersurface. In the second $\sfp$ is an embedding of algebras of a special kind. Maps ${\sfp}$ that will appear in this paper will be factored into such elementary maps.

\begin{proposition}\label{P:classofextreg}
\begin{enumerate}
\item {\hRo} regular homogeneous element $x$ in a graded Cohen-Macaulay algebra $\hRo$
defines the 
homomorphism $\sfp:{\hRo}\to {\hRt}={\hRo}/x{\hRo}$. Then $\Th_{\sfp}$ (\ref{E:classt}) is the class of extension
\[0\to\omega_{\hRo}\to\omega_{\hRo}\to\omega_{\hRt}\to0\]
of the dualizing modules.
The map $\omega_{\hRo}\to \omega_{\hRo}$
is multiplication on $x$.
\item If $\hRo$ is Gorenstein and $\bar{x}=(x_1,\dots,x_n)$ is regular, then ${\hRo}/\bar{x}{\hRo}$ is Gorenstein.
\end{enumerate}
\end{proposition}
\begin{proof}
\begin{enumerate}
\item By abuse of notations, I denote the lift of $x\in {\hRo}$ to $P$ by the same symbol. The operator of multiplication on $x$ in $\hRo$ will be denoted by $\sfs$.
The complex (\ref{E:resolutionB}) is quasi isomorphic to the cone $C^{\bullet}(\sfs^{\bullet})$.
 Dualization transforms the canonical map $F^{\bullet}(\hRo)\to C^{\bullet}(\sfs^{\bullet})$ to $\grdif:C^{\bullet}(\sfs^{*\bullet})\to F^{\bullet}(\omega_{\hRo})[1]$, where $\sfs^{*\bullet}$ is the operator of multiplication on $x$ in $F^{\bullet}(\omega_{\hRo})$. The map $\grdif$ is the boundary map in the distinguished triangle 
\[F^{\bullet}(\omega_{\hRo})\to F^{\bullet}(\omega_{\hRo}) \to C^{\bullet}(\sfs^{*\bullet})\to F^{\bullet}(\omega_{\hRo})[1].\] By definition, the element \[\grdif\in \Hom_P(C^{\bullet}(\sfs^{*\bullet}),F^{\bullet}(\omega_{\hRo})[1] )=\Ext_P^1(\omega_{\hRt},\omega_{\hRo})\] is equal to $\Th_{\sfp}$. In the last equality I used that $\hRo$ and $\hRt$ are Cohen-Macaulay. I also used Theorem \ref{T:Twenty}. The triangle is quasi isomorphic to 
\[ \omega_{\hRo}\to\omega_{\hRo}\to\omega_{\hRt}\overset{\grdif}\to\omega_{\hRo}[1].\] In such description interpretation of $\grdif$ as a class of extension becomes tautologous.
\item Follows from Proposition 3.1.19 \cite{BrunsandHerzog}.
\end{enumerate}
\end{proof}

\begin{proposition}\label{P:extmaintheorem}
Let $\hRo$ be a graded Gorenstein algebra. Suppose that annihilator $\Ann(x)\neq 0$ for a homogenous $x\in {\hRo}$. Suppose also that ${\hRt}={\hRo}/\Ann(x)$ is Gorenstein, ${\hRth}={\hRo}/(x)$ is Cohen-Macaulay, and $ \dim(\hRo)=\dim(\hRt)=\dim({\hRth})$. 
 The map $\sfs:{\hRo}\to {\hRo}, a\to ax$ induces a homomorphism of $\hRo$-modules which is a part of the exact sequence 
\begin{equation}\label{E:Binclusionexact}
0\to {\hRt}\overset{\sfs}\to {\hRo}\to {\hRth}\to 0.
\end{equation}
 Denote by $\sfp:{\hRo}\to {\hRt}={\hRo}/\Ann(x)$ the homomorphism of algebras.

\begin{enumerate}
\item Under such assumptions, there is a short exact sequence of canonical modules
\begin{equation}\label{E:omegaexact}
0\leftarrow \omega_{\hRt}\leftarrow \omega_{\hRo}\leftarrow \omega_{\hRth}\leftarrow 0,
\end{equation}
which is isomorphic to 
\begin{equation}\label{E:omegaexacttwo}
0\leftarrow {\hRt}\overset{\sfp}{\leftarrow} {\hRo}\leftarrow \Ann(x) \leftarrow 0.
\end{equation}
\item Conversely, the short exact sequence of $\hRo$-modules (\ref{E:omegaexacttwo}) gives rise to the short exact sequence of the modules
\[0\to \omega_{\hRt}\overset{\Th_{\sfp}}{\to} \omega_{\hRo}\to \omega_{\hRo}/\omega_{\hRt}\to 0\]
isomorphic to (\ref{E:Binclusionexact}).
\end{enumerate}
\end{proposition}
\begin{proof}
Under isomorphisms $\hRo=P/\fr_{\hRo}$, ${\hRt}=P/\fr_{\hRt}$, ${\hRth}=P/\fr_{\hRth}$ algebras ${\hRo},{\hRt}$ and ${\hRth}$ become modules over $P$.
By using the fact that $\omega_X=\Ext^{s-\dim(X)}_P(X,P)$ for Cohen-Macaulay $X={\hRo},{\hRt},{\hRth}$, I derive (\ref{E:omegaexact}) from the segment of the long exact sequence of $\Ext$ groups. 

The map $\omega_{\hRo}\to \omega_{\hRt}$ in (\ref{E:omegaexact}) is compatible with the grading. It must, in the view of identifications $\omega_{\hRo}\cong {\hRo}$ and $\omega_{\hRt}\cong {\hRt}$, coincide (up to a multiplicative constant $c\in \C$) with $\sfp$. Thus $\omega_{\hRth}\cong \Ann(x)$.

Maps in (\ref{E:Binclusionexact}) can be extended to maps of the minimal free $P$-resolutions 
\[0\to F^{\bullet}(\hRt)\overset{\sfs^{\bullet}}\to F^{\bullet}(\hRo)\to F^{\bullet}({\hRth})\to 0.\]
 As $\hRo$ and $\hRt$ are Gorenstein $F^{*\bullet}(\hRt)$ and $ F^{*\bullet}(\hRo)$ are resolutions of $\hRt$ and $\hRo$ respectively. The map of complexes $\sfs^{*\bullet}$ in 
\[ 0\leftarrow F^{*\bullet}(\hRt)\overset{\sfs^{*\bullet}}\leftarrow F^{*\bullet}(\hRo)\leftarrow F^{*\bullet}({\hRth})\leftarrow 0\]
up to a constant should be equal to $\sfp^{\bullet}$. By symmetry $\sfp^{*\bullet}=\sfs^{\bullet}$, which proves the last statement.
\end{proof}

 Let us fix isomorphisms 
\begin{equation}\label{E:duality}
{\hRo}\cong \omega_{\hRo}
\end{equation}
 ${\hRt}\cong \omega_{\hRt}$ for Gorenstein $R,S$ and interpret $\Th_{\sfp} $ (\ref{E:classt}) as 
\begin{equation}\label{E:thasext}
\Th_{\sfp} \in \Ext_{P}^{\codim}({\hRt},{\hRo}).
\end{equation}
 \begin{proposition}\label{P;rightchoice}
Let a pair of algebras ${\hRo},{\hRt}$ satisfy assumptions of Proposition \ref{P:classofextreg} or Proposition \ref{P:extmaintheorem}. For any choice of isomorphism $\sfq:{\hRo}\cong \omega_{\hRo}$, there is a unique isomorphism $\sfq':{\hRt}\cong \omega_{\hRt}$ which is a part of a commutative diagram in the derived category of $\hRo$-modules
\[\begin{CD}
{\hRt}@>\sfq'>>\omega_{{\hRt}} \\
@VV\Th_{\sfq} V @VV\Th_{\sfq'} V \\
 {\hRo}[\dim(\hRo)-\dim(\hRt)] @>\sfq>> \omega_{{\hRo}}[\dim(\hRo)-\dim(\hRt)].
\end{CD}\]
The choice of an isomorphism ${\hRt}\cong \omega_{\hRt}$ determines an isomorphism $\hRo\cong \omega_{\hRo}$ with the same compatibility properties.
\end{proposition}
\begin{proof}
Suppose ${\hRo},{\hRt}$ satisfy conditions of Proposition \ref{P:classofextreg}. The isomorphism (\ref{E:duality})
 induces an isomorphism 
\begin{equation}\label{E:seqiso}
{\hRt}={\hRo}\underset{{\hRo}}{\otimes}{\hRt}\cong \omega_{\hRo}\underset{{\hRo}}{\otimes}{\hRt}=\omega_{\hRt}
\end{equation} 
with the required properties because it identifies $0\to {\hRo}\to {\hRo}\to {\hRt}\to0$ with $0\to \omega_{{\hRo}}\to \omega_{{\hRo}}\to \omega_{{\hRt}}\to0$. 

The isomorphism (\ref{E:duality}) in general is defined up to a multiplicative constant. If I alternatively fix an isomorphism ${\hRt}\cong \omega_{\hRt}$, I can determine the constant by applying (\ref{E:seqiso}) to (\ref{E:duality}). 

Suppose that now ${\hRo},{\hRt}$ satisfy conditions of Proposition \ref{P:extmaintheorem}. I can read off the isomorphism ${\hRt}\cong \omega_{\hRt}$ as the isomorphism of submodules in $\hRo$ and $\omega_{\hRo}$ induced by (\ref{E:duality}). Conversely, if I am given ${\hRt}\cong \omega_{\hRt}$, I can fix the ambiguity constant in (\ref{E:duality}) by comparing restriction of (\ref{E:duality}) on ${\hRt}\subset {\hRo}, \omega_{\hRt}\subset \omega_{\hRo}$ with ${\hRt}\cong \omega_{\hRt}$.
\end{proof}

Let $\fb\subset {\hRt}$ be a graded ideal, $\fc$ its preimage in $P$. Composition with $\Th_{\sfp} $ (\ref{E:classt}),
defines the map of local cohomology
\begin{equation}\label{E:THdef}
 \Th_{{\hRt}\leftarrow {\hRo}}:H_{\fb}^{i}(\hRt)\cong \lim_{\overset{n}{\longrightarrow}}\Ext_P^i(P/\fc^n,\hRt)\overset{\Th_{\sfp}}\to \lim_{\overset{n}{\longrightarrow}}\Ext_P^{i+\codim}(P/\fc^n,{\hRo})\cong H_{\sf\sfp^{-1}(\fb)}^{\codim+i}(\hRo).
 \end{equation}
The isomorphisms $\cong$ have been justified in the Theorem 7.11 \cite{Twenty}.
\begin{remark}\label{R:nullaction}
The map $\Th_{\sfp}$ commutes with the action of the algebra $P$. Moreover the action of $P$ on $\Ext_P^i(P/\fc^n,\hRt)$ factors through $\hRt$. Let $\fr$ be the kernel of projection map ${\hRo}\to {\hRt}$. Then the image of $\Th_{{\hRt}\leftarrow {\hRo}}$ in $H_{\sfp^{-1}(\fb)}^{\codim+i}(\hRo)$ is annihilated by $\fr$.
\end{remark}

In our applications, I will also be interested in the cohomology of the double complex (see Appendix \ref{S:local} for notations ) \[\fT_{\hRo}^{i}(M):=\bigoplus_{k=j-i} B_i(\Gamma_{\fa}I_{\hRo}^{j}(M),\{x^s\})\] where $I^{\bullet}(M)$ is an injective resolution, $\{x^s\}$ are the homogeneous degree one generators of $\hRo$. Presentation $\hRo=P/\fr_{\hRo}$ and isomorphism (\ref{E:THdef}) lets to identify cohomology $H^i(\fT_{R}(M))$ with the cohomology of
\begin{equation}\label{E:ferm}
\fT_{P}^{k}(\fc,M)=\bigoplus_{k=j-i} B_i(\Gamma_{\fc}I_{P}^{j}(M),\{x^s\}).
\end{equation}
Now $\{x^s\}$ are the generators of $P$ and $\fc$ is a preimage of $\fa$ in $P$. Let $\mathrm{R}\Gamma_{\fc}$ be the right derived functor of $\Gamma_{\fc}$ and $\underset{L}{\otimes}\,\C$ be the left derived functor of $\otimes\C$ in the bounded from the left derived category of $D(Mod_{P})$. The complex (\ref{E:ferm}) is a representative of $\mathrm{R}\Gamma_{\fc}(M)\underset{L}{\otimes}\C$.

A surjective graded homomorphism $\hRo=P_1/\fr_{\hRo}\overset{\sfp}{\to} \hRt=P_2/\fr_{\hRt}$ of minimal presentations can be lifted to a surjective graded map $\sfp:P_1\to P_2$ of polynomial algebras. Let us choose an isomorphism $P_1\cong P_2\otimes P'_2$. Following Section \ref{S:kos} I use construction (\ref{E:kosmap}) to define
\[\rmk:H^k\fT_{P_2}(\hRt)\to H^{k-(\dim P_1-\dim P_2)}\fT_{P_1}(\hRt).\]
Note that now I am using cohomological grading, which is responsible for the negative shift.
The map \[\Th_{\sfp}:H^{k}\fT_{P_1}(\hRt)\to H^{k+\dim\hRo-\dim\hRt}\fT_{P_1}(\hRo)\]
is induced by (\ref{E:thasext}) if I think about it as a map in derived category.
A composition of the last two maps 
\begin{equation}\label{E:thprime}
H^k\fT_{P_2}(\hRt)\to H^{k+(\dim\hRo-\dim\hRt)-(\dim P_1-\dim P_2)}\fT_{P_1}(\hRo)
\end{equation}
will be denoted by $\Th'_{\sfp}$.

\subsection{Thom class and the group action}\label{S:groupaction}
In our application, the algebras will be equipped with the action of some symmetry group. In this section, I will discuss commutation relations of $\Th$ with this group action.
The reader may wish to consult \cite{Thomason} about the general theory of resolutions of sheaves equivariant with respect to a group scheme action. In our application, when the group $\Aut$ is reductive and the ground scheme is $Spec\, \C$ much of the theory becomes trivial. 
 If we assume that the map (\ref{E:surjective}) 
 commutes with the $\Aut$-action, the general theory immediately implies that the map of complexes (\ref{E:resmapind}) can be chosen to be $\Aut$-equivariant and $\Th_{\sfp}$ (\ref{E:classt}) is $\Aut$-invariant.
 \begin{equation}\label{E:chiadef}
 \begin{split}
&\text{ In the graded Gorenstein case, the generating space }\\
&\text{of the canonical module $\omega_{\hRo}$ is a one-dimensional representation $\chi'_{\hRo}$ of $\Aut$.}
\end{split}
 \end{equation} 
Note that $\chi'_{{\hRo},P}$ depends on the presentations $\hRo=P/\fr_{\hRo}$. To make the definition of $\chi'_{\hRo}$ less presentation-dependent, I define
\begin{equation}\label{E:chireddef}
\chi_{{\hRo},P}:=\chi'_{{\hRo},P}/\det P, 
\end{equation}
\begin{equation}\label{E:det}
\det P:=\chi'_{\C,P}.
\end{equation}
\begin{proposition}
The $P$-module structure on the algebra $\hRo=P/\fr$ can be extended to the structure of $P\otimes P'$-module, where variables $x_1,\dots,x_n\in P'=\C[x_1,\dots,x_n]$ act trivially on $\hRo$.
Then 
\begin{equation}\label{E:chindep}
\chi_{{\hRo},P}=\chi_{{\hRo},P\otimes P'}.
\end{equation}
\end{proposition}
\begin{proof}
The $P\otimes P'$-resolution of $\hRo$ is the tensor product of the $P$-resolution $F_{\bullet}(\hRo)$ and the Koszul complex $K(P')$. Thus $\chi'_{{\hRo},P\otimes P'}=\chi'_{{\hRo},P}\chi'_{\C,P'}$, which implies (\ref{E:chindep}).
\end{proof}

 From the point of view of representation theory, we have an isomorphism \[\omega_{\hRo}={\hRo}\otimes \chi_{\hRo}'.\]
\begin{proposition}\label{R:equiv}
Let us assume that $p:R\to S$ is a surjective map of degree-one generated Gorenstein algebras and $x\in \hRo$ has degree one.
\begin{enumerate}
\item \label{I:equiv1}After identification of $\Ext_{P}^{\codim}(\omega_{\hRt},\omega_{\hRo})$ with $\Ext_{P}^{\codim}({\hRt},{\hRo})$ the group $\Aut$ acts on $\Th_{\sfp}$ through the character 
\begin{equation}\label{E:absrel}
\chi_{{\hRt}\leftarrow {\hRo}}:=
\chi_{\hRt}\chi_{\hRo}^{-1}.
\end{equation}
\item \label{I:equiv2}
If \begin{equation}\label{E:xeq}gx=\chi_x(g)x\end{equation} for degree-one $x$ from Proposition \ref{P:classofextreg}, then $\chi_{{\hRt}\leftarrow {\hRo}}=\chi^{-1}_x.$
\item \label{I:equiv3}
If $x$  from Proposition \ref{P:extmaintheorem}, $\deg x=1$ satisfies (\ref{E:xeq}), then $\chi_{{\hRt}\leftarrow {\hRo}}=\chi_x.$
\end{enumerate}
\end{proposition}
\begin{proof}
Item (\ref{I:equiv1}) is obvious. 
To prove item (\ref{I:equiv2}) I choose  a minimal free $P$-resolution $\hRo\ot F_{\bullet}$. I also assume that the maps commute with $\Aut$-action.  Under my assumptions  $F_{\bullet}/(x)$ is a minimal $P/(x)$ resolution of $S$. 
Let $c$ be the generator of the top degree group $F_d$. $c$ also generated  $F_d/(x)$. From this I conclude that $\chi'_S=\chi'_R$. Also  $\det(P/(x))=\chi_x\det(P)$ and  
\[\chi_S=\chi'_S/\det(P/(x))=\chi^{-1}_x\chi'_R/\det(P)=\chi^{-1}_x\chi_R\]
Item \ref{I:equiv3}  follows from Proposition \ref{P:extmaintheorem}.
\end{proof}

\subsection{An interpretation of the local cohomology as a $\Tor$ functor}
In this section, I give a reformulation of local cohomology in terms of $\Tor$-functors. Such reformulation has an advantage because projective resolutions are much more accessible than their injective analogs.
\begin{proposition}\label{P:localtorisom}
\begin{enumerate}
\item Let $\fc\subset {\hRo}$ be an ideal in a finitely generated graded commutative algebra over $\C$. Fix a polynomial algebra $P=\C[Y\oplus Z]$ and a graded surjective homomorphism
\[\sfs:P\to {\hRo}\]
such that $\sfs(ZP)=\fc$. 
 The homomorphism $\sfs$ makes $\hRo$ into a $P$-module. Let 
\begin{equation}\label{E:minres}
\{F_i=P\otimes V_i| i=0,\dots,d\}
\end{equation}
 be its minimal $P$-resolution. Then in the notations of Appendix \ref{S:Dmodulenoteconstr} , there is an isomorphism
\begin{equation}\label{E:isolocaltor}
H^i_{\fc}(\hRo)\cong H_{\dim Z-i}(F\underset{P}{\otimes} \C[Y]\otimes \C[Z]^{-1} ).
\end{equation}
In particular 
\[H^i_{\fc}(\hRo)\cong\Tor_{\dim Z-i}^P(\hRo,\C[Y]\otimes \C[Z]^{-1}).\]
\item Let $\Aut$ be an algebraic reductive group of automorphisms of $\hRo$ and $P$. Suppose that $\sfs$ is $\Aut$-map 
and $Z\subset Y+Z$ is a pair of $\Aut$-invariant subspaces inside of $P$. Then (\ref{E:isolocaltor}) induces an isomorphism of $\Aut$-representations.

\end{enumerate} 
\end{proposition}
\begin{proof}
I will use the complex $K(\hRo,\sfs(\bar{x}),n)$ ($\bar{x}$ is a basis for $Z$), justified by Proposition \ref{P:locKoszul}, for computation of the local cohomology.
For this purpose I replace $\hRo$ with $F_{\bullet}$. This way I get a $F_i$-valued Koszul complex $F_i\otimes \Lambda^j[Z]$. It is a bi-complex.
The canonical map
\[\underset{\overset{n}{\longrightarrow}}\lim K^{\dim Z}(F_i,\bar{x},n)\to \C[Y]\otimes \C[Z]^{-1}\underset{P}{\otimes} F_i=\C[Y]\otimes \C[Z]^{-1}{\otimes} V_i\]
( an obvious generalization of (\ref{E:quasi1}) ) defines a map of complexes
\begin{equation}\label{E:isofisrtspage}
 B^k=\bigoplus_{k=j-i} \underset{\overset{n}{\longrightarrow}}\lim K^{j}(F_i,\bar{x},n)\to \C[Y]\otimes \C[Z]^{-1}{\otimes} V_{\dim Z-k}.
 \end{equation}
The spectral sequence of $B$ collapses in the first page, 
 which contains only one row of groups
\[\underset{\overset{n}{\longrightarrow}}\lim H^{\dim Z} K(F_i,\bar{x},n)=
\C[Y]\otimes \C[Z]^{-1}{\otimes} V_i\]
(see Example \ref{E:polyloc}).
It proves that (\ref{E:isofisrtspage}) is a quasi-isomorphism.

\begin{lemma}\label{E:eqres}
In the above assumptions $\hRo$ admits a finite $\Aut$-equivariant $P$-resolution. Moreover any $\Aut$- equivariant map $M\to N$ between   finitely generated $\hRo$ modules admits an $\Aut$-equivariant lift to a map between  $P$-resolutions.
\end{lemma}
\begin{proof}
It can be proven directly or deduced from the  infinitely stronger results \cite{Thomason} 
\end{proof} 

If we use $\Ext$ definition (\ref{E:loccohextdef}) of local cohomology the previous proof can be repeated. The advantage of (\ref{E:loccohextdef}) is that the limiting groups are representations of the symmetry group of the pair $\fc\subset {\hRo}$. This automatically implies that all the identifications used in my proof are compatible with the $\Aut$-action. Hence (\ref{E:isolocaltor}) is a $\Aut$-map.

\end{proof}
\begin{remark}\label{R:groupaction}
The action of group of automorphisms $\Aut$ of the algebra $\hRo$ typically lifts to a formal group action on $H_{\fc}^i(\hRo)$. In zero characteristics, it is encoded by the Lie algebra $\Lie(G)$ representation in $H_{\fc}^i(\hRo)$. If, however, a subgroup $H\subset G$ is a symmetry of the pair $\fc\subset {\hRo}$, then restriction of representation of $\Lie(G)$ on $\Lie(H)$ becomes integrable (see \cite{Kempf} for details). 

\end{remark}

\subsection{Some remarks on the Poincar\'{e} pairing}
The class of Gorenstein algebras is a source of various cohomology groups with the nondegenerate Poincar\'{e} pairing. All the results in this paper are the expansions of this observation. This section contains some general results concerning such pairings. It starts with the theorem of Avramov and Golod. 

In this section, I assume that $\hRo=P/\fr$,$\fr$ is graded, $\hRo$ is Gorenstein, $P=\C[U]$. 
 \begin{proposition}\label{C:PoincareGen}
The cohomology of the Koszul complex $B(\hRo)$ (see notation convention \ref{C:simpl}) are nonzero in the range from $0$ to $d$
\begin{equation}\label{E:ddef}
d:= \dim_{\C} U-\dim R,
\end{equation}
$H_d(B)\cong \C$. The product 
\[V_i \otimes V_{d-i}\to V_{d}, V_i=H_i(B)\]
 defines a nondegenerate pairing. In fact, nondegeneracy of the pairing is equivalent to Gorenstein property of $\hRo$. 
\begin{equation}\label{E:toriso}
H_i(B)\cong\Tor_i^P(\hRo,\C)
\end{equation}

\end{proposition}
\begin{proof}
The proof is a trivial application of the main theorem in \cite{AvramovGolod}.
The only additional work that is required, and which I omit, is a consistent insertion of the word "graded" in the original proof because Avramov and Golod proved their theorem 
for Gorenstein local rings. Isomorphism $H_i(B)\cong\Tor_i^P(\hRo,\C)$ immediately follows if I resolve $\C$ with the standard Koszul resolution and apply ${\hRo}\underset{P}{\otimes}?$.
\end{proof}

Multiplication ${\hRo}\otimes {\hRo}\to {\hRo}$ induces a $P$-bilinear map on the minimal resolution (\ref{E:minres})
\begin{equation}\label{E:resprod}
\times:F_{\bullet}\otimes F_{\bullet}\to F_{\bullet}.
\end{equation} 
\begin{remark}
 The induced map on $V_i=F_i\underset{P}{\otimes} \C$:
$V_i\otimes V_j\to V_{i+j}$
is identified by (\ref{E:toriso}) with the product $H_i(B)\otimes H_j(B)\to H_{i+j}(B)$.
 \end{remark}
 By definition (\ref{E:minres}), $V_i$ is a linear subspace $F_i$. It is not true in general that $\times(V_i\otimes V_j)$ contains in $V_{i+j}\subset F_{i+j}$.
 Suppose, however, that $i+j=d$ (\ref{E:ddef}). Then 
 \begin{equation}\label{E:muisodef}
 \begin{split}
 &\times(V_i\otimes V_j)\subset V_d\subset F_d\\
 &V_d\overset{\mu}{\widetilde{\longrightarrow}} \C.
 \end{split}
 \end{equation}
 This immediately follows from the $\C^{\times}$-grading considerations and the fact that $F_i$ are a free graded cyclic modules. For this reason the pairing 
 \begin{equation}\label{E:pairingchains}
 \begin{split}
&\times:{V}_i\otimes {V}_{d-i}\to {V}_{d}\\
&\text{ defined on the level of chains coincides with the nondegenerate}\\
&\text{(Proposition \ref{C:Poincare}) cohomological pairing.} 
\end{split} \end{equation}

I will use the following proposition as a rudimental dictionary between the languages of local algebra and $\Tor$-functors.
\begin{proposition} \label{P:pairing}
 $\fc,\fe\subset {\hRo}$ be two graded ideals such that $\sfs(ZP)=\fc$, $\sfs(WP)=\fe$,$U=Y \oplus Z\oplus W$. Under identification (\ref{E:isolocaltor}), the product map 
 
\begin{equation}\label{E:product}
 H^i_{\fc}(\hRo)\otimes H^{i}_{\fe}(\hRo)\to H^{i+j}_{\fc+\fe}(\hRo)
\end{equation}
 coincides with the map
\[\begin{split}
&\Tor^P_{\dim Z-i}(\hRo,\C[Y\oplus W]\otimes \C[Z]^{-1})\otimes \Tor^P_{\dim W-j}(\hRo,\C[Y\oplus Z]\otimes \C[W]^{-1})\to \\
&\to \Tor^P_{\dim Z+\dim W-i-j}(\hRo,\C[Y]\otimes\C[ Z\oplus W]^{-1})
.\end{split}
\]
It is induced by (cf. (\ref{E:prodelem}))
\begin{equation}\label{E:multsimple}
\begin{split}
& \times: \left(V_i\otimes \C[Y\oplus W]\otimes \C[Z]^{-1}\right)\otimes \left(V_j\otimes\C[Y\oplus Z]\otimes \C[W]^{-1}\right) \to \\
&V_{i+j}\otimes\C[Y]\otimes\C[ Z\oplus W]^{-1}
.\end{split}
\end{equation}

\end{proposition}
\begin{proof}
Fix a basis $\overline{x}$ for $Z$ and $\overline{y}$ for $W$. An isomorphism 
\[K(P,\bar{x},n)\otimes K(P,\bar{y},n)\to K(P,(\bar{x},\bar{y}),n)\] 
can be used to define the product 
\begin{equation}\label{E:prodinterm}
\times:K(F_{\bullet},\bar{x},n)\otimes K(F_{\bullet},\bar{y},n)\to K(F_{\bullet},(\bar{x},\bar{y}),n).
\end{equation}

By the definition of local cohomology which uses Koszul complexes Appendix \ref{S:local}, the map (\ref{E:prodinterm}) induces (\ref{E:product}). The map (\ref{E:isofisrtspage}) gives identification of (\ref{E:prodinterm}) with (\ref{E:multsimple}).
\end{proof}
\section{Finite approximations to the space of states}\label{S:spaceofstates}
In this section, the cohomology group $H^{i}_{\cZ[(0)^0,\hdelta']}(\cZ[\hdelta,\hdelta'],\O)$, the building block for the space of states, will be examined in some details.
 
 $\cZ[\hdelta,\hdelta']$ is an affine scheme. For such schemes most of the sheaf-theoretic constructions reduce to constructions from the commutative algebra. 

Here is the precise statement concerning local cohomology of sheaves. Fix a module $M$ over a ring $\hRo$. Denote by $\widetilde{M}$ the corresponding quasi-coherent sheaf on $Spec\, {\hRo}$. 
\begin{theorem}(\cite{Twenty} Theorem 12.47). Let $\fc$ be an ideal in a Noetherian ring $\hRo$. Denote $Spec\, {\hRo}$ by $X$
and the closed set $V (\fc)\subset X$ by $Z$.
 For each $\hRo$-module $M$ and an integer $j$ \[H_{\fc}^j (M ) \cong H_{Z}^j ( X , \widetilde{M}).\]
\end{theorem}
From this theorem I infer that 
\[H^{i}_{\cZ[(0)^0,\hdelta']}(\cZ[\hdelta,\hdelta'],\O)=H^{i}_{\fa}[\hdelta,\hdelta']\]
where 
\[H^{i}_{\fa}[\hdelta,\hdelta']:=H^{i}_{\fa}(\hA[\hdelta,\hdelta']),\quad \fa=\fri[\hdelta,(1)^{-1}]\]
 (see Definition \ref{D:algAdelta} for notations). 

\begin{remark}\label{R:groupactionloccoh}
By Remarks \ref{R:symmetries} and  \ref{R:groupaction} the group $\C^{\times}\times\Spin(10)$ acts on $H^{i}_{\fa}[(0)^N,(1)^{N'}], N<N'$.
\end{remark}

Let $\hA$ be the algebra based on $[\hdelta,\hdelta']$. 

\begin{convention}{\rm
It is time to extend Convention \ref{C:mainconvention} slightly and add to the list of abbreviations $P^{-1}=\C[\hdelta,\hdelta']^{-1}$ (see Appendix \ref{S:Dmodulenoteconstr} equation \ref{E:inverse} for explanation of the notation).
}\end{convention}

The pairing (\ref{E:pairinggen}) in my finite algebraic setup becomes
\begin{equation}\label{E:multmap}
\begin{split}
&\sfm:H_{\fa}^i(\hA)\otimes H_{\fb}^j(\hA)\to H_{\fa+\fb}^{i+j}(\hA) \cong H^{i+j}_{\fm}(\hA)\\
&\bigoplus_{i}H^{i}_{\fm}(\hA)\overset{\sfres_{\hA}}{\longrightarrow}\C.
\end{split}
\end{equation}
The ideal 
\[\fb=\fri[(0)^0,\hdelta']\]
defines the closed subscheme $\cZ[\hdelta,(1)^{-1}]\subset \cZ$. The sum \[\fm:=\fa+\fb\] is the maximal homogeneous ideal in $\hA$. 
Ideal
\[\fa'=\fri[(0)^1, \hdelta'].\]
will be used for defining field-antifield pairing. For  technical needs I will  be using ideals
\[\ff=\fri[\hdelta,\hdelta']^{\leq-1},\quad \ff'=\fri[\hdelta,\hdelta']^{\geq 0}\]
and the pairing
\begin{equation}\label{E:multmapi}
\begin{split}
&\sfm:H_{\fri}^i(\hA)\otimes H_{\fri'}^j(\hA)\to H_{\fri+\fri'}^{i+j}(\hA) \cong H^{i+j}_{\fm}(\hA)\\
&\bigoplus_{i}H^{i}_{\fm}(\hA)\overset{\sfres_{\hA}}{\longrightarrow}\C.
\end{split}
\end{equation}

Here are some observations about $H^{i}_{\fa}(\hA)$. The maps (\ref{E:tisomorphism}), (\ref{E:mapsintr}) induce isomorphisms
\begin{equation}\label{E:sigmaTcoh}
\begin{split}
&\shift:H^{i}_{\fa}[\hdelta,\hdelta']\cong H^{i}_{\shift(\fa)}[\shift(\hdelta),\shift(\hdelta')]\\
&\reflection: H^{i}_{\fa}[\hdelta,\hdelta']\cong H^{i}_{\fa'}[\reflection(\hdelta'),\reflection(\hdelta)]\cong H^{i}_{\fb}[\shift^{-1}\reflection(\hdelta'),\shift^{-1}\reflection(\hdelta)].
\end{split}
\end{equation}
I used that $\reflection(\fa)=\fa'$.

The simplest question that might be asked about the pairing (\ref{E:multmap}) are:
\begin{enumerate}
 \item \label{I:q1} What is "size" of the groups $H^{i}_{\fm}(\hA)$. 
\item \label{I:q2} The space $\cZ$ is equipped with the group action, that multiplies the functional $\sfres$ on the character. What is this character? 
\item \label{I:q3} What properties of $\cZ$ and its subschemes ensure nondegeneracy of the pairing $\sfres(\sfm(a,b))$? 
\end{enumerate}

There are gradings on $\hA$ and the derived objects associated with $\Aut$ (\ref{E:symmstries0}).
\begin{equation}\label{E:cdeg}
\begin{split}
&\text{
The $\C^{\times}$-weight of a homogeneous element 
will be denoted by $\deg_{\C^{\times}}$. }\\
&\text{I will also refer to $\C^{\times}$-weight decomposition as to $\C^{\times}$-grading. }\\
&\text{The $\bT$-weight will be denoted by $\deg_{\bT}$ and there will be the $\bT$-grading.}
\end{split}
\end{equation}
Let us introduce notations that we will use throughout the paper.

The space of formal linear combinations of elements $\{\lambda^{\halpha}| \halpha \in \setN\}, \setN\subset \sethE$ with coefficients in $\C$ will be denoted by $\Span{\setN}$. To simplify the exposition,
I equip $\Span{\setN}$ with the dot product. 
\begin{equation}\label{E:dotprdef}
\text{By definition, $\{\lambda^{\halpha}\}$ is an orthonormal basis for $\Span{\setN}$.}
\end{equation}
By abuse of notations, the linear span of the set
\begin{equation}\label{E:regseqgenverynew}
\begin{split}
& \alt \setN :=\left\{
\lambda_-^i= \lambda^{\halpha^+}- \lambda^{\halpha^-}| \rho(\halpha^{\pm})=i\in \Z ,\halpha^{\pm}\in \setN
\right\}
\end{split}
\end{equation}
will be also denoted by $\alt \setN $.
Introduce a notation for the quotient:
\begin{equation}\label{E:regconstr}
\rReg \setN:=\Span{ \setN}/\alt(\setN)\cong \alt(\setN)^{\perp}=\Span{\lambda_+^i= \sum_{\rho(\halpha)=i,\halpha\in \setN }\lambda^{\halpha}|i\in \rho(\setN) }.
\end{equation}
Inclusion $\setN\subset \setN'$ defines a linear inclusion $\mathsf{inc}:\Span{ \setN}\subset \Span{ \setN}$. Let $\mathsf{pr}$ be the orthogonal projection $\Span{ \setN}\to \rReg \setN$.
The maps $\mathsf{inc}$ and $\mathsf{pr}$ define inclusions
\begin{equation}\label{E:functorialpr}
\begin{split}
&\mathsf{inc}:\alt \setN\to \alt \setN',\\
&\mathsf{pr}\circ \mathsf{inc}: \rReg \setN\to \rReg \setN' \\
&\text{ and }\Span{ \setN}= \alt \setN+\rReg \setN\overset{\inclusiono}{\to} \alt \setN'+\rReg \setN'=\Span{\setN'},\\
& \inclusiono:=\mathsf{inc}\oplus \mathsf{pr}\circ \mathsf{inc}.
\end{split}
\end{equation}
I have to alert the reader that the map $\inclusiono:\Span{ \setN}\to \Span{ \setN'}$ in general {\it is not} equal to $\mathsf{inc}:\Span{ \setN}\to \Span{ \setN'}$
(c.f. Example (\ref{EX:regalt}) item (\ref{I:spanexample})).

 Throughout this section, I will assume that all the intervals, if not mentioned otherwise, satisfy (\ref{E:purity}). 
 In particular, $\hA$ will be always assumed Gorenstein if not mentioned otherwise.

\subsection{The structure of the groups $H^{i}_{\fm}({\hA})$}
In this section I use notations from Section \ref{S:groupaction}. 
Let $\hA$ be an algebra based on the interval $[\hdelta,\hdelta']$.

I start my analysis of the structure of the group $H^{i}_{\fm}({\hA})$ with the answer on the question (\ref{I:q1}) page (\pageref{I:q1}). 
 \begin{proposition}\label{E:diality}
 \begin{enumerate}
\item \[H_{\fm}^i({\hA})=\{0\}\text{ if }i\neq \rk(\hdelta,\hdelta')+1,\]
\begin{equation}\label{E:loccohnew}
H_{\fm}^i({\hA})\cong \bigoplus_{i\geq 0} \hA^*_i\text{ if } i=\rk(\hdelta,\hdelta')+1.
\end{equation}
There is the integer $a\in \Z_{>0}$ (the a-invariant of $\hA$) such that the $\C^{\times}$-grading of $\hA^*_i$ is equal to $-a-i$.
\item The previous statement also holds for $\hA[\hdelta,(1)^{-1}]^{\leq-1}$.
\end{enumerate}
\end{proposition}
\begin{proof}

By Corollary \ref{C:Hgorenstein}, ${\hA}$ is Gorenstein.
The proof follows from Theorem 18.7 \cite{Twenty}.

The algebra $\hA[\hdelta,(1)^{-1}]^{\leq-1}$ is a quotient $\hA[\hdelta,(3)^{-1}]/(\lambda^{(3)^{-1}})$. By Proposition \ref{P:GorensteinCL}
$\hA[\hdelta,(3)^{-1}]$ is Gorenstein. 
\begin{lemma}\label{L:regularity}(Proposition 35 \cite{MovStr})
For any $\hdelta<\hdelta'\in \sethE$ $\rReg([\hdelta,\hdelta'])$ (\ref{E:regconstr}) is a regular sequence in ${\hA[\hdelta,\hdelta']}$.
The statement remains valid if $[\hdelta,\hdelta']$ is replaced by a (semi-) open interval.
 \end{lemma}

By Lemma \ref{L:regularity} $\lambda^{(3)^{-1}}$ is regular in $\hA[\hdelta,(3)^{-1}]$. By Proposition \ref{P:classofextreg}
$\hA[\hdelta,(3)^{-1}]/(\lambda^{(3)^{-1}})$ is Gorenstein. The second statement of the proposition  also follows from Theorem 18.7 \cite{Twenty}.
\end{proof}
\begin{remark}\label{E:indexa}
{\rm
In the case of algebra $\sA$ the integer $a$ coincides with the index $\mathrm{ind}\, \OGr^{+}(5,10)$
}\end{remark}

The map $\sfres_{\hA}$ is an isomorphism 
\begin{equation}\label{E:varthetadef0}
\sfres_{\hA}:\hA^{*}_0\to \C 
\end{equation} in (\ref{E:diality}).
Thus the $\C^{\times}$-grading of $\sfres$ coincides with the $a$-invariant. My next objective will be to answer question (\ref{I:q2}) and to compute 
the character of the $\Aut$-action (\ref{E:symmstries0}) on $\hA^{*}_0$. 

The $\Aut$-character associated with $\hA^{*}_0$ will be denoted by $a_{\hA}(t,q,z)$ and will be called the generalized $a$-invariant. $a_{\hA}$ contains more information than the $a$-invariant:
\begin{equation}\label{E:sigma}
a_{\hA}(t,q,z)=t^{-a}q^{-u}z^{-r}.
\end{equation}
A vector $r$ is an element of weight lattice of $\widetilde{\bT}^5$.
This action on $\hA^{*}_0$ is closely related to the action of $\Aut$ on the generator of the dualizing module $\omega_{\hA}$ (\ref{E:chiadef}).
\begin{proposition}\label{C:Poincare}
The cohomology $H_i(B)$ of the Koszul complex $B_{\bullet}(\hA)$ (see notation convention \ref{C:simpl}) are zero outside of the range $0\leq i\leq d=|[\hdelta,\hdelta']|-(\rk(\hdelta,\hdelta')+1)$. The product \[H_i(B) \otimes H_{d-i}(B)\to H_d(B)\cong \C\] defines a nondegenerate pairing. The group $H_i(B)$ is isomorphic to $\Tor_i^P(\hA,\C)$.
\end{proposition}

\begin{proof}
By Corollary \ref{C:Hgorenstein} , $\hA$ is Gorenstein. The proof follows from Proposition \ref{C:PoincareGen}.
\end{proof}

\begin{proposition}\label{P:charmatch}
Let 
\begin{equation}\label{E:relacharnot}
\chi:=\chi_{{\hA},P}
\end{equation} be the relative character (\ref{E:chireddef}) of $\Aut$-action on the generating space of the dualizing module $\omega_A$. Then the generalized $a$-invariant $a:=a_{{\hA}}$ is equal
\[a=\chi^{-1}.\] 
\end{proposition}
\begin{proof}
By Proposition \ref{P:localtorisom}, $H_{\fm}^i({\hA})\cong \Tor_{|[\hdelta,\hdelta']|-i}^P(\hA,P^{-1})$.
Let $F_{\bullet}$ (\ref{E:minres}), $d= |[\hdelta,\hdelta']|-\rk(\hdelta,\hdelta')+1$,
 be the minimal $P$-resolution of $\hA$, which I will be using for computation of $\Tor$-groups.
The group $\Tor_{d}^P(\hA,P^{-1})$ is a subgroup in the tensor product $P^{-1}\otimes V_{d}$. By Proposition \ref{C:Poincare}, the group $V_d=\Tor_{d}^P(\hA,\C)=H_d(B)$ is one-dimensional. Denote its generator by $v$. The element $\lambda^{-1}\otimes v$, where \[\lambda^{-1}:=\prod_{\halpha\in [\hdelta,\hdelta']} \frac{1}{\lambda^{\halpha}}\in P^{-1}, \] is a cocycle in $F_{\bullet}\underset{P}{\otimes} \hA$. It corresponds to the generator of $\hA_0^{*}$ in (\ref{E:loccohnew}). This is true because $\lambda^{-1}$ is characterized by the property that it is the only element that is annihilated by all $\lambda^{\halpha}$. 
The element $\lambda^{-1}$ spans a one dimensional $\Aut$-representation (see formula (\ref{E:trace})). 
 It follows from (\ref{E:trace}) that under $\Aut$ action $\lambda^{-1}$ scales by the character $\det P$ (\ref{E:det}). By Definition \ref{D:gorensteinserre}, representation of $\Aut$ in $V_d$ is dual to the representation spanned by the generator of $\omega_{\hA}$. Thus 
\[a=\chi'^{-1}\det P=\chi^{-1}.\] 
\end{proof}

It is clear now that it is essential for determining scaling properties of $\sfres$ to know the action of $\Aut$ on the generator of the dualizing module $\omega_{\hA}$.

\subsubsection{The action of $\bT$ on the dualizing module $\omega_{\hA}$}
 The plan for this section is to derive a formula for the character $\chi$ (\ref{E:relacharnot}).
 The idea is to compute $\chi[\hdelta,\hdelta']$ inductively. Corollary \ref{R:equiv} will be used as an inductive step. To put this idea to work,
I have to produce a sequence of surjective homomorphisms of Gorenstein algebras
\[\hA_{n+1} \leftarrow \hA_n\leftarrow \hA_{n-1}\leftarrow \cdots \leftarrow \hA_1\]
such that $\hA_{i+1}\leftarrow \hA_{i}$ satisfies assumptions of Corollary \ref{R:equiv}, $\hA_1={\hA[\hdelta,\hdelta']}$ $\hA_n=\hA[\hdelta',\hdelta']\cong \C[\lambda^{\hdelta'}]$ and $\hA_{n+1}=\C$.

I will give a construction of such a sequence in the series of propositions. The reader should consult Appendix \ref{S:techlemmas} for the definition of $\rM_i^{\pm}$ and $\setCL^{\pm}$. 

In this paper, we will mostly encounter algebra homomorphisms $\sfp:\hA[\setX]\to \hA[\setX']$ whose domain and codomain are labeled by intervals $\setX'\subset \setX\subset \sethE$. Sometimes, however, it will be useful to factor $\sfp$ into simpler homomorphisms $\hA[\setX]\overset{\sfp'}\to \hA[\setY]\overset{\sfp''}{\to} \hA[\setX']$ with $\setX'\subset \setY\subset \setX$, such that $\setY$ is a semi-interval.

Here is an example of such factorization. Fix
\begin{equation}\label{E:proceduregammaconditions}
 \hgamma\lessdot \hgamma' <\hbeta\text{ and }|\setCL^{-}(\hgamma)|=1.
 \end{equation}
 I factor the embedding 
 
\begin{equation}\label{E:proceduregamma}
[\hgamma,\hbeta]\overset{\sfp}{\supset} [\hgamma',\hbeta]\text{ into }[\hgamma,\hbeta]\overset{\sfp'}{\supset}(\hgamma,\hbeta]\overset{\sfp''}{\supset} [\hgamma',\hbeta].\end{equation}
I generalize this in the following construction.
 \begin{construction}\label{E:mainconstruction}{\rm

Fix a system of embedded intervals
\[[\hdelta_{1},\hbeta]\overset{\sfp_1}{\supset} [\hdelta_{2},\hbeta]\supset\cdots \overset{\sfp_{n-1}}{\supset} [\hdelta_{n},\hbeta].\]

I expand it by applying (\ref{E:proceduregamma}) to intervals based on $\hgamma=\hdelta_{i},\hgamma'=\hdelta_{i+1}$ repeatedly whenever (\ref{E:proceduregammaconditions}) is satisfied. In notations (\ref{E:proceduregamma}) $\sfp=\sfp_i$, $\sfp'=\sfp'_i,\sfp''=\sfp''_i$.
This way I get a sequence of (semi)-intervals
\begin{equation}\label{E:modseq}
[\hdelta,\hbeta]=\Delta_1\supset \cdots\supset\Delta_{n'}=[\hdelta',\hbeta]. 
\end{equation}
}\end{construction}
\begin{lemma}\label{L:saturate0}
\begin{enumerate}
\item \label{I:saturate0first} 
The homomorphism
\begin{equation}\label{E:group1}
\begin{split}
&\sfp:{\hA[\hdelta,\hbeta]}\to {\hA[\hdelta',\hbeta]}\text{ such that } \hdelta\lessdot \hdelta' , \setCL^{-}(\hdelta)=2,3, \text{ and } \\
&\hdelta,\hdelta'\in \rM_1^+\sqcup \rM_3^+, \hbeta\in \rM_1^-\sqcup \rM_3^-
\end{split}
\end{equation}
 satisfies conditions of Proposition (\ref{P:classofextreg}). $\Ker\ \sfp:=(\lambda^{\hdelta})$.
 
\item \label{I:saturate0second} 
The homomorphism
\begin{equation}\label{E:group2}
\begin{split}
&\sfp':{\hA[\hdelta,\hbeta]}\to {\hA(\hdelta,\hbeta]}\text{ such that } |\setCL^{-}(\hdelta)|=1 \text{ and } \\
&\hdelta,\in \rM_1^+\sqcup \rM_3^+, \hbeta\in \rM_1^-\sqcup \rM_3^-
\end{split}
\end{equation}
satisfies conditions of Proposition (\ref{P:classofextreg}). $\Ker\ \sfp:=(\lambda^{\hdelta})$.

The homomorphism
\begin{equation}\label{E:group3}
\begin{split}
&\sfp'':\hA(\hdelta,\hbeta]\to \hA[\hdelta',\hbeta]\text{ such that } |\setCL^{-}(\hdelta)|=1,\hdelta\lessdot \hdelta', \text{ and } \\
&\hdelta,\hdelta'\in \rM_1^+\sqcup \rM_3^+, \hbeta\in \rM_1^-\sqcup \rM_3^-
\end{split}
\end{equation}
satisfies conditions of Proposition (\ref{P:extmaintheorem}).
\end{enumerate}
In the following I will refer to  (\ref{E:group1},\ref{E:group1}) as regular  and to (\ref{E:group1}) as irregular homomorphisms.
\end{lemma}
\begin{proof}
As ${\hA(\hdelta,\hbeta]}\cong {\hA[\hdelta,\hbeta]}/(\lambda^{\hdelta})$ where $\lambda^{\hdelta}$ is regular (Proposition \ref{P:integraldomain}) and ${\hA[\hdelta,\hbeta]}$ is Gorenstein by assumption
, then, by Proposition \ref{P:classofextreg} $, {\hA(\hdelta,\hbeta]}$ is Gorenstein. I conclude that:
\begin{equation}\label{E:gorimplication}
\text{If } {\hA[\hdelta,\hbeta]}\text{ is Gorenstein, then so is } {\hA(\hdelta,\hbeta]}.
\end{equation}
\begin{enumerate}
\item By item (\ref{I:m3}) of Lemma \ref{L:multstatM} , $(\hdelta,\hbeta]=[\hdelta',\hbeta]$. By (\ref{E:gorimplication}), the map $\sfp:{\hA[\hdelta,\hbeta]}\to {\hA[\hdelta',\hbeta]}$ satisfies conditions of Proposition (\ref{P:classofextreg}).
The same way I prove the statement about $\sfp':{\hA[\hdelta,\hbeta]}\to {\hA(\hdelta,\hbeta]}$ when $\setCL^{-}(\hdelta)=1$.
\item 

From items (\ref{I:four}) and (\ref{I:five}) of Lemma \ref{L:degrel}, I deduce that 
 the are two elements $\lambda^{\hgamma},\lambda^{\hgamma'}$ in $\Ker\, \sfp''$ belong to $\Ann(\lambda^{\hdelta'})\subset {\hA(\hdelta,\hbeta]}$ and
$\hgamma,\hgamma'$ are comparable. I can assume that $\hgamma<\hgamma'$. By using item (\ref{I:six}) and the standard monomial basis in ${\hA(\hdelta,\hbeta]}$, I conclude that $\lambda^{\hgamma},\lambda^{\hgamma'}$ generate $\Ker\, \sfp''$.
By using the standard monomial basis for ${\hA(\hdelta,\hbeta]}$ (Proposition \ref{R:stbasis}) and item (\ref{I:seven}) from Lemma \ref{L:degrel} , I conclude that ${\hA(\hdelta,\hbeta]}/\lambda^{\hdelta'}\cong {\hA[\hgamma,\hbeta]}$. By Proposition \ref{E:degreenull}, $ {\hA[\hgamma,\hbeta]}$ is Cohen-Macaulay. Thus $\sfp''$ satisfies assumptions of Proposition \ref{P:extmaintheorem}.
\end{enumerate}
\end{proof}

The following lemma is a powerful instrument. It will enables me to reduce many difficult questions related to the map $\sfp:\hA[\hdelta',\hbeta]\to \hA[\hdelta,\hbeta]$ to a series of easy questions about the factors $\sfp_i$ in the factorization of $\sfp$.

 \begin{lemma}\label{L:saturate}
Fix pair of intervals 
\begin{equation}\label{E:twointrvals}
[\hdelta,\hbeta]\supset [\hdelta',\hbeta]
\end{equation}
 $\hdelta'\in \rM_1^+$
 (\ref{E:Mdef}). By using Proposition \ref{P:seqint}, it can be completed to the sequence of embedded intervals (\ref{E:intervalsneed}). With the help of Construction \ref{E:mainconstruction} , I expand it to the sequence of sets (\ref{E:modseq}). Inclusions of sets (\ref{E:modseq}) define a series of homomorphisms of algebras
 \begin{equation}\label{E:refignseq}
\hA[\Delta_{n'}]\overset{\sfp_{n'-1}}{\ot} \cdots \overset{\sfp_1}{\ot} \hA[\Delta_{1}].
\end{equation} 
 
I claim that any such homomorphism $\hA[\Delta_i]\to \hA[\Delta_{i+1}]$
 satisfies assumptions of Proposition \ref{P:classofextreg} or Proposition \ref{P:extmaintheorem}.
\end{lemma}
\begin{proof}
By reading carefully conditions of Lemma \ref{L:saturate0}, I conclude that $|\setCL^{-}(\hdelta_i)|$ has to be $1,2$ or $3$, which are automatically satisfied.
Thus, the statement of lemma follows from Lemma \ref{L:saturate0}.

\end{proof}
\begin{definition}\label{D:thdeffora}{\rm
Fix $ [\hdelta,\hbeta]\subset [\hdelta',\hbeta']$ such that $ \hA[\hdelta,\hbeta],\hA[\hdelta',\hbeta']$ are Gorenstein (Theorem \ref{P:GorensteinCL}). Let $\sfp:\hA[\hdelta',\hbeta']\to \hA[\hdelta,\hbeta]$ be projection the projection (Definition \ref{D:algAdelta}). Equation (\ref{E:classconstruction}) defines the element
\begin{equation}\label{E:initdefThom}
\Th_{\sfp}\in \Ext_{P}^{\codim}( \hA[\hdelta,\hbeta], \hA[\hdelta',\hbeta']), \codim=\rho(\hbeta')-\rho(\hbeta) -\rho(\hdelta')+\rho(\hdelta),P=\C[\hdelta',\hbeta']
\end{equation}
By (\ref{E:THdef}) $\Th_{\sfp}$ defines a map of local cohomology
\[H_{\fc}^{i}[\hdelta,\hbeta]\to  H_{\fc}^{\codim+i}[\hdelta',\hbeta']\]
$\fc\subset \hA[\hdelta',\hbeta']$ is an ideal. 
By abuse of notations I denote $\sfp(\fc)$ by $\fc$.
}\end{definition}
Introduce a notation:
\[\chi'[\halpha,\hbeta]:= \chi'_{{\hA[\halpha,\hbeta]},\C[\halpha,\hbeta]}, \quad \chi[\halpha,\hbeta]:= \chi_{{\hA[\halpha,\hbeta]}}\]
The results of the following proposition will be used heavily in the
 computation of the $\Aut$-weight of the pairing (\ref{E:pairinggen}).
\begin{proposition}
Fix pair of intervals $[\hdelta,\hbeta]\supset [\hdelta',\hbeta]$. The group $\Aut$ commutes with the homomorphism $\sfp:{\hA[\hdelta,\hbeta]}\to {\hA[\hdelta',\hbeta]}$. 
\begin{enumerate}
\item The element \[\Th_{\sfp}\in \Ext_{P}^{\codim}( {\hA[\hdelta',\hbeta]}, {\hA[\hdelta,\hbeta]}), \codim=\rho(\hdelta')-\rho(\hdelta),P=\C[\hdelta,\hbeta]\] scales under the action of $\Aut$. More precisely if $g=(t,q,z)$, then $\Th_{\sfp}$ satisfies
\begin{equation}\label{E:chath}
\Th_{\sfp}^g=\Th_{\sfp}\chi[\hdelta',\hbeta](g){\chi[\hdelta,\hbeta]}^{-1}(g),g\in \Aut.
\end{equation}
\item If $\Th_{\sfp}$ corresponds to the elementary homomorphisms of algebras (\ref{E:group1},\ref{E:group2},\ref{E:group2}),
then in notations (\ref{E:weights}) 
\begin{equation}\label{E:elcharacters}
\Th_{\sfp}^g=
\begin{cases}
\hat{v}_{\halpha}^{-1}\Th_{\sfp}& \sfp \text{ is the map } {\hA[\halpha,\hbeta]}\to {\hA(\halpha,\hbeta]} (\ref{E:group1},\ref{E:group2})
 \\
\hat{v}_{\hgamma}\Th_{\sfp}& \sfp \text{ is the map } {\hA(\halpha,\hbeta]}\to {\hA[\hgamma,\hbeta]} (\ref{E:group3}).
 \\
\end{cases}\end{equation}
 See formula (\ref{E:weights}) for the character $\hat{v}_{\hdelta}$.
In particular, the products $\chi'(\halpha,\hbeta]{\chi'[\halpha,\hbeta]}^{-1}$, $\chi'[\hgamma,\hbeta]{\chi'(\halpha,\hbeta]}^{-1}$ in (\ref{E:elcharacters}) do not depend on $\hbeta$.

\item The character $\det \C[\hdelta,\hbeta]$ (\ref{E:chiadef}) 
satisfies
\begin{equation}\label{E:trivmod}
\left(\det \C[\hdelta,\hbeta]\right)^{-1}=\prod_{\halpha\in [\hdelta,\hbeta]}\hat{v}_{\halpha}.
\end{equation}
\end{enumerate}

\end{proposition}
\begin{proof}
Follows from Corollary \ref{R:equiv}. The last item follow from a simple computation with the Koszul resolution for $\C$.
\end{proof}

I define quantities $(a_{\rL}, u_{\rL}, r_{\rL})$ by the formula $t^{-a_{\rL}}q^{-u_{\rL}}z^{-r_{\rL}}:=\det \C[\hdelta,\hbeta]$.
In more details
\[a_S[\hdelta,\hbeta]=|[\hdelta,\hbeta]|,\quad u_S[\hdelta,\hbeta]=\sum_{\halpha\in [\hdelta,\hbeta]}u(\halpha),\quad z^{r_S[\hdelta,\hbeta]}=\prod_{\halpha\in [\hdelta,\hbeta]}v_{\halpha}(z)\]

The next proposition reduces computation of $\chi[\hdelta,\hbeta]$ to the problem of combinatorics.
\begin{proposition}\label{P:charmaincomp}
Fix an interval $[\hdelta,\hbeta]$. 
\begin{equation}\label{E:Lset}
 \rL[\hdelta,\hbeta]=\rM\cap\mathrm{Core}[\hdelta,\hbeta],\quad \overline{\mathrm{Core}}[\hdelta,\hbeta]:= [\hdelta,\hbeta]\backslash \mathrm{Core}[\hdelta,\hbeta].
\end{equation}

See Lemma \ref{L:multstatM} for definition of $\rM$ and (\ref{E:coredef}) for $\mathrm{Core}$.
Then
\begin{equation}\label{E:sigmaalt0}
\chi[\hdelta,\hbeta]=\prod_{\halpha\in \rL[\hdelta,\hbeta]}\hat{v}_{\halpha} \prod_{\halpha\in \overline{\mathrm{Core}}[\hdelta,\hbeta]}\hat{v}_{\halpha}
\end{equation}
\end{proposition}
\begin{proof}
I leave as an exercise the trivial verifications based on Corollary \ref{R:equiv} and formulas
 (\ref{E:chath}), (\ref{E:elcharacters}) that:
\begin{enumerate}
\item If $\hdelta,\hdelta'$ satisfy conditions (\ref{E:group1}), then 
\begin{equation}\label{E:chtthcompatibility1}
\chi[\hdelta,\hbeta]=\hat{v}_{\hdelta}\chi[\hdelta',\hbeta].
\end{equation}
\item If $\hdelta,\hdelta'$ satisfy conditions (\ref{E:group2}), then 
\begin{equation}\label{E:chtthcompatibility2}
\chi[\hdelta,\hbeta]=\hat{v}_{\hdelta}\chi(\hdelta,\hbeta].
\end{equation}
\item If $\hdelta,\hdelta'$ satisfy conditions (\ref{E:group3}), then 
\[\begin{split}
&(\hdelta,\hbeta]=[\hdelta',\hbeta]\cup\{\hdelta''\}, \text{ where }\hdelta''\neq \hdelta', \rho(\hdelta'')=\rho(\hdelta'),
\end{split}\]
\begin{equation}\label{E:chtthcompatibility3}
\chi(\hdelta,\hbeta]=\hat{v}^{-1}_{\hdelta'}\chi[\hdelta',\hbeta] \text{ and } \chi[\hdelta,\hbeta]=\hat{v}_{\hdelta} \hat{v}^{-1}_{\hdelta'}\chi[\hdelta',\hbeta].
\end{equation}
\end{enumerate}
The rest of the proof is the induction on the cardinality of $\mathrm{Core}[\hdelta,\hbeta]\cap \rM_1^+$, which I omit.  During induction I  get a presentation of $\chi[\hdelta,\hbeta]$ as a product of $\hat{v}^{\pm}_{\halpha}$. The key moment is  that $\hat{v}_{\halpha}$, $\halpha\in \rM_{1}^{+}\backslash \rM$ appears in this product twice, one time in the numerator and the other in the denominator.

\end{proof}

It often happens that an equation satisfied by a function is more illuminating than an explicit formula for the function. In the next proposition, I derive such an equation for $\chi[\hdelta,\hdelta']$. 
\begin{proposition}\label{P:funequation} Let as assume hypothesis of Proposition \ref{P:charmaincomp}.
\begin{enumerate}
\item 
Let $\rL=\rL[\hdelta,\hbeta]$ be as in (\ref{E:Lset}). Then the exponents of the character $\chi[\hdelta,\hbeta]
=t^{-a}q^{-u}z^{-r}$ are
\begin{equation}\label{E:sigmaalt}
\begin{split}
&a[\hdelta,\hbeta]=-\left|\rL\cup\overline{\mathrm{Core}}\right| ,\\
&u[\hdelta,\hbeta]=-\sum_{\halpha\in \rL\cup\overline{\mathrm{Core}}}u(\halpha),\\
&z^{r[\hdelta,\hbeta]}=\prod_{\halpha\in \rL\cup\overline{\mathrm{Core}}}v^{-1}_{\halpha}(z),
\end{split}
\end{equation}

which coincide with the components of the $\Aut$-weight of $\sfres$.

\item \label{I:chareq2} Suppose 
that $\hdelta<(0)^l,(1)^{l-1}<\hbeta$.
Then 
\[\chi[\hdelta,\hbeta]=\chi[\hdelta,(1)^{l-1}]\chi[(0)^l,\hbeta]q^{2-4l} t^{-4}.\]
In particular, if $l=0$, then 
\begin{equation}\label{E:chiprop}
\chi[\hdelta ,\hdelta']=t^{-4}q^{2}\chi[\hdelta ,(1)^{-1}]\chi[(0)^0 ,\hdelta']. 
\end{equation}
\item
\[\chi[\hdelta,\hbeta]=\chi[\hdelta,\hbeta]^{\leq l}\chi[\hdelta,\hbeta]^{\geq l+1}.\]

\item \label{I:chareq3}
\begin{equation}
\begin{split}
&a[\hdelta,\hbeta]=a[\reflection(\hbeta), \reflection(\hdelta)],\quad a[\shift(\hdelta),\shift(\hbeta)]=a[\hdelta,\hbeta],\\
&u[\hdelta,\hbeta]=-u[\reflection(\hbeta), \reflection(\hdelta)],\quad u[\shift(\hdelta),\shift(\hbeta)]=u[\hdelta,\hbeta]+a[\hdelta,\hbeta],\\
&r[\hdelta,\hbeta]=S(r[\reflection(\hbeta)), \reflection(\hdelta)],\quad r[\shift(\hdelta),\shift(\hbeta)]=r[\hdelta,\hbeta].
\end{split}
 \end{equation}

\end{enumerate}
\end{proposition}
\begin{proof}
Only the last item in the proposition requires a comment. By Lemma \ref{L:multstatM}, item \ref{I:sixmultstatM}, the set $\rM$
 is invariant under $\reflection$ and $\shift$. Thus $\reflection$ induces a bijection between $\rL[\hdelta,\hbeta]$ and $\rL[\reflection(\hbeta), \reflection(\hdelta)]$ and $\shift$ between $\rL[\hdelta,\hbeta]$ and $\rL[\shift(\hdelta),\shift(\hbeta)]$.
\end{proof}

Here is a result of computation of $\chi$ for the most elementary sets.
\begin{proposition}\label{P:cheleminterval}
\[\chi[(0)^r,(1)^r]=t^{8}q^{8r}.\]

\end{proposition}
\begin{proof}
$\C^{\times}\times\bT\times\Spin(10)$ is the symmetry group of $A[(0)^r,(1)^r]$. It means that the $\widetilde{\bT}^5$-character $\chi[(0)^r,(1)^r]$ is a restriction of one-dimensional representation of $\Spin(10)$. As $\Spin(10)$ is semisimple, any one-dimensional representation is trivial. The character of $\C^{\times}\times\bT$ is computed directly with (\ref{E:sigmaalt0}).

\end{proof}

The following corollary is a combination of the results of Propositions \ref{P:funequation} and \ref{P:cheleminterval}.
\begin{corollary}\label{C:specialcharacter}
Suppose $N<N'$. Then 
\[\chi[(0)^N,(1)^{N'}]=
t^{8+4(N'-N)}q^{2 {N'}^2+4 {N'}-2 {N}^2+4 {N}}
\]
\end{corollary}

\begin{remark}
If I choose $[\hdelta, \hdelta']$ such that $\reflection[\hdelta, \hdelta']=[\hdelta, \hdelta']$ for $\reflection$ as in (\ref{E:involution}), then $\chi[\hdelta, \hdelta'](t,q,z)=\chi[\hdelta, \hdelta'](t,q^{-1},S(z))$ and $ \chi[\hdelta, \hdelta'](t,q,z)$ doesn't depend on $q$ and $z$.
\end{remark}

\subsection{The groups $H_{\fa}^i[\hdelta,\hdelta']$}
In this section, I will answer question \ref{I:q3} from the page \pageref{I:q3} about nondegeneracy of the pairing $\sfres(\sfm(a,b))$.
To do this I need to explore the groups $H_{\fa}^i[\hdelta,\hdelta']$ more closely. My exposition will becomes more accessible with the use of certain $D$-modules, which I introduce in the next section.

 In Sections \ref{SS:Computation}, \ref{SS:thestardualitypairing}, \ref{S:fourterm}, \ref{S:virtchar},
all the intervals $[\hdelta,\hdelta']$ satisfy $\hdelta\leq (0)^0,(1)^{-1}\leq \hdelta'$ and the purity condition (\ref{E:purity}).

\subsubsection{The auxiliary D-modules}\label{E:dmod}
Proposition \ref{P:localtorisom} enables me to identify local cohomology of $\hA$ with the $\Tor$ functor over the free commutative algebra $P$. 
 In the isomorphism $H^i_{\fc}(\hA)=\Tor_j^P(\hA,M)$ (\ref{E:isolocaltor}), the non-finitely generated $P$-module $M$ depends on the presentation $\hA=P/\fr$ and the ideal $\fc$. 
In order to describe $M$ in finite terms, I will introduce the algebra of differential operators $D\supset P$. The module $M$ will become a cyclic $D$-module. My present goal is to describe a number of $D$ modules that will be used in the outlined above computation of local cohomology. I will use non uniqueness of the choice of $M$ to my advantage: one choice (type T) of $M$ will illuminate the gradings on $H^i_{\fc}(\hA)$, the other (type S) will make $H^i_{\fc}(\hA)$ more accessible for computations. The reader may check from time to time with Appendix \ref{S:Dmodulenoteconstr} for notations and some elementary constructions that will be used in the rest of the paper. I fix the interval $[\hdelta,\hdelta']$ which satisfy $\hdelta<(0)^0, (1)^{-1}<\hdelta'$, for the rest of the section so that all the constructions will be based on this choice.

 The differential operators based on $\Span{[\hdelta,\hdelta']}$ (I use notations (\ref{E:dotprdef}, \ref{E:regseqgenverynew}, \ref{E:regconstr})) will be denoted by $D[\hdelta,\hdelta']$. With the help of the dot product 
 (\ref{E:dotprdef}) I identify $\lambda^{\halpha}$ with coordinates on $\Span{[\hdelta,\hdelta']}$. 
 
 It is customary in the present context to denote the dual basis to $\{\lambda^{\halpha}\}$ for $\Span{[\hdelta,\hdelta']}^*$ by $\{w_{\reflection(\halpha)}| \halpha \in [\hdelta,\hdelta']\}$ (the involution $\reflection$ was introduced in (\ref{E:involution}). In the presence of the dot product, $w_{\reflection(\halpha)}$ has an interpretation of the partial derivative $\frac{\sd}{\sd \lambda^{\halpha}}$.
The algebra $ D[\hdelta,\hdelta']$ is generated by $\{\lambda^{\halpha},w_{\reflection(\halpha)}| \halpha \in [\hdelta,\hdelta']\}$ that are subject to relations
\[\begin{split}
&[\lambda^{\halpha},\lambda^{\halpha'}]=[w_{\hbeta},w_{\hbeta'}]=0\\
&[\lambda^{\halpha},w_{\hbeta}]=\delta_{\halpha,\reflection(\hbeta)}.
\end{split}\]

Introduce the following subsets of $[\hdelta,\hdelta']$.
\begin{equation}\label{E:Ndef}
\setN(\fc)=\begin{cases}
\emptyset&\text{ if } \fc=\{0\}\\
[\hdelta,(1)^{-1}]&\text{ if } \fc=\fa\\
[(0)^1,\hdelta']&\text{ if } \fc=\fa'\\
[(0)^{0},\hdelta']&\text{ if } \fc=\fb\\
[\hdelta,\hdelta']^{\leq-1}&\text{ if } \fc=\ff\\
[\hdelta,\hdelta']^{\geq-1}&\text{ if } \fc=\ff'\\
[\hdelta,\hdelta']&\text{ if } \fc=\fm\\
\setE&\text{ if } \fc=\fp.\\
\end{cases}
\end{equation}
Sets $\setN(\fc)$ satisfy
\[\setN(\fm)=[\hdelta,\hdelta']=\setN(\fa)\cup \setN(\fp)\cup \setN(\fa'),\quad \setN(\fp)\cup \setN(\fa')=\setN(\fb).\]

Components of the orthogonal sum decomposition 
\[\Span{ [\hdelta,\hdelta' ]}= \alt[\setN(\fa)]+\rReg[\setN(\fa)]+\rReg[\setN(\fb)]+\alt[\setN(\fb)]\]
were defined in (\ref{E:regseqgenverynew}) and (\ref{E:regconstr}).
Maps (\ref{E:functorialpr})
induce inclusion of algebras 
\begin{equation}\label{E:kappadef}
\begin{split}
&\sfk^{\alt}:D(\alt[\hdelta,\hgamma])\to D(\alt[\hdelta',\hgamma]),\\
& \sfk^{\rReg}:D(\rReg[\hdelta,\hgamma])\to D(\rReg[\hdelta',\hgamma]),\\
&\hdelta'<\hdelta<\hgamma.
\end{split}
\end{equation}
To become more familiar with functorial properties of $\alt$ and $\rReg$ (\ref{E:functorialpr}), let us consider an example:
\begin{example}\label{EX:regalt}
\begin{enumerate}
\item $\alt[(0)^0,(13)^{0}]=0$, $\rReg [(0)^0,(13)^{0}]=\Span{ [(0)^0,(13)^{0}]}$.
\item The map (\ref{E:functorialpr}) of $\rReg [(0)^0,(13)^{0}]$ to $\rReg[(45)^{-1},(13)^{0}]$ 
 is 
\[\inclusiono(\lambda^{(0)^0})=\frac{1}{\sqrt{2}}( \lambda^{(0)^0}+\lambda^{(3)^{-1}}), \quad \inclusiono(\lambda^{(12)^0})= \frac{1}{\sqrt{2}}(\lambda^{(12)^0}+\lambda^{(2)^{-1}}),\quad \inclusiono(\lambda^{(13)^0})= \lambda^{(13)^0}.\]
\item \label{I:spanexample} The map $\inclusiono:\Span{[(4)^{-1},(13)^{0}]}\to \Span{((35)^{-1},(13)^{0}]}$ (\ref{E:functorialpr}) transforms
$\lambda^{(4)^{-1}}$ to $ \frac{1}{\sqrt{2}}(\lambda^{(4)^{-1}}+\lambda^{(45)^{-1}})$.
\end{enumerate}
\end{example}

\paragraph{Modules $\moduleT_{\fc}$}\label{S:modulesT}
 \bigskip

By using notations (\ref{E:Ndef}) and (\ref{E:algdistrnot}), I define a family $\moduleT_{\fc},\fc=(0),\fa,\fa',\fb,\fp,\fm$
of $P$-modules 
 \begin{equation}\label{E:TIJdef}
 \begin{split}
 &\moduleT_{(0)}:=\C[\hdelta,\hdelta']\\
 &\moduleT_{\fa}:=\C[\setN(\fa)]^{-1}\otimes \C[ \setN(\fp)]\otimes \C[\setN(\fa')]
 \\
 &\moduleT_{\fa'}:=\C[\setN(\fa)]\otimes \C[\setN(\fp)]\otimes \C[\setN(\fa')]^{-1}
 \\
 &\moduleT_{\fb}:= \C[\setN(\fa)]\otimes \C[\setN(\fb)]^{-1}
 \\
 &\moduleT_{\ff}:=\C[\setN(\ff)]^{-1}\otimes \C[\setN(\ff')]
 \\
 &\moduleT_{\ff'}:=\C[\setN(\ff)]\otimes \C[\setN(\ff')]^{-1}
 \\
 &\moduleT_{\fm}:=\C[\setN(\fm)]^{-1}.
 \end{split}\end{equation}
 In the following, I occasionally will use an abbreviation
 \[\cL:=\C[z,z^{-1}]\] 
 for the ring of Laurent polynomials. If $N$ is a $\C$ vector space, $N[z,z^{-1}]$ will stand for a free $\C[z,z^{-1}]$-module $N\otimes \C[z,z^{-1}]$. A span of a set of elements $\{x_i\}$ in a $\C[z,z^{-1}]$-module $M$ will be denoted by $\LSpan{x_i}$.
 A free $\C[z,z^{-1}]$-module $\LSpan{[\hdelta,\hdelta']}$ contains a submodule spanned over $\C[z,z^{-1}]$ by 
 \begin{equation}\label{E:lambdaz}
 \lambda^{\rbeta}(z):=\sum_{\rbeta^k\in [\hdelta,\hdelta']} \lambda^{\rbeta^k}z^{k},\quad \rbeta\in \setE.
 \end{equation} 
 \begin{lemma}\label{L:sumdecomposition}
 There is a direct sum decomposition of free $\C[z,z^{-1}]$ modules.
 \begin{equation}\label{E:decomposition}
 \LSpan{[\hdelta,\hdelta']}\cong \LSpan{\lambda^{\rbeta}(z)}\oplus\LSpan{ \setN(\fa)}\oplus \LSpan{\setN(\fa')}.
 \end{equation}
 \end{lemma}
 \begin{proof}The proof becomes obvious if I use the fact that the map
\[\begin{split} 
&\LSpan{[\hdelta,\hdelta']}\to \LSpan{[\lambda^{\rbeta}]} \quad \lambda^{\rbeta^k}\to \delta_{0,k}\lambda^{\rbeta}
\end{split}\]
after restriction on $\LSpan{\lambda^{\rbeta}(z)}$ induces an isomorphism \[\LSpan{\lambda^{\rbeta}(z)}\cong \LSpan{\lambda^{\rbeta}}.\]
 \end{proof}
\begin{corollary} 
 The $P[z,z^{-1}]$ module 
 \[\moduleT_{\fp}:=\C[\lambda^{\rbeta}(z)]^{-1}\otimes \C[\setN(\fa)]\otimes \C[\setN(\fa')]\]
is free over $\C[z,z^{-1}]$.
\end{corollary} 
 A $P$-modules $\moduleT_{\fc}$ is a $D$-modules. The elements $w_{\reflection(\halpha)}$ act by $\frac{\sd }{\sd \lambda^{\halpha}}$.
 The products 
 \[\begin{split}
 &\varpi_{\fc}:=\prod_{\halpha\in \setN(\fc)}(\lambda^{\halpha})^{-1}\fc\neq \fp, \\
 &\varpi_{\fp}:=\prod_{\rbeta\in \setE}(\lambda_{[\hdelta,\hdelta']}^{\rbeta}(z))^{-1}
 \end{split}\] are $D$ generators for $\moduleT_{\fc}$ and $\moduleT_{\fp}$ respectively.
Infinitesimal symmetries of the algebra ${\hA}$ act on $\moduleT_{\fc}$. Some of them can be integrated to an algebraic action of the corresponding Lie group. In particular, this applies to $\C^{\times}\times\bT$. 
Introduce a function
\begin{equation}\label{E:functionsdef}
s(\fc):=|\setN(\fc)|.
\end{equation}
Degree $\deg_{\C^{\times}}$ (\ref{E:cdeg}) and $\bT$-weight $\deg_{\bT}$ of the generators $\varpi_{\fc}$ $\fc\neq \fp$ is equal to 
\[\begin{split}
&\deg_{\C^{\times}} \varpi_{\fc}=-s(\fc), \\
&\deg_{\bT} \varpi_{\fc} =-\sum_{\rbeta^k\in \setN(\fc)}k. 
\end{split}\]

The isomorphism described in the next proposition will be used in the definition of the field - antifield pairing.
 \begin{proposition}\label{P:tensorproduct}
 There is an isomorphism
 \begin{equation}\label{E:tensorproduct1}
 \moduleT_{\fm}\overset{\phi}{\longrightarrow}
\moduleT_{\fa}\underset{P}{\otimes}\moduleT_{\fp}\underset{P}{\otimes} \moduleT_{\fa'}.
 \end{equation}
 In the above equation I extended scalars from $\C$ to $\C[z,z^{-1}]$ where it was necessary.
 \end{proposition}
 \begin{proof}
 It follows from decomposition (\ref{E:decomposition}) that the right-hand-side of (\ref{E:tensorproduct1}) is nonzero. Both sides of (\ref{E:tensorproduct1}) are modules over $D[z,z^{-1}]$. The generators $w_{\reflection(\halpha)}$ act 
diagonally. The product $\varpi_{\fa}\varpi_{\fp}\varpi_{\fa'}$ is annihilated by $\lambda^{\rbeta}(z)$ and $\lambda^{\halpha},\halpha\in \setN(\fa)\cup \setN(\fa')
$. By Lemma \ref{L:sumdecomposition}, it is also annihilated by $\lambda^{\halpha},\halpha\in [(0)^0,(1)^{0}]$. By construction, $\moduleT_{\fm}$ is a $D$-cyclic module. The easiest way to see that $\varpi_{\fa}\varpi_{\fp}\varpi_{\fa'}$ is a generator of the tensor product is to replace $\{w_{\reflection{\rbeta^k}}\}\subset D$ with a different set of generators $\{w_{\reflection{\rbeta^0}},w_{\reflection{\rbeta^k}}-z^kw_{\reflection{\rbeta^0}},k\neq 0\}$. For this new set verification is trivial.
 The map $\phi$ of cyclic $D[z,z^{-1}]$-modules which sends $\varpi_{\fm}$ to $\varpi_{\fp}\varpi_{\fa}\varpi_{\fa'}$ is self-consistent because $\lambda^{\halpha}\varpi_{\fm}=0$ are the defining relation for $\moduleT_{\fm}$. 
 \end{proof}
 \begin{remark}\label{R:lowest}
 
\begin{enumerate}
\item The $\bT$-weight subspaces of $\moduleT_{\fa}$ and $\moduleT_{\fb}$ are finite-dimensional. 
 $\bT$-weights of $\moduleT_{\fa}$ are bounded from below by the weight of $\varpi_{\fa}$. In $\moduleT_{\fb}$ case the weights are bounded from above by the weight of $\varpi_{\fb}$.
\item 
If $\hdelta<(1)^{-1},(0)^0<\hdelta'$ the $\bT$-weight subspaces of $\moduleT_{\fm}$ are infinite-dimensional. 
 If $\hdelta'=\rdelta'^k,k<0$ then $\varpi_{\fm}$ is the lowest $\bT$-weight vector in $\moduleT_{\fm}$.
\end{enumerate}
 \end{remark}
 
\begin{definition}\label{D:upperlower}
{\em
 Fix a linear space $M=\bigoplus_{u\in \Z} M_{u}, \dim M_{u}<\infty$. $M$ is a {\it positive energy} space if $\exists u_0$ such that $M_{u_0}\neq \{0\}, M_{u}=\{0\}, u<u_0$. I denote $u_0$ by $\lu(M)$. The {\it negative energy} space is defined similarly. The weights upper bound is denoted by $\uu(M)$.
 }
\end{definition}

For example, $\moduleT_{\fa}$ is a positive energy space with respect to $\deg_{\bT}$-grading. $\moduleT_{\fb}$ is a negative energy space.
The isomorphism 
\[\moduleT_{\fa}\underset{P}{\otimes} \moduleT_{\fb} \cong \moduleT_{\fm}\]
follows easily from the isomorphisms $x^{-1}\C[x]\underset{\C[x]}\otimes \C[x]\cong x^{-1}\C[x]$, (\ref{E:algdistrnot}) and (\ref{E:TIJdef}). 
 I use it to define map 
 \begin{equation}\label{E:Tmap}
 \begin{split}
&\moduleT_{\fa}{\otimes} \moduleT_{\fb} \to \moduleT_{\fm},\quad \moduleT_{\ff}{\otimes} \moduleT_{\ff'} \to \moduleT_{\fm}\text{, and }
\end{split}
 \end{equation}
\[\moduleT_{(0)}{\otimes} \moduleT_{\fm} \to \moduleT_{\fm}.\]

 The residue functional 
 \begin{equation}\label{E:rhodef}
\sfres_{\moduleT_{\fm}}:\moduleT_{\fm}\to \C
 \end{equation}
 is 
 \[\quad \sfres_{\moduleT_{\fm}}(a):=\oint\cdots\oint a \prod_{\halpha\in [\hdelta,\hdelta'] } d\lambda^{\halpha}.\]
 
 $\sfres_{\moduleT}$ and the product (\ref{E:Tmap}) define the pairing:
 \begin{equation}\label{E:pairingT}
\moduleT_{\fa}{\otimes} \moduleT_{\fb}\to \C,\quad \moduleT_{(0)}{\otimes} \moduleT_{\fm}\to \C,\quad (a,b)_{\moduleT_{\fm}}=
 \sfres_{\moduleT_{\fm}}(ab).
 \end{equation}
 
 \begin{proposition}\label{P:Tpairingnondeg}
 The pairings (\ref{E:pairingT}) are not degenerate.
 \end{proposition}
 \begin{proof}
 Elements 
 \[\begin{split}
 &\prod_{\halpha\in \setN(\fb) }\left(\lambda^{\halpha}\right)^{k_{\halpha}} \prod_{\hbeta\in \setN(\fa) }\left(\lambda^{\hbeta}\right)^{-1-n_{\hbeta}}\in \moduleT_{\fa}, \\
 &\prod_{\halpha\in \setN(\fb) }\left(\lambda^{\halpha}\right)^{-1-k_{\halpha}} \prod_{\hbeta\in \setN(\fa) }\left(\lambda^{\hbeta}\right)^{n_{\hbeta}}\in \moduleT_{\fb}
 \end{split}\]
 form a pair of dual bases in modules $\moduleT_{\fa}$ and $\moduleT_{\fb}$ respectively. Similar bases exist in $\moduleT_{(0)}$ and $\moduleT_{\fm}$.
 \end{proof}
\paragraph{Modules $\moduleS_{\fa}$, $\moduleS_{\fb}$
}
Description of modules of type $\moduleS
$ is less straightforward than of type $\moduleT$.
Besides giving the definition, 
I will devote some efforts exploring properties $\moduleS[\hdelta,\hdelta']$ as a function of the interval.
\begin{definition}\label{D:lessdef}{\em
$\setK^{\leq n}:=\{\halpha\in \setK|\rho(\halpha)\leq n\}$. This construction will also be useed with other binary relations $=, <, >, \geq$. Here $\rho$ is the function (\ref{E:LdefE}).
I will also use partial order in $\hat{\setE}$ to define $\setK^{\leq \hbeta}:=\{\halpha\in \setK|\halpha\leq \hbeta\}$ and similar constructs based on $ <, >, \geq$.
}\end{definition}

By using $\alt$ (\ref{E:regseqgenverynew}) and $\rReg$ (\ref{E:regconstr}) constructions and a finite subset $\setK\subset \sethE$, I define
\begin{equation}\label{E:DCdef}
\begin{split}
&D_{-}[\setK]:=D(\alt\, \setK),\\
&D_{+}[\setK]:=D(\rReg\, \setK),\\
&\C_{-}[\setK]:=\C[\alt\, \setK],\\
&\C_{+}[\setK]:=\C[\rReg\, \setK],\\
&\C_{+}[\setK]^{-1}:=\C[\rReg\, \setK]^{-1}.\\
\end{split}
\end{equation}
The modules $\moduleS_{\fa}[\setK],\moduleS_{\fa'}[\setK]$ over 
\[D[\setK]\cong D_{-}[\setK^{\leq (1)^{-1}}]\otimes D_{+}[\setK^{\leq (1)^{-1}}]\otimes D[\setK^{\geq (0)^{0}}]\] 
are the tensor products
\begin{equation}\label{E:FIdef}
\begin{split}
&\moduleS_{\fa}[\setK]:=\C_{-}[\setK^{\leq (1)^{-1}}]\otimes \C_{+}[\setK^{\leq (1)^{-1}}]^{-1}\otimes \C[\setK^{\geq (0)^{0}}],\\
&\moduleS_{\fa'}[\setK]:=\C[\setK^{\leq (1)^{0}}]\otimes \C_{-}[\setK^{\geq (0)^{1}}]\otimes \C_{+}[\setK^{\geq (0)^{1}}]^{-1}. \\
\end{split}
\end{equation}

Fix finite $\setQ\subset \sethE$ such that $\setQ^{\geq (0)^{0}}={\setQ'}^{\geq (0)^{0}}$. Let us apply construction (\ref{E:inclusion1}) to spaces $V=\alt\, \setQ^{\leq (1)^{-1}}, V'=\alt\, {\setQ'}^{\leq (1)^{-1}}$, $U=\rReg\, \setQ^{\leq (1)^{-1}}, U'=\rReg\, \setQ'^{\leq (1)^{-1}}$.
 The maps $\sfa=\mathsf{inc}$,$\sfb=\mathsf{pr}\circ \mathsf{inc}$ are taken from \ref{E:functorialpr}.
After taking isomorphisms (\ref{E:FIdef}) into account, I assemble a map
\begin{equation}\label{E:inclFI}
\begin{split}
&\sfi:\moduleS_{\fa}[\setQ]\cong \moduleS_{\fa}[\setQ^{\leq (1)^{-1}}]\otimes \moduleS_{\fa}[\setQ^{\geq (0)^{0}}]{ \longrightarrow }\moduleS_{\fa}[\setQ'^{\leq (1)^{-1}}]\otimes \moduleS_{\fa}[\setQ'^{\geq (0)^{0}}]\cong \moduleS_{\fa}[\setQ'],\\
&\sfi:=\mathsf{j}\otimes \id.
\end{split}
\end{equation}
The component $\mathsf{j}$ is $\mathsf{inc}\otimes (\mathsf{pr}\circ \mathsf{inc})$.

In the notations of Appendix \ref{S:Dmodulenoteconstr} the $D[\setK]$ -generator $\varpi_{\fa}(\setK)$ of $\moduleS_{\fa}[\setK]$ is 
\[\varpi_{\fa}(\setK)=\prod _{\lambda_+^i \in \rReg \setK^{\leq (1)^{-1}} } \frac{1}{\lambda_+^i}.\]

From now on to the end of the section I assume that $\hdelta,\hdelta' \in \rM_1^{+}$ (\ref{E:Mdef}) and 
$\hdelta,\hdelta'<(0)^0,(1)^{-1}<\hgamma$.
My next plan is to modify inclusions (\ref{E:inclFI}).

\begin{definition}\label{D:tau}{\em
Suppose $\hdelta' \lessdot \hdelta$ and $
[\hdelta',\hgamma]\supset [\hdelta,\hgamma]
$
\begin{enumerate}
\item 
If $\hdelta'\in 
\rM_2^{-}\subset \rM_1^{+}$ ( Definition \ref{D:Mpmidef0})

\begin{equation}\label{E:inlnonreg}
\inclusiont(a):= \lambda^{\hdelta'}\inclusiono(a), 
\end{equation}
(c.f. the group of homomorphisms (\ref{E:group3}) and the irregular map $\Thl$ (\ref{E:irregmap})).
\item If $\hdelta'\in \rM_1^{+}\backslash \rM_2^{-}
$
 \begin{equation}\label{E:inlreg}
\inclusiont(a):= \inclusiono(a)
\end{equation}
(c.f. the group of homomorphisms (\ref{E:group1},\ref{E:group2}) and the regular map $\Thl$ (\ref{E:regmap})).
\end{enumerate}
}\end{definition}
Define the element
\[
m_{\hdelta}:=\prod_{\halpha\in 
\rM_2^{-}\cap \setN(\fa)} \lambda^{\halpha}.
\]

\begin{definition}{\em
 $\sfk$-maps (\ref{E:kappadef}) and inclusions $[\hdelta,\hgamma]\subset [\hdelta',\hgamma]$ $\hdelta,\hdelta'\in \rM_1^{+}$ define inclusions of algebras $D(\hdelta,\hgamma]\subset D(\hdelta',\hgamma]$
\[a\to \sfk(a).\]
}\end{definition}

Inclusion $[(0)^0,\hgamma]\subset [(0)^0,\hgamma']$ define a homomorphism
\[\C[(0)^0,\hgamma']\to \C[(0)^0,\hgamma]\]
which is the identity on $\lambda^{\halpha}, \halpha\in [(0)^0,\hgamma]$ and zero on $\lambda^{\halpha}, \halpha\in [(0)^0,\hgamma']\backslash [(0)^0,\hgamma]$.
I use it to define the maps 
\begin{equation}\label{E:Fres}
\projectiono:\moduleS_{\fa}[\hdelta,\hgamma']\to \moduleS_{\fa}[\hdelta,\hgamma].
\end{equation}
Note that maps $\inclusiono$ and $\projectiono$ commute and I can define the double limit
\begin{equation}\label{E:directinv}\moduleS_{\fa}[\infty,\infty]=\underset{\underset{\hdelta}{\longrightarrow}}{\lim}\underset{\underset{\hgamma}{\longleftarrow}}{\lim}\moduleS_{\fa}[\hdelta,\hgamma].
\end{equation}
$\moduleS_{\fa}[\infty,\infty]$ is a module over $D(\infty,\infty)$.

\begin{remark}\label{R:dualssystem}{\rm
The modules 
\begin{equation}\label{E:FJdef}
\moduleS_{\fb}[\hdelta,\hgamma]:= \C[\hdelta ,(1)^{-1}] \otimes \C_{-}[(0)^0,\hgamma]\otimes \C_{+}[(0)^0,\hgamma]^{-1}
\end{equation}
and $\moduleS_{\fa'}[\hdelta,\hgamma]$ form a bidirect system similar to (\ref{E:directinv}). To define the structure involved in its definition, we have to conjugate all the constructions from Definition \ref{D:tau} with the automorphism $\shift\reflection$ (\ref{E:shift},\ref{E:involution}) in case of $\moduleS_{\fb}$ and with $\reflection$ in case of $\moduleS_{\fa'}$. In particular the pair $
\rM_2^{-}\subset \rM_1^{+}$, that appear in Definition \ref{D:tau}, has to be replaced by $
\rM_2^{+}\subset \rM_1^{-}$. 
}
\end{remark} 
Finally 
\begin{equation}\label{E:FIdefF}
\begin{split}
&\moduleS_{\ff}[\setK]:=\C_{-}[\setK^{\leq-1}]\otimes \C_{+}[\setK^{\leq-1}]^{-1}\otimes \C[\setK^{\geq 0}],\\
&\moduleS_{\ff'}[\setK]:=\C[\setK^{\leq-1}]\otimes \C_{-}[\setK^{\geq 0}]\otimes \C_{+}[\setK^{\geq 0}]^{-1}. \\
\end{split}
\end{equation}
There is an obvious map
\[\moduleS_{\ff}[\hdelta,\hgamma]{\otimes} \moduleS_{\ff'}[\hdelta,\hgamma]\to \moduleT_{\fm}[\hdelta,\hgamma]\] 
Modules  $\moduleS_{\ff}[\hdelta,\hgamma],\moduleS_{\ff'}[\hdelta,\hgamma]$ forma bidirect system with structure maps the same as $\inclusiont$ (Definition \ref{D:tau}) and $\projectiono$ (\ref{E:Fres}). See also  Remark \ref{R:dualssystem}.

There is a straightforward generalization of the above constructions for semi-closed and open intervals $(\hdelta,\hgamma],[\hdelta,\hgamma),(\hdelta,\hgamma)$, which I will use freely.

\subsubsection{Computation of $H^i_{\fc}(A)$ with complexes $\moduleT_{\bullet}(\fc)$ and $\moduleS_{\bullet}(\fc)$}\label{SS:Computation}
I will implement the idea of Proposition \ref{P:localtorisom} to use $\Tor$-functors and modules from Section \ref{E:dmod} for computation of $H_{\fa}^i({\hA})$ for the algebra $\hA$ based on the interval $[\hdelta,\hdelta']$.

 \begin{definition}
 {\em 
 With the help of the minimal $P$-resolution $F_{\bullet}(\hA)$ (\ref{E:minres}), I define the complexes 
\[\moduleT_{\bullet}(\fc):=F_{\bullet}\underset{P}{\otimes} \moduleT_{\fc}\quad \moduleS_{\bullet}(\fc)=F_{\bullet}\underset{P}{\otimes} \moduleS_{\fc}\]
$\fc=(0),\fa,\fa',\fb,\fp, \ff,\ff'$ or $\fm$. See (\ref{E:TIJdef}), (\ref{E:FIdef}), (\ref{E:FJdef}) and
for definition of $\moduleT$ and $\moduleS$.
}\end{definition}

Before I formulate precise relation of the cohomology $H_i(\moduleT(\fc))$ and $H_i(\moduleS(\fc))$ to $H^i_{\fc}(\hA)$, I will exhibit a system of parameters in $\fc$, which will justify construction of $\moduleS_{\bullet}(\fc)$.

Recall that a sequence of elements $\{x_1,\dots,x_n\}$ in the ideal $\fc\subset R$ is a {\it system of parameters} if $\Rad(x_1,\dots,x_n)=\fc$

\begin{proposition}\label{P:parameters}
Let  $\fz$  be an ideal of ${\hA[\hgamma,\halpha]}$ generated by

\begin{equation}\label{E:parametrs}
\rReg([\hgamma,\halpha]\backslash [\hbeta,\halpha]), \hbeta\leq \halpha
 \end{equation}
See (\ref{E:regconstr}) for definition of $\rReg$. $\fz$ contains in the ideal $\fc=\fri([\hgamma,\halpha]\backslash [\hbeta,\halpha])\subset {\hA[\hgamma,\halpha]}$. 
There is  also an ideal $\fy=\rReg([\hgamma,\halpha]^{\leq-1})$, which  contains in $\ff=\fri([\hgamma,\halpha]^{\leq-1})$. Then $\Rad\, \fz=\fc$, $\Rad\, \fy=\ff$.
\end{proposition} 
\begin{proof}
I set $k=\min_{[\hgamma,\halpha]\backslash [\hbeta,\halpha]}\rho(\hgamma)$.
Denote by $\fz'$ the preimage of $\fz$ in $\C[\hgamma,\halpha]$. $\fz'$ is generated by $(\lambda^l)$ and the defining relations (\ref{E:equationsaf0}). Consider a filtration $\fz'_0\subset \fz'_1\subset \cdots \fz'$, where $\fz'_i$ is generated by $\lambda^{k},\dots,\lambda^{i+k}$ and (\ref{E:equationsaf0}). 

 Let $\fc'$ be the preimage of $\fc$ in $\C[\hgamma,\halpha]$. Define the filtration $\fc'_0\subset \fc'_1\subset \cdots \subset \fc'$. $\fc'_i$ is generated by 
\[G_i=\{\lambda^{\halpha}|\halpha\in [\hgamma,\halpha]\backslash [\hbeta,\halpha], k\leq \rho(\halpha)\leq k+i\}.\]
For the proof it suffice to show that $\Rad\, \fz'=\fc'$. Note that $\Rad\, \fc'=\fc'$ because $P/\fc'={\hA[\hgamma,\halpha]}/\fc=\hA[\hbeta,\halpha]$ has no zero divisors \cite{MovStr}.

The proposition will follow if I prove that $\Rad\, \fz'_i\supset G_i$ for all $i$. Note that if $G_{i+1}\backslash G_{i}$ consists of one element, then this element is $\lambda^{k+i}$. This is true, for example, when $i=0$. If, under assumption $|G_{i+1}\backslash G_{i}|=1$, I have already proven that $\Rad\, \fz'_i\supset G_i$, then $\Rad\,\fz'_{i+1}\supset G_{i+1}$ will follow trivially. The remaining nontrivial case is $|G_{i+1}\backslash G_{i}|=2$. Note that in this case there is one relation in (\ref{E:equationsaf0}) that contains a clutter $\lambda^{\hdelta}\lambda^{\hdelta'}$ with $\lambda^{\hdelta},\lambda^{\hdelta'}\in G_{i+1}\backslash G_{i}$. By using straightening relations and working under inductive assumptions, I see that all other terms in this relation contain in $\Rad\,\fz'_i$. By using $\lambda^{k+i}$ for the Gr\"{o}bner reductions, I transform the cluttering monomial $\lambda^{\hdelta}\lambda^{\hdelta'}$ to the forms $(\lambda^{\hdelta})^2$ or $(\lambda^{\hdelta'})^2$. This implies that $\lambda^{\hdelta},\lambda^{\hdelta'}\in \Rad\, \fz'_{i+1}$.

The proof of the second equality is similar. $\hA[\hgamma,\halpha]/\ff=\hA[\hgamma,\halpha]^{\geq 0}$. The algebra $\hA[\hgamma,\halpha]^{\geq 0}$ inherits standard basis from $\hA[\hgamma,\halpha]$. By using this basis I construct inclusion $\hA[\hgamma,\halpha]^{\geq 0}\to \hA[(0)^0,\halpha]+ \hA[(3)^{-1},\halpha]$. The summand have no zero devisors \cite{MovStr} thus the radical of $\hA[\hgamma,\halpha]^{\geq 0}$ is trivial and $\ff=\Rad\, \ff$. The remainder of the proof is similar to the previous case and omitted. 
\end{proof}

The following proposition explains me how to manipulate with the ideal in local cohomology group without changing the content of the cohomology. 
 \begin{proposition}\label{P:radeq} (\cite{Twenty},Proposition 7.2 )
Let $M$ be an $\hRo$-module.
If radicals of two ideals $\fa,\fb\subset {\hRo}$ coincide $\Rad\, \fa=\Rad\, \fb=\fc$, then
\[H^i_{\fa}(M)=H^i_{\fb}(M)=H^i_{\fc}(M).\]
\end{proposition}
Here is an example of how this theorem can be used.
\begin{proposition}\label{P:identification}
Let $\fc$ be $\fa,\fa',\fb,\ff,\ff' \text{ or }\fm$.
\begin{enumerate}
\item
Let 
 $\moduleT_{\fc}$ be as in (\ref{E:TIJdef})
. Then 
\begin{equation}\label{E:loctor}
 H^{i}_{\fc}(A)\cong\Tor^{P}_{s(\fc)-i}(A,\moduleT_{\fc})
,\quad \Tor^P_{i}(A,\moduleT_{\fc})=H_{i}(\moduleT(\fc))
\end{equation}
where 
\[s(\fc)=|\setN(\fc)|.\]
 
The isomorphisms are compatible with the $\Aut$ action.
\item Let $\moduleS_{\fc}$ be as in (\ref{E:FIdef}),(\ref{E:FJdef}), 
 or (\ref{E:FIdefF}).
Then
\[H^{i}_{\fc}(\hA)\cong\Tor^{P}_{s'(\fc)-i}(\hA,\moduleS_{\fc})
, \quad \Tor_i^P(\hA,\moduleS_{\fc})=H_{i}(\moduleS(\fc))\]
where 
\begin{equation}\label{E:sschiftdef}
s'(\fc)=\begin{cases}
3-\rho(\hdelta)&\text{ if } \fc=\fa\\
\rho(\hdelta')-8&\text{ if } \fc=\fa'\\
\rho(\hdelta')-2&\text{ if } \fc=\fb\\
-\rho(\hdelta)&\text{ if } \fc=\ff\\
\rho(\hdelta')+1&\text{ if } \fc=\ff'\\
\end{cases}
\end{equation}

\end{enumerate}
\end{proposition}
\begin{proof}
The proof of the first item follow from Proposition \ref{P:localtorisom}. The functions $s$ give the number of generators of the ideal.

I prove the second item only for $\fa$.
 \begin{lemma}\label{L:idealreduction}
 $H^i_{\fa}(\hA)=H^i_{(\bar{\lambda})}(\hA)$, where $\bar{\lambda}=\rReg[\hdelta,(1)^{-1}]$,(see (\ref{E:regconstr}) for the definition of $\rReg$).
\end{lemma}
\begin{proof}
 The ideal $\fa$ contains a system of parameters $(\bar{\lambda})=\rReg[\hdelta,(1)^{-1}]$ (see (\ref{E:parametrs}) and Proposition \ref{P:parameters}). By Proposition \ref{P:radeq} $H^i_{\fa}(\hA)=H^i_{(\bar{\lambda})}(\hA)$.
\end{proof}
 
 To finish the proof, it remains to apply Proposition \ref{P:localtorisom}. Note that $s'(\fa)=|\rReg[\hdelta,(1)^{-1}]|$.
\end{proof}

\subsubsection{The $*$-duality pairing}\label{SS:thestardualitypairing}
In this section, I will study the pairing between the groups $H_{\fa}^i(\hA)$ and $H_{\fb}^j(\hA)$. It is a composition of the multiplication $\sfm$ (\ref{E:multmap}), isomorphism (\ref{E:loccohnew}) and the residue map $\sfres_{\hA}$ (\ref{E:varthetadef0}):
\begin{equation}\label{E:fundpairingdef}
(a,b):=\sfres_{\hA}(\sfm(a,b)).
\end{equation} 

I will also discuss the pairing $(a,b)$ between the groups $H_{\ff}^i(\hA)$ and $H_{\ff'}^j(\hA)$ defined by composition of (\ref{E:multmapi}) and $\sfres_{\hA}$.
\begin{proposition}\label{P:pairingdegree}
\begin{enumerate}
\item The bilinear form $(a,b)$ (\ref{E:fundpairingdef}) defines a perfect pairing between $H_{\fa}^i(\hA)$ and $H_{\fb}^j(A)$ of degree $\rk(\hdelta,\hdelta')+1$. 
The $\Aut$-weight of the pairing is equal to 
\begin{equation}\label{E:awr}
(a[\hdelta,\hdelta'],u[\hdelta,\hdelta'], r[\hdelta,\hdelta']).\quad (\text{For definition see }\ref{E:sigmaalt}).
\end{equation}
\item  $(a,b)$ defines a perfect pairing between $H_{\ff}^i(\hA)$ and $H_{\ff'}^j(\hA)$ of degree $\rk(\hdelta,\hdelta')+1$. It has weight (\ref{E:awr}).
\item \label{I:pairingdegree3} Suppose that the interval $[\hdelta,\hdelta']=[(0)^N,(1)^{N'}]$. By Remark \ref{R:groupactionloccoh} the groups  $H_{\fa}^i[(0)^N,(1)^{N'}], H_{\fb}^i[(0)^N,(1)^{N'}]$ are equipped with $\Spin(10)$-action. The pairing is compatible with this action.
\end{enumerate}
\end{proposition}
\begin{proof}

By using Proposition 
\ref{P:identification}, I identify local cohomology with the $\Tor$ groups, which I compute with the complexes $\moduleT_{\bullet}(\fa)$ and $\moduleT_{\bullet}(\fb)$.

By Proposition \ref{P:pairing}, the lift of the map $\sfm$ (\ref{E:multmap}) on the level of chains $\moduleT_{i}(\fa)$ and $\moduleT_{i}(\fb)$ 
is the product 
\begin{equation}\label{E:Dproduct}
{V}_i\otimes \moduleT_{\fa} \otimes {V}_j\otimes \moduleT_{\fb}\to {V}_{i+j}\otimes \moduleT_{\fm}
\end{equation}
of maps (\ref{E:pairingchains}) and (\ref{E:Tmap}). 

 \begin{lemma}\label{L:pairinge}
There are nondegenerate pairings 
\[\begin{split}
&\moduleT_{i}(\fa)\otimes \moduleT_{d-i}(\fb)\to \C\\
&\moduleT_{i}((0))\otimes \moduleT_{d-i}(\fm)\to \C\\
& d:=s(\fm)-s'(\fm).
\end{split}
\]
\end{lemma}
\begin{proof}
The analogue of (\ref{E:Dproduct}) for $\moduleT_{\bullet}((0))$ and $\moduleT_{\bullet}(\fm)$ is
\begin{equation}\label{E:oDproduct}
{V}_i\otimes \moduleT_{(0)} \otimes {V}_j\otimes \moduleT_{\fm}\to {V}_{i+j}\otimes \moduleT_{\fm}.
\end{equation}

The pairing between the pair $\moduleT_{\bullet}(\fa)$ and $\moduleT_{\bullet}(\fb)$ is the composition of (\ref{E:Dproduct}) with 
\begin{equation}\label{E:functional}
{V}_{d}\otimes \moduleT_{\fm}\overset{\mu\otimes \sfres_T}{\longrightarrow} \C.
\end{equation}
For the pair $\moduleT_{\bullet}((0))$ and $\moduleT_{\bullet}(\fm)$ I have to compose (\ref{E:oDproduct}) with (\ref{E:functional}).
The map $\sfres_T$ is defined in (\ref{E:rhodef}) and $\mu$ in (\ref{E:muisodef}). 
The resulting pairings are the tensor product of two nondegenerate pairings 
(see \ref{E:pairingchains} and Proposition \ref{P:Tpairingnondeg}).

\end{proof}

The map (\ref{E:functional}) defined on cohomology $H_d(\moduleT(\fm))$ is proportional to (\ref{E:varthetadef0}). Indeed, $\mu\otimes \sfres_T$ is homogeneous. It is nonzero only on the graded component of ${V}_{d}\otimes \moduleT_{\fm}$ spanned, in notations of the proof of Proposition \ref{P:charmatch}, by $\lambda^{-1}\otimes v$. By the same proposition, 
$\deg_{\C^{\times}}\lambda^{-1}\otimes v=-a$ -the grading of $\hA_0^{*}\subset H_{\fm}^{\rk(\hdelta,\hdelta')+1}({\hA})$.

The groups $\Aut$ acts on $\hA_0^{*}$ through the character $a_{\hA}(t,q,z)$ (\ref{E:sigma}). By Proposition \ref{E:relacharnot} , $a_{\hA}(t,q,z)=\chi^{-1}$. Exponents of the monomial in $t,q,z$ that define $\chi$ are tabulated in Corollary \ref{P:funequation}. From it I deduce the weight of $\sfres_{\hA}$, which is equal to the weight of the pairing.

The proof of the second statement is similar and I omit it.

To prove the  third statement I observe that by Lemma \ref{E:eqres} the maps (\ref{E:pairingchains}) are $\Spin(10)$-equivariant (I don't have to work with minimal resolutions). The map (\ref{E:Tmap}) is $\Spin(10)$-equivariant equivariant by construction. Thus the map $\sfm(a,b)$ is $\Spin(10)$-equivariant on the level of cohomology. The group $\Spin(10)$ is simple. It must  act trivially on  $\hA_0^{*}\subset H^{8(N'-N)+11}_{\fm}[(0)^N,(1)^{N'}]$. Hence $\res$ is $\Spin(10)$-equivariant map. From this I conclude that $\res\sfm(a,b)$ is $\Spin(10)$-equivariant pairing.
\end{proof}
\begin{remark}\label{R:degreeshift1}{\rm
The cohomology pairing (\ref{E:fundpairingdef}) and the map $\sfres_{{\hA[\hdelta,\hdelta']}}$ transform under $\Aut$ by the character $\chi[\hdelta,\hdelta'] $ (see \ref{E:sigmaalt}
).
 It is convenient to twist the action of $\Aut$ on $H^i_{\fa}(\hA)$ by the character 
 $\chi^{-1}[\hdelta,(1)^{-1}]$,
  $H^i_{\fb}(\hA)$ by $\chi^{-1}[(0)^{0},\hbeta]$, 
 $H^i_{\fri}(\hA)$ by  $\chi^{-1}[\hdelta,\hdelta']^{\leq −1}$,
  $H^i_{\fri'}(\hA)$ by $\chi^{-1}[\hdelta,\hdelta']^{\geq 0}$, and
  $H^i_{\fm}(\hA)$ by $\chi^{-1}[\hdelta,\hdelta']$.
Then according to  (\ref{E:chiprop})  after the twist the $\Aut$-weight  of the pairing (\ref{E:multmap}) is $t^{-4}q^{2}$. The pairing (\ref{E:multmapi}) is is equivariant with respect to the twisted action. I will also refer to this twist as  to {\it renormalized action} of $\Aut$. The grading by weight spaces of the renormalized action will be called {\it the renormalized grading}.
}\end{remark}

\subsubsection{The four-term complex $\Frg^{\fa}_{\bullet}[\hdelta,\hdelta']$}\label{S:fourterm}
 In this section, I will define a very compact four-term complex $\Frg^{\fa}_{\bullet}$ that computes $H_{\fa}^i(\hA)$ for algebra $\hA$ based on an interval $[\hdelta,\hdelta']$.
 Let us fix some notations and recall the old (\ref{D:lessdef}). Denote the subset of $\sethE$
$\sethE^{\leq 2}\cap \sethE^{\geq 0}\text{ by } \rA.$
\[\rA_1\cup \rA_2:=\{(0)^0,(12)^0,(13)^0\}\cup\{(3)^{-1},(2)^{-1},(1)^{-1}\}=\rA\]
\[\rA^{\fc}:=\begin{cases}
\rA_2&\text{ if } \fc=\fa\\
\rA_1&\text{ if } \fc=\fb
\end{cases}\]
Introduce a tensor factorization of $P=\C[\hdelta,\hdelta']=R^{\fc}\otimes Q^{\fc},\fc=\fa,\fb:$
\[
R^{\fa}:=\C[[\hdelta,\hdelta']\backslash \rA] \otimes \C[\rA_1], Q^{\fa}:= \C[\rA^{\fa}],\quad 
 R^{\fb}:=\C[[\hdelta,\hdelta']\backslash \rA] \otimes \C[ \rA_2], Q^{\fb}:= \C[\rA^{\fb}].\]
There will be  two modifications of $\Frg^{\fc}_{\bullet}[\hdelta,\hdelta']$. All of them are Koszul complexes based on $Q^{\fc}$-modules $\Mrg_{\fc}$: 
\begin{equation}\label{E:Mrgdef}
\Mrg_{\fc}:=\moduleS_{\fc}\underset{R^{\fc}}{\otimes} {\hA}.
\end{equation}
The generators $\lambda^{\halpha},\halpha\in \rA^{\fc}$
act on $\Mrg_{\fc}$ by multiplication on $\underline{\lambda}^{\halpha}$ (see (\ref{E:underlinenotation}) for notations).

In the next proposition, I show that the four-term 
Koszul complexes 
\begin{equation}\label{E:Rdef}
\Frg^{\fc}_{\bullet}:=B_{\bullet}(\Mrg_{\fc},\underline{\lambda}^{\halpha}|\halpha\in \rA^{\fc}\})
\end{equation}
compute $H_{\fc}^i(\hA)$.

\begin{proposition}\label{P:complexreduction}

There is an isomorphism
\begin{equation}\label{E:redisomorphism}
H_{\fc}^{s'(\fc)-i}[\hdelta,\hdelta']\cong\Tor_i^{P}(\hA[\hdelta,\hdelta'],\moduleS_{\fc}[\hdelta,\hdelta'])\cong H_i(\Frg^{\fc}[\hdelta,\hdelta'])
\end{equation}
$
s'(\fc),\fc=\fa,\fb$(see (\ref{E:sschiftdef}) for the definition of $s'$).
\end{proposition}
\begin{proof}
I prove the statement only for $\fc=\fa$. By Proposition \ref{P:identification}, $H_{\fa}^{-\rho(\hdelta)+3-i}\cong\Tor_i^P(\hA,\moduleS_{\fa})$. The constant $-\rho(\hdelta)+3$ is the number of elements in the sequence $\rReg[\hdelta,(1)^{-1}]$. For computation of $\Tor$ I use the Koszul resolution
 \begin{equation}\label{E:diagres}
 B_{\bullet}(\hA\otimes P,\{\underline{\lambda}^{\halpha}\}),\halpha\in[\hdelta,\hdelta'].
\end{equation}
It immediately implies that the cohomology of 
\begin{equation}\label{E:complexfortor}
B_{\bullet}(\hA\otimes \moduleS_{\fa},\{\underline{\lambda}^{\halpha}\})
\end{equation} computes $\Tor_i^P(\hA,\moduleS_{\fa})$. 

My plan is to apply construction (\ref{C:bicomplex}) to (\ref{E:complexfortor}): 
\[\begin{split}
&\sd_1a=\sum_{\halpha\in [\hdelta,\hdelta']\backslash \rA}\underline{\lambda}^{\halpha}\frac{\sd a}{\sd\theta^{\halpha}}+\sum_{\halpha\in \rA_1}\underline{\lambda}^{\halpha}\frac{\sd a}{\sd\theta^{\halpha}}\\
&\sd_2a=\sum_{\halpha\in \rA_2}\underline{\lambda}^{\halpha}\frac{\sd a}{\sd\theta^{\halpha}}.\\
\end{split}\]
See (\ref{E:underlinenotation}) for notations.
The proof is based on consideration of the spectral sequence of the bicomplex $B_{\bullet,\bullet}$. 
\begin{lemma}
The cohomology of $(B_{\bullet,j},\sd_1)$ is the first page $E_{i,j}^1$ of the spectral sequence of the bicomplex. 
I claim that $E_{i,j}^1=0, i>0$. $E_{0,j}^1=\Frg^{\fa}_{j}$.
\end{lemma}
\begin{proof}
I will use the third item of Lemma \ref{L:stopiat}. I have the following identifications $R^{\fa}=H=H_1\otimes H_2$, \[H_1=
\C_-[[\hdelta,\hdelta']^{\leq -1}]\otimes \C[\rA_1] \otimes \C_-[[\hdelta,\hdelta']^{\geq 3}],\] 
\[H_2=\C_+[[\hdelta,\hdelta']^{\leq -1}] \otimes \C_+[[\hdelta,\hdelta']^{\geq 3}]=H'_2\otimes H''_2.
\]

\begin{equation}\label{E:Sdecomp}
\moduleS_{\fa}\cong H_1\otimes H_2^{'-1}\otimes H_2^{''}\otimes \C[\rA_2]^{-1}.
\end{equation}

 $\moduleS_{\fa}$ is by construction a free $H_1$-module.
Lemma \ref{L:regularity} implies that ${\hA}$ is a free $H_2$ module. Conditions of Lemma \ref{L:stopiat} are satisfied. I conclude that $E_{i,j}^1=0, i\geq 1$. Equality $E_{0,j}^1=\Frg^{\fa}_{j}$ trivially verifies.

\end{proof}

This finishes the proof.

\end{proof}

The next proposition gives an interpretation to space of chains of $\Frg_i^{\fc}$ in terms of local cohomology.

\begin{proposition}\label{P:fockasloc}
There is an isomorphism 
\begin{equation}\label{E:IHidentification}
\Mrg_{\fa}\cong H_{\ff}^{-\rho(\hdelta)}({\hA})\otimes \C[\rA^{\fa}]^{-1},\quad \Mrg_{\fb}\cong H_{\ff'}^{\rho(\hdelta')+1}({\hA})\otimes \C[\rA^{\fb}]^{-1}
\end{equation}

\end{proposition}
\begin{proof}
I prove only for $\fa$. I denote $H_1\otimes H_2^{'-1}\otimes H_2^{''}$ in  (\ref{E:Sdecomp})  by $X$.
The tensor product $\hA$ and $X$ over $R^{\fa}$ coincides with the tensor product over $\C[\hdelta,\hdelta']$ of $\hA$ and $S_{\ff}[\hdelta,\hdelta']$ (\ref{E:FIdefF}).
By Proposition \ref{P:localtorisom} there is an isomorphism 
\begin{equation}\label{E:IeqHtensprod}
\hA[\hdelta,\hdelta']\underset{P}{\otimes} S_{\ff}[\hdelta,\hdelta']= H_{\fy}^{-\rho(\hdelta)}({\hA}),
\end{equation} where $\fy$ is generated by the regular sequence $\rReg[\hdelta,\hdelta']^{\leq -1}$. The proof follows from Proposition \ref{P:parameters}, Proposition \ref{P:radeq}.
\end{proof}

Here is an immediate corollary of the proof.
\begin{corollary}\label{C:Ionedegree}
\[H_{\ff}^i({\hA})\neq 0\text{ only for } i=-\rho(\hdelta),\]
\[H_{\ff'}^i({\hA})\neq 0\text{ only for } i=\rho(\hdelta')+1.\]
\end{corollary}

The fact that local cohomology $H_{\fa}^i({\hA})$ are nontrivial only for four consecutive values of the cohomological index is established in the following proposition.
\begin{proposition}\label{P:vanishingmain}
\[H_{\fa}^i({\hA})\neq 0\text{ only for } -\rho(\hdelta)\leq i \leq -\rho(\hdelta)+3,\]
\[H_{\fb}^i({\hA})\neq 0\text{ only for } \rho(\hdelta')-2\leq i \leq \rho(\hdelta')+1.\]
\end{proposition}
\begin{proof}
The 
statement  is a corollary of Proposition \ref{P:complexreduction}.
\end{proof}

\subsubsection{Computations of the virtual character of $H^i_{\fa}(\hA)$}\label{S:virtchar}
Fix an algebra $\hA$ based on the interval $[\hdelta,\hdelta']$.
The plan for this section is to derive functional equations satisfied by the virtual character 
\[Z^{bare}_{\fc}(t,q,z):=\sum_{i\geq 0} (-1)^{i}\chi_{H_{\fc}^i(\hA)}(t,q,z),\fc=\fa,\fb\]
of $\Aut$ (\ref{E:symmstries0}).
The character $\chi_{H_{\fc}^i(\hA)}(t,q,z)$ is a generating function of dimensions of weight subspaces $\sum_{j,u,r}\dim H_{\fc}^{i,j,r}(A)t^jq^uz^r$.
By using sets $\setN(\fc)$ (\ref{E:Ndef}) and functions (\ref{E:sigmaalt}),I define
\[a(\fc):=a(\setN(\fc)),u(\fc):=u(\setN(\fc)),r(\fc):=r(\setN(\fc)). \]
Together with $Z^{bare}_{\fa}$ I also study its renormalized version
\begin{equation}\label{E:renormchar}
 Z_{\fa}:=(-1)^{s'(\fa)}t^{a(\fa)}q^{u(\fa)}z^{r(\fa)}Z^{bare}_{\fa}.
\end{equation}
The function $s'$ is as in (\ref{E:sschiftdef}). This virtual character corresponds to the twisted action described in Remark \ref{R:degreeshift1}.

Here is the functional equation for $Z$ in the most basic form.
\begin{proposition}\label{P:eqderivation0}
$Z^{bare}_{\fa}$ and $Z_{\fa}$ satisfy
\begin{equation}\label{E:fbarebasic2}
\begin{split}
 & Z^{bare}_{\fa}(t,q,z)=\\
&=(-1)^{s'(\fm)}t^{-a(\fm)}q^{-u(\fm)}z^{-r(\fm)} Z^{bare}_{\fb}(t^{-1},q^{-1},z^{-1})
\end{split}
\end{equation}
and
\begin{equation}\label{E:frenorm3}
\begin{split}
Z^{bare}_{\fb}[\hdelta,\hdelta'](t,q,z)= Z^{bare}_{\fa}[\shift^{-1}\reflection(\hdelta'),\shift^{-1}\reflection(\hdelta)](q^{-1}t,q^{-1},Sz).
\end{split}
\end{equation}
\end{proposition}
\begin{proof}
The formula (\ref{E:fbarebasic2}) follows from Proposition \ref{P:pairingdegree} and formula (\ref{E:htands}).

The composition $\shift^{-1}$ and $\reflection$ (\ref{E:sigmaTcoh}) induce an isomorphism
\begin{equation}\label{E:tsigmaiso}
\shift^{-1}\reflection:H^i_{\fb}[\hdelta,\hdelta']\to H^i_{\fa}[\shift^{-1}\reflection(\hdelta'),\shift^{-1}\reflection(\hdelta)].
\end{equation}

Here are the rules of commuting elements of $\Aut$ and $\reflection$ and $\shift$. 
\begin{equation}
 \begin{split}
& \reflection((t,q,z)a)=(t,q^{-1},Sz)\reflection(a)\\
&\shift((t,q,z)a)=(tq^{-1},q,z)T(a)\\
&(t,q,z)\in \Aut,a\in A\\
&S\in D_5 \text{ is a certain automorphism of $\widetilde{\bT}^5$}\\
&\text{ (see Appendix A in \cite{MovStr} for details)}.
\end{split}
\end{equation}
These facts imply the formula (\ref{E:frenorm3}).
\end{proof}

Here is an alternative way to write the same set of equations.
\begin{proposition}\label{P:eqderivation}
The virtual characters $Z^{bare}_{\fa}$ and $Z_{\fa}$ are solutions of the functional equations
 \begin{equation}\label{E:fbare}
\begin{split}
 & Z^{bare}_{\fa}[\hdelta,\hdelta'](t,q,z)=\\
&=(-1)^{s'(\fm)}t^{-a(\fm)}q^{-u(\fm)}z^{-r(\fm)} Z^{bare}_{\fa}[\shift^{-1}\reflection(\hdelta'),\shift^{-1}\reflection (\hdelta)](qt^{-1},q,Sz^{-1})
\end{split}
\end{equation}
and
\begin{equation}\label{E:frenorm}
\begin{split}
Z_{\fa}[\hdelta,\hdelta'](t,q,z)=-t^{-4}q^{2}Z_{\fa}[\shift^{-1}\reflection(\hdelta'),\shift^{-1}\reflection(\hdelta)](qt^{-1},q,Sz^{-1}).
\end{split}
\end{equation}
\end{proposition}
\begin{proof}
The first formula is a direct corollary of Proposition \ref{P:eqderivation0}.

Derivation of the second formula uses formula (\ref{E:fbare}), items (\ref{I:chareq2}, \ref{I:chareq3}) from Proposition \ref{P:funequation} and equation (\ref{E:htands}), which sum up to
\[
\begin{split}
 &a(\fm)=a(\fa)+a(\fb)+4,\\
 &u(\fm)=u(\fa)+u(\fb)-2,\\
 &r(\fm)=r(\fa)+r(\fb).\\
\end{split}
\]
\end{proof}

Equation (\ref{E:frenorm}) simplifies if I set $[\hdelta,\hdelta']=[(0)^{-N-1},(1)^N]$. Denote $Z_{\fa}[(0)^{-N-1},(1)^N]$ by $Z[N]$. Then (\ref{E:frenorm}) becomes (\ref{E:ZN}).
Operator $S$ disappears in $Sz^{-1}$ because representation of $\widetilde{\bT}^5$ in $H_{\fa}^i(A[(0)^{-N-1},(1)^N])$ is a restriction of $\Spin(10,\C)$-representation and $S(z)=SzS^{-1}, S\in \Spin(10,\C)$ is an inner automorphism.
\paragraph{An explicit formula for $Z_{\fa}$} Isomorphism (\ref{E:loctor}) can be used for derivation of an explicit formula for $Z_{\fa}$:
\[Z^{bare}_{\fc}(t,q,z)=(-1)^{|[\hdelta,(1)^{-1}]|}\sum (-1)^i\Tor^{P}_{i}({\hA},\moduleT_{\fc})(t,q,z).\]
For computation of $\Tor$ groups I can use resolution (\ref{E:diagres}). The corresponding complex that computes $\Tor$ groups is
$\{{\hA}\otimes \moduleT_{\fc}\otimes \Lambda^i\spinor\}$. I conclude that 
\begin{equation}\label{E:formulapartition}
Z^{bare}_{\fa}(t,q,z)=(-1)^{|[\hdelta,(1)^{-1}]|}{\hA}(t,q,z)\moduleT_{\fa}(t,q,z)\Lambda \spinor(-t,q,z).
\end{equation}
I computed ${\hA}(t,q,z)$ in \cite{MovStr}.
\[
\begin{split}
&\moduleT_{\fa}(t,q,z)=\prod_{\halpha\in \setN(\fa)}\frac{\hat{v}_{\halpha}^{-1}}{1-\hat{v}_{\halpha}^{-1}}
\prod_{\halpha\in \setE}\frac{1}{1-\hat{v}_{\halpha}}
\prod_{\halpha\in\setN(\fa')}\frac{1}{1-\hat{v}_{\halpha}}\\
&\Lambda \spinor(t,q,z)=\prod_{\halpha\in \setN(\fm)}(1+\hat{v}_{\halpha})
\end{split}
,\]
(see notation
(\ref{E:weights})). 
One has to use caution with the formula \ref{E:formulapartition}. The problem is that the left-hand-side is  defined as a rational function for any ideal $\fc=\fri[\hdelta,(1)^k]$ its components, specifically $\moduleT_{\fri}(t,q,z)$,    makes sense as a formal power series in $q$ only when $k=-1$.
In this case $\moduleT_{\fa}(t,q,z)\Lambda \spinor(-t,q,z)=1$ in the ring of formal power series and
\[Z^{bare}_{\fa}(t,q,z)=(-1)^{|[\hdelta,(1)^{-1}]|}{\hA}(t,q,z)\]

The formula for $\hA_N^{N'}:=\hA[(0)^N,(1)^{N'}](t,q,1), B_N^{N'}:=\hA[(0)^N,(0)^{N'}](t,q,1), N<0<N'$ can be simplified (see \cite{MovStr} for details).
In \cite{MovStr} I established that
\begin{equation}\label{E:hilbertini}
\begin{split}
&B_0^1=\frac{1+3t+t^2}{(1-t)^{8}(1-qt)}, \quad  \hA_0^0=\frac{1+5t+5t^2+t^3}{(1-t)^{11}} ,\\
&K(t,q)=\left(
\begin{array}{cc}
 \frac{t \left(t^2+3 t+1\right)}{(t-1)^7 (q t-1)} & \frac{\left(t^2+3 t+1\right) (t^3+q^2)-5 q (t+1) t^2}{q^2 (t-1)^7 (q t-1)} \\
 \frac{t (t+1) \left(t^2+4 t+1\right)}{(t-1)^{10}} & \frac{\left(t^3+5 t^2+5 t+1\right) (t^3+q^2)-q \left(5 t^2+14 t+5\right) t^2}{q^2 (t-1)^{10}} \\
\end{array}
\right)
\end{split}
\end{equation}

\[\left(
\begin{array}{c}
 B_0^{r+1} \\
 \hA_0^r \\
\end{array}
\right)=
K(q^{r}t,q) \cdots K(qt,q) 
\left(
\begin{array}{c}
 B_0^1 \\
 \hA_0^0 \\
\end{array}
\right),\]
\[ \hA_N^{N'}(t,q)=\hA_0^{N'-N}(tq^{N},q).\]
The partition function 
\begin{equation}\label{E:explicitform}
{Z_{\fa}}_N^{N'}= \hA_N^{N'}(t,q)\chi[(0)^{N},(1)^{-1}]=\hA_N^{N'}(t,q)t^{4-4N}q^{-2+4N-2N^2}, N<0
\end{equation}
The character $\chi[(0)^{N},(1)^{-1}]$ is taken from Corollary \ref{C:specialcharacter}.

Here are some terms of the $q$-expansion of $Z_{\fa}=Z_{-2}q^{-2}+Z_{-1}q^{-1}+Z_{0}q^{0}+Z_{1}q^{1}+\cdots$
\begin{equation}\label{E:qexpansion}
\begin{split}
&\frac{t^4 \left(1 - 10 t^2 + 16 t^3 - 16 t^5 + 10 t^6 - t^8\right)}{q^2 (1-t)^{16}}+\\
&+\frac{t^4 \left(46 - 144 t + 116 t^2 + 16 t^3 - 16 t^5 - 116 t^6 + 144 t^7 - 46 t^8\right)}{q (1-t)^{16}}+\\
&+\frac{1}{(1-t)^{16}}\left(-1 + 16 t - 120 t^2 + 576 t^3 - 1003 t^4 + 528 t^5 - 214 t^6 + 
 592 t^7 \right.\\
& \left.- 592 t^9 + 214 t^{10}+
 - 528 t^{11} + 1003 t^{12} - 576 t^{13} + 120 t^{14} - 16 t^{15} + t^{16}
 \right)+\\
&+\frac{q \left(-16 + 210 t - 1200 t^2 + 3696 t^3 - 4704 t^4 + 2630 t^5 - 3312 t^6 + 
 3148 t^7 + 1328 t^8 - 1328 t^{10} - 3148 t^{11}+\cdots
 \right)}{t(1-t)^{16}}+\cdots
 \end{split}
 \end{equation}
 See Remark \ref{R:interpretationofcoefficients} for interpretation of the coefficients of this expansion.

\subsection{$H^i_{\fa}[\hdelta,\hdelta']$ as a function of $\hdelta$, and $\hdelta'$}
To define the limiting space of states $H^{i+\itwo}_{\fa}(\hA)$, 
I have to glue together spaces $H^i_{\fa}[\hdelta,\hdelta']$. This is not an obvious procedure because that range $-\rho(\hdelta)\leq i\leq -\rho(\hdelta)+3$, for which the groups are nontrivial, change with $\hdelta$ (Proposition \ref{P:vanishingmain}). The space $H^{i+\itwo}_{\fa}(\hA)$ will be a limit of bidirect system formed by $H^i_{\fa}[\hdelta,\hdelta']$. My plan for this section is to define the structure maps of this system and study their basic properties.

In this section all intervals $[\hdelta_i,\hdelta_j]$ satisfy $\hdelta_i\leq (0)^0,(1)^{-1}\leq \hdelta_j$.

 Let us consider a sequence $\bl =\rReg([\hdelta,(1)^{-1}])(\ref{E:regconstr}).$
By Lemma \ref{L:idealreduction} and equation (\ref{E:directlimitkoszul}) 
\begin{equation}\label{E:localviaKoszul}
H^i_{\fa}[\hdelta,\hdelta']=H^i_{(\bl)}[\hdelta,\hdelta']=\underset{\underset{n}{\longrightarrow}}{\lim}H^i(K({\hA[\hdelta,\hdelta']},\bl^n)).
\end{equation} 

Results of Subsection \ref{S:CM} will enable me to define bidirect system of linear spaces $\{H^i(K({\hA[\hdelta,\hdelta']},\bl^n))\}$ with varying $\hdelta,\hdelta',n$. Let us see how this can be implemented. 

In the following proposition I will suppress $\bl^n$-dependence in cohomology of Koszul complexes, which will be denoted by $HK^i[\hdelta,\hdelta']$.
\begin{proposition}\label{P:commutativity}
Fix $\hdelta_1<\hdelta_2<(0)^{-1},(1)^{-1}<\hdelta_3<\hdelta_4$.

There is a commutative diagram 
\begin{equation}\label{E:diagramindprob}
\begin{CD}
HK^i[\hdelta_2,\hdelta_4] @>\sfp'>> HK^i[\hdelta_2,\hdelta_3]\\
@VV\Th_{\sfq} V @VV\Th_{\sfq'} V \\
HK^{i+\codim}[\hdelta_1,\hdelta_4] @>\sfp>> HK^{i+\codim}[\hdelta_1,\hdelta_3]. 
\end{CD}
\end{equation}
The maps are induced from the commutative diagram of algebras
\begin{equation}\label{E:diagramalg}
\begin{CD}
\hA[\hdelta_2,\hdelta_4] @>\sfp'>> \hA[\hdelta_2,\hdelta_3] \\
@AA\sfq A @AA\sfq' A \\
\hA[\hdelta_1,\hdelta_4] @>\sfp>> \hA[\hdelta_1,\hdelta_3]
\end{CD}
\end{equation}
 and maps of resolutions $\Th_{\sfq},\Th_{\sfq'}$ (\ref{E:classt},Definition \ref{D:thdeffora}). \[\codim=\rho(\hdelta_2)-\rho(\hdelta_1)=\rk(\hdelta_1,\hdelta_2).\]

\end{proposition}
\begin{proof}

First, I assume that $\sfq$ and $\sfq'$ is defined by the formulas (\ref{E:group1}, \ref{E:group2}). 
By Lemma \ref{L:saturate0}, items \ref{I:saturate0first}, \ref{I:saturate0second} and Proposition \ref{P:classofextreg} , $\Th_{\sfq}$ is encoded by the class of the extension 
\begin{equation}\label{E:extsimple}
\{0\}\to \hA[\hdelta_1,\hdelta_4]\overset{\lambda^{\hdelta_1}\times ?}\longrightarrow \hA[\hdelta_1,\hdelta_4]\overset{\sfp}\to \hA(\hdelta_1,\hdelta_4]\to \{0\}.
\end{equation}
$\lambda^{\hdelta_1}\times ?$ stands for operator of multiplication on $\lambda^{\hdelta_1}$.
In (\ref{E:diagramindprob}) $\Th_{\sfq}$ is the boundary homomorphism in the Koszul cohomology corresponding to this extension.
After applying homomorphisms 
\begin{equation}\label{E:specialization}
\sfp:\hA[\hdelta_1,\hdelta_4]\to \hA[\hdelta_1,\hdelta_3], \sfp':\hA(\hdelta_1,\hdelta_4]\to \hA(\hdelta_1,\hdelta_3]
\end{equation} (\ref{E:extsimple}) transforms to \[\{0\}\to \hA[\hdelta_1,\hdelta_3]\overset{\lambda^{\hdelta_1}\times ?}\longrightarrow \hA[\hdelta_1,\hdelta_3]\to \hA(\hdelta_1,\hdelta_3]\to \{0\},\] which implies commutativity of (\ref{E:diagramindprob}).

Let us do the case when $\sfq$ and $\sfq'$ belong to the class (\ref{E:group3}). By Lemma \ref{L:saturate0}, item \ref{I:saturate0second}, conditions of 
Proposition \ref{P:extmaintheorem} are satisfied. By this proposition, the maps $\Th_{\sfq}$, $\Th_{\sfq'}$ (\ref{E:classt}) are 
inclusions 
\begin{equation}\label{E:inclusionofA}
\sfs:\hA[\hdelta_2,\hdelta_i]\overset{\lambda^{\hdelta_2}\times ?}\longrightarrow \hA(\hdelta_1,\hdelta_i],
\end{equation}
 which are obviously compatible with the specialization (\ref{E:specialization}). See Proposition \ref{P:extmaintheorem} for details on $\sfs$. This implies commutativity of (\ref{E:diagramindprob}) in the case when maps $\sfq$ and $\sfq'$ are defined by (\ref{E:group3}).

By my assumptions, conditions of Lemma \ref{L:saturate} for the pair $[\hdelta_2,\hdelta_4]\subset [\hdelta_1,\hdelta_4]$ are satisfied. It remains to apply my previous arguments to elementary homomorphisms (\ref{E:refignseq}) and use Proposition \ref{P:composition}.

\end{proof}
\begin{remark}\label{R:mainremark}
The proof of the Proposition would have been easier, had I known that in the commutative diagram 
\[\begin{CD}
x\in \Ext^{\codim}(\hA[\hdelta_2,\hdelta_3],\hA[\hdelta_1,\hdelta_4]) @>\id\otimes \sfp >> \Ext^{\codim}(\hA[\hdelta_2,\hdelta_3],\hA[\hdelta_1,\hdelta_3]) \ni \Th_{\sfq'} \\
@VV{\sfp'}^*\otimes \id V @VV{\sfp'}^*\otimes \id V \\
\Th_{\sfq}\in \Ext^{\codim}(\hA[\hdelta_2,\hdelta_4],\hA[\hdelta_1,\hdelta_4]) @>\id\otimes \sfp>> \Ext^{\codim}(\hA[\hdelta_2,\hdelta_4],\hA[\hdelta_1,\hdelta_3]) \ni y
\end{CD}\]
exist elements $x, y$ such that \[\Th_{\sfq}={\sfp'}^*\otimes \id (x) \text{ and }\Th_{\sfq'}=\id\otimes \sfp (x).\]
Then
\[y={\sfp'}^*\otimes \id(\Th_{\sfq'})=\id\otimes \sfp(\Th_{\sfq}).\]
In my formulas $\codim=\rk(\hdelta_1,\hdelta_2)$. This would imply commutativity of (\ref{E:diagramindprob}) and (\ref{E:diagramindprodual}).
\end{remark}
\begin{corollary}\label{C:comdiagramdir} Under assumptions of Proposition \ref{P:commutativity} there is a commutative diagram
\begin{equation}\label{E:diagramindprolocal}
\begin{CD}
H_{\fa}^i[\hdelta_2,\hdelta_4] @>\sfp'>> H_{\fa}^i[\hdelta_2,\hdelta_3] \\
@VV\Th_{\sfq} V @VV\Th_{\sfq'} V \\
H_{\fa}^{i+\codim}[\hdelta_1,\hdelta_4] @>\sfp>> H_{\fa}^{i+\codim}[\hdelta_1,\hdelta_3]. 
\end{CD}
\end{equation}
\end{corollary}
\begin{remark}\label{R:weightcompatibility}{\rm
The formulas (\ref{E:absrel}), (\ref{E:chath}), and (\ref{E:sigmaalt0}) imply that arrows in (\ref{E:diagramindprolocal}) commute with the renormalized $\Aut$ action defined in Remark \ref{R:degreeshift1}. 
}\end{remark}
\subsubsection{Compatibility of the $*$-duality pairing and $\Th$}
If we wish to extend pairing which exits on $H_{\fa}^i[\hdelta,\hdelta']$ to $H_{\fa}^{i+\itwo}$, we have to verify consistency of the system of maps $\Th$ and the pairing. This is what is done in this section. 

The pairing (\ref{E:fundpairingdef}) is the composition of $\sfm$ and $\sfres$. I check compatibility of $\Th$ with $\sfm$ is the next proposition.
 \begin{proposition}\label{P:pairingcomp}
Fix $[\hdelta_2,\hdelta_4]\subset [\hdelta_1,\hdelta_4]$.
 There is a commutative diagram
\[\begin{CD}
H_{\fa}^i[\hdelta_2,\hdelta_4]\otimes H_{\fb}^j[\hdelta_1,\hdelta_4] @>\sfm\circ (\id\otimes\, \sfp)>> H_{\fm}^{i+j}[\hdelta_2,\hdelta_4] \\
@VV\Th_{\sfp}\otimes\, \id V @VV\Th_{\sfp} V \\
H_{\fa}^{i+\codim}[\hdelta_1,\hdelta_4]\otimes H_{\fb}^j[\hdelta_1,\hdelta_4] @>\sfm>> H_{\fm}^{i+j+\codim}[\hdelta_1,\hdelta_4] 
\end{CD}
\]
where $\sfp$ is the map induced by the restriction homomorphism $\hA[\hdelta_1,\hdelta_4]\to \hA[\hdelta_2,\hdelta_4]$,\\ $\Th_{\sfp}\in \Ext_{\C[\hdelta_1,\hdelta_4]}^{\codim}(\hA[\hdelta_2,\hdelta_4],\hA[\hdelta_1,\hdelta_4])$ is the corresponding Thom class (Definition \ref{D:thdeffora}),
 $\sfm$ is the product map (\ref{E:multmap}).
\end{proposition}
\begin{proof}
By using Proposition \ref{P:identification}, I identify local cohomology with the $\Tor$ groups. 
It suffice to prove commutativity of 
\[\begin{CD}
\Tor^{P}_i(\hA[\hdelta_2,\hdelta_4], \moduleT_{\fa})\otimes \Tor_j^{P}(\hA[\hdelta_1,\hdelta_4],\moduleT_{\fb}) @>\sfm\circ (\id\otimes\, \sfp)>> \Tor_{i+j}^{P}(\hA[\hdelta_2,\hdelta_4],\moduleT_{\fm}) \\
@VV\Th_{\sfp}\otimes\, \id V @VV\Th_{\sfp} V \\
\Tor^{P}_{i-\codim}(\hA[\hdelta_1,\hdelta_4], \moduleT_{\fa})\otimes \Tor_j^{P}(\hA[\hdelta_1,\hdelta_4],\moduleT_{\fb}) @>\sfm>> \Tor_{i+j-\codim}^{P}(\hA[\hdelta_1,\hdelta_4],\moduleT_{\fm}) 
\end{CD}\]
where $\moduleT_{\fa}= \moduleT_{\fa}[\hdelta_1,\hdelta_4],\moduleT_{\fb}= \moduleT_{\fb}[\hdelta_1,\hdelta_4],\moduleT_{\fm}= \moduleT_{\fm}[\hdelta_1,\hdelta_4]$ as in Section \ref{S:modulesT} and $P=\C[\hdelta_1,\hdelta_4]$.
Note that by Proposition \ref{P:localtorisom}, the groups $\{\Tor^{\C[\hdelta_1,\hdelta_4]}_i(\hA[\hdelta_2,\hdelta_4], \moduleT_{\fa}[\hdelta_1,\hdelta_4])\}$ are equal to $\{\Tor^{\C[\hdelta_2,\hdelta_4]}_i(\hA[\hdelta_2,\hdelta_4], \moduleT_{\fa}[\hdelta_2,\hdelta_4])\}$ with some index shift.

As in the proof of Proposition \ref{P:commutativity}, I first verify the statement for $\Th_{\sfp}$ which belongs to one of the groups (\ref{E:group1}, \ref{E:group2} or \ref{E:group3}). $\Th_{\sfp}$ from (\ref{E:group1}, \ref{E:group2}) are the boundary homomorphisms. By \cite{CE} Proposition 2.5 Chap XI, I have a commutative diagram that is related to the boundary homomorphism corresponding to the extension (\ref{E:extsimple})

\begin{equation}\label{E:diagramproductCE}
\begin{CD}
\Tor^{P}_i(\hA(\hdelta_1,\hdelta_4], \moduleT_{\fa})\otimes \Tor_j^{P}(\hA[\hdelta_1,\hdelta_4],\moduleT_{\fb}) @>>> \Tor_{i+j}^{P\otimes P}(\hA(\hdelta_1,\hdelta_4]\otimes \hA[\hdelta_1,\hdelta_4],\moduleT_{\fa}\otimes \moduleT_{\fb} ) \\
@VV\Th_{\sfp}\otimes\, \id V @VV\widetilde{\Th}_{\sfp} V \\
\Tor^{P}_{i-1}(\hA[\hdelta_1,\hdelta_4], \moduleT_{\fa})\otimes \Tor_j^{P}(\hA[\hdelta_1,\hdelta_4],\moduleT_{\fb}) @>>> \Tor_{i+j-1}^{P\otimes P}(\hA[\hdelta_1,\hdelta_4]\otimes \hA[\hdelta_1,\hdelta_4],\moduleT_{\fa}\otimes \moduleT_{\fb}) 
\end{CD}
\end{equation}
$\widetilde{\Th}_{\sfp}$ is boundary map corresponding to extension (\ref{E:extsimple}) tensored over $\C$ on $\hA[\hdelta_1,\hdelta_4]$. Observe that for any homomorphism $C\to B$ of commutative algebras the natural map 
\begin{equation}\label{E:functmap}
 \Tor^C_i(M,N)\to \Tor^B_i(M\underset{C}{\otimes} B,N \underset{C}{\otimes} B)
\end{equation}
 is compatible with the boundary homomorphisms (whenever $?\underset{A}{\otimes} B$ transforms an exact sequence to an exact sequence). I apply this comment to $C=P\otimes P$, $B=P$, and the right column of \ref{E:diagramproductCE}. This way obtain the claim for $\Th_{\sfp}$ from (\ref{E:group1}, \ref{E:group2}). For $\sfp$ from (\ref{E:group3}) the proof is similar but relies only on the fact that (\ref{E:functmap}) is a map of functors.
 
I finish the proof as Proposition \ref{P:commutativity} by applying Lemma \ref{L:saturate} for the pair $[\hdelta_2,\hdelta_4]\subset [\hdelta_1,\hdelta_4]$, decomposing $\Th_{\sfp}$ into the Yoneda product of elementary $\Th_{\sfp_i}$ I already analyzed. By using associativity of the product, I prove by induction $\Th_{\sfp_n}\cdots \Th_{\sfp_{i+1}}\sfm((\Th_{\sfp_{i}}\cdots \Th_{\sfp_{1}}a)\sfp(b))=\Th_{\sfp_n}\cdots \Th_{\sfp_{i+2}}\sfm((\Th_{\sfp_{i+1}}\Th_{\sfp_{i}}\cdots \Th_{\sfp_{1}}a)\sfp(b))$.

\end{proof}
\begin{proposition}\label{P:injecttop}
In the assumptions of Proposition \ref{P:pairingcomp}, the map \[\Th_{\sfp}:H_{\fm}^{\rk(\hdelta_2,\hdelta_4)+1}[\hdelta_2,\hdelta_4]\to H_{\fm}^{\rk(\hdelta_1,\hdelta_4)+1}[\hdelta_1,\hdelta_4]\]
 is injective.
\end{proposition}
\begin{proof}
I will use the complex $\moduleT_{\bullet}(\fm)$ for computation of $H_{\fm}^{i}(A)$ (Proposition \ref{P:identification}).
The map $\Th_{\sfp}:\moduleT_{i}(\fm)[\hdelta_2,\hdelta_4]\to \moduleT_{i-\codim}(\fm)[\hdelta_1,\hdelta_4]$ is the the $\C$-adjoint to the map $\sfp:F_{\bullet}(\hA[\hdelta_1,\hdelta_4])\to F_{\bullet}(\hA[\hdelta_2,\hdelta_4])$, which is understood in the sense of the graded duality relative to $\deg_{\C^{\times}}$-grading. Then on the level of cohomology $\Th_{\sfp}$ is $\C$-adjoint to the onto map $\sfp:\hA[\hdelta_1,\hdelta_4]\to \hA[\hdelta_2,\hdelta_4]$. 
The adjoint must be injective.
\end{proof}

The central result of this section is the formula for the adjoint for $\Th$. It appears in the next proposition. 
 \begin{proposition}\label{C:adjoint}
In the assumptions of Proposition \ref{P:pairingcomp} the pairing (\ref{E:fundpairingdef}) satisfies
\[(\Th_{\sfp} (a),b)=(a,\sfp(b))\]
where $a\in H_{\fa}^{i}[\hdelta_2,\hdelta_4]$,$b\in H_{\fb}^{j}[\hdelta_1,\hdelta_4]$, $i+j+\rho(\hdelta_2)-\rho(\hdelta_1)=\rk(\hdelta_1,\hdelta_4)+1$.

The weight  of $\Th_{\sfp}$ with respect to the renormalized $\Aut$ action (see Remark \ref{R:degreeshift1}) is equal to zero.
\end{proposition}
\begin{proof}It follows from Propositions \ref{P:pairingcomp}, \ref{P:injecttop} because
\[\begin{split}
&(\Th_{\sfp} (a),b)=\sfres_{\hA[\hdelta_1,\hdelta_4]}\circ \sfm(\Th_{\sfp} (a),b)=\sfres_{\hA[\hdelta_1,\hdelta_4]}\circ\Th_{\sfp}\circ \sfm( a,\sfp(b))=\\
&=\sfres_{\hA[\hdelta_2,\hdelta_4]}\circ \sfm( a,\sfp(b))=(a,\sfp(b)).
\end{split}\]
By Propositions \ref{P;rightchoice}, \ref{P:injecttop} , I can normalize $\sfres_{\hA[\hdelta_2,\hdelta_4]}$ so that it is equal to $\sfres_{\hA[\hdelta_1,\hdelta_4]}\circ\Th_{\sfp}$.

It is suffice to prove the second statement for elementary maps $\sfp_i$ from Lemma \ref{L:saturate}. By (\ref{E:absrel}) and  (\ref{E:chtthcompatibility1},\ref{E:chtthcompatibility2},\ref{E:chtthcompatibility3}) the renormalized weight of $\Th_{\sfp_i}$ is equal to zero.
\end{proof}

The following statement is a modification of Corollary \ref{C:adjoint}. It could be interpreted as the statement of compatibility of the bidirect system defined by (\ref{E:diagramindprolocal}) and the pairing.

\begin{proposition}\label{P:adjointdiagran}
The graded dual (with respect to $\Aut$-weight subspaces decomposition) of the diagram (\ref{E:diagramindprolocal}) is isomorphic to
 \begin{equation}\label{E:diagramindprodual}
\begin{CD}
H_{\fb}^{i'}[\hdelta_2,\hdelta_4] @>\Th_{\sfp'}>> H_{\fb}^{i'+\codim'}[\hdelta_2,\hdelta_3] \\
@AA\sfq A @AA\sfq'A \\
H_{\fb}^{i'}[\hdelta_1,\hdelta_4] @>\Th_{\sfp}>> H_{\fb}^{i'+\codim'}[\hdelta_1,\hdelta_3]. 
\end{CD}
\end{equation}
 The maps are induced from the commutative diagram of algebras (\ref{E:diagramalg})
\[\codim'=
\rho( \hdelta_4)-\rho( \hdelta_3),i+i'=\rho (\hdelta_4)-\rho( \hdelta_2).\]
\end{proposition}
\begin{proof}
I start the proof with a lemma.
 \begin{lemma}\label{L:weights}
\begin{enumerate}
\item $H_{\fa}^{i}(\hA)$ is a positive energy $\bT$-space.
\item $H_{\fb}^{i}(\hA)$ is a negative energy $\bT$-space.
\item \label{I:weights3} $H_{\fm}^{i}[\hdelta,(1)^{-1}]$ is a positive energy $\bT$-space with the lowest $\bT$ weight $u[\hdelta,(1)^{-1}]$ (\ref{E:sigmaalt}).
\end{enumerate}
\end{lemma}
\begin{proof}
 By Lemma \ref{E:eqres}, $F_{i}(\hA[\hgamma,\hgamma'])=V_i\otimes \C[\hgamma,\hgamma']$ where $V_i$ are finite-dimensional $\bT$-representations. By Remark \ref{R:lowest}), $\moduleT_{\fa}$ ($\moduleT_{\fb}$) is a positive (negative) energy space with respect to the $\deg_{\bT}$-grading.
 From this I conclude that $\moduleT_i(\fa)$ and $ H_{\fa}^{i}(\hA)$ are positive energy spaces. Spaces $\moduleT_i(\fb)$ and $H_{\fb}^{i}(\hA)$ have negative energy. 
 
 The proof of the third item follows from Proposition \ref{E:diality} and Proposition \ref{P:charmatch}.
\end{proof}

 I conclude that the pairing (\ref{E:fundpairingdef}) splits into a sum of nondegenerate pairings between finite-dimensional weight spaces.

The statement follow from Corollary \ref{C:comdiagramdir} and Corollary \ref{C:adjoint}.

\end{proof}

Commutation rules for $\Th$, $\reflection$ and $\shift$ are given below.
 \begin{proposition}\label{E:morecompatibility}
Fix $\hdelta<\hdelta'<\hgamma<\hgamma'$.

Let \[\sfq:\hA[\reflection(\hgamma'),\reflection(\hdelta)]\to \hA[\reflection(\hgamma'),\reflection(\hdelta)],\quad q=\reflection\circ p\circ \reflection\]
 where $\sfp:\hA[\hdelta,\hgamma']\to \hA[\hdelta',\hgamma']$. Then
\[\reflection \circ\Th_{\sfp} =\Th_{\sfq}\circ\reflection.\]
 The map $\shift$ (\ref{E:tisomorphism}) satisfies
\[\shift \circ\Th_{\sfp} =\Th_{\sfp}\circ \shift.\]
There are similar identities for projections $\sfp:\hA[\hdelta,\hgamma']\to \hA[\hdelta,\hgamma]$.
\end{proposition}
\begin{proof}
I leave the proof as an exercise.
\end{proof}

 \section{Bounds on weights and on dimensions of weight spaces of $H_{\fa}^{i}[\hdelta,\hdelta']$
 }\label{S:lowerbound}
The spaces of states $H_{\fa}^{i+\itwo}$ which I will define in Section \ref{S:limiting} carry the $\deg_{\bT}$-grading. The finite approximations $H_{\fa}^{i}[\hdelta,\hdelta']$ are positive energy $\bT$-representations (Lemma \ref{L:weights}). Does $H_{\fa}^{i+\itwo}$ share this property? It should, if we hope that $H_{\fa}^{i+\itwo}$ is a representation of the Virasoro algebra with weights bounded from below, as discussed in the introduction.
In this section I will do all the technical work needed to answer this question.

Let $\hA$ be the algebra based on $[\hdelta,\hdelta']$. 
I will recast, as usual, my local cohomology computations to computations of certain $\Tor$-functors (Proposition \ref{P:localtorisom}). It will be convenient to do computations of $\Tor$ groups by using a special non-minimal free resolution of ${\hA}$. Algebra ${\hA}$ has  Koszul property \cite{MovStr}. I will use a special $P$-resolutions of ${\hA}$ which comes with any Koszul algebra. 
I will spend the next section reviewing this.

\subsection{The quadratic dual}
In this section, I will describe the Koszul resolution mentioned  in the previous paragraph.
Recall (see e.g. \cite{RoccoSturmfels}) that an algebra $B$ is {\it G-quadratic} if its defining ideal  has a Gr\"{o}bner basis of quadrics with respect to some
system of generators and some term order.
In \cite{MovStr} I showed that ${\hA}$ is G-quadratic in the order (\ref{E:orderaf}).

The quadratic dual $\hA^{!}$ (see e.g. \cite{PP},p. 5 for the definition )has the following description.
Let \[T:=\bigoplus_{n\geq 0} \spinor[\hdelta,\hdelta']^{\otimes n}\] be the free graded algebra on $\spinor[\hdelta,\hdelta']$ (\ref{E:spinordef}) in degree minus one. I choose negative grading to make it compatible with Chevalley-Eilenberg construction.

Let $\Gamma(5)$ be the $\Spin(10)$-intertwiner $\Lambda^5V\to \Sym^2\sspinor$ (see \cite{BVinbergALOnishchik}). In the base $\{v_{s}\},\{\theta_{\rbeta}\}$ it defines the tensor $\Gamma^{\ralpha\rbeta}_{s_1,\dots,s_5}$.
Define a set of relations 
\begin{equation}\label{E:mainrelations}
r^{l,l'}_{s_1,\dots,s_5}=\sum_{\ralpha\rbeta\in E}\Gamma^{\ralpha\rbeta}_{s_1,\dots,s_5}[\theta_{\ralpha^l},\theta_{\rbeta^{l'}}]\sim 0\quad 1\leq s_1<\dots<s_5\leq 10, \ralpha^l,\rbeta^{l'}\in [\hdelta,\hdelta']
\end{equation}
$[\cdot,\cdot]$ stands for the graded commutator.

Here is the homological interpretation of $\hA^{!}$.
\begin{proposition}\label{P:Koszul}\cite{MovStr}
The algebra $\Ext_{A}(\C,\C)$ is isomorphic to the Koszul dual $\hA^{!}=T/\fr$, where ideal $\fr$ is generated by relations (\ref{E:mainrelations}).
\end{proposition}
It is a standard fact about Koszul duality (see \cite{PP}) that if $\hA$ is commutative, then $\hA^!$ is a universal enveloping of a graded Lie algebra . In my case, I denote this Lie algebra by $L$. It is generated by the odd vector space $\spinor[\hdelta,\hdelta']$. Then 
\begin{equation}\label{E:lalgdef}
\hA^!=U(L),
\end{equation} where $U$ stands for the universal enveloping algebra.

Fix a graded vector space $V=\bigoplus_{i\leq 0} V_i$. I use standard notations in homological algebra.
$V[1]_i=V_{i+1}$
Denote by 
\[ \begin{split}
&\Sym(V[1])
= k \oplus (V_0) \oplus (V_{-1} \oplus V_0 \wedge V_0) \oplus (V_{-2} \oplus V_{-1} \otimes V_0 \oplus V_0 \wedge V_0 \wedge V_0)
 \oplus \cdots\\
&=\Sym(V[1])_0\oplus\Sym(V[1])_{-1}\oplus\cdots
\end{split}\]
the graded exterior algebra. Chevalley-Eilenberg differential graded algebra is
$\Hom(\Sym( L[1]),\C)$. Its differential satisfies graded Leibniz rule. On generators $t^a\in\Hom( L,\C) $ it is defined as\[d t^a = - \frac{1}{2} C^a_{b c} t^b \wedge t^c.\] 
By construction, the differential in such grading has the degree $+1$. If I follow this practice, the linear space $\rHom(L_1,\C)\cong {\hA}_1$ will have cohomological degree two. For me, it will be convenient to make the suspension shift in the opposite direction 
\begin{equation}\label{E:CEdef}
CE:=\rHom(\Sym(L[-1]),\C).
\end{equation}
 $CE$ and $\Hom(\Sym( L[1]),\C)$ are isomorphic as $\Z_2$-graded algebras. The differential in $CE_{\bullet}=\{CE_{\bullet}\}_{i\geq 0}$ has degree $-1$ and ${\spinor[\hdelta,\hdelta']^*}_1\subset P$ acquires degree zero. Components of $CE_{\bullet}$ are free modules over $P$. 
 \begin{proposition}
The cohomology of $CE_{\bullet}$ is isomorphic to ${\hA}$. 
 \end{proposition}
 \begin{proof}
This is a standard fact in the theory of commutative Koszul algebras. In my case, it follows from Proposition \ref{P:Koszul}. 
 For details see \cite{PP}.
 \end{proof}
\subsection{Stanley-Reisner algebra $SR$}
The estimate of the lowest $\bT$-weight in $H^i_{\fa}(\hA)$ will rely on the following result.
\begin{proposition}\label{P:injection}
Fix $[\hgamma,\hgamma' ]\subset [\hdelta,\hdelta' ]$. Let $\sfp^!:\hA^![\hgamma,\hgamma' ]\to \hA^![\hdelta,\hdelta' ]$ be the map induced by projection $\sfp:\hA[\hdelta,\hdelta' ]\to \hA[\hgamma,\hgamma' ]$. Then $\sfp^!$ is injective.
\end{proposition}
\begin{proof}
The total order (\ref{E:orderaf})
 on $\sethE$ induces a total order on $[\hdelta,\hdelta']\subset \sethE$. It defines degree-lexicographic filtration $\{F^i\}$ on $\hA[\hdelta,\hdelta' ]$. It follows directly from straightening law in ${\hA[\hdelta,\hdelta']}$ that $Gr{\hA[\hdelta,\hdelta']}$ is the Stanley-Reisner algebra $SR[\hdelta,\hdelta']$ (see \cite{MillerandSturmfels} for details). It is the quotient of $P$ by the ideal generated by 
\begin{equation}\label{E:SRrelations}
\lambda^{\halpha}\lambda^{\hbeta}=0\quad \forall\halpha,\hbeta\in [\hdelta,\hdelta' ]| \halpha \nless \hbeta \text{ and } \halpha \ngtr \hbeta.
\end{equation}
 Suppose that 
 \begin{equation}\label{E:inclusion}
 [\hgamma,\hgamma' ]\subset [\hdelta,\hdelta' ].
 \end{equation}
 The homomorphism $\sfp$ induces a homomorphism
 \[\sfp:SR[\hdelta,\hdelta' ]\to SR[\hgamma,\hgamma' ].\]
Inclusion (\ref{E:inclusion}) induces a wrong direction emedding of algebras \[\sfp':SR[\hgamma,\hgamma' ]\to SR[\hdelta,\hdelta' ]\] such that $\sfp\circ \sfp'=\id$. Thus $SR[\hgamma,\hgamma' ]$ is a retract of $SR[\hdelta,\hdelta' ]$. 

Denote $\hA[\hdelta,\hdelta' ]$ by $A$ and $\hA[\hgamma,\hgamma' ]$ by $B$. We know from \cite{MovStr} that the algebras $A$ and $B$ are Koszul, that is $A^!\cong\bigoplus_{i\geq 0} \Ext_A^i(\C,\C)$ and $B^!\cong \bigoplus_{i\geq 0} \Ext_B^i(\C,\C)$. 
I will prove the statement which is dual to injectivity. It is the statement that the map $\sfp^{!*}:\Tor^A_i(\C,\C)\to \Tor^B_i(\C,\C)$ is onto. Algebras $Gr A$, $Gr B$ that  defined with respect to filtrations $\{F_A^i\},\{F_B^i\}$ coincide with $SR[\hdelta,\hdelta' ]$ and $ SR[\hgamma,\hgamma' ]$. Filtration $\{F_A^i\}$ and $\{F_B^i\}$ defines filtrations on the bar-complexes that compute 
$\Tor$-groups. These filtrations lead to spectral sequences. By Proposition 7.1 \cite{PP}, $Gr A$ and $Gr B$ are also Koszul. Straightening law in the algebra and its $Gr$ implies equality of the Hilbert series $A(t)=Gr A(t)$, $B(t)=Gr B(t)$. Thus 
\begin{equation}\label{E:compgen}
A^!(t)=Gr A^!(t)\text{ and }B^!(t)=Gr B^!(t).
\end{equation}
The first pages of the spectral sequences are equal to $\Tor^{Gr A}(\C,\C)$ and $\Tor^{Gr B}(\C,\C)$ respectively.
By Koszulness, the first pages are also equal to $Gr A^!$ and $Gr B^!$. Equalities of  the generating functions (\ref{E:compgen}) imply that spectral sequences collapse on the first page.

The map $Gr\, \sfp^{!*}:\Tor^{Gr A}(\C,\C)\to\Tor^{Gr B}(\C,\C)$ has the left inverse $\sfp'^{!*}:\Tor^{Gr B}(\C,\C)\to \Tor^{Gr A}(\C,\C)$. Thus the map $\sfp^{!*}:\Tor^{A}(\C,\C)\to\Tor^{B}(\C,\C)$ must be surjective.
\end{proof}
\begin{remark}
The proof of Proposition \ref{P:injection} goes through if I replace the pair of closed intervals by (semi)-open intervals.
\end{remark}

The group $\Aut$ acts on $ \spinor[\hdelta,\hdelta']$ and $L[\hdelta,\hdelta']$. The action on generators in the notations (\ref{E:weights}) is given by
\[\Pi^*(g)\theta_{\halpha}:=\hat{v}^{-1}_{\halpha}(t,q,z)\theta_{\halpha}\]

It is fairly easy to describe Koszul dual $SR^![\hdelta,\hdelta']$ by  using the general definition of $A^!$ from \cite{PP} p.6.

\begin{proposition}\label{E:pathalg}
Algebra $SR^![\hdelta,\hdelta']$ has a basis that consists of monomials $\theta_{\halpha_1}\dots\theta_{\halpha_n}$ such that $(\halpha_i, \halpha_{i+1})$ is a clutter in $\sethE$. The product of two monomials $\theta_{\halpha_1}\dots\theta_{\halpha_n}$ ,$\theta_{\hbeta_1}\dots\theta_{\hbeta_{n'}}$ is  concatenation of monomials  if $(\halpha_n,\hbeta_1)$ is a clutter and is zero if  $(\halpha_n,\hbeta_1)$ is not.
 The algebra $SR^![-\infty,\infty]$ is a path algebra (without  local units) of the quiver obtained from the graph (\ref{P:kxccvgw14}) by replacing each edge by a pair of oppositely oriented edges. $SR^![\hdelta,\hdelta']$ is a subalgebra generated by $\theta_{\halpha}, \halpha\in [\hdelta,\hdelta']$
\end{proposition}
\begin{figure}[h]
\centering
\begin{tikzpicture}[ scale=0.8]
	\begin{pgfonlayer}{nodelayer}
		\node [circle, fill=black, draw, style=none,label=180:\tiny{$(25)^{r-1}$}] (0) at (-7.5, -3.5) {.};
		\node [circle, fill=black, draw, style=none,label=0:\tiny{$(35)^{r-1}$}] (1) at (-4.5, -3) {.};
		\node [circle, fill=black, draw, style=none,label=180:\tiny{$(34)^{r-1}$}] (2) at (-7.5, -4.5) {.};
		\node [circle, fill=black, draw, style=none,label=180:\tiny{$(5)^{r-1}$}] (3) at (-6, -3) {.};
		\node [circle, fill=black, draw, style=none,label=180:\tiny{$(4)^{r-1}$}] (4) at (-7.5, -1.5) {.};
		\node [circle, fill=black, draw, style=none,label=180:\tiny{$(45)^{r-1}$}] (5) at (-7.5, -2.5) {.};
		\node [circle, fill=black, draw, style=none,label=0:\tiny{$(3)^{r-1}$}] (6) at (-4.5, -1) {.};
		\node [circle, fill=black, draw, style=none,label=180:\tiny{$(12)^{r}$}] (7) at (-7.5, 0.5) {.};
		\node [circle, fill=black, draw, style=none,label=180:\tiny{$(2)^{r-1}$}] (8) at (-7.5, -0.5) {.};
		\node [circle, fill=black, draw, style=none,label=180:\tiny{$(1)^{r-1}$}] (9) at (-6, 1) {.};
		\node [circle, fill=black, draw, style=none,label=0:\tiny{$(13)^{r}$}] (10) at (-4.5, 1) {.};
		\node [circle, fill=black, draw, style=none,label=180:\tiny{$(0)^{r}$}] (11) at (-6, -1) {.};
		\node [circle, fill=black, draw, style=none,label=180:\tiny{$(23)^{r}$}] (12) at (-7.5, 2.5) {.};
		\node [circle, fill=black, draw, style=none,label=180:\tiny{$(14)^{r}$}] (13) at (-7.5, 1.5) {.};
		\node [circle, fill=black, draw, style=none,label=0:\tiny{$(24)^{r}$}] (14) at (-4.5, 3) {.};
		\node [circle, fill=black, draw, style=none,label=180:\tiny{$(15)^{r}$}] (15) at (-6, 3) {.};
		\node [circle, fill=black, draw, style=none,label=180:\tiny{$(25)^{r}$}] (16) at (-7.5, 4.5) {.};
		\node [circle, fill=black, draw, style=none,label=180:\tiny{$(34)^{r}$}] (17) at (-7.5, 3.5) {.};
		\node [circle, fill=black, draw, style=none,label=0:\tiny{$(35)^{r}$}] (18) at (-4.5, 5) {.};
		\node [circle, fill=black, draw, style=none,label=180:\tiny{$(5)^{r}$}] (19) at (-6, 5) {.};
		\node [ style=none] (20) at (-6, -5) {\tiny $\cdots$};
		\node [ style=none] (21) at (-7.5, -5.5) {\tiny $\cdots$};
		\node [style=none] (22) at (-6, 6) {\tiny $\cdots$};
		\node [style=none] (23) at (-7.5, 6) {\tiny $\cdots$};
		\node [style=none] (24) at (-4.5, -5) {\tiny $\cdots$};
	\end{pgfonlayer}
	\begin{pgfonlayer}{edgelayer}
		\draw [thick] (19.center) to (23.center);
		\draw [thick] (15.center) to (14.center);
		\draw [thick] (16.center) to (17.center);
		\draw [thick] (18.center) to (19.center);
		\draw [thick] (13.center) to (12.center);
		\draw [thick] (10.center) to (9.center);
		\draw [thick] (7.center) to (8.center);
		\draw [thick] (11.center) to (6.center);
		\draw [thick] (5.center) to (4.center);
		\draw [thick] (1.center) to (3.center);
		\draw [thick] (0.center) to (2.center);
		\draw [thick] (15.center) to (12.center);
		\draw [thick] (13.center) to (9.center);
		\draw [thick] (7.center) to (9.center);
		\draw [thick] (11.center) to (8.center);
		\draw [thick] (11.center) to (4.center);
		\draw [thick] (5.center) to (3.center);
		\draw [thick] (0.center) to (3.center);
		\draw [thick] (15.center) to (17.center);
		\draw [thick] (16.center) to (19.center);
		\draw [thick] (22.center) to (19.center);
		\draw [thick] (15.center) to (19.center);
		\draw [thick] (15.center) to (9.center);
		\draw [thick] (9.center) to (11.center);
		\draw [thick] (11.center) to (3.center);
		\draw [thick] (20.center) to (2.center);
		\draw [thick] (3.center) to (20.center);
		\draw [thick] (20.center) to (24.center);
		\draw [thick] (21.center) to (20.center);
	\end{pgfonlayer}
\end{tikzpicture}

\begin{equation}
\label{P:kxccvgw14}
\end{equation}
\end{figure}

\begin{proof}
Follows directly form (\ref{E:SRrelations}), definition of $A^!$ and Koszul property of $SR^![\hdelta,\hdelta']$. I leave derivation of  the  graph \ref{P:kxccvgw14} form \ref{P:kxccvgw13} to the reader.
\end{proof}

 Any directed path on the diagram \ref{P:kxccvgw14} is characterized by a sequence of vertices $\setP:=(\hdelta_1,\dots,\hdelta_n)$ it traverses.  The weight $w=(a,u,r)$ of the corresponding monomial $\theta_{\hdelta_1}\cdots,\theta_{\hdelta_{n}}, n=-a$ in $SR^![\hdelta,\hdelta']$ is  \[\left(-n,\sum_{\hdelta\in \setP}-u(\hdelta),\sum_{\hdelta\in \setP}-r(\hdelta)\right).\] 
 \begin{proposition}\label{P:pathestimates}
 Dimension of $\Aut$-weight space $SR^![\hdelta,\hdelta']^w,w=(a,u,r)$ is bounded by a constant $C_a$ that doesn't depend on  $u$ and $r$.
 
In addition to that,  the number   of basis monomials in $SR^![\hdelta,\hdelta']$ of degree $a$ that contain  $\theta_{\halpha}\in \spinor[\hdelta,(1)^{-1}]$ and   $\theta_{\halpha}\in \spinor[(0)^0,\hdelta']$ is bounded by the constant $C_a$.

All the constants do not depend on $[\hdelta,\hdelta']$.
 \end{proposition}
 \begin{proof}
 By Proposition \ref{E:pathalg} $SR^![\hdelta,\hdelta']^w$ has a basis labelled by paths, whose vertices belong to $[\hdelta,\hdelta']$. Let us fix such a path $\setP$.   The choice of $w$ determines the length $|a|$ (by construction $a$ is negative) of  $\setP$. 
 Denote $\lu(\setP):=\min_{\setP} u(\hdelta)$,$\uu(\setP):=\max_{\setP} u(\hdelta)$.
  it follows from the structure of the graph (\ref{P:kxccvgw14}) that $|u(\hdelta_i)-u(\hdelta_{i+1})|\leq 1$. From this I conclude that $\uu(\setP)-u\leq  |a|,$ and $ u-\lu(\setP)\leq |a|$. Thus \[u-|a| \leq \lu(\setP) \leq \uu(\setP)\leq u+|a|.\]
There is a finite number  $C_{|a|}$ of sequences $\setP\subset \sethE$ that with $|a|$ elements that satisfy these conditions.

To prove the second statement I notice that in this case $\lu(\setP)\leq 0\leq \uu(\setP)$. Thus  $-|a|<\lu(\setP)\leq 0\leq \uu(\setP)<|a|$. The rest of the proof repeats the previous arguments.
\end{proof}
\begin{remark}\label{R:isorepresentations}
It follows from the proof of Proposition \ref{P:injection} that there is an isomorphism  of $\Aut$ representations $\hA^![\hdelta,\hdelta']$ and  $SR^![\hdelta,\hdelta']$. There an an isomorphism $\Aut$ representations in $L[\hdelta,\hdelta']$ and  the $LSR^![\hdelta,\hdelta']$.
\end{remark}
\begin{proposition}\label{E:random}
 Let $L^u$ be $\bT$-weight subspace of $L[\hdelta,\hdelta']. $There is a direct sum decomposition $L=L^{\geq 0}+L^{< 0}=\bigoplus_{u\geq 0} L^u+\bigoplus_{u< 0}L^u$.
 By Proposition \ref{P:injection}  the Lie subalgebra in $L=L[\hdelta,\hdelta']$ that is generated by $\spinor[\hdelta,(1)^{-1}]$ is isomorphic to $L[\hdelta,(1)^{-1}]$. 
 By construction $L[\hdelta,(1)^{-1}]\subset L^{> 0}, L[(0)^0,\hdelta]\subset L^{\leq 0}$. 
 Let $L=\bigoplus_{a\leq 0}L_a$ be $\C^{\times}$-grading. 

\begin{enumerate}
\item $\dim\left(L_a^{> 0}/L[\hdelta,(1)^{-1}]_a\right)\leq C_a$  $\uu\left(L_a^{> 0}/L[\hdelta,(1)^{-1}]_a\right)\leq C'_a $for some  $C_a,C'_a\geq 0$. 
\item Let $w=(a,u,r)$ be an $\Aut$-weight.
$\dim_{\C}L^{w}\leq C''_a$ .
\item $\dim_{\C}\left(L_a^{\leq 0}/L[(0)^0,\hdelta]_a\right)\leq D_a $ and  $|\uu\left(L_a^{\leq 0}/L[(0)^0,\hdelta]_a\right)|\leq D'_a$
for some $D_a,D'_a\geq 0$
\end{enumerate}
All the constants do not depend on $[\hdelta,\hdelta']$.
\end{proposition}
\begin{proof}
Follows from Remark \ref{R:isorepresentations} and Proposition \ref{P:pathestimates}.
\end{proof}
\begin{proposition}\label{P:boundextweight}
\[\dim_{\C} \Ext^l_{\hA[\hdelta,\delta']}(\C,\C)^{w}\leq C_{w,l}\] where $C_{w,l}$ doesn't depend on $\hdelta,\delta'$.
\end{proposition}
By Koszul property of $\hA[\hdelta,\delta']$ (\cite{MovStr}) $\bigoplus_l \Ext^l_{\hA[\hdelta,\delta']}(\C,\C)\cong \hA^![\hdelta,\delta']$. The theorem follows from Remark \ref{R:isorepresentations} and Proposition \ref{P:pathestimates}.
\subsection{Estimates of weights}
In this section I will establish a universal bounds on renormalized $\bT$-weights of $H_{\fa}^{i}[\hdelta,\hdelta']$ . 
To do this, I would like interpret the complex $CE_{\bullet}$ (\ref{E:CEdef}) as a non minimal free resolution of the $P$-module ${\hA}$. 
 It will be the main technical tool which enable me to find the lower bound on weights of $\bT$-action in $H_{\fa}^{i}[\hdelta,\hdelta']$. 
 
I will  start with discussion of dimensions of $\C^{\times}\times\bT$-weight spaces of  local cohomology of the free commutative  algebra 
 $\C[\setX]$, $\setX\subset \sethE$. Fix an ideal $\fa=\fri[[-\infty,(1)^{-1}]\cap \setX]$. Obviously $H_{\fa}^i(\C[X])\neq \{0\}, i\neq 
 |[-\infty,(1)^{-1}]\cap \setX|$. The cohomology can be computed by using K\"{u}neth formula and an isomorphism (\ref{E:loccohaffine}). I conclude that the character of $\C^{\times}\times\bT$ action is equal to
 \[\theta_{\fa}^{bare}[\setX]=\prod_{\ralpha^i\in [-\infty,(1)^{-1}]\cap \setX} \frac{t^{-1}q^{-i}}{(1-t^{-1}q^{i})}\prod_{\ralpha^i\in [(0)^{0},\infty]\cap \setX} \frac{1}{(1-tq^{i})} \]
 The renormalized character $\theta_{\fa}[\setX](t,q):=\theta_{\fa}^{bare}[\setX](t,q) \prod_{\ralpha^i\in [-\infty,(1)^{-1}]\cap \setX} tq^{i}$ has a limit when $\setX=[\hdelta,\hdelta']$ and  $\hdelta\to-\infty, \hdelta'\to\infty$:
 \[\theta_{\fa}(t,q)=\prod_{i=1}^{\infty} \frac{1}{(1-t^{-1}q^{i})^{16}}\prod_{i=0}^{\infty} \frac{1}{(1-tq^{i})^{16}}\in \Z((t))[[q]]\]

Fix a notation \[Chs_{\bullet}^{\fc}:=CE_{\bullet}[\hdelta,\hdelta'] \underset{P}{\otimes} \moduleT_{\fc}[\hdelta,\hdelta'], \fc=\fa,\fb,\fm.\]
\begin{proposition}\label{P:importantprop}
Let $b_{\fc}:=s(\fc)-s'(\fc)$ ( see Proposition \ref{P:identification} for notations). 
\begin{enumerate}
\item 
\[\begin{split}
&H^{i}_{\fc}({\hA})=H_{i'}(Chs^{\fc}), i+i'=s(\fc)\\
&H_{b_{\fc}+i}(Chs^{\fc})= 0, i\neq 0 ,\dots,3, \fc=\fa,\fb, i\neq 0 \text{ for }\fc=\ff,\ff',\fm.
\end{split}\]
\item \label{I:lowerandupperbounds} Define 
\[r(i)=\begin{cases}
0&i\leq 0\\
1&i=1,2\\
2&i\geq 3
\end{cases}\]
In the notations of Definition \ref{D:upperlower}
$\lu\left(H^{3-\rho(\hdelta)-i}_{\fa}({\hA})\right)\geq -r(i), $ where I consider  the renormalized $\bT$-action (see Remark \ref{R:degreeshift1}).
Likewise, $\uu\left( H^{\rho(\hdelta')+1-i}_{\fb}({\hA})\right) \leq-r(i).$

 More generally, if only $\hdelta\in \rM_1^+\sqcup \rM_3^+$ and $\hdelta'$ is arbitrary, then there is $i, \hdelta,\hdelta'$-independent constant $C$ such that   $\lu\left(H^{i}_{\fa}({\hA})\right)\geq C$.
Similar in case    only $\hdelta'\in \rM_1^-\sqcup \rM_3^-, \exists C'$ such that  $\uu\left(H^{\rho(\hdelta)-2+i}_{\fb}({\hA})\right)\leq C'$.
\item \label{I:Pimportantprop} 
Let $H^{i}_{\fc}[\hdelta,\hdelta']^w$ be the weight subspace   for the renormalized action of $\Aut$. Then
\begin{equation}\label{E:dimbound}
\exists C_w\geq 0\text{ such that }\dim H^i_{\fc}[\hdelta,\hbeta]^{w}\leq C_w,  \fc=\fa,\fb  \text{ for any } [\hdelta,\hbeta].
\end{equation}

\end{enumerate}
\end{proposition}

\begin{proof}
\begin{enumerate}
\item The first statement is a corollary of Proposition \ref{P:localtorisom}. It is different form Proposition \ref{P:identification} only in the choice of resolution. 
\item By Proposition \ref{P:injection}, the map ${\hA[\hdelta,\hdelta']}\to \hA[\hdelta,(1)^{-1}]$ induces inclusion of universal enveloping algebras $\hA^![\hdelta,(1)^{-1}]\to \hA^![\hdelta,\hdelta']$ and of the Lie algebras $L[\hdelta,(1)^{-1}]\subset L[\hdelta,\hdelta']$. Let $\Ker$ be the kernel of the homomorphism $CE[\hdelta,\hdelta']\to CE[\hdelta,(1)^{-1}]$. The powers $\{\Ker^{\times k}\}$ define decreasing filtration of $CE_{\bullet}[\hdelta,\hdelta']$ (of Serre-Hochschild type). It defines a filtration in $Chs_{\bullet}^{\fa}[\hdelta,\hdelta']$. 
Introduce a notation
\[B_{j}^{k}:=\rHom_j\left(\Sym^k\left((L[\hdelta,\hdelta']/L[\hdelta,(1)^{-1}])[-1]\right),\C\right).\]
Gradation $j$ is defined in equation (\ref{E:CEdef}). $L[\hdelta,\hdelta']/L[\hdelta,(1)^{-1}]$ and $B^k:=\oplus_jB_{j}^{k}$ are  $L[\hdelta,(1)^{-1}]$-module, where the generators are acting by commutations with   $ \theta_{\halpha}, \halpha\in [\hdelta,(1)^{-1}]$.   

By the first item of this proposition and Theorem \ref{E:diality}, the only nontrivial cohomology group in $Chs^{\fa}[\hdelta,(1)^{-1}]$ is 
\[\begin{split}
& H_{|[\hdelta,(1)^{-1}]|+\rho(\hdelta)-3}(Chs^{\fa}[\hdelta,(1)^{-1}])\cong H^{3-\rho(\hdelta)}_{\fa}[\hdelta,(1)^{-1}].\\
\end{split}\]

The initial page of the spectral sequence $E^{1}_{sk}=H_{s+k}(Chs^{\fa}[\hdelta,(1)^{-1}]\otimes B^{k})\Rightarrow H_{s+k}(Chs^{\fa}[\hdelta,\hdelta'])$ is equal  to
cohomology of the complex \[  H_{|[\hdelta,(1)^{-1}]|+\rho(\hdelta)-3}(Chs^{\fa}[\hdelta,(1)^{-1}])\otimes B_{\bullet}^k.\]\
 The differential is of Koszul type $\sum_{\halpha\in [\hdelta,(1)^{-1}]}\lambda^{\halpha}\otimes \theta_{\halpha}$.

By Lemma \ref{L:weights} item \ref{I:weights3}, the lower bound on (nonrenormalized) weights of $H^{3-\rho(\hdelta)}_{\fa}[\hdelta,(1)^{-1}]=H^{3-\rho(\hdelta)}_{\fm}[\hdelta,(1)^{-1}]$ is $u[\hdelta,(1)^{-1}]$.

 Let \[K
 \subset Chs_{|[\hdelta,(1)^{-1}]|+\rho(\hdelta)-3}^{\fa}[\hdelta,(1)^{-1}]
 \] be a $\bT$ - sub-representation, consisting of cycles, 
that projects isomorphically $H_{|[\hdelta,(1)^{-1}]|+\rho(\hdelta)-3}(Chs^{\fa})$. 

The group 
$H_{\fa}^i[\hdelta,\hdelta']$must be isomorphic to a sub-quotient of $ K\otimes B^{k}$.

Denote by $TL$ the Lie subalgebra $\bigoplus_{i\leq-2}L_i\subset L$. Notations $L^{<}$ is explained in Proposition \ref{E:random}. Denote 
\[N:=TL[\hdelta,\hdelta']/TL[\hdelta,(1)^{-1}][-1], \quad M_j:=\bigoplus_{t_s|\sum_{s\geq 1} st_s=-j} \bigotimes \rHom(\Sym^{t_s} N_s,\C). \]
The groups $M^{< 0}_j$ is constructed by the same formula with $N_s$  replaced by $N_s^{>0}$(taking $\rHom$ changes signs of weights). Likewise $M^{\geq 0}_j$ is constructed from $N_s^{\leq0}$. There is an isomorphism
\[B_j\cong\bigoplus_{j_1+j_2=j}M^{< 0}_{j_1}\otimes M^{\geq 0}_{j_2}\otimes \rHom(\Sym \spinor[(0)^0,\hdelta'],\C)\]
By Proposition \ref{E:random} dimensions of $N_s^{>0}$ and of $M^{< 0}_{j}$ are bounded from above by the constants that don't depend on $[\hdelta,\hdelta']$. By the same proposition  the $\bT$-weights of $N_s^{>0}$( and automatically the weights of $M^{< 0}_{j}$) are bounded from below by $[\hdelta,\hdelta']$-independent constant $C_j$. 
 In addition  $\lu\left( M^{\geq 0}_{j_2}\otimes \rHom(\Sym \spinor[(0)^0,\hdelta'],\C)\right)\geq 0$ . From this I derive  that $C_j\leq \lu\left(B_j\right)$. This implies that for the  renormalized $\bT$-action   $\lu\left(H_{\fa}^{-\rho(\hdelta)+3-j}[\hdelta,\hdelta']\right)\geq C_j$.  The cohomology groups are nontrivial only in the finite range of $j$. The universal bound is $\min C_j$. Similar arguments work for $\uu\left(H_{\fb}^{\rho(\hdelta)-2+j}[\hdelta,\hdelta']\right)$.

One can obtain a more accurate bounds on the weight of  renormalized $\bT$-action  in assumptions that $[\hdelta,\hdelta']$ is Gorenstein (\ref{E:purity}). As $\lu\left(B_0\right)=\lu\left(\rHom(\Sym \spinor[(0)^0,\hdelta'],\C)\right)=0$  the previous arguments imply that $\lu\left(H_{\fa}^{-\rho(\hdelta)+3}[\hdelta,\hdelta']\right)\geq 0$. 
$B_1=(M^{< 0}_{1}+M^{\geq 0}_{1})\otimes \rHom(\Sym \spinor[(0)^0,\hdelta'],\C)$. Computation with the path algebra from Proposition \ref{E:pathalg} shows that $\dim_{\C}M^{< 0}_{1}=10$ and all the elements in $M^{< 0}_{1}$ have $\bT$-weight $-1$. This verifies the bound for $H_{\fa}^{-\rho(\hdelta)+2}[\hdelta,\hdelta']$. Similar arguments establish the upper bounds for $H_{\fb}^{\rho(\hdelta')+1}[\hdelta,\hdelta']$ and $H_{\fb}^{\rho(\hdelta')}[\hdelta,\hdelta']$

Finally, I use Poincar\'{e} duality between pairs $(H_{\fb}^{\rho(\hdelta')+1},  H_{\fa}^{(-\rho(\hdelta)+3)-3})$ and  $(H_{\fb}^{\rho(\hdelta')},  H_{\fa}^{(-\rho(\hdelta)+3)-2})$, Proposition \ref{P:pairingdegree} and Remark \ref{R:degreeshift1} to verify the bounds for $ \lu\left(H_{\fa}^{(-\rho(\hdelta)+3)-i}\right),$ $i=2,3$.
\item I will prove the statement only for $\fc=\fa$. By the previous item the group $H_{\fa}^{-\rho(\hdelta)+3}(\hA)$ is a sub-quotient of the tensor product $H^{3-\rho(\hdelta)}_{\fm}[\hdelta,(1)^{-1}]\otimes B_0$. 
The linear space $H^{3-\rho(\hdelta)}_{\fm}[\hdelta,(1)^{-1}]$ is dual to $\hA[\hdelta,(1)^{-1}]$ (\ref{E:loccohnew}). The later space has a basis formed by standard monomials. It determines a dual weight basis $H^{3-\rho(\hdelta)}_{\fm}[\hdelta,(1)^{-1}]$.
After taking renormalization of $\Aut$ action into account I use  the dual basis to define an $\Aut$-equivariant  embedding of $H^{3-\rho(\hdelta)}_{\fm}[\hdelta,(1)^{-1}]$ into $\C[\hdelta,(1)^{-1}]^{-1}$. I conclude that $\chi_{H^{3-\rho(\hdelta)}_{\fm}[\hdelta,(1)^{-1}]}(t,q)\in \Z[t,t^{-1}][[q]]$.  By  combining this observation  with the results of  the previous item I get $\Aut$-equivariant  linear embedding of $H_{\fa}^{-\rho(\hdelta)+3}(\hA)$ into $\C[\hdelta,(1)^{-1}]^{-1}\otimes \C[(0)^0,\hdelta']$.
Thus $\chi_{H_{\fa}^{-\rho(\hdelta)+3}(\hA)}(t,q)\in \Z((t))[[q]]$.
\begin{definition}
I will be using a partial order on  series with real coefficients. It is  defined by the rule $f(t,q)\leq g(t,g)$ iff $g(t,g)-f(t,q)$ has positive series coefficients. Obviously, if $f(t,q)=\sum_{ij=-\infty}^{\infty}c_{ij}t^iq^j, c_{ij}\in \R^{\geq 0}$, $g\in \R((t))[[q]]$ and $f\leq g$ then  $f\in \R((t))[[q]]$.
\end{definition}
I conclude   $ \theta_{\fa}(t,q)\geq \chi_{H_{\fa}^{-\rho(\hdelta)+3}(\hA)}(t,q)$.  As $\theta(t,q)$ doesn't depend on $[\hdelta,\hdelta']$ this  proves the claim for $H_{\fa}^{-\rho(\hdelta)+3}$. 
The combination of the previous argument with the the proof of the last item leads to inequality. 
\[\chi_{H_{\fa}^{-\rho(\hdelta)3-i}}(t,q)\leq \theta_{\fa}(t,q)\sum_{j_i+j_2=i} \chi_{M^{< 0}_{j_1}}(t,q) \chi_{M^{\geq 0}_{j_2}}(t,q) \]
  $M^{\geq  0}_{j_2}$ is a graded components  in a free graded commutative algebra. Though dimension of the  generators is infinite, by  Proposition \ref{E:random} the generating function is bounded by $\sum_{i=2}^j \frac{C_{i}t^i}{1-q}$. From this I derive a bound  
\[\chi_{M^{\geq  0}_{j}}\leq  \prod_{s=2}^j \prod_{k\geq0}(1-(-1)^st^iq^k)^{(-1)^{s+1}C_s}=\theta''_j(t,q)\] 
It follows from from the same proposition  that $M^{< 0}_{j_1}$ is a graded subspace in a free graded commutative algebra on a finite number of generators. Proposition \ref{E:random}  gives a $[\hdelta,\hdelta']$-independent bound on this number.
Thus  $\chi_{M^{< 0}_{j}}\leq \theta'_j(t,q)\in \Z[t,t^{-1}, q,q^{-1}]$. Finally \[\chi_{H_{\fa}^{-\rho(\hdelta)3-i}}(t,q)\leq \theta_{\fa}(t,q)\sum_{j_i+j_2=i}\theta'_{j_1}(t,q)\theta''_{j_2}(t,q)\in \Z((t))[[q]]\] where the right-hand-side doesn't depend on $[\hdelta,\hdelta']$.

\end{enumerate}
\end{proof}
\begin{remark}\label{R:interpretationofcoefficients}{\rm
Here some interpretation of coefficients of the series (\ref{E:qexpansion}).
The $-1$ in the numerator of the coefficient of $q^0$ accounts for the generator $\omega_3$ in $H_{\fa}^{3-\rho(\hdelta)}[\hdelta,\hdelta']$ with the smallest $\bT$-weight. In my choice of renormalization the  weight of $\omega_3$ is set to zero. The group $H_{\fa}^{-\rho(\hdelta)}[\hdelta,\hdelta']$ contains a homogeneous element $\omega_0$ that pairs nontrivially with $\omega_3$. By taking renormalized degree of the pairing into account I conclude that $\omega_3$  has $\C^{\times}\times \bT$ weight equal to $(4,-2)$.  I interpret $\omega_0$ as being responsible for the unit in the denominator of the coefficient of $q^{-2}$.

Note that up to division  on $t^4$ and shift $Z_{i}\to Z_{i-2}$ the coefficients that I found coincide with $Z_0(t)$, $Z_{1}(t)$ and  $Z_{2}(t)$  formulas  (3.12),(3.22) and (3.32) \cite{AABN}. 
}\end{remark}

Here is the final result.
\begin{proposition}\label{P:surjectivity9}
For every $\Aut$-weight $w=(a,u,r)$ 
the maps 
\begin{equation}\label{E:maponweightspaces}
\begin{split}
&\sfp:H_{\fb}^i[\hdelta,\hbeta]^w\to H_{\fb}^i[\hdelta',\hbeta]^w,\hdelta \leq \hdelta'\leq \hbeta , \text{ is onto if }u(\hdelta')< u\\
&\sfp:H_{\fa}^i[\hdelta,\hbeta']^w\to H_{\fa}^i[\hdelta,\hbeta]^w,\hdelta \leq \hbeta \leq \hbeta' , \text{ is onto if }u(\hbeta)> u
\end{split}
\end{equation}
\end{proposition}
\begin{proof}
I will prove the statement only for ideal $\fb$. First I assume that $\hdelta\in \rM_1^+\sqcup \rM_3^+$ and $\hdelta'\in \rM_1^+$.
 Lemma \ref{L:saturate} allows to decompose $\sfp$ into a product of elementary projections $\sfp_i$. 
It is sufficient to  prove the statement for  $\sfp=\sfp_i$. Possible $\sfp_i$ are described in Lemma \ref{L:saturate0}.
If $\sfp$ is a regular  homomorphism (equations (\ref{E:group1}) and (\ref{E:group2})) then $\sfp$ is a map in the short exact sequence
\[\{0\}\to \hA[\hdelta,\hbeta]\overset{\lambda^{\hdelta}\times ?}\longrightarrow \hA[\hdelta,\hbeta]\overset{\sfp}\to \hA(\hdelta,\hbeta]\to \{0\}.\]
It leads to the long exact sequence of  $\bT$-weight components of  local cohomology 
\begin{equation}\label{EPlongexact}
\cdots \to H^i_{\fb}[\hdelta,\hbeta]^{u-u(\hdelta)}\overset{\lambda^{\hdelta}\times ?}\longrightarrow H^i_{\fb}[\hdelta,\hbeta]^{u}\overset{\sfp}\to H^i_{\fb}(\hdelta,\hbeta]^{u}\to H^{i+1}_{\fb}[\hdelta,\hbeta]^{u-u(\hdelta)}\to\cdots
\end{equation}
By Proposition \ref{P:importantprop}  $\uu(H^i_{\fb}[\hdelta,\hbeta])\leq 0$
 . The  (renormalized) $\bT$ weight of $\lambda^{\hdelta}$ is $u(\hdelta)$.
If the condition (\ref{E:maponweightspaces}) is satisfied 
the group $H^i_{\fb}[\hdelta,\hbeta]^{u-u(\hdelta)}$ is zero and 
$\sfp$ in (\ref{EPlongexact}) is an isomorphism.

An  irregular  homomorphism (\ref{E:group3}) $\sfp$ defines  a short exact sequence
 \[\{0\}\to \fe(\hdelta,\hbeta] \to  \hA(\hdelta,\hbeta]\overset{\sfp}{\to}  \hA[\hdelta',\hbeta]\to\{0\},\quad \rho(\hdelta')=\rho(\hdelta'')=\rho(\hdelta)+1,\hdelta''\in \rM_2^+,\hdelta,\hdelta''\in \rM_1^+ \]
of $\C(\hdelta,\hbeta]$-modules. It is explained in the proof of the second part of Lemma \ref{L:saturate0} that  the ideal $I(\hdelta,\hbeta]$ is generates by $\lambda^{\hgamma},\lambda^{\hgamma'}, \hgamma\lessdot\hgamma'\in \setCL^-(\hdelta')$. There a projection $\sfq:\hA(\hdelta,\hbeta]\to \hA[\hgamma,\hbeta]$. In the same lemma  it is also explained 
that $\fe(\hdelta,\hbeta]$  maps isomorphically by $\sfq$ to ideal $\fe[\hgamma,\hbeta]$ in $\hA[\hgamma,\hbeta]$. $\fe[\hgamma,\hbeta]$ is  generated by  $\lambda^{\hgamma},\lambda^{\hgamma'}$. The isomorphism $H^i_{\fb}[\hgamma,\hbeta]\cong H^i_{\fb}(\fe[\hgamma,\hbeta])$ is established in Proposition \ref{E:CMcohcomp} item (\ref{E:idealloccohcomp}). 
Thus the composition   $H^i_{\fb}(\fe(\hdelta,\hbeta])\overset{\sfj}{\to} H^i_{\fb}(\hdelta,\hbeta]\to H^i_{\fb}[\hgamma,\hbeta]$ is an isomorphism. It implies that $\sfj$ is an inclusion, the boundary differential in the  long exact sequence 
\[
\cdots \to  H^i_{\fb}(\fe(\hdelta,\hbeta]){\longrightarrow} H^i_{\fb}(\hdelta,\hbeta]\overset{H(\sfp)}{\longrightarrow} H^i_{\fb}[\hdelta',\hbeta]\to \cdots
\]
 is trivial and  $H(\sfp)$ is surjective. 
 
Suppose that now  $\hdelta\in \rM_1^+\sqcup \rM_3^+$ and $\hdelta'\in \rM_3^+$. The diagram (\ref{P:kxccvgw13}) helps to find  $\hdelta''\in \rM_1^+$ such that $ \hdelta'<\hdelta''$. The composition of the maps
$H^i_{\fb}[\hdelta,\hbeta]\overset{\sfp}{\to} H^i_{\fb}[\hdelta',\hbeta]\overset{\sfq}\to H^i_{\fb}[\hdelta'',\hbeta]$
coincides with $H^i_{\fb}[\hdelta,\hbeta]\overset{\sfr}{\to} H^i_{\fb}[\hdelta'',\hbeta]$. Maps $\sfp$ and $\sfr$ are surjective. Then so is $\sfq$.

\end{proof}
\section{The limiting space of states}\label{S:limiting}
Closer to the end this section, I will define the limiting groups $H_{\fa}^{i+\itwo}$ and establish its basic properties. 
It is also desirable to have a simple model for practical computations with $H_{\fa}^{i}[\hdelta,\hdelta']$ and $H_{\fa}^{i+\itwo}$. 

The bidirect  system of complexes  $\Frg^{\fa}_{\bullet}[\hdelta,\hdelta']$ whose cohomology coincide with $H_{\fa}^{i}[\hdelta,\hdelta']$ serves this purpose.
 I will devote the next subsection to this complex.
Throughout this section all the intervals satisfy purity condition (\ref{E:purity}).

\subsection{$\Frg^{\fa}_{\bullet}[\hdelta,\hdelta']$ as a function of $\hdelta$ and $\hdelta'$}
 
In Section \ref{S:fourterm}, I introduced the complex $\Frg^{\fa}_{\bullet}[\hdelta,\hdelta']$ (\ref{E:Rdef}). My plan is to define a structure bidirect system on $\{\Frg^{\fa}_{\bullet}[\hdelta,\hdelta']\}$ whose limit will be $\Frg^{\fa}_{\bullet}$. The important difference from the system $H_{\fa}^{i}[\hdelta,\hdelta']$ (\ref{E:diagramindprolocal}), that I already have, is that the bidirect system will now be defined on the level of chains.  This explains my interest in functorial properties of the correspondence 
$[\hdelta,\hdelta']\Rightarrow \Frg^{\fc}_{\bullet}[\hdelta,\hdelta'].$

Recall that there is identification of cohomology of $\Frg^{\fc}_{\bullet}[\hdelta,\hdelta']$ and $H_{\fa}^{i}[\hdelta,\hdelta']$ (\ref{E:redisomorphism}). 
One can think about the correspondence 
\begin{equation}
[\hdelta,\hdelta']\Rightarrow H_{\fa}^{i}[\hdelta,\hdelta']
\end{equation}
as a functor from the category of of intervals $\{[\hdelta,\hdelta']\}$($\hdelta'$ is fixed) to linear spaces.
$\Th$ (\ref{E:classt})are the structure maps of this functor.
The question I will address now is how to see the lift $\Thl_p$ of $\Th_{\sfp}$ 
 to the level of chains of $\Frg^{\fa}_{\bullet}[\hdelta,\hdelta']$.
 
I am going to use one more time the familiar idea that $\Th$ can be decomposed into product of elementary $\Th_{\sfp_i}$. This time I will lift $\Th_{\sfp_i}$ on the level of chains.  Presently I will define such a lifting, which I denote by $\Thl_{\sfp_i}$. For this purpose I will applied Lemma \ref{L:saturate} to the pair $ [\hdelta,\hbeta]\supset [\hdelta',\hbeta]$ to construct maps $\sfp_i$ (\ref{E:refignseq}). From this lemma I know that there are three kinds of homomorphisms $\sfp_i$ (\ref{E:group1}, \ref{E:group2}, \ref{E:group3}). I will deal with the regular cases (\ref{E:group1}, \ref{E:group2}) first.

\paragraph{ The regular case} Let us assume that $\hdelta$ is the same as one of the left-end points of the intervals (\ref{E:group1},\ref{E:group2}). This means that $\lambda^{\hdelta}$ is not a zero divisor in ${\hA[\hdelta,\hdelta']}$ and ${\hA[\hdelta,\hdelta']}/(\lambda^{\hdelta})\cong \hA(\hdelta,\hdelta']$. I define the map $\Thl:\Mrg_{\fa}(\hdelta,\hdelta']\to \Mrg_{\fa}[\hdelta,\hdelta']$ on the decomposable tensors $a\underset{R^{\fa}}{\otimes} f\in \hA\underset{R^{\fa}}{\otimes} \moduleS_{\fa}$ by the formula
\begin{equation}\label{E:regmap}
\Thl(a\underset{R^{\fa}}{\otimes} f):= \tilde{a}\underset{R^{\fa}}{\otimes} \sfk(f)= \tilde{a}\underset{R^{\fa}}{\otimes} \frac{1}{\lambda^{\hdelta}}f
\end{equation}
where the element $\tilde{a}$ is any preimage $a$ in ${\hA[\hdelta,\hdelta']}$ with respect to $\sfp$ (\ref{E:extsimple}), $\sfk$ is taken from Definition \ref{D:tau}. $\frac{1}{\lambda^{\hdelta}}=\varpi_{\Imm\, (\mathsf{pr}\circ \mathsf{inc})^{\perp}}$
in the formula (\ref{E:elemmaps}), that describes the map $\sfk$. As $\lambda^{\hdelta}\frac{1}{\lambda^{\hdelta}}f=0\in \moduleS_{\fa}[\hdelta,\hdelta']$, it follows 
that $\Thl$ doesn't depend on the possible choices of $\tilde{a}$. I extend this to a map of free $\Lambda[\theta^{0},\theta^{1},\theta^{2}]$-modules
$\Frg^{\fa}_i(\hdelta,\hdelta']\to \Frg_i^{\fa}[\hdelta,\hdelta']$. 

\paragraph{ The irregular case}
Let us assume that $\hdelta\in \rM_2^{-}$ (Definition \ref{D:Mpmidef0}), which means that $\hdelta$ is one of the left-end points of the closed intervals (\ref{E:group3}). $\hgamma$ is the corresponding left-end point of the open interval.

The map $\Thl:\Mrg_{\fa}[\hdelta,\hdelta']\to \Mrg_{\fa}(\hgamma,\hdelta']$ on a decomposable element $a\underset{P}{\otimes} f$ is defined as 
\begin{equation}\label{E:irregmap}
\Thl(a\otimes f):= \sfv(a) \underset{R^{\fa}}{\otimes} \inclusiono(f) = \tilde{a}\underset{R^{\fa}}{\otimes} \inclusiont(f) 
\end{equation}
where $\sfv$ and $\inclusiont$ are introduced in equation (\ref{E:inclusionofA}) and 
 Definition \ref{D:tau} respectively. 
 As in the regular case, I extend this to a map of free $\Lambda[\theta^{0},\theta^{1},\theta^{2}]$-modules
$\Frg^{\fa}_i(\hgamma,\hdelta']\to \Frg_i^{\fa}[\hdelta,\hdelta'].$ 

\paragraph{$\Thl$ in the general case}
\begin{definition}\label{D:thldef}{\rm
I assume that $ [\hdelta,\hbeta]\supset [\hdelta',\hbeta]$.
\begin{enumerate}
\item I define the map 
\begin{equation}\label{E:inddef}
\Thl:\Frg_i^{\fa}[\hdelta',\hbeta]\to \Frg_i^{\fa}[\hdelta,\hbeta]
\end{equation} as a composition
\[\Frg_i^{\fa}[\Delta_{n'}]\overset{\Thl_{\sfp_{n'-1}}}{\longrightarrow} \cdots \overset{\Thl_{\sfp_1}}{\longrightarrow} \Frg_i^{\fa}[\Delta_{1}].\]
where $\sfp_i$ are taken from (\ref{E:refignseq}). To construct (semi)intervals $\Delta_i$ I applied Lemma \ref{L:saturate} to $ [\hdelta,\hbeta]\supset [\hdelta',\hbeta]$.
\item Suppose $\hbeta'<\hbeta$. The map 
\begin{equation}\label{E:resdef}
\pl:\Frg_i^{\fa}[\hdelta,\hbeta]\to \Frg_i^{\fa}[\hdelta,\hbeta']
\end{equation} is induced by the maps of the tensor components of the module $\Mrg_{\fa}$ \[\sfp: {\hA[\hdelta,\hbeta]}\to \hA[\hdelta,\hbeta']\text{ and }\projectiono:\moduleS_{\fa}[\hdelta,\hbeta]\to \moduleS_{\fa}[\hdelta,\hbeta'] (\ref{E:Fres}).\]
The map (\ref{E:resdef}) is also defined for semi-intervals $(\hdelta,\hbeta]\supset (\hdelta,\hbeta']$.
\end{enumerate}
}
\end{definition}
\begin{remark}
The maps $\Thl$ and $\pl$ commute.
\end{remark}
Compatibility of $\Thl$ and $\Th$ is verified in the following statement.
\begin{proposition}\label{P:colimitmaps}
 In the assumptions of Proposition \ref{P:commutativity} ($\hbeta=\hdelta_3,\hdelta_4$) and after identification (\ref{E:redisomorphism}),(\ref{E:IHidentification})
\begin{enumerate}
\item The map $\Thl_{\sfq}$ on $\Mrg_{\fa}$ coincides with
\[\Th_{\sfq}\otimes \id:H_{\ff}^{-\rho(\hdelta_2)}[\hdelta_2,\hbeta]\otimes \C[\rA^{\fa}]^{-1}\to H_{\ff}^{-\rho(\hdelta_1)}[\hdelta_1,\hbeta]\otimes \C[\rA^{\fa}]^{-1}\]
\item The induced map \[\Thl_{\sfq}:H_j(\Frg^{\fa}[\hdelta_2,\hbeta])\to H_{j}(\Frg^{\fa}[\hdelta_1,\hbeta]), \hdelta_1< \hdelta_2\]
coincides with the map
\[\Th_{\sfq}:H_{\fa}^{3-\rho(\hdelta_2)-j}[\hdelta_2,\hbeta]\to H_{\fa}^{3-\rho(\hdelta_1)+j}[\hdelta_1,\hbeta].\] 

\item The map \[\pl:H_j(\Frg^{\fa}[\hdelta,\hbeta'])\to H_{j}(\Frg^{\fa}[\hdelta,\hbeta]), \hbeta<\hbeta'(\hbeta=\hdelta_3,\hbeta'=\hdelta_4)\]
 coincides with the map
\[\sfp:H_{\fa}^{3-\rho(\hdelta)-j}[\hdelta,\hbeta']\to H_{\fa}^{3-\rho(\hdelta)+j}[\hdelta,\hbeta].\] 

\end{enumerate}
\end{proposition}
\begin{proof}
I start with the proof of the second statement.  For this purpose , I replace $[\hdelta',\hbeta]\supset [\hdelta,\hbeta]$ by elementary extensions of sets $\Delta_i\supset \Delta_{i+1}$ from Definition \ref{D:thldef}. 
If $\sfq:\hA(\Delta_i)\supset \hA(\Delta_{i+1})$
belongs to the groups (\ref{E:group1}),(\ref{E:group2}), then, by Proposition \ref{P:classofextreg}, $\Th_{\sfq}$ is the class of the extension 
(\ref{E:extsimple}). 

By (\ref{E:THdef}), $\Th$ is the boundary homomorphism $H_{\fa}^i[\hdelta,\hbeta]\to H_{\fa}^{i+1}(\hdelta,\hbeta]$ corresponding to (\ref{E:extsimple}). 
I replace ${\hA[\hdelta,\hbeta]}$ in (\ref{E:extsimple}) by its $\C[\hdelta,\hbeta]$ resolution $F_{\bullet}[\hdelta,\hbeta]$. The extension lifts to an exact triangle
\[F_{\bullet}\overset{\lambda^{\hdelta}\times?}{\longrightarrow} F_{\bullet}\to B_{\bullet}(F, \lambda^{\hdelta})\to F_{\bullet}[-1].\]
$B_{\bullet}(F, \lambda^{\hdelta})$ is the diagonal complex of the Koszul (bi)-complex $B_{\bullet}(F_{\bullet}, \lambda^{\hdelta})$ which coincides with the cone of the operator $\lambda^{\hdelta}\times ?$ on $F_{\bullet}$. 
The diagonal complex $\Frg^{\fa}_{\bullet}(F)$ of the bicomplex $\Frg^{\fa}_{\bullet}(F_{\bullet})$ is a modification of $\Frg^{\fa}_{\bullet}$ in which $\hA$ in (\ref{E:Mrgdef}) is replaced $F_{\bullet}$. Evidently, $\Frg^{\fa}_{\bullet}$ and $\Frg^{\fa}_{\bullet}(F)$ are quasi-isomorphic.
My task is to trace the action of the boundary map 
\[\Frg_{\bullet}^{\fa}(B(F))=B_{\bullet}(\Frg^{\fa}(F))\to \Frg_{\bullet}^{\fa}(F[-1]).\] 
The Koszul construction $B_{\bullet}(\Frg^{\fa}(F))=\Frg^{\fa}(F)\otimes \Lambda[\theta^{\hdelta}]$ is built on the operator $\lambda^{\hdelta}\times ?$ on $\Mrg_{\fa}(F_i)=F_i\underset{R}{\otimes}\moduleS_{\fa}$, which is then extended to $\Frg_{\bullet}^{\fa}(F)$. Originally, it is induced from the action of $\lambda^{\hdelta}\times ?$ on $F_i$. Notice that $\lambda^{\hdelta}\times ?$ acting on $\moduleS_{\fa}$ defines the same operator in $\Mrg_{\fa}(F_i)$. In the spectral sequence of the bicomplex $B_i((\Frg^{\fa} F)_j)$, the zero cohomology $H_0B((\Frg^{\fa} F)_{\bullet})$ is zero, for $\lambda^{\hdelta}\times ?$ is surjective in $\moduleS_{\fa}[\hdelta,\hbeta]$. The kernel of $\lambda^{\hdelta}\times ?$ in $\Frg_{\bullet}^{\fa} F[\hdelta,\hbeta]$ consists of elements of the form $\frac{1}{\lambda^{\hdelta}}b,b\in \Frg_{\bullet}^{\fa} F(\hdelta,\hbeta]$. The assignment $b\to \frac{1}{\lambda^{\hdelta}}b$ allows me to identify the cohomology of $B_{\bullet}(\Frg^{\fa}(F[\hdelta,\hbeta]))$ with the cohomology of $\Frg_{\bullet}^{\fa}(F(\hdelta,\hbeta])$.
The boundary map takes $\frac{1}{\lambda^{\hdelta}}b\xi\in B_{\bullet}(\Frg^{\fa}(F))$ to $\frac{1}{\lambda^{\hdelta}}b\in \Frg_{\bullet}^{\fa}(F)$. Now it becomes evident that the boundary map will match with (\ref{E:regmap}) after I compress resolutions $F_{\bullet}[\hdelta,\hbeta]$ and $F_{\bullet}(\hdelta,\hbeta]$ back to ${\hA[\hdelta,\hbeta]}$ and ${\hA(\hdelta,\hbeta]}$.

It remains to understand the structure of the map in local cohomology induced by inclusion $[\hdelta',\hbeta] \subset (\hdelta,\hbeta]$ (\ref{E:inclusionofA}). I study it by replacing local cohomology by more manageable $\Tor$ groups. I factor the map into a composition of two maps. The first is
\[H_{\fa}^{3-\rho(\hdelta')-i}[\hdelta',\hbeta]=\Tor_{i}^{\C[\hdelta',\hbeta]}({\hA[\hdelta',\hbeta]}, \moduleS_{\fa}[\hdelta',\hbeta])\overset{\sfi}{\cong}\Tor_i^{\C(\hdelta,\hbeta]}({\hA[\hdelta',\hbeta]}, \moduleS_{\fa}(\hdelta,\hbeta]).\]
The isomorphism is produced by maps $\C[\hdelta',\hbeta]\to \C(\hdelta,\hbeta]$, $\moduleS_{\fa}[\hdelta',\hbeta]\to \moduleS_{\fa}(\hdelta,\hbeta]$ which are induced by $\inclusiono$ (\ref{E:functorialpr}).

Arguing as in the proof of Proposition \ref{E:redisomorphism} and taking into account that $[\hdelta',\hbeta] \subset (\hdelta,\hbeta]$ belong to the class (\ref{E:group3}),
I verify that the isomorphism $\sfm$ is induced by the map of the close relative of the complex $\Frg_{\bullet}^{\fa}$:
\begin{equation}\label{E:iotafr}
\inclusiono:B_{\bullet}(\moduleS_{\fa}[\hdelta',\hbeta]\underset{R[\hdelta',\hbeta]}{\otimes} {\hA[\hdelta',\hbeta]},\{\underline{\lambda}_-^0,\underline{\lambda}_-^1, \underline{\lambda}_-^2\})\to B_{\bullet}(\moduleS_{\fa}(\hdelta,\hbeta]\underset{R(\hdelta,\hbeta]}{\otimes} {\hA[\hdelta',\hbeta]},\{\underline{\lambda}_-^0,\underline{\lambda}_-^1, \underline{\lambda}_-^2\}).
\end{equation}

The second map in the factorisation 
\[\Tor_i^{\C(\hdelta,\hbeta]}({\hA[\hdelta',\hbeta]}, \moduleS_{\fa}(\hdelta,\hbeta])\overset{\sfv}{\to} \Tor_i^{\C(\hdelta,\hbeta]}({\hA(\hdelta,\hbeta]}, \moduleS_{\fa}(\hdelta,\hbeta])\cong H_{\fa}^{3-\rho(\hdelta')-i}(\hdelta,\hbeta]\]
is induced by the map of $\C(\hdelta,\hbeta]$-modules $\lambda^{\hdelta'}\times ?$ (\ref{E:inclusionofA}). 
Again, repeating the arguments of the proof of Proposition \ref{E:redisomorphism}, we 
see that $\sfv$ lifted to he level of map of chains complexes 
\[B_{\bullet}(\moduleS_{\fa}(\hdelta,\hbeta]\underset{R(\hdelta,\hbeta]}{\otimes} {\hA[\hdelta',\hbeta]},\{\underline{\lambda}_-^0,\underline{\lambda}_-^1, \underline{\lambda}_-^2\})\overset{\sfv}{\to} B_{\bullet}(\moduleS_{\fa}(\hdelta,\hbeta]\underset{R(\hdelta,\hbeta]}{\otimes} {\hA(\hdelta,\hbeta]},\{\underline{\lambda}_-^0,\underline{\lambda}_-^1, \underline{\lambda}_-^2\})\]
is equal to $\lambda^{\hdelta'}\times ?$. 

We see that $\sfv\circ \inclusiono$ coincides with right hand side of (\ref{E:irregmap}).
The first statement uses isomorphism (\ref{E:IeqHtensprod}) and follows the same lines as the second.

I leave the proof of the third statement as an exercise.

\end{proof}
\begin{example}
 To make an illustration I suppose $k=\rho(\lambda^{\hdelta_2}),k-1=\rho(\lambda^{\hdelta'_2})$. Then the map (\ref{E:iotafr}) does the following transformation
\[a\otimes \varpi[\hdelta_2,\hbeta]=a\otimes \frac{1}{\lambda^{\hdelta_2}}\frac{1}{\lambda^{\hdelta'_2}}\prod_{k-1<i\leq -1}\frac{1}{\lambda^i}\overset{\inclusiono}\to
a\otimes \frac{1}{\lambda^k} \frac{1}{\lambda^{k-1}}\prod_{k-1<i\leq -1}\frac{1}{\lambda^i}=a\otimes \varpi(\hdelta_1,\hbeta]. \]
Elements $\lambda^{i}$ are defined in (\ref{E:regconstr}).
The values of the composition of $\inclusiono$ and $\lambda^{\hdelta_2}\times ?$ is 
\[a\otimes \frac{1}{\lambda^{\hdelta_2}}\frac{1}{\lambda^{\hdelta'_2}}\prod_{k-1<i\leq -1}\frac{1}{\lambda^i} \to \tilde{a}\otimes \frac{\lambda^{\hdelta_2}}{\lambda^k} \frac{1}{\lambda^{k+1}}\prod_{k+1<i\leq -1}\frac{1}{\lambda^i}.\]
\end{example}

\begin{lemma}\label{E:injective}
The maps (\ref{E:regmap}) and (\ref{E:irregmap}) and their compositions (\ref{E:inddef}) are injective. The map \ref{E:resdef} is surjective.
\end{lemma}
\begin{proof}
Let us consider the regular case (\ref{E:regmap}) first.
By Propositions \ref{P:colimitmaps} and \ref{P:fockasloc} the map (\ref{E:regmap}) is a map of free $\Lambda[\theta^{0},\theta^{1},\theta^{2}]$-modules  induced on the generating spaces by  the boundary differential in the exact sequence
\[\cdots \to H_{\ff}^{-\rho(\hdelta)-1}[\hdelta,\hdelta']\to H_{\ff}^{-\rho(\hdelta)-1}(\hdelta,\hdelta']\to H_{\ff}^{-\rho(\hdelta)}[\hdelta,\hdelta']\to \cdots\]
By Corollary \ref{C:Ionedegree}  $H_{\ff}^{i}[\hdelta,\hdelta']$ is nonzero only  in degree $-\rho(\hdelta)$ and $H_{\ff}^{i}(\hdelta,\hdelta']$ is nonzero only  in degree $-\rho(\hdelta)-1$. Thus the boundary differential must be an inclusion.

Consider now the irregular case. Arguing as in the regular case we see that the map (\ref{E:irregmap}) is induced by the map
\[H_{\fc}^{-\rho(\hdelta)-1}[\hdelta',\hbeta]\to H_{\fc}^{-\rho(\hdelta)-1}(\hdelta,\hbeta]\]
It originates from the  inclusion map in the short exact sequence $0\to \hA[\hdelta',\hbeta]\overset{\lambda^{\hdelta_2}\times ?}\longrightarrow \hA(\hdelta,\hbeta]\to \hA[\hdelta'',\hbeta]\to0$. Note that $\rho(\hdelta')=\rho(\hdelta'')$. We see  that the argument of the regular case can be repeated. The long exact sequence of local cohomology is in fact a short exact sequence nontrivial in degree $-\rho(\hdelta')$. Thus the map in question is an inclusion.

Similar arguments explain surjectivity of the map $\pl$. The only difference is that  the long exact sequences reduce to short exact sequences of local cohomology in degree $-\rho(\hdelta)$.
\end{proof}

\subsection{The complex $\Frg_{\bullet}^{\fa}$}\label{S:fockinf}
My present goal is to construct a complex that would compute cohomology $H_{\fa}^{3-i+\itwo}$ in one step.
This complex will be denoted by $\Frg_{\bullet}^{\fa}$.

The  maps (\ref{E:inddef}) and (\ref{E:resdef}) define a bidirect system on $\Frg_{\bullet}^{\fa}[\hdelta,\hgamma]$ parametrized by  $\hdelta\in  \rM_1^{+}$,$\hgamma\in  \rM_1^{-}$. I use   reductive action of  group $\Aut$  to decompose this complex into a direct sum  
\[\Frg_{\bullet}^{\fc}[\hdelta,\hgamma]=\bigoplus_{w} \Frg_{\bullet}^{\fc}[\hdelta,\hgamma]^{w},\fc=\fa,\fb\]
Here  $w$ stands for the weight of $\Aut$.

Introduce notations for   limits:
\begin{equation}
\begin{split}
& \overline{\Frg}^{\fc,w}_{\bullet}:=\underset{\underset{\hdelta}\longrightarrow}{\lim}\underset{\underset{\hgamma}\longleftarrow}{\lim}\Frg_{\bullet}^{\fc}[\hdelta,\hgamma]^w,\\
&\underline{\Frg}^{\fc,w}_{\bullet}:=\underset{\underset{\hgamma}\longleftarrow}{\lim}\underset{\underset{\hdelta}\longrightarrow}{\lim}\Frg_{\bullet}^{\fc}[\hdelta,\hgamma]^w.
\end{split}
\end{equation}
\begin{proposition}
The map 
\begin{equation}\label{E:kappaf}
\kappa:\overline{\Frg}^{\fc,w}_{\bullet}\to \underline{\Frg}^{\fc,w}_{\bullet},\fc=\fa,\fb
\end{equation} is an embedding.
\end{proposition}
\begin{proof}
By Lemma \ref{E:injective} conditions of Proposition \ref{P:kappaisomorphism} on the structure maps of the bidirect system  are satisfied.
\end{proof}

\begin{definition}
Introduce notations for   limits:
\begin{equation}
\begin{split}
& \overline{\Frg}^{\fc,w}_{\bullet}:=\underset{\underset{\hdelta}\longrightarrow}{\lim}\underset{\underset{\hgamma}\longleftarrow}{\lim}\Frg_{\bullet}^{\fc}[\hdelta,\hgamma]^w,\\
&\underline{\Frg}^{\fc,w}_{\bullet}:=\underset{\underset{\hgamma}\longleftarrow}{\lim}\underset{\underset{\hdelta}\longrightarrow}{\lim}\Frg_{\bullet}^{\fc}[\hdelta,\hgamma]^w.
\end{split}
\end{equation}
\end{definition}

\subsection{The definition of semi-infinite local cohomology 
 }\label{SS:definitionh}
 Our definition of semi-infinite local cohomology uses  bidirect systems of linear spaces.
The reader can  consult \cite{EckmannHilton},\cite{FreiMacdonald}, \cite{CieliebakFrauenfelder}, for details on this topic.

Let $w$ be a weight of the group $\Aut$.
By Remark \ref{R:weightcompatibility} maps in (\ref{E:diagramindprolocal}) commute with the renormalized  $\Aut$ action. This is why weight spaces $H_{\fc}^{i}[\hdelta,\hdelta']^w$ form a bidirect system labelled by $\hdelta\in \setA^+:= \rM_1^{+}\sqcup \rM_3^{+}$ and $\hdelta'\in   \setA^-:=\rM_1^{-}\sqcup \rM_3^{-}$. The sets of indices have a partial order induced from $\sethE$.

There are several types of semi-infinite local cohomology groups.
\begin{definition}\label{D:defloccoh}{\em
 The $\Aut$ weight subspace of  semi-infinite local cohomology ($\fc=\fa,\fb$) of the directed system of algebras $\hA[\hdelta,\hdelta']$ (Definition \ref{D:algAdelta})  corresponding to the weight $w$ are the limits
\begin{equation}\label{E:semiinfinitecohdef}
\begin{split}
&  \overline{H}_{\fc}^{i+\itwo,w}:=\begin{cases}&\underset{\underset{\hdelta\in  \setA^+ }{\longrightarrow}}{\lim}\underset{\underset{\hdelta'\in  \setA^-}{\longleftarrow}}{\lim}H_{\fa}^{i+s'(\fa)}[\hdelta,\hdelta']^w, \fc=\fa\\ 
&\underset{\underset{\hdelta\in  \setA^-}{\longrightarrow}}{\lim}\underset{\underset{\hdelta'\in  \setA^+}{\longleftarrow}}{\lim}H_{\fb}^{i+s'(\fb)}[\hdelta,\hdelta']^w,\fc=\fb
\end{cases}
,\quad  \underline{H}_{\fc}^{i+\itwo,w}:=\begin{cases}&\underset{\underset{\hdelta'\in  \setA^-}{\longleftarrow}}{\lim}\underset{\underset{\hdelta\in  \setA^+}{\longrightarrow}}{\lim}H_{\fa}^{i+s'(\fa)}[\hdelta,\hdelta']^w,\fc=\fa\\
&\underset{\underset{\hdelta'\in  \setA^+}{\longleftarrow}}{\lim}\underset{\underset{\hdelta\in  \setA^-}{\longrightarrow}}{\lim}H_{\fb}^{i+s'(\fb)}[\hdelta,\hdelta']^w,\fc=\fb.
\end{cases}
\end{split}
\end{equation}

The structure maps of the bidirect system $\hdelta\in  \rM_1^{+}\sqcup \rM_3^{+}$,$\hdelta'\in  \rM_1^{-}\sqcup \rM_3^{-}$ come from the diagram (\ref{E:diagramindprolocal}). The function $s'$ is taken from (\ref{E:sschiftdef}). 
Similar definition can be made using  subsets of indices  $\setB^+:=\{(0)^u|u\in \Z\}\subset \rM_3^{+}$,$\setB^-:=\{(1)^u|u\in \Z\}\subset \rM_3^{-}$.

}\end{definition}
There is a canonical homomorphism 
\begin{equation}\label{E:kappaiso}
\kappa^w_{\fc}:\overline{H}_{\fc}^{i+\itwo,w}\to \underline{H}_{\fc}^{i+\itwo,w}, \fc=\fa,\fb,
\end{equation} 
(see (\ref{E:fourterm}) definition). For  more substantial discussion of $\kappa$ see  \cite{EckmannHilton} Section 3.2. Direct and inverse limits usually do not commute. See Example \ref{EX:dirinve} for discussion of this phenomenon. This is why  I can't take for granted that  $\kappa_{\fc}$ is  an isomorphism. Still it is true:
\begin{proposition}\label{E:kappaisomorphism}
\begin{enumerate}
\item The map (\ref{E:kappaiso}) is an isomorphism. 
\item \label{I:kappaisomorphism} The limits taken over the sets $\setA^+,\setA^-$ and $\setB^+,\setB^-$ coincide.
\item By by virtue of the previous item the following   notations make sense
 \[\begin{split}
 &H_{\fc}^{i+\itwo,w}:= \overline{H}_{\fc}^{i+\itwo,w}= \underline{H}_{\fc}^{i+\itwo,w}\\
 &H_{\fc}^{i+\itwo}:=\bigoplus_{w} H_{\fc}^{i+\itwo,w},\fc=\fa,\fb
 \end{split}\]
 The space $H_{\fc}^{i+\itwo}$ is equipped with the $\Spin(10)$-action.
\end{enumerate}
\end{proposition}
\begin{proof}
Fix the weight $w=(a,u,r)$. By Proposition \ref{P:surjectivity9} and $*$-duality (Proposition \ref{C:adjoint})  the map 
\begin{equation}\label{E:proj23}
\sfp:H_{\fa}^i[\hdelta,\hbeta']^w\to H_{\fa}^i[\hdelta,\hbeta]^w,u(\hbeta)>u
\end{equation} is onto and the map 
\begin{equation}\label{E:thinjection}
\Th_{\sfp}:H_{\fa}^i[\hdelta,\hbeta']^w\to H_{\fa}^{i+\codim}[\hdelta',\hbeta']^w, \codim=\rho(\hdelta)-\rho(\hdelta'), u<-u(\hdelta)
\end{equation} is an injection.

Let $k(\hbeta')$ be the maximal integer such that for any $\hdelta$ with $\rho(\hdelta)\leq k(\hbeta')$ the map (\ref{E:thinjection}) is an isomorphism. Such $k$ exists  because of the bound  (\ref{E:dimbound}). I claim that $\inf_{\hbeta'>(0)^0} k(\hbeta')>-\infty$. Let us assume otherwise. Then there is a sequence
\begin{equation}\label{E:assumption}
 \hbeta_1<\hbeta_2<\cdots\text{ such that } k(\hbeta_i)>k(\hbeta_{i+1}).
 \end{equation} I can assume that $u(\hbeta_i)>u$ for all $i$. I pick a sequence $\hdelta_i$ such that $\rho(\hdelta_i)=k(\hbeta_i)$. By the assumption (\ref{E:assumption}) I have a sequence of maps
\[H_{\fa}[\hdelta_1,\hbeta_1]^w \twoheadleftarrow  H_{\fa}[\hdelta_1,\hbeta_2]^w\rightarrowtail H_{\fa}[\hdelta_2,\hbeta_2]^w \twoheadleftarrow H_{\fa}[\hdelta_2,\hbeta_3]^w\rightarrowtail \cdots\]
In the above formula I omitted cohomological degrees for simplicity.
As $k(\hbeta_{i+1})<\rho(\hdelta_i)$ in the  diagram inclusions, defined by $\Th$,  are {\it never} isomorphisms. This diagram contradicts to the bound (\ref{E:dimbound}).

I conclude that there is a constant $k=\inf_{\hbeta'>(0)^0} k(\hbeta')$ such that  the maps (\ref{E:thinjection}) are  isomorphisms if  
\begin{equation}\label{E:condisomorphism10}
u(\hdelta)<-u\text{ and } \rho(\hdelta)<k.
\end{equation}  In these range of indices inverse systems $H^i_{\fa}[\hdelta_i,\hbeta]^w$ ($i=1,2$) are isomorphic. The structure maps of the systems are surjections. By using the bound  (\ref{E:dimbound}) one more time I conclude that there is an integer $l$ such that if 
\begin{equation}\label{E:conditionbeta}
\rho(\hbeta)>l, u(\hbeta)>u
\end{equation}
 and (\ref{E:condisomorphism10}) hold then (\ref{E:proj23}) and (\ref{E:thinjection})are isomorphisms. It is clear that the limits of the bidirect system $H_{\fc}^{i+s'(\fa)}[\hdelta,\hbeta]^w$ won't change if I restrict the range of indices that satisfy (\ref{E:condisomorphism10},\ref{E:conditionbeta}). Such  bidirect system satisfies conditions of the  Theorem 5.6 \cite{FreiMacdonald}, from which the first statement follows.
 
 I just constructed a range of indices $\setA^{'\pm}$ such that all the structure maps for the bidirect system $H^i_{\fa}[\hdelta,\hbeta]^w$ are isomorphisms. The double limit is isomorphic to  $H^i_{\fa}[\hdelta_0,\hbeta_0]^w$ $\hdelta_0\in \setA^{'+}\cap \rM_3^+$, $\hbeta_0\in \setA^{'-}\cap \rM_3^0$. Later space is isomorphic  to the double limit along $\setB^{'\pm}:=\setA^{'\pm}\cap \setB^{\pm}$. Limits over the sets $\setB^{'\pm}$ and $\setB^{\pm}$ obviously coincide. This proves the second statement of the theorem. 
 
 By Remark \ref{R:groupactionloccoh} the group acts on $H^{i}_{\fa}[(0)^N,(1)^{N'}]$. It is evident that all the structure maps of  the system labelled by $\setB^{\pm}$ are compatible with $\Spin(10)$-action. This proves the last assertion.

 \end{proof}
 \begin{definition}\label{E:semiinfdefsum}
 By by virtue of Proposition \ref{E:kappaisomorphism} the following   notations make sense
 \[\begin{split}
 &H_{\fc}^{i+\itwo,w}:= \overline{H}_{\fc}^{i+\itwo,w}= \underline{H}_{\fc}^{i+\itwo,w}\\
 &H_{\fc}^{i+\itwo}:=\bigoplus_{w} H_{\fc}^{i+\itwo,w},\fc=\fa,\fb
 \end{split}\]
 \end{definition}
\begin{corollary}\label{C:limitchardim}
Let $w=(a,u,r)$ be an $\Aut$-weight.
\begin{enumerate}
\item \label{I:infinityweightpositivity} $H_{\fa}^{i+\itwo,w}=\{0\}$ for $u<0$, $H_{\fb}^{i+\itwo,w}=\{0\}$ for $u>0$.  (Proposition \ref{P:importantprop} item \ref{I:lowerandupperbounds})
\item \label{I:infinityweightfinitness} $\dim_{\C}H_{\fc}^{i+\itwo,w}<\infty$.
\item Let $Z_{\fc}[\hdelta,\hdelta'],Z_{\fc}(t,q,z)$ be the renormalized virtual character (\ref{E:renormchar}) of $H^i_{\fc}[\hdelta,\hdelta']$ and $H_{\fc}^{i+\itwo}$ respectively. Then
 \[\lim_{\hdelta\to-\infty,\hdelta'\to \infty} Z_{\fc}[\hdelta,\hdelta'](t,q,z)=Z_{\fc}(t,q,z),\fc=\fa,\fb\]
The limit is understood in the sense of coefficient-wise convergence in  $\Z[\widetilde{\bT}]((t))((q))$. $\Z[\widetilde{\bT}]$ is a group algebra of the group of $\widetilde{\bT}$-characters.
\end{enumerate}
\end{corollary}
\begin{remark}
It would be interesting to explore point-wise convergence of 
\[Z_{\fa}(t,q,z)=\lim_{N\to-\infty,N\to \infty} {Z_{\fa}}_N^{N'}(t,q,z)\]
In particular It it true that $Z_{\fa}(t,q,z)$ is a meromorphic function on $\Aut$ with poles only at $\hat{v}_{\halpha}(t,q,z)=1$? 
\end{remark}
\begin{proposition}\label{C:complexcohiso}
\[H_i\underline{\Frg}^{\fc}\cong H_i\overline{\Frg}^{\fc} \cong H_{\fc}^{3-i+\itwo},\fc=\fa,\fb\]
\end{proposition}
\begin{proof}
I will give a proof for $\fc=\fa$. By Proposition \ref{P:complexreduction}  $H_i(\Frg^{\fa}[\hdelta,\hgamma]^w)\cong H_{\fa}^{s'(\fa)-i}[\hdelta,\hgamma]^w$. By Lemma \ref{E:injective} $\Frg^{\fa}[\hdelta,\hgamma]^w$ with fixed $\hdelta$ satisfies Mittag-Leffler condition. 
Denote temporarily \[F_{\gamma,\bullet}:=\underrightarrow{\lim}\Frg^{\fa}[\hdelta,\hgamma]^w,\quad  G^i_{\hgamma}:=\underrightarrow{\lim}H_{\fa}^{s'(\fa)-i}[\hdelta,\hgamma]^w.\]
The Lemma \ref{E:injective}  also implies Mittag-Leffler condition for $F_{\hgamma,\bullet}$.

  {\it Milnor's exact sequence} (see e.g. \cite{Weibel}) gives 
\[\{0\}\to \underleftarrow{\lim}^1H_{\fa}^{s'(\fa)-i-1}[\hdelta,\hgamma]^w \to H_i(\underleftarrow{\lim}\Frg^{\fa}[\hdelta,\hgamma]^w)\overset{\sfq}\to \underleftarrow{\lim}H_{\fa}^{s'(\fa)-i}[\hdelta,\hgamma]^w \to \{0\}\]
By Proposition \ref{P:surjectivity9} $H_{\fa}^{s'(\fa)-i}[\hdelta,\hgamma]^w$ satisfies Mittag-Leffler condition and $\underleftarrow{\lim}^1=\{0\}$. This implies that $\sfq$ is an isomorphism. Direct limits in the category of linear spaces commute with homology (\cite{Weibel} Theorem 2.6.15). From this I deduce  isomorphisms
\begin{equation}\label{E:firstiso}
H_i(\overline{\Frg}^{\fc,w)})=H_i(\underrightarrow{\lim}\underleftarrow{\lim}\Frg^{\fa}[\hdelta,\hgamma]^w)\cong \underrightarrow{\lim}\underleftarrow{\lim}H_{\fa}^{s'(\fa)-i}[\hdelta,\hgamma]^w=\overline{H}_{\fa}^{3-i+\itwo,w} \text{ and }
\end{equation}
\[H_i(F_{\gamma})\cong G^i_{\hgamma}.\]

Again by Proposition \ref{P:importantprop} $\dim_{\C} G^i_{\hgamma}<C_w$. This automatically implies Mittag-Leffler condition for $G^i_{\hgamma}$.  By using Milnor sequence one more time I get an isomorphism

\begin{equation}\label{E:secondiso}
H_i(\underline{\Frg}^{\fc,w)})=H_i(\underleftarrow{\lim}F_{\gamma})\cong \underleftarrow{\lim}G^i_{\hgamma}=\underline{H}_{\fa}^{3-i+\itwo,w}
\end{equation}

The proof follows from isomorphisms (\ref{E:kappaf}), (\ref{E:firstiso}), and (\ref{E:secondiso}). 
\end{proof}

From now on
\[H_{\fa}^{i+\itwo}:=\overline{H}_{\fa}^{i+\itwo}\cong \underline{H}_{\fa}^{i+\itwo}\]
The pairing 
on $H_{\fa}^{i+\itwo}$ 
will be introduced presently.
\begin{proposition}\label{C:pairingdegree}
There is a nondegenerate pairing
\[H^{i+\itwo}_{\fa}\otimes H^{3-i+\itwo}_{\fb}\to \C.\]
Its $\Aut$-renormalized weight 
is $(-4,2,0)$. The pairing is compatible with $\Spin(10)$-action.
\end{proposition}
\begin{proof}
I define the pairing on the level of weight components. The fact that $H_{\fa}^{s'(\fa)-i}[\hdelta,\hgamma]^w$ and $H_{\fb}^{s'(\fb)-j}[\hdelta,\hgamma]^{w'}, \hdelta\in \rM_1^+\sqcup  \rM_3^+,\hgamma\in \rM_1^-\sqcup  \rM_3^-$ form two bidirect systems with a pairing in the sense of Appendix \ref{S:pairingbidirect} follows from Corollary \ref{C:adjoint}. Nondegeneracy of the pairing between components as well as the degree of the pairing  follows from Proposition \ref{P:pairingdegree} and Remark \ref{R:degreeshift1}. 

The first statement follows from  Proposition \ref{P:abstrpairing} and an isomorphism (\ref{E:kappaiso}). Conditions of Proposition \ref{P:abstrpairing} follow  from Proposition \ref{P:importantprop} item \ref{I:Pimportantprop}. 

The statement about $\Spin(10)$ action follows from Proposition \ref{P:pairingdegree} item \ref{I:pairingdegree3} and Proposition \ref{E:kappaisomorphism} item \ref{I:kappaisomorphism}. 

\end{proof}

Another structure on the limiting groups is an isomorphism from the next proposition. 
\begin{proposition}\label{P:shiftcong2}
There is an isomorphism
\[\shift^{-1}\reflection:{H}^{i+\itwo}_{\fb}\cong {H}^{i+\itwo}_{\fa}
.\]
The isomorphisms are compatible with the grading shift described in Remark \ref{R:degreeshift1}. A composition of this map with the pairing from Proposition \ref{C:pairingdegree} give a nondegenerate pairing

\[H^{i+\itwo}_{\fa}\otimes H^{3-i+\itwo}_{\fa}\to \C.\]
Its $\Aut$-renormalized weight is $(-4,2,0)$. The pairing is compatible with the $\Spin(10)$-action.
\end{proposition}
This verifies Conjecture \ref{CON:main} item \ref{CON:main4} for $\tspace=\cone$.

\begin{proof}
The isomorphism is induced by (\ref{E:tsigmaiso}). Its compatibility with the bidirect system on $H^{i}_{\fa}[\hdelta,\hdelta']$ and $H^{i}_{\fb}[\hdelta,\hdelta']$ follows from Proposition \ref{E:morecompatibility}. The statement about compatibility with the shifts copies the end of the proof of Proposition \ref{P:eqderivation}.
The rest follows from Proposition \ref{C:pairingdegree}.
\end{proof}
\begin{corollary}\label{C:equationZ}
The virtual character $Z_{\fa}(t,q,z)$ of $\{H^{i+\itwo}_{\fa}\}$ in the sense of Section \ref{S:virtchar}  satisfies $Z_{\fa}(t,q,z)=-t^{-4}q^{2}Z_{\fa}(qt^{-1},q,z^{-1})$. Equation  (\ref{E:stardual}) follows from it. The claim follows from the first item of Corollary \ref{C:limitchardim} and Proposition \ref{P:shiftcong2}. 
\end{corollary}
Finally there is a vanishing result:
\begin{proposition}\label{P:main}
The groups ${H}^{i+\itwo}_{\fc},\fc=\fa,\fb$
are zero for $i\neq 0,\dots,3$.
\end{proposition}
\begin{proof}
Follows from Proposition \ref{P:vanishingmain}.
\end{proof}
\begin{remark}
By Remark \ref{E:indexa} and Proposition \ref{P:cheleminterval}, the index $\mathrm{ind}\, \OGr^{+}(5,10)$ is equal to $8$. By (\ref{E:hilbertini}), $\dim \OGr^{+}(5,10)=10$. 
I conclude that $\mathrm{coind}\, \OGr^{+}(5,10) +1=10-8+1=3$. It matches with the degree of the pairing (Proposition \ref{C:pairingdegree}) 
 as promised in Conjecture \ref{CON:main} item \ref{CON:main4}.
\end{remark}
In the next proposition I continue to use notations of Definition \ref{D:upperlower}.
\begin{proposition}\label{E:renormboundes}
Define \[\lu^{\fa}_i:=\lu\left(H^{i+\itwo}_{\fa}(\hA)\right),i=0,\dots,3\] for renormalized $\bT$-action (Remark \ref{R:degreeshift1} ). It follows from Proposition \ref{P:importantprop} item \ref{I:lowerandupperbounds} that 
 \[\lu^{\fa}_0\geq-2,\lu^{\fa}_1\geq-1,\lu^{\fa}_2\geq-1, \lu^{\fa}_3\geq 0.\]
The constants $\uu^{\fb}_i$ satisfy
\[\uu^{\fb}_0\leq-2,\uu^{\fb}_1\leq -1,\uu^{\fb}_2\leq-1,\uu^{\fb}_3\leq 0.\]
\end{proposition}
\begin{remark}\label{R:formalseries9}
The  pairing in Proposition \ref{P:shiftcong2} determines a symmetry of the series coefficients $Z_{\fa}=\sum_{u\geq0} c_{a,u}t^{a}q^{u}$. Constraint on  $u$ follows from Proposition \ref{E:renormboundes}. Functional equation on $Z_{\fa}$ Corollary \ref{C:equationZ} implies that $c_{a,u}=-c_{-4-a,2+a+u}$ and $c_{a,u}=0$ for $a<-u-2$. This implies that $Z_{\fa}(t,q)\in \Z((t))[[q]]$. A similar argument also works for $\chi_{H^{i+\itwo}_{\fa}}(t,q,z)$.
\end{remark}
\paragraph{Explicit description of $H^{3+\itwo}(\hA)$}
Fix the algebra $D$ generated by $\lambda^{\ralpha^r}, w_{\rbeta^l}, \ralpha^r,\rbeta^l\in \sethE$ subject to commutation relations
\[\begin{split}
&[\lambda^{\ralpha^r},\lambda^{{\ralpha'}^{r'}}]=[w_{\rbeta^l},w_{{\rbeta'}^{l'}}]=0,\\
&[\lambda^{\ralpha^r},w_{\rbeta^l}]=\delta^{\ralpha}_{\rbeta}\delta_{r+l,0}.
\end{split}\]
The set $\rho^{-1}(i)\subset \sethE$ consists of two elements, which I denote by $\halpha_{\pm}$.
I choose in $D$ a new set of generators:
\begin{equation}
\begin{split}
&\lambda_+^i=\lambda^{\halpha_+}+\lambda^{\halpha_-},\quad w^+_i=w_{\halpha_+}+w_{\halpha_-},\\
&\lambda_-^i=\lambda^{\halpha_+}-\lambda^{\halpha_-},\quad w^-_i=w_{\halpha_+}-w_{\halpha_-},\\
& \rho(\halpha_{\pm})=i.
\end{split}
 \end{equation} 
Note that $+$ and $-$ labelled generators mutually commute. 
The $D$-submodule $U$ of $\moduleS_{\fa}[\sethE]$ (\ref{E:FIdef}) cyclicly generated by
\[vac=\frac{1}{\lambda^{(3)^{-1}}}\frac{1}{\lambda^{(2)^{-1}}}\frac{1}{\lambda^{(1)^{-1}}}\prod_{i<0,i\equiv 1 \mod 2} \lambda^i_{-}\prod_{i<0}
 \frac{1}{\lambda_+^i}\]
can be characterized by relations
\begin{equation}\label{E:defrel2}
\begin{split}
&\lambda_+^ivac=0,\quad i<0,\\
&\lambda^{3^{-1}}vac=\lambda^{2^{-1}}vac=\lambda^{1^{-1}}vac=0,\\
&w_{\ralpha^i}vac=0,\quad i<0,\\
&(w^-_i)^2vac=0,\quad i>0,\quad i\equiv 1 \mod2,\\
&w^-_ivac=0,\quad i>0,\quad i\equiv 0 \mod2.\\
\end{split}
\end{equation} 
The term "submodule" is not quite appropriate because of the infinite product in the numerator that defined $vac$. The defining relations (\ref{E:defrel2}) contain only finite expression in generators and could be used as the definition of the $\C[\sethE]$-module without any references to (\ref{E:defrel2}).
Here the promised elementary description of one of the semi-infinite local cohomology groups.
 \begin{proposition}\label{P:tensorproduct1}
 Denote by $N$ the tensor product $\hA[-\infty,\infty]\underset{P}\otimes U$. There is a map
\begin{equation}\label{E:tensorproduct}
\begin{split}
H^{3+\itwo}_{\fa}\cong \widehat{N}
, \quad P=\C[\sethE]
\end{split}
\end{equation} 
Completion $\widehat{N}$ of $N$ is taken with respect to the system of submodules $\fri(-k,\hdelta)N\subset N$ defined by  ideals $\fri(-k,\hdelta):=\fri[{\lambda^i_-|i\leq -k}]+\fri[\sethE\backslash \sethE^{\leq \hdelta}]\subset P$.
Conjecturally this map is an isomorphism.
\end{proposition}
\begin{proof}
To simplify notations I denote $\widehat{N}/\widehat{\fri(-k,\hdelta)N}=N/\fri(-k,\hdelta)N$ by $N(-k,\hdelta)$.

I will   outline the proof first. By Proposition \ref{P:complexreduction} the group $H_{\fa}^{3-\rho(\hdelta)}[\hdelta,\hdelta']$ is the tensor product $\moduleS_{\fa}[\hdelta,\hdelta']\underset{\C[\hdelta,\hdelta']}{\otimes} \hA[\hdelta,\hdelta']$. I will fix $k$ and $\hdelta'$ and construct a coherent family of maps 
\begin{equation}\label{E:comparemap}
\moduleS_{\fa}[\hdelta,\hdelta']\underset{\C[\hdelta,\hdelta']}{\otimes} \hA[\hdelta,\hdelta']\to N(-k,\hdelta')
\end{equation}
This will enable me to define a map
\[\underset{\underset{\hdelta}{\longrightarrow}}{\lim}  H_{\fa}^{3-\rho(\hdelta)}[\hdelta,\hdelta']\to  \underset{\underset{k}{\longleftarrow}}{\lim}  N(-k,\hdelta')\]
In order to define the desired map  I  pass to inverse limit with respect to  $\hdelta'$.

In the rest of the proof I will concern myself with a construction of the map (\ref{E:comparemap}).

I will use the following notations 
\[U[-\infty,\hdelta']:=U/\fri[\sethE\backslash \sethE^{\leq \hdelta'}]U,\quad U(-k,\hdelta'):=U[-\infty,\hdelta']/\fri[{\lambda^i_-|i< -k}]U[-\infty,\hdelta']\] 
$U[-\infty,\hdelta']$ is a $D[-\infty,\hdelta']$-module. Inclusion of Weyl algebras $D[\hdelta,\hdelta']\to D[-\infty,\hdelta']$ turns $U[-\infty,\hdelta']$ into a $D[\hdelta,\hdelta']$-module. I define an inclusion  of $D[\hdelta,\hdelta']$-modules $\sff:\moduleS_{\fa}[\hdelta,\hdelta']\to U[-\infty,\hdelta']$ on the generator by 
\[\varpi_{\fa}[\hdelta,\hdelta']\to \prod_{\rho(\hdelta)\leq i<0,i\equiv 1 \mod 2}w_i^- vac.\]
By restriction of scalars from $\C[-\infty,\infty]$ to $\C[\hdelta,\hdelta']$ $U[-\infty,\hdelta']$ and $U(-k,\hdelta')$
become $\C[\hdelta,\hdelta']$-modules.
By abuse of notations I denote by $\sff$  the composition of maps of $\C[\hdelta,\hdelta']$-modules
\[\moduleS_{\fa}[\hdelta,\hdelta']\to U[-\infty,\hdelta']\to U(-k,\hdelta')\]
I denote by $\moduleS^{-k}_{\fa}[\hdelta,\hdelta']$ the image $\Imm\sff\subset U(-k,\hdelta')$.
The map $\sff$ induces a map of tensor products
\[ \hA[\hdelta,\hdelta'] \underset{\C[\hdelta,\hdelta']}{\otimes} \moduleS_{\fa}[\hdelta,\hdelta']\to \hA[\hdelta,\hdelta'] \underset{\C[\hdelta,\hdelta']}{\otimes} \moduleS^{-k}_{\fa}[\hdelta,\hdelta']\]
I denote by $\moduleS_{\fa}(-k,\hdelta')$ a $\C[-\infty,\hdelta']$-submodule in $ U(-k,\hdelta')$ generated by $\moduleS^{-k}_{\fa}[\hdelta,\hdelta']$. 
\begin{equation}\label{E:actionfactorization}
\text{The action of $\C[-\infty,\hdelta']$ on $\moduleS_{\fa}(-k,\hdelta')$ factors through $Q:=\C[\lambda_{-}^{-k},\dots,\lambda_{-}^{\rho(\hdelta)}]\otimes \C[\hdelta,\hdelta']$.}
\end{equation}
 Moreover 
 \begin{equation}\label{E:freemod1}
 \moduleS_{\fa}(-k,\hdelta')\cong \C[\lambda_{-}^{-k},\dots,\lambda_{-}^{\rho(\hdelta)}]\otimes \moduleS^{-k}_{\fa}[\hdelta,\hdelta'].
 \end{equation}
The nontrivial part of the construction is the map
\[\sfg:\hA[\hdelta,\hdelta'] \underset{\C[\hdelta,\hdelta']}{\otimes} \moduleS^{-k}_{\fa}[\hdelta,\hdelta'] \to \hA[-\infty,\hdelta'] \underset{\C[-\infty,\hdelta']}{\otimes} \moduleS_{\fa}(-k,\hdelta')\subset \hA\underset{P}\otimes U(-k,\hdelta')=N(-k,\hdelta')\]

The map (\ref{E:comparemap}) is the composition $\sfg\circ \sff$. 

Let $\fri^A:=\fri^A[\sethE\backslash \sethE^{\geq \hdelta}]$ be an ideal in $\hA[-\infty,\hdelta']$. Obviously $\hA[-\infty,\hdelta']/\fri^A\cong \hA[\hdelta,\hdelta']$. If I manage to prove that 
\begin{equation}\label{E:idealvanishing}
\fri^A \underset{\C[-\infty,\hdelta']}{\otimes} \moduleS_{\fa}(-k,\hdelta')=\{0\}
\end{equation} 
it would mean that 
\[\hA[-\infty,\hdelta'] \underset{\C[-\infty,\hdelta']}{\otimes} \moduleS_{\fa}(-k,\hdelta')\cong \hA[\hdelta,\hdelta']\underset{Q}{\otimes} \moduleS_{\fa}(-k,\hdelta')\cong \hA[\hdelta,\hdelta']\underset{\C[\hdelta,\hdelta']}{\otimes} \moduleS^{-k}_{\fa}[\hdelta,\hdelta'] \]
I the last identification I used isomorphism (\ref{E:freemod1}). Then I would take $\sfg$ to be the composition of the above isomorphisms.
\begin{lemma}
The isomorphism (\ref{E:idealvanishing}) holds true.
\end{lemma}
\begin{proof}
By (\ref{E:actionfactorization}) I can replace $\hA[-\infty,\hdelta']$ and $\C[-\infty,\hdelta']$ in $B(-k,\hdelta'):=\hA[-\infty,\hdelta'] \underset{\C[-\infty,\hdelta']}{\otimes} \moduleS_{\fa}(-k,\hdelta')$ by $\hA[-\infty,\hdelta']^{\geq -k}$ and $\C[-\infty,\hdelta']^{\geq -k}$. By definition $\lambda^i_+, i<0$ act trivially in $\moduleS_{\fa}(-k,\hdelta')$. This is why 
\begin{equation}\label{E:simpleidentity}
\lambda^{\halpha}a=-\lambda^{\halpha'}a,a\in \moduleS_{\fa}(-k,\hdelta')
\end{equation} when  $\rho(\halpha)=\rho(\halpha')<0$ and $\halpha\neq \halpha'$. By using isomorphism (\ref{E:freemod1}) any element in $B(-k,\hdelta')$ is a sum of the elements of the form 
\[\begin{split}
&a s=a \prod_{-k\leq i\leq \rho(\hdelta),i\equiv 1 \mod 2}\lambda^i_-w_i^- s'=a \prod_{\halpha\in \setX}\lambda^{\halpha}t, \\
&t=\prod_{-k\leq i\leq \rho(\hdelta),i\equiv 1 \mod 2}w_i^- s''\in U(-k,\hdelta')\\
&a\in \hA[-\infty,\hdelta']^{\geq -k}, s,s',s''\in \moduleS^{-k}_{\fa}[\hdelta,\hdelta']
\end{split}\]
$\setX$ is any subset $\sethE$ such that $\rho:\setX\to [-k,\rho(\hdelta)]\cap(1+ 2\Z)$ is a bijection. The last equality follows from (\ref{E:simpleidentity}). It is a matter of direct inspection of the diagram (\ref{P:kxccvgw13}) to verify that for any $\halpha$ such that $-k\leq \rho(\halpha)\leq \rho(\hdelta)$ there is $\setX$ as above and $\hbeta\in \setX$ such that $(\halpha,\hbeta)$ is a clutter. Let us pick $\halpha$ such that $\rho(\halpha)=-k$. Then by straightening relations (\ref{E:reluniv}) 
\[\lambda^{\halpha}s=\lambda^{\halpha}\lambda^{\hbeta}\prod_{\hbeta'\in \setX\backslash\{\hbeta\}}\lambda^{\hbeta'}t=0\]
It is true because if  $\rho(\hgamma)<-k$, then  $\lambda^{\hgamma}t=0$. The same argument works when $-k<\rho(\halpha)<\rho(\hdelta)$. The only difference is that straightening rules and (\ref{E:simpleidentity}) has to be repeated several times.
\end{proof}

\end{proof}

\section{ The field - antifield pairing}\label{S:faf}
The field - antifield pairing is a companion structure to the $*$-duality pairing. Thought their construction are formally distinct
, the framework of the proofs in both cases is very similar. This is why I leave some of the verifications to the reader.
In this section all the intervals satisfy purity condition (\ref{E:purity}).
\subsection{ The field - antifield pairing for a finite interval}
Let $\hA$ be an algebra based on the interval $[\hdelta,\hdelta']$. The field - antifield pairing is a pairing between cohomology of Koszul complexes
$H_{i}(B(H^j_{\fa}(A)))$ and $H_{i'}(B(H^{j'}_{\fa'}(A)))$. The pair of indices suggests that these cohomologies are pages of spectral sequences. Recall that by Proposition \ref{P:identification}, $H^i_{\fc}(A)$ can be computed by using the complex $\moduleT_{\bullet}(\fc)$. I define $\fT_{\bullet}(\fc)$ to be the diagonal complex of the Koszul bicomplex \begin{equation}
\begin{split}
&\fT_{i,j}(\fc):=B_i(\moduleT_{j}(\fc),\{\lambda^{\halpha}|\halpha\in \setN(\fc)\}),\fc\neq \fp, \setN(\fc) \text{ as in \ref{E:Ndef}} , \\
&\fT_{i,j}(\fp):=B_i(\moduleT_{j}(\fp),\{\lambda^{\rbeta}(z)|\rbeta\in \setE\}).
\end{split}
\end{equation}
There are also closely related complexes $\fT_{Kos}^{\bullet}(\fc)$, the diagonal complex of
\begin{equation}\label{E:bvkos}
\begin{split}
&\fT_{Kos}^{i,j}(\fc):=\underset{\underset{n}{\longrightarrow}}{\lim}
B_{-i}(K^{j}({\hA},\{\left(\lambda^{\halpha}\right)^n\}),\{\lambda^{\hbeta}\}), \halpha,\hbeta\in \setN(\fc) ,\fc\neq \fp \\
&\fT_{Kos}^{i,j}(\fp):=\underset{\underset{n}{\longrightarrow}}{\lim}B_{-i}(K^{j}({\hA},\{\left(\lambda^{\ralpha}(z)\right)^n\}),\{\lambda^{\rbeta}(z)\}), \ralpha,\rbeta\in \setE.
\end{split}
\end{equation}

It follows from Proposition \ref{P:identification} that 

\begin{equation}\label{E:BV}
\BV^i_{\fc}:=H^i(\fT_{Kos}(\fc))=H_{s(\fc)-i}(\fT(\fc)).
\end{equation} 
The following result gives the rough idea of the size of the groups (\ref{E:BV}) and their relations to bi-indexed groups $H_{i}(B(H^j_{\fa}(A)))$.
 \begin{proposition}\label{P:fT}
\begin{enumerate}
\item The first spectral sequence of the bicomplex $\fT_{i,j}(\fc)$ has the first page $E^1_{ij}:=B_i(\Tor_j^P(\hA,\moduleT_{\fc}))\cong B_i(H_{\fc}^{s(\fc)-j}(\hA))$.
\item \label{I:fTtwo}The second spectral sequence of the same bicomplex has the first page isomorphic to $\delta_{s(\fc),i}\otimes V_j$ (\ref{E:minres}). The rank one $\C$-linear space $\delta_{s(\fc),i}$ (or rank one $\C[z,z^{-1}]$-modules in $\fp$-case) is concentrated in cohomological degree $s(\fc)$.
\item $H_{i}(\fT(\fc))\cong V_{i-s(\fc)}$.
\end{enumerate}
\end{proposition}
\begin{proof}
An easy computation with the K\"{u}nneth's formula shows that 
\begin{equation}
H_i(B(\moduleT_{\fc},\{\lambda^{\halpha}|\halpha\in \setN(\fc)\}))=\delta_{s(\fc),i}.
\end{equation}
The remaining proof, which uses standard spectral sequence arguments, is left to the reader. 
\end{proof}

I will construct a pairing 
\[\begin{split}
&\fT_{i}(\fa){\otimes}\fT_{k-16-i}(\fa')\to \C[z,z^{-1}],\\
&k=d+|[\hdelta,\hdelta']|,d=\rk(\hdelta,\hdelta')+1
\end{split}\] which will descend to a field-anti-field prefect pairing on cohomology. Existence of such pairing is not surprising knowing that there is a pairing on $V_i$ (Proposition \ref{C:PoincareGen}). My construction will use three ingredients:
\begin{enumerate}
\item A $\C[z,z^{-1}]$-linear map of complexes
 \begin{equation}\label{E:Upsilon}
 \begin{split}
&\sfU: \fT_{\bullet}(\fa)\otimes\fT_{\bullet}(\fp)\otimes \fT_{\bullet}(\fa'){\to} \fT_{\bullet}(\fm) \\
&\sfU_{Kos}: \fT_{Kos}^{\bullet}(\fa)\otimes\fT_{Kos}^{\bullet}(\fp)\otimes \fT_{Kos}^{\bullet}(\fa'){\to} \fT_{Kos}^{\bullet}(\fm). 
\end{split}
\end{equation}
\item A $d$-closed $\C[z,z^{-1}]$ linear map 
\begin{equation}\label{E:tfunct}
\begin{split}
&\sfres_{\fT}:\fT_{k}(\fm)\to \C[z,z^{-1}].
\end{split}
\end{equation}
\item
A cocycle 
\begin{equation}\label{E:tcoc}
\Theta(z)\in H_{16}(\fT(\fp))\cong H^{0}(\fT_{Kos}(\fp)).
\end{equation}
\end{enumerate}
The pairing will be the composition
\begin{equation}\label{E:threecomp}
(a,b)_{\sfU}=\sfres_{\fT}\circ \sfU(a,\Theta(z),b).
\end{equation}
 \paragraph{The map $\sfU$ }
 In my construction of $\sfU$, I use that $\fT_{\bullet}(\fc)$ is a tensor product $F_{\bullet}(A)\underset{P}{\otimes} \moduleT_{\fc}\otimes \Lambda[\setN(\fc)]$.
To construct $\sfU$ one has to extend scalars from $\C$ to $\C[z,z^{-1}]$ and replace $P$ by $P[z,z^{-1}]$. The map $\sfU$ (\ref{E:Upsilon}) is the tensor product of three three-linear maps. The first map is the triple product of resolutions (\ref{E:resprod}) $\cdot \times(\cdot \times\cdot)$.
 The second is 
\[ \moduleT_{\fa}{\otimes}\moduleT_{\fp}{\otimes} \moduleT_{\fa'}\to \moduleT_{\fa}\underset{P}{\otimes}\moduleT_{\fp}\underset{P}{\otimes} \moduleT_{\fa'}=\moduleT_{\fm}.\]
It uses an isomorphism (\ref{E:tensorproduct1}).
 The third is the tautological product \[\Lambda[\setN(\fp)] \otimes \Lambda[\setN(\fa)] \otimes \Lambda[\setN(\fa')]\to \Lambda[\hdelta,\hdelta']=\Lambda[\setN(\fm)].\]

\paragraph{The functional $\sfres_{\fT}$} 
 By definition, the groups $\fT_{i}(\fm)$ are nontrivial in the range $0\leq i\leq k,$ (\ref{E:tfunct}). The top degree group 
 \begin{equation}\label{E:thetasub}
 \begin{split}
& \fT_{k}(\fm):=\moduleT_{\fm}\otimes \Span{\theta \otimes c},\\
&\theta_{\setX}:=\bigwedge_{\halpha\in \setX} \theta^{\halpha}, \theta:=\theta_{[\hdelta,\hdelta']}\\
&\Span{ c}= V_{d}
 \end{split}
\end{equation}
 contains a one-dimensional cohomology group $H_{k}(\fT(\fm))\cong \C$ generated by the cocycle 
 \begin{equation}\label{E:topelementf}
 \varpi_{\fm}\otimes \theta \otimes c.
\end{equation}
 The map $\sfres_{\fT}$ is defined as
\[ \sfres_{\fT}:\fT_{k}(\fm)\to \C[z,z^{-1}],\quad p(\lambda)\otimes \theta \otimes c\to \sfres_{\moduleT_{\fm}}(p(\lambda)),\quad p(\lambda)\in \moduleT_{\fm} \]
is nonzero on $\varpi_{\fm}\otimes \theta \otimes c$.
It is $d$-closed by trivial reasons.

 \subsection{The cocycle $\Theta(z)$}
Let $\hA$ be an algebra based on the interval $[\hdelta,\hdelta']$. The cocycle $\Theta(z)$ (\ref{E:tcoc}) is the essential part of the pairing (\ref{E:threecomp}).
Its construction 
will depend on the choice of an element $\bar{c}$ in cohomology of $B_{\bullet}(\sA)$ (c.f. formula 5.19 in \cite{AABN}). I have a two-step construction of $\Theta(z)$. First, I define a complex $\sfT_{\bullet}(\fm)$ whose cohomology pairs perfectly with the cohomology of $B_{\bullet}(\sA)$. I define $\Theta$ as cocycle in complementary cohomology group such that 
\begin{equation}\label{E:Thetadef2}
\langle \bar{c},\Theta\rangle=1
\end{equation}
where $\bar{c}$ is some distinguished element in $B_{\bullet}(\sA)$.
 Second, I devise a map 
\[\sf\sfq_{*}:\sfT_{\bullet}(\fm)\to \fT_{\bullet}(\fp)\]
which I will use to produce $\Theta(z)$:
\begin{equation}\label{E:Thetazdef2}
\Theta(z):=\sfq_{*}\Theta\in \BV^0_{\fp}=H^0(\fT_{Kos}(\fp))=H_{16}(\fT(\fp)).
\end{equation}

\paragraph{Cohomology of $B_{\bullet}(\sA)$}
I will use the notations: $\sP$ will stand for $\C[(0)^0,(1)^0]$, $\sLambda$ for $\Lambda[(0)^0,(1)^0]$ and $\sD$ for $D[(0)^0,(1)^0]$. 

Cohomology of the Koszul complex $B(\sA)=\sA\otimes \sLambda$ has been studied extensively (\cite{CNT},\cite{BerNek},\cite{CortiReid}).
By Propositions \ref{C:PoincareGen} and Corollary \ref{C:dimalg},
cohomology $\sV_i:=H_i(B(\sA))$ has a perfect pairing $\sV_i\otimes \sV_{5-i}\to \C$. The generator $\bar{c}$ of $\sV_{5}$ is represented by the cocycle
\[\Gamma_{\rbeta \rbeta}^m\lambda^{\rbeta}\theta^{\rbeta} \Gamma_{\rgamma \rdelta}^n\lambda^{\rgamma}\theta^{\rdelta} \Gamma_{\repsilon \rvarepsilon}^k\lambda^{\repsilon}\theta^{\rvarepsilon} \Gamma_{\rmu \rnu}^{mnk}\theta^{\rmu}\theta^{\rnu}.\]
By Proposition \ref{C:PoincareGen}, linear spaces $\sV_i$ are the generating spaces of the minimal free $\sP$-resolution $F_{\bullet}(\sA)$ of $\sA$. 
\paragraph{The pairing $H_{21-i}(\sfT(\fm))\otimes \sV_i\to \C$}.
Here is the notations that I will use for some $\sD$-modules
\[\smoduleT_{\fm}:=\moduleT_{\fm}[(0)^0,(1)^0],\quad \smoduleT_{(0)}:=\sP\]
and for some complexes
\[\begin{split}
&\sfT_{\bullet}(0):=B_{\bullet}(F_{\bullet}(\sA)),\\
&\sfT_{\bullet}(\fm):=B_{\bullet}(\smoduleT_{\bullet}(\fm)),\quad \smoduleT_{\bullet}(\fm):=F_{\bullet}(\sA)\underset{\sP}{\otimes}\smoduleT_{\fm}.
\end{split}\]
The pairing that I am about to construct will be used in (\ref{E:Thetadef2}) for definition of $\Theta$.
\begin{proposition}
\begin{enumerate}
\item Cohomology of $\sfT_{\bullet}(0)$ coincide with cohomology of $B(\sA)$.
\item There is a nondegenerate pairing $H_{21-i}(\sfT(\fm))\otimes H_i(B(\sA))\to \C$.
\end{enumerate}
\end{proposition}
\begin{proof}
The first statement is obvious. 

The complex $F_{\bullet}(\sA)$ has length $5$. This is why 
\begin{equation}\label{E:completensordecompositon}
\sfT_{i}(\fc)=\bigoplus_{k+l=i}\smoduleT_{k}(\fc)\otimes \sLambda^l,\fc=(0),\fm 
\end{equation}
have length $21$. 
It is easy to construct a nondegenerate pairing between complexes 
\begin{equation}\label{E:pairingsmall}
\sfT_{21-i}(\fm) \otimes \sfT_{i}(0) \to \C.
\end{equation}
The complex \ref{E:completensordecompositon} is a tensor product. The $\smoduleT$-tensor factors are equipped with the pairing from Lemma \ref{L:pairinge}. Exterior algebra has the inner product 
\[\sLambda^i\otimes \sLambda^{16-i}\to \sLambda^{16} \cong \C.\]
The tensor product of the last two inner products define (\ref{E:pairingsmall}).
The pairing that I just constructed indices the pairing in cohomologies stated in the second item of the proposition.
\end{proof}

\begin{definition}{\em
I define $\Theta$ to be a nontrivial cocycle in $\sfT_{16}(\fm)$ that satisfies (\ref{E:Thetadef2}).
}\end{definition}

\paragraph{Definition of the map $\sfq_{*}$}
The map 
\begin{equation}\label{E:mapqdef}
\begin{split}
&\sfq:\sP\to P[z,z^{-1}]\\
&\sfq(\lambda^{\rbeta}):= \lambda_{[\hdelta,\hdelta']}^{\rbeta}(z) (\ref{E:lambdaz})
\end{split}
\end{equation}
descends to a homomorphism of algebras
\begin{equation}\label{E:mapofalgebras}
\sA\to \hA[z,z^{-1}].
\end{equation}

$\sfq$ induces a map of complexes $\sfT_{\bullet}(\fm)\to \fT_{\bullet}(\fp)$.
Here are the details.
The minimal free resolution $F_{\bullet}(\hA)[z,z^{-1}]$ of $\hA[z,z^{-1}]$ is a module over $P[z,z^{-1}]$. After restriction of scalars via $\sfq$ it becomes a complex of $\sP[z,z^{-1}]$-modules. Freeness of the resolution allows to lift the map (\ref{E:mapofalgebras}) to a map of the complexes
\begin{equation}\label{E:resmapc}
F_{\bullet}(\sA)[z,z^{-1}]\to F_{\bullet}(\hA)[z,z^{-1}]
\end{equation}
of the $\sP[z,z^{-1}]$-modules.

The map \ref{E:mapqdef} induces a map of $\sP[z,z^{-1}]$-modules:
\begin{equation}\label{E:fmfpmap}
\begin{split}
&\smoduleT(\fm)[z,z^{-1}]\to \moduleT(\fp) \\
& \prod_{\rbeta\in \setE }\left(\lambda^{\rbeta}\right)^{-1-k_{\rbeta}} \to \prod_{\rbeta\in \setE }\left(\lambda^{\rbeta}(z)\right)^{-1-k_{\rbeta}} 
\prod_{\halpha\in \setN(\fa)\cup \setN(\fa')}\left(\lambda^{\halpha}\right)^{-1}. 
\end{split}
\end{equation}
I define the map of complexes
\[\sfq_{*}:\sfT_{\bullet}(\fm)=F(\sA)\underset{\sP}{\otimes} \smoduleT(\fm)\otimes \sLambda\to F(\hA)\underset{P[z,z^{-1}]}{\otimes} \moduleT(\fp)\otimes \sLambda=\fT_{\bullet}(\fp)\]
as a tensor product of (\ref{E:resmapc}), (\ref{E:fmfpmap}) and the identity map on $\sLambda$.
I define $\Theta(z)$ by the formula (\ref{E:Thetazdef2}).
Before we can move on we have to check nontrivially of $\Theta(z)$.
\begin{proposition}
\[\Theta(z)\neq 0\in H_{16}(\fT(\fp)).\]
\end{proposition}
\begin{proof}
Composition of (\ref{E:mapofalgebras}) with projection $\sfp:\hA[z,z^{-1}]\to \sA[z,z^{-1}]$ is the identity on $\sA[z,z^{-1}]$. I use this to define a map $\sfp_{*}:\fT_{\bullet}(\fp)\to \sfT_{\bullet}(\fm)$ such that $\sfp_{*}\circ \sfq_{*}$ is homotopic to identity. From this I conclude that $\sfq_{*}\Theta \neq 0$
\end{proof}

 \subsection{Nondegeneracy of the field-antifield pairing}
 Now we are in position to put together all the ingredients, define $(a,b)_{\sfU}$ and verify nondegeneracy of $(a,b)_{\sfU}$.
In the notations  (\ref{E:weights}) I define 
\begin{equation}
\begin{split}
&u(\fm):=\sum_{\halpha\in \setN(\fm)}u(\halpha)\\
&r'(\fm)\left| z^{r'(\fm)}:=\prod_{\halpha\in \setN(\fm)}v_{\halpha}(z) \right. .
\end{split}
\end{equation}
where $\setN$ is defined in (\ref{E:Ndef}).
Here is the main statement of this section.
\begin{proposition}\label{E:nondegeneracyf}
\begin{enumerate}
\item Bilinear form (\ref{E:threecomp}) defines nondegenerate pairings 
 \begin{equation}\label{E:fafpairingcoh}
 \begin{split}
& H_{i}(\fT(\fa)){\otimes}H_{k-16-i}(\fT(\fa'))\overset{(a,b)_{\sfU}}{\longrightarrow} \C[z,z^{-1}]\\
& BV^{j}_{\fa}{\otimes}BV^{k'-j}_{\fa'}\overset{(a,b)_{\sfU}}{\longrightarrow} \C[z,z^{-1}]\\
&k=s(\fm)+s'(\fm),\quad k'=s'(\fm)-s(\fm).
\end{split}\end{equation}
\item The non-renormalized $\Aut$-weight   of the pairing is equal to 
\[(a[\hdelta,\hdelta']-s(\fm),u[\hdelta,\hdelta']-u(\fm), r[\hdelta,\hdelta']-r(\fm))\quad (\text{see }\ref{E:functionsdef},\ref{E:sigmaalt}\text{ for notations}).\]

\end{enumerate}
\end{proposition}
\begin{proof}
The two pairings are related by the isomorphism (\ref{E:BV}).
The spectral sequence from item \ref{I:fTtwo}, Proposition \ref{P:fT} is induced by the filtration
\[F_j(\fT_i(\fc))=\bigoplus_{k+l=i, k\leq j}\moduleT_{k}(\fc)\otimes \Lambda^l, \fc\neq \fp,\]
\[F_j(\fT_i(\fp))=\bigoplus_{k+l=i, k\leq j}\moduleT_{k}(\fp)\otimes \sLambda^l.\]
Filtrations $F_j(\fT_i(\fp)),F_j(\fT_i(\fa)),F_j(\fT_i(\fa'))$ are compatible with the triple product (\ref{E:Upsilon}). That is why $\sfU$ induces a map of the first pages:
\begin{equation}\label{E:Upsilonred}
\left( \delta_{s(\fa),i}\otimes V_{j}\right)\otimes \left(\delta_{s(\fp),i'}\otimes V_{j'}\right)\otimes \left(\delta_{s(\fa'),{i''}}\otimes V_{j''}\right)\to \delta_{s(\fm),{i+i'+i''}}\otimes V_{j+j'+j''}.
\end{equation}
Only one group in the sum $\bigoplus_{i+j=16}\delta_{s(\fp),i}\otimes V_j$ is nontrivial. It is $\delta_{s(\fp),16}\otimes V_{0}\cong\C$. The top component of \[\Theta(z)=\sum_{i=0}^{16} \Theta_i(z), \Theta_i(z)\in \moduleT_{16-i}(\fp)\otimes \sLambda^i \] with respect to filtration must be equal to (up to a nonzero factor) 
\begin{equation}\label{E:Thetazdef3}
 \Theta_{16}(z)=\varpi_{\fp} \theta\in \moduleT_{\fp}\otimes \sLambda^{16}= \moduleT_0(\fp)\otimes \sLambda^{16}.
\end{equation}

The inner product (\ref{E:threecomp}) induces a pairing between first pages of spectral sequences. If I prove that this derived pairing is not degenerate, then it will automatically imply that (\ref{E:threecomp}) is not degenerate.

To define the pairing between first pages, I use the same scheme as for (\ref{E:threecomp}) but replace (\ref{E:Upsilon}) by (\ref{E:Upsilonred}), (\ref{E:Thetazdef2}) by (\ref{E:Thetazdef3}). 
The resulting pairing 
\[ \varpi_{\fa} \theta_{\setN(\fa)}\otimes V_{d-j}\otimes \varpi_{\fp} \theta_{\setE}\otimes V_0\otimes \varpi_{\fa'} \theta_{\setN(\fa')}\otimes V_{j}\to \varpi_{\fm}\theta_{[\hdelta,\hdelta']} \otimes V_{d}\cong \C\]
is not degenerate because it is constructed from nondegenerate pairing on $\bigoplus_i V_i$ (Proposition \ref{C:PoincareGen}) and because $\varpi_{\fa}\varpi_{\fp}\varpi_{\fa'}=\varpi_{\fm}$ (see the proof of Proposition \ref{P:tensorproduct}).
\begin{lemma}\label{L:weightth}
The $\Aut$-weight of $\Theta(z)$ is equal to $(0,0,0)$. 
\end{lemma}
\begin{proof}
 The easiest way to determine this is to observe that the spectral sequence from Proposition \ref{E:nondegeneracyf} is compatible with $\Aut$-action. Thus it is safe to compute the weight of $\Theta(z)$ by calculating the weight of its leading component $\Theta_{16}(z)$(\ref{E:Thetazdef3}) for which the answer is obvious.
\end{proof}

Lemma \ref{L:weightth} implies that $\Theta(z)$ doesn't contribute to the $\Aut$-weight of the pairing. I conclude that the weight in question is equal to the weight of the element (\ref{E:topelementf}), which is the sum of $(-a[\hdelta,\hdelta'],-u[\hdelta,\hdelta'],-r[\hdelta,\hdelta'])$-the weight of $\varpi_{\fm}\otimes c$ (see (\ref{E:sigmaalt}) and Proposition \ref{P:pairingdegree}) and of
\[\left(s[\hdelta,\hdelta'],u[\hdelta,\hdelta'],r'[\hdelta,\hdelta']\right) ,\]
 the weight of $\theta_{[\hdelta,\hdelta']}$. 
\end{proof}
\begin{remark}{\rm (c.f. \ref{R:degreeshift1})
I define a renormalized $\Aut$-action on $\fT_{\bullet}(\fa)$ by twisting on the character $\chi'[\hdelta,(1)^{-1}]^{-1} =\chi^{-1}[\hdelta,(1)^{-1}]\det \C[\hdelta,(1)^{-1}]$ (Proposition \ref{P:funequation}, equation (\ref{E:trivmod})) and $\fT_{\bullet}(\fa')$ by $\chi'[(0)^1, \hdelta']^{-1}$
}\end{remark}
\begin{remark}
Proposition \ref{P:fT} implies equality of virtual characters

\begin{equation}\label{E:zbvbareeq}
 ZBV^{bare}_{\fc}(t,q,z):=\sum_{i}(-1)^i \chi_{BV^i_{\fc}}(t,q,z)=Z^{bare}_{\fc}(t,q,z)\Lambda \spinor (-t,q,z).
\end{equation}

The pairing (\ref{E:fafpairingcoh}) implies equality of the virtual characters:

\begin{equation}\label{E:BVdualitychar}
\begin{split}
& ZBV^{bare}_{\fa}(t,q,z)=\\
&=(-1)^{s'(\fm)-s(\fm)}
t^{s(\fm)-a(\fm)}q^{u(\fm)-u(\fm)}z^{r'(\fm)-r(\fm)}ZBV^{bare}_{\fa'}(t^{-1},q^{-1},z^{-1}).
\end{split}
\end{equation}
\[\begin{split}
&ZBV^{bare}_{\fa}[\hdelta,\hdelta'](t,q,z)= ZBV^{bare}_{\fa'}[\reflection(\hdelta'),\reflection(\hdelta)](t,q^{-1},Sz)
\end{split}\]
the last equality is induced by the isomorphism $\reflection$ (\ref{E:mapsintr}).
\end{remark}
 Equality 
\[\Lambda \spinor (t,q,z)=
t^{s(\fm) }q^{u(\fm) }z^{r'(\fm) }\Lambda \spinor (t^{-1},q^{-1},z^{-1})\]
comes from the pairing 
\[\Lambda^{i} \spinor\otimes \Lambda^{|[\hdelta,\hdelta']|-i} \spinor\to \Lambda^{|[\hdelta,\hdelta']|} \spinor.\]
\begin{proposition}\label{P:maintwosf}
The virtual character $Z_{\fa}[\hdelta,\hdelta']$ (\ref{E:renormchar}) satisfies 
\begin{equation}\label{E:feqfieldandifield}
 \begin{split}
&Z_{\fa}[\hdelta,\hdelta'](t,q,z)=\\
&=-t^{-8}Z_{\fa}[\reflection(\hdelta'),\reflection(\hdelta)](t^{-1},q,Sz^{-1}).
\end{split}
\end{equation}
\end{proposition}
\begin{proof}
From (\ref{E:zbvbareeq}) and (\ref{E:BVdualitychar}), I deduce that
\[
\begin{split}
&Z^{bare}_{\fa}(t,q,z)=(-1)^{s'(\fm)}
t^{-a(\fm)}q^{-u(\fm)}z^{-r(\fm)}\times\\
&\times Z^{bare}_{\fa'}(t^{-1},q^{-1},z^{-1}).
\end{split}
\]

This implies equation (\ref{E:feqfieldandifield}) for $Z_{\fa}[\hdelta,\hdelta']$.
Derivation of (\ref{E:feqfieldandifield}) uses formula (\ref{E:fbare}), items (\ref{I:chareq2}, \ref{I:chareq3}) from Proposition \ref{P:funequation} and equation (\ref{E:htandsw}), which summarize to
\[
\begin{split}
 &a(\fm)=a(\fa)+a(\fa'),\\
 &u(\fm)=u(\fa)+u(\fa'),\\
 &r(\fm)=r(\fa)+r(\fa').\\
\end{split}
\]
\end{proof}

Equation (\ref{E:feqfieldandifield}) simplifies if I set $[\hdelta,\hdelta']=[(0)^{-N},(1)^N]$. Denote $Z_{\fa}[(0)^{-N},(1)^N]$ by $Z'[N]$. Then (\ref{E:feqfieldandifield}) becomes (\ref{E:ZprimeN}).
Operator $S$ disappears in $Sz^{-1}$ because representation of $\widetilde{\bT}^5$ in $H_{\fa}^i(A[(0)^{-N},(1)^N])$ is a restriction of $\Spin(10,\C)$-representation and $S(z)=SzS^{-1}, S\in \Spin(10,\C)$.

\subsection{The definition of $\BV^{i+\itwo}(\fc)$}\label{S:bvdef}
Exposition of this section follows closely Section \ref{SS:definitionh}.

Fix $[\hdelta_2,\hdelta_4]\subset [\hdelta_1,\hdelta_4]$. The map $\Th_{\sfp}$, corresponding to $\sfp:\hA[\hdelta_1,\hdelta_4]\to \hA[\hdelta_2,\hdelta_4]$ induces the map of complexes
\[\begin{split}
&\fT_{Kos}^{\bullet}(\fa)[\hdelta_2,\hdelta_4]\to \fT_{Kos}^{\bullet}(\fa)[\hdelta_1,\hdelta_4][\codim], \\
&\codim=s'[\hdelta_1,\hdelta_4]-s'[\hdelta_2,\hdelta_4]. 
\end{split}\]
I modify the map by multiplying the image of $\Th$ on 
$\theta_{[\hdelta_1,\hdelta_4]\backslash [\hdelta_2,\hdelta_4]}$ (\ref{E:thetasub}). 
As a result, I get the map
\begin{equation}\label{E:thmodalt}
 \begin{split}
&\Th'_{\sfp}:\fT_{Kos}^{\bullet}(\fa)[\hdelta_2,\hdelta_4]\to \fT_{Kos}^{\bullet}(\fa)[\hdelta_1,\hdelta_4][\hcodim], \\
&\hcodim=s'[\hdelta_1,\hdelta_4]-s'[\hdelta_2,\hdelta_4] -(s[\hdelta_1,\hdelta_4]-s[\hdelta_2,\hdelta_4]).
\end{split}
\end{equation}

I leave to the reader verification that $\fT_{Kos}^{\bullet}(\fa)$ (\ref{E:bvkos}) and $\fT_{P}^{k}(\fa,M),M=\hA$ from (\ref{E:ferm}) compute the same cohomology. I also won't prove a simple statement that map (\ref{E:thmodalt}) coincides with (\ref{E:thprime}) when $\hRo=\hA[\hdelta_1,\hdelta_4]$,$\hRt=\hA[\hdelta_2,\hdelta_4]$.

\begin{proposition}
Fix $\hdelta_1<\hdelta_2<(0)^{-1},(1)^{-1}<\hdelta_3<\hdelta_4$ and $\fc=\fa,\fa',\fp,\fm,\fb$. 

There is a commutative diagram 
\begin{equation}\label{E:diagramindpro2}
\begin{CD}
\BV_{\fc}^i[\hdelta_2,\hdelta_4] @>\sfp'>> \BV_{\fc}^i[\hdelta_2,\hdelta_3] \\
@VV\Th'_{\sfq} V @VV\Th'_{\sfq'} V \\
\BV_{\fc}^{i+\hcodim}[\hdelta_1,\hdelta_4] @>\sfp>> \BV_{\fc}^{i+\hcodim}[\hdelta_1,\hdelta_3]. 
\end{CD}
\end{equation}
The maps are induced from the commutative diagram of algebras (\ref{E:diagramalg})
 and the maps $\Th'_{\sfq},\Th'_{\sfq'}$ (\ref{E:thmodalt}). \[\hcodim=\rho(\hdelta_2)-\rho(\hdelta_1)-|[\hdelta_1,(1)^{-1}]\backslash [\hdelta_2,(1)^{-1}] |.\]

\end{proposition}
\begin{proof}
The proof is similar to the proof of Proposition \ref{P:commutativity} and will be omitted.
\end{proof}

Here is formal definition of the limits group.
\begin{equation}\label{E:semiinfinitecohdef1}
\begin{split}
& BV_{\fc}^{i+\itwo}(\hA):=\bigoplus_{w}\underset{\underset{\hdelta}{\longrightarrow}}{\lim}\underset{\underset{\hdelta'}{\longleftarrow}}{\lim}BV_{\fc}^{i+\hd(\fc)}[\hdelta,\hdelta']^w, \fc=\fa,\fa',\hd(\fc)=s'(\fc)-s(\fc)
\end{split}
\end{equation}
where the structure maps come from the diagram (\ref{E:diagramindpro2}) and the function $s'$ from (\ref{E:sschiftdef}).
I give only statements of the next three proposition. The interested reader will recover the proofs without a difficulty if he will use the proofs of the analogous statements from Section \ref{SS:definitionh} as a template.
\begin{proposition}\label{P:pairingcomp2}
Suppose $[\hdelta_2,\hdelta_4]\subset [\hdelta_1,\hdelta_4]$ and $\hdelta_2<(1)^{-1},(0)^0<\hdelta_4$, and that $\hA[\hdelta_i,\hdelta_4],i=1,2$ are Gorenstein. Then there are commutative diagrams
\[\begin{CD}
\BV_{\fa}^i[\hdelta_2,\hdelta_4]\otimes \BV_{\fp}^j[\hdelta_1,\hdelta_4]\otimes \BV_{\fa'}^k[\hdelta_1,\hdelta_4] @>\sfU_{Kos}\circ (\id\otimes\, \sfp\otimes\, \sfp)>> \BV_{\fm}^{i+j+k}[\hdelta_2,\hdelta_4] \\
@VV\Th_{\sfp}\otimes\, \id\otimes\, \id V @VV\Th_{\sfp} V \\
 \BV_{\fa}^{i+\hcodim}[\hdelta_1,\hdelta_4]\otimes \BV_{\fp}^j[\hdelta_1,\hdelta_4]\otimes \BV_{\fa'}^k[\hdelta_1,\hdelta_4] @>\sfU_{Kos}>> \BV_{\fm}^{i+j+k+\hcodim}[\hdelta_1,\hdelta_4] 
\end{CD}\]

\[\begin{CD}
\BV_{\fa}^i[\hdelta_1,\hdelta_4]\otimes \BV_{\fp}^j[\hdelta_1,\hdelta_4]\otimes \BV_{\fa'}^k[\hdelta_2,\hdelta_4] @>\sfU_{Kos}\circ (p\otimes\, \sfp\otimes\,\id )>> \BV_{\fm}^{i+j+k}[\hdelta_2,\hdelta_4]\\
@VV\id\otimes\, \id\otimes\, \Th_{\sfp} V @VV\Th_{\sfp} V \\
 \BV_{\fa}^{i}[\hdelta_1,\hdelta_4]\otimes \BV_{\fp}^j[\hdelta_1,\hdelta_4]\otimes \BV_{\fa'}^{k+\hcodim}[\hdelta_1,\hdelta_4] @>\sfU_{Kos}>> \BV_{\fm}^{i+j+k+\hcodim}[\hdelta_1,\hdelta_4]
\end{CD}\]
where $\sfp$ is the map induced by the restriction homomorphism $\hA[\hdelta_1,\hdelta_4]\to \hA[\hdelta_2,\hdelta_4]$, $\Th'_{\sfp}$
 is the modified Thom map (\ref{E:thmodalt}),
 $\sfU_{Kos}$ is the product map (\ref{E:Upsilon}).
\end{proposition}
\begin{proposition}
In the assumptions of Proposition \ref{P:pairingcomp2}, the pairing (\ref{E:threecomp}) satisfies
\[(\Th'_{\sfq} (a),b)_{\sfU}=(a,q(b))_{\sfU}\]
where $a\in BV_{\fa}^{i}[\hdelta_2,\hdelta_4]$,$b\in BV_{\fa'}^{j}[\hdelta_1,\hdelta_4]$, $i+j+\rho(\hdelta_2)-\rho(\hdelta_1)-(|\hdelta_1,\hdelta_4|-|\hdelta_2,\hdelta_4|)=\rk(\hdelta_1,\hdelta_4)+1-|\hdelta_1,\hdelta_4|$. 
\end{proposition}
There is an isomorphism
$\reflection:BV_{\fa}^{i+\itwo}(\hA)\to BV_{\fa'}^{i+\itwo}(\hA)$
induced by $\reflection$ (\ref{E:mapsintr})
The above information enables
 me to define the pairing
\begin{equation}\label{E:pairingferm}
 BV_{\fa}^{i+\itwo}(\hA)\otimes BV_{\fa}^{3-i+\itwo}(\hA)\to \C.
\end{equation}
\begin{corollary}
\begin{equation}\label{E:ZprimeN}
Z_{\fa}(t,q,z)=-t^{-8}Z_{\fa}(t^{-1},q,z^{-1}). 
\end{equation}
\end{corollary}
\begin{remark}Spectral sequence arguments shows that
\[ZBV_{\fa}(t,q)=\sum_{u\geq 0,a}b_{a,u}t^aq^u=Z_{\fa}(t,q)Z_{S}(t,q),\quad Z_{S}(t,q)=\prod_{u>0} (1-t^{-1}q^u)^{16} \prod_{u\geq 0} (1-tq^u)^{16}\]
$Z_{S}(t,q)\in \Z[t,t^{-1}][[q]]$ is the renormalized virtual character of  spinors. The above equality and Remark \ref{R:formalseries9} implies that $ZBV_{\fa}(t,q)\in \Z((t))[[q]]$. The symmetry of the coefficients $b_{a,u}=-b_{-8-a,u}$ of $ZBV_{\fa}(t,q)$  that is induced by the pairing  \ref{E:pairingferm} implies that in fact  $ZBV_{\fa}(t,q)\in \Z[t,t^{-1}][[q]]$. From this I conclude that $Z_{\fa}(t,q)\in \Z((t))[[q]]\cap \Q(t)[[q]]$.

\end{remark}

\section{ $H^i_{\fa}({\hA})$ and localization}\label{S:localization}
Let $\hA$ be the algebra based on $[\hdelta,\hdelta']$. In addition, all the intervals in this section satisfy (\ref{E:purity}).
In this section, I will address the problem of denominators outlined in the introduction.

The linear space $H^i_{\fa}({\hA})$ is by construction a $P$-module. In this section, I assume that $[(0)^0,(1)^0]\subset [\hdelta,\hdelta']$. The linear space $H^i_{\fa}({\hA})$ becomes $\sP$-module. In practice, it is often convenient to 
localize elements of $H^i_{\fa}({\hA})$ by inverting some of the generators $\lambda^{\rbeta}:=\lambda^{\rbeta^0}$ without significantly changing the size of group $H^i_{\fa}({\hA})$. A possible way to do this is in the formalism of the \v{C}ech complex.

Recall that the extended \v{C}ech complex of an $\hRo$-module $M$ is the complex
\[M \to \prod\nolimits_{i_0} M_{f_{i_0}} \to
\prod\nolimits_{i_0 < i_1} M_{f_{i_0}f_{i_1}} \to
\ldots \to M_{f_1\ldots f_r}=:\Cech_e^{\bullet}(\mathfrak{U},M)\]
where $M_{f_1\ldots f_r}$ is a localization with respect to $f_1\ldots f_r\in \hRo$. $\mathfrak{U}$ is the open cover of $Spec\, \hRo$ by 
$U_i=\{x\in Spec\, \hRo|f_i(x)\neq 0\}$. 
A chain from $\Cech_e^m(\mathfrak{U},\F)$ is a family of elements
$\phi = \{ \phi_{\alpha_0,\dots,\alpha_m}\} \in \bigoplus M_{f_{\alpha_0} \cdots f_{\alpha_m} }$.
The differential is defined by the formula 
\[(d\phi)_{\alpha_0\dots\alpha_{m+1}}=\sum_{i=0}^{m+1}(-1)^i \res( \phi_{\alpha_0\dots \hat{\alpha}_{i}\dots \alpha_{m+1}}).\]
There is a short exact sequence of complexes

\begin{equation}\label{E:cechexact}
0\to \Cech^{\bullet}(\mathfrak{U},M)\to \Cech_e^{\bullet}(\mathfrak{U},M)\to M\to 0
\end{equation}
where $\Cech^{\bullet}(\mathfrak{U},M)$ is the classical \v{C}ech complex.

In our application I fix an open covering $\mathfrak{U}=\{U_{\ralpha}|\lambda^{\ralpha}\neq 0,\ralpha\in \setE \}$ of $\C^{16}\backslash \{0\}\subset \C^{16}$.
 I use order 
 \begin{equation}\label{E:orderaf}
\begin{split}&\dots <(4)^{r-1}< (0)^r <(3)^{r-1}<(12)^r<(2)^{r-1}<(13)^r<(1)^{r-1}<(14)^r<(23)^r<\\
&<(15)^r<(24)^r<(25)^r<(34)^r<(35)^r<(5)^r<(45)^r<(4)^r<(0)^{r+1}<\dots\end{split}\end{equation}
 on $\setE\subset \sethE$ to order indices. The algebra $R$ in our case is $\sP$ and sequence $\{f_i\}$ is $\{\lambda^{\ralpha}\}$, $M$ is some $\sP$-module.

$\Cech_{e, \leq n}^{\bullet}(\mathfrak{U},M)$ is a sub complex of $\Cech_e^{\bullet}(\mathfrak{U},M)$ spanned by chains whose denominators have degrees $\leq n$. It is explained in the Appendix A4.1 \cite{ComEisenbud} that the complex $\Cech_{e, \leq n}^{\bullet}(\mathfrak{U},M)$ is isomorphic to Koszul complex $K(M, \{\lambda^{\rbeta}\},n)$ and 

\begin{equation}\label{cechkoszul}
\Cech_e^{\bullet}(\mathfrak{U},M)\cong \underset{\overset{n}{\longrightarrow} }{\lim} K(M, \{\lambda^{\rbeta}\},n).
\end{equation}

After I apply to exact sequence (\ref{E:cechexact}) with $M={\hA[\hdelta,\hdelta']}$ the functor $\underset{\overset{n}{\longrightarrow} }{\lim} K(?, [\hdelta,(1)^{-1}],n)$ and use identification (\ref{cechkoszul}), I will get the long exact sequence
\begin{equation}\label{E:exactcechC}
0\to H^0(C)\to H^0_{\fa'}({\hA})\to H^0_{\fa}({\hA})\to H^1(C)\to\cdots 
\end{equation}
The ideal $\fa'$ is equal $\fri[\hdelta,(1)^0]$, $C^{\bullet}$ is the total complex of the bicomplex $ \underset{\overset{n}{\longrightarrow} }{\lim} K^{\bullet}(\Cech^{\bullet}(\mathfrak{U},\hA), [\hdelta,(1)^{-1}],n)$.

The cohomology of $C^{\bullet}$ have a simple structure.
\begin{proposition}\label{E:isomorphism1}
\[H^{i+1}(C)=H^i_{\fa}({\hA})\text{ for }-\rho(\hdelta)\leq i\leq -\rho(\hdelta)+3\] and \[H^{i}(C)=H^i_{\fa'}({\hA})\text{ for }-\rho(\hdelta)+8\leq i\leq -\rho(\hdelta)+11.\]
For other values of $i$ $H^i(C)=0$.

There is also an obvious modifications of these isomorphisms when $\hdelta\to-\infty,\hdelta'\to\infty$.
\end{proposition}
\begin{proof}
I know (Proposition \ref{P:vanishingmain}) that $H^i_{\fa}({\hA})\neq 0$ for $-\rho(\hdelta)\leq i\leq -\rho(\hdelta)+3$ and $H^i_{\fa'}({\hA})\neq 0$
for $-\rho(\hdelta)+8\leq i\leq -\rho(\hdelta)+11$ ( use (\ref{E:sigmaTcoh}) and Proposition \ref{P:vanishingmain}). The result follows from the exact sequence (\ref{E:exactcechC}).
\end{proof}

Let $\widetilde{M}$ be the coherent sheaf on $\C^{16}$ associated with the module $M$. I will use the same notation for the restriction of $\widetilde{M}$ on $\C^{16}\backslash \{0\}$. I denote the \v{C}ech cohomology $H^i\Cech^{\bullet}(\mathfrak{U},M)$ of $\widetilde{M}$ by $\rH^i(\C^{16}\backslash \{0\}, \widetilde{M})$. With the module $M=H^i_{\fa}({\hA})$ I associate the coherent sheaf $\widetilde{H}^i_{\fa}({\hA})$ on $\C^{16}\backslash \{0\}$. 
\begin{proposition}\label{E:cechspectral}
\begin{enumerate}
\item There is a spectral sequence
\begin{equation}\label{E:spectralincreas}
\rH^i(\C^{16}\backslash \{0\}, \widetilde{H}^j_{\fa}({\hA}))\Rightarrow H^{i+j+1}(C).
\end{equation}
\item There is a similar spectral sequence
\[\rH^i(\C^{16}\backslash \{0\}, \widetilde{H}^{j+\itwo}_{\fa}(\hA))\Rightarrow H^{i+j+1+\itwo}(C[-\infty,\infty]).\]
\end{enumerate}
\end{proposition}
\begin{proof}
Decomposition $[\hdelta,(1)^0]=[\hdelta,(1)^{-1}]\cup [(0)^0,(1)^0]$ enables me to apply Construction \ref{C:bicomplex} to Koszul complex $K(\hA,[\hdelta,(1)^0],n)$. By the above discussion the corresponding spectral sequence is isomorphic to (\ref{E:spectralincreas}).

The second statement follows from the first because $\Th$ is a homomorphism of $\sP$-modules.
\end{proof}

\section{Spaces of formal maps}\label{S:formalmaps}
There is an informal way to interpret elements of the space $\cZ^{poly}_{\tspace}$ as probes of the target $\tspace$. Elements of $\cZ^{poly}_{\tspace}$, which are represented by Laurent polynomials, are the true maps from $\C^{\times}$ to $\tspace$. Alternatively, one can take maps from the infinitesimal punctured disk to $\tspace$ as the probes. Intuitively it is clear that infinitesimal maps do not penetrate far into $\tspace$. In particular, they should be less sensitive to the fundamental group of $\tspace$ than Laurent maps. 
In this section I propose some conjectures about the relations of the groups $H^{i+\itwo}$ based on spaces of maps of these two types.

I formalize these descriptions of these spaces by using the language of algebra.
There are three versions of the algebra ${\hA[\hdelta,\hdelta']}$ when $\hdelta\to -\infty, \hdelta'\to \infty$, which I will use in this section. My exposition is less formal than in previous sections. The reason is that this section contains mostly conjectures.

Inclusions of intervals $[\hdelta,\hdelta']\subset [\hgamma,\gamma']$ define 
an embedding $\C[\hdelta,\hdelta']\to \C[\hgamma,\gamma']$ which is tautological on the generators. I define $\C(-\infty,\infty)=\C[\sethE]$ to be the direct limit $\underset{\overset{\hdelta,\hdelta'}{\longrightarrow}}{\lim} \C[\hdelta,\hdelta']$.
The polynomial algebra $\C[\sethE]$
have three particularly interesting linear topologies. The bases of these topologies are generated by the system of ideals:
\begin{enumerate}
\item $\fri_{\pm}(n)=\fri((-\infty,(1)^{-n}]\cup[(0)^n,\infty))\subset $,
\item $\fri_{+}(n)=\fri([(1)^{n},\infty)$,
\item $\fri_{-}(n)=\fri((-\infty,(1)^{-n}])$.
\end{enumerate}
By construction $\fri_{\pm}(n+1)\subset \fri_{\pm}(n)$,$\fri_{+}(n+1)\subset \fri_{+}(n)$,$\fri_{-}(n+1)\subset \fri_{-}(n)$. Define the completions
\begin{equation}\label{E:polalg}
\begin{split}
&\C[-\infty,\infty]=\underset{\overset{n}{\longleftarrow}}{\lim} \C(-\infty,\infty)/\fri_{\pm}(n),\\
&\C(-\infty,\infty]=\underset{\overset{n}{\longleftarrow}}{\lim} \C(-\infty,\infty)/\fri_{+}(n),\\
&\C[-\infty,\infty)=\underset{\overset{n}{\longleftarrow}}{\lim} \C(-\infty,\infty)/\fri_{-}(n).
\end{split}
\end{equation}
 
Maps which are identity on the generators defines continuous homomorphisms 
 \[m_+:\C(-\infty,\infty]\to \C[-\infty,\infty],\quad m_-:\C[-\infty,\infty)\to \C[-\infty,\infty].\] 
 
Relations $\Gamma^{\bm{s}^k}$ (\ref{E:equationsaf0}) ($N=-\infty, N'=\infty$) are the elements of all the three algebras (\ref{E:polalg}). The ideals $\fr\subset \C[-\infty,\infty]$, $\fq_{+}\subset \C(-\infty,\infty],\fq_{-}\subset \C[-\infty,\infty)$ are the topological closure of the ideals generated by $\Gamma^{\bm{s}^k}$. Introduce the algebras
\[\hA[-\infty,\infty]=\C[-\infty,\infty]/\fr,\]
\[\hA(-\infty,\infty]=\C[-\infty,\infty)/\fq_{+},\]
\[\hA[-\infty,\infty)=\C[-\infty,\infty)/\fq_{-}.\]
The homomorphisms $m_+$, $m_-$ define the maps
\begin{equation}\label{E:mdef}
m_+:\hA(-\infty,\infty]\to \hA[-\infty,\infty], \quad m_-:\hA[-\infty,\infty)\to \hA[-\infty,\infty].
\end{equation}

\begin{question}
What is the kernels of the maps (\ref{E:mdef})?
\end{question}

The algebra $\hA[-\infty,\infty]$ admits a continuous homomorphism $\hA[-\infty,\infty]\to {\hA[\hdelta,\hdelta']}$, where $ {\hA[\hdelta,\hdelta']}$ is equipped with the discrete topology. I leave 
without the proof of the following pair of simple propositions.
\begin{proposition}
$\hA[-\infty,\infty]=\underset{\overset{\hdelta,\hdelta'}{\longleftarrow}}{\lim}{\hA[\hdelta,\hdelta']}$.
\end{proposition}
The quotient of $\hA[-\infty,\infty)$ by the ideal $\fri([-\infty,\infty)\backslash [\hdelta,\infty))$ gives me the algebra $\hA[\hdelta,\infty)$.
The shift (\ref{E:shift})
 identifies $\hA[\hdelta,\infty)$ with $\hA[\shift(\hdelta),\infty)$.
\begin{proposition}
\[\hA[-\infty,\infty)=\underset{\overset{\hdelta}{\longleftarrow}}{\lim}\hA[\hdelta,\infty)=\underset{\overset{n}{\longleftarrow}}{\lim}\hA[(0)^{-n},\infty).\]
\end{proposition}

The anti-involution $\reflection$ of $\sethE$ (\ref{E:involution}) comes from an automorphism of the linear space $\spinor[z,z^{-1}]$. \\$\reflection(-\infty,(1)^{-n}]=[(0)^{n},\infty)\supset [(1)^{n-2},\infty)$. The action of $\reflection$ on $\sspinor[z,z^{-1}]$ is compatible with the relations $\Gamma^{\bm{s}^k}$. It implies that 
there is a continuous isomorphism
\[\reflection:\hA(-\infty,\infty]\to \hA[-\infty,\infty).\]

$\hA[(0)^0,\infty)$ is the algebra of function on the space of arcs 
\[\Arc(0)=\left\{\theta(z)=\sum_{\rbeta\in \setE}\lambda^{\rbeta}(z)\theta_{\rbeta}=\sum_{k\geq 0}\sum_{\rbeta\in \setE}\lambda^{\rbeta^k}\theta_{\rbeta}z^k \text{ that satisfy }(\ref{E:pure})\right\}.\]
In the above formula $z$ is treated as a formal variable and no convergence of the series is expected.
The space related to $\hA(-\infty, (1)^0]$ is
\[\Arc(\infty)=\left\{\theta(z)=\sum_{\rbeta\in \setE}\lambda^{\rbeta}(z)\theta_{\rbeta}=\sum_{k\geq 0}\sum_{\rbeta\in \setE}\lambda^{\rbeta^k}\theta_{\rbeta}z^{-k} \text{ that satisfy }(\ref{E:pure})\right\}.\] 

The spaces of arcs or spaces of jets are well studied in the literature. See e.g. \cite{DenefLoeser}. These spaces are the inverse limits of the spaces of finite jets. In case of the cone $\cone$, these spaces are 
 \[\Arc^{N}=\Maps(Spec\, \C[z]/z^{N+1},\cone)= \left\{\theta(z)=\sum_{k= 0}^N\sum_{\rbeta\in \setE}\lambda^{\rbeta^k}\theta_{\rbeta}z^k \text{ that satisfy }(\ref{E:pure}) \mod z^{N+1}.\right\},N\geq 0\]
 \[\Arc_{N}=\left\{\theta(z)=\sum_{k= N}^0\sum_{\rbeta\in \setE}\lambda^{\rbeta^k}\theta_{\rbeta}z^{k} \text{ that satisfy }(\ref{E:pure}) \mod z^{N-1}.\right\}, N\leq 0.\]
The projections are 
\[\projectiono^N:\Arc^{N}\rightarrow \Arc^{N-1},\quad \projectiono^N:\theta(z) \mod z^{N+1} \to \theta(z) \mod z^{N}.\]
As the equations (\ref{E:pure}) are homogeneous, the map
\[\inclusionth^{N-1}(\theta(z))= z\theta(z),\quad \inclusionth^{N-1}:\Arc^{N-1}\rightarrow \Arc^{N}\]
is well defined. Similar structures exist on $\Arc_{N}$
 The algebra of polynomial functions \[\tjA^N=\jA[(0)^0,(1)^N]=\O[\Arc^{N}]\]
 shares its generators \[\{\lambda^{\rbeta^l}|\rbeta\in \setE, 0\leq l\leq N\}\] with the algebra $\hA[(0)^0,(1)^N]$. The main difference with $\hA[(0)^0,(1)^N]$ is that it has as half as many relations. More precisely, the algebra $\tjA^{N}$ is the quotient of the polynomial algebra $\mathbb{C}[(0)^0,(1)^N]$ by the ideal $\fq[(0)^0,(1)^N]$ of relations generated by 
$\Gamma^{\bm{s}^k}$ $ 0 \leq l,l',k\leq N \quad \bm{s}\in \rG$ (\ref{E:equationsaf0}). 
More generally, define 
\[\begin{split}
&\jA[\hdelta,(1)^N]=\jA[(0)^0,(1)^N]/\fri([(0)^0,(1)^N]\backslash [\hdelta,(1)^N])\\
&\text{ for } [\hdelta,(1)^N]\subset [(0)^0,(1)^N].
\end{split}\] 
I postulate that automorphism $\shift$ (\ref{E:shift}) induces an isomorphism $\jA[\hdelta,(1)^N]\cong \jA[\shift(\hdelta),(1)^{N+1}]$. This way I define $\jA[\hdelta,(1)^N]$ for the general $\hdelta< (1)^N$.
The algebras $\bjA_{N}=\jA[(0)^N,\hdelta]=\O[\Arc_{N}]$ are defined along the same lines.
Here is the characterization of the algebra of functions on the space of formal loops.
\begin{proposition}
\begin{enumerate}
\item The homomorphisms $\projectiono^{N*}$,$\projectiono_N^{*}$ induce a direct system of algebras $\projectiono^{N*}:\tjA^{N-1}\to \tjA^{N}$ $\projectiono_N^{*}:\bjA^{N-1}\to \bjA^{N}$ such that
\[\underset{\longrightarrow}{\lim} \tjA^{N}=\hA[(0)^0,\infty)\quad \underset{\longrightarrow}{\lim} \bjA^{N}=\hA(-\infty, (1)^0]\]
and
\begin{equation}\label{E:dirdef}
\underset{\longrightarrow}{\lim} \jA[\hdelta,(1)^N]=\hA[\hdelta,\infty)\quad \underset{\longrightarrow}{\lim} \jA[(0)^N,\hdelta]=\hA(-\infty, \hdelta].
\end{equation}
\item The homomorphisms $\inclusionth_N^*$ induce an inverse system of algebras $\inclusionth^{N-1*}:\tjA^{N}\to \tjA^{N-1}$, $\inclusionth_{N-1}^*:\bjA^{N}\to \bjA^{N-1}$ such that
\[\underset{\longleftarrow}{\lim} \tjA^{N}=\hA[-\infty,(1)^0]\quad \underset{\longleftarrow}{\lim} \bjA^{N}=\hA[(0)^0,\infty].\]
\item \[\underset{\overset{\projectiono^*}{\longrightarrow}}{\lim} \underset{\overset{\inclusionth^*}{\longleftarrow}}{\lim} \tjA^{N+N'}=\hA[-\infty,\infty)\quad \underset{\overset{\projectiono^*}{\longrightarrow}}{\lim} \underset{\overset{\inclusionth^*}{\longleftarrow}}{\lim} \bjA_{N+N'}=\hA(-\infty,\infty].\]
\end{enumerate}
\end{proposition}

The map $\jA[\hdelta,(1)^N]\to \hA[\hdelta,\infty)$ is the part of the structure maps defining the direct limit (\ref{E:dirdef}). 
The composition
\[\jA[\hdelta,(1)^N]\to \jA[\hdelta,\infty)\overset{}{\to} \hA[\hdelta,\infty]\to \hA[\hdelta,(1)^N]\]
is nothing else but the projection
\[\C[\hdelta,(1)^N]/\fq[\hdelta,(1)^N]\to \C[\hdelta,(1)^N]/\fr[\hdelta,(1)^N].\]
\begin{question}
Is $\rReg[\hdelta,(1)^N]$ (\ref{E:regconstr}) regular in $\jA[\hdelta,(1)^N]$? What is the simplest maximal regular sequence in the augmentation ideal.
\end{question}
Note that the map $\reflection$ induces an isomorphism \[\reflection:\jA[\hdelta,(1)^N]\cong {\jA}[(0)^{-N},\reflection(\hdelta)].\]

\begin{question}
Is $Spec\, \jA[\hdelta,(1)^N]$ reduced and irreducible?
\end{question}

The algebra $\C[z]/z^{N+1}$ admits a deformation $\C[z,q]/(z^{N+1}-q^{N+1})$. Since $\Arc^N$ is\\ $\Maps(Spec\,\C[z]/z^{N+1},\cone)$, the space $\Arc^N$ admits a "deformation" in the form \[\Arc_q^N=\Maps(Spec\,\C[z,q]/(z^{N+1}-q^{N+1}),\cone).\] Note that $Spec\, \C[z,q,q^{-1}]/(z^{N+1}-q^{N+1})$ is a union of $N$ components. This is why \[\Maps(Spec\, \C[z,q,q^{-1}]/(z^{N+1}-q^{N+1}),\cone)=\cone^N\times \C^{\times}.\]

The algebra $\tjA_q^{N}=\O[\Arc_q^{N}]$ is isomorphic to the quotient of 
$\mathbb{C} [(0)^0,(1)^N]\otimes \C[q]$ by the ideal of relations $\fq_q[(0)^0,(1)^N]$ generated by 
\[\Gamma^{\bm{s}^k}+q^{N+1}\Gamma^{\bm{s}^{N+1+k}} 0 \leq k\leq N-1 \text{ and } \Gamma^{\bm{s}^N}\quad \bm{s}\in G.\]
From the geometric considerations I know that the algebra $\tjA_q^{N}\otimes \C[q,q^{-1}]$ is isomorphic to $\O[\cone]^{\otimes N}\otimes \C[q,q^{-1}]$. 
\begin{conjecture}
\begin{enumerate}
\item $\tjA_q^{N}$ is flat over $\C[q]$.
\item $\tjA_q^{N}$ is a Cohen-Macaulay integral domain over $\C[q]$.
\end{enumerate}
An immediate corollary from the previous conjecture is the formula for the Hilbert series of $\O[\Arc_0^{N}(\cone)]$:
\end{conjecture}
 \begin{corollary}{\ \rm
\begin{enumerate}
\item \[\O[\Arc_0^{N}(\cone)](t)=\frac{(1+5t+5t^2+t^3)^N}{(1-t)^{11N}}.\]
\item $\O[\Arc_0^{N}(\cone)]$ is Gorenstein.
\item The index of the dualizing line bundle on $Proj\, \O[\Arc_0^{N}(\cone)]$ is equal to $-(11N-3N)=-8N$.
\end{enumerate}
}\end{corollary}

I have only a small number of computations with {\it Macaulay2} to support the following conjecture.
\begin{conjecture}
\begin{enumerate}
\item More generally $\jA[\hdelta,(1)^N]$ is Cohen-Macaulay.
\item $\jA[\hdelta,(1)^N]$ is Gorenstein iff $\hA[\hdelta,(1)^N]$ is.
\end{enumerate}
\end{conjecture}

\begin{remark}
\begin{enumerate}
\item The generating series from the conjecture have been checked with {\it Macaulay2} for $N=0,1,2,3$.
\item In the order (\ref{E:orderaf}), defining relations of the ideal $\fq[\hdelta,(1)^N]$ do not define a quadratic Gr\"{o}bner basis. It means that most likely $\jA[\hdelta,(1)^N]$ is not Koszul.
\end{enumerate}
\end{remark}

The map of algebras
\[\jA[0^{N},(1)^{N'}]\to \hA[0^{N},(1)^{N'}]\]
induces a map of the minimal $P=\C[0^{N},(1)^{N'}]$-resolutions
\[{G}_{\bullet}[0^{-N},\hdelta]\to F_{\bullet}[(0)^{-N},\hdelta].\]
Dualizing the map and by using the conjectural Gorenstein property of $\jA$, I get an element
\[\Th_{\hA\leftarrow {\jA}}\in \Ext_P^{5(N'-N)}({\hA},\jA).\]
The number $5(N'-N)$ is the codimension of $Spec\,\hA[0^{N},(1)^{N'}]$ in $Spec\,{\jA}[0^{N},(1)^{N'}]$, which conjecturally has dimension $11(N'-N+1)$.
If $N<0<N'$, I can use $\Th_{\hA\leftarrow \jA}$ to construct the map of local cohomology
\begin{equation}\label{E:mapslochilb}
H_{\fa}^i[0^{N},(1)^{N'}]\to H_{\fa}^{i+5(N'-N)}({\jA}[0^{N},(1)^{N'}]).
\end{equation}
As usual, $\fa=\fri[0^{N},(1)^{-1}]$.

The next three conjectures are the central statements of this section.
\begin{conjecture}
\begin{enumerate}
\item The cohomology groups $H_{\fa}^{i}({\jA}[0^{N},(1)^{N'}])$ for different $N,N'$ are related by homomorphisms $\Th$ as in Proposition \ref{P:commutativity}.
\item They satisfy an analogue of the compatibility relation (\ref{E:diagramindprob}) that would enable me to define the limiting groups
$H_{\fa}^{i+\frac{\infty}{2}}(\hA_f)$.
\item The map (\ref{E:mapslochilb}) could be extended to 
\[H_{\fa}^{i+\frac{\infty}{2}}(\hA)\to H_{\fa}^{i+\frac{\infty}{2}}(\hA_f).\]
\end{enumerate}
\end{conjecture}
\begin{conjecture}
In a series of works (\cite{GMSch},\cite{MalikovSchechtmanVaintrob},\cite{MalikovSchechtman}, \cite{GerbGMsch}) the authors have studied geometrically defined vertex operator algebras. These are cohomology of certain sheaves of vertex operator algebras. I conjecture that their groups defined for the smooth part of the cone $\cone$ coincide with $H_{\fa}^{i+\frac{\infty}{2}}(\hA_f)$.
\end{conjecture}

Let $\fri_k$ be the ideal generated by $\lambda^{\rbeta^s}, s=0,\dots,k-1, \rbeta\in \setE$. It defines a closed subscheme of $
\Arc^{N}$ isomorphic to $\Arc^{N-k}$. 

Let $\fj_k$ be the ideal generated by $\lambda^{\rbeta^s}, s=k,\dots,N, \rbeta\in \setE$. It defines a closed subscheme of $
\Arc^{N}$ isomorphic to $\cZ(0,k-1)
$ when $k<N/2$. 

\begin{conjecture}
There is a perfect pairing
\[ H^{i}_{\fri_k}(\tjA^N)\otimes H^{-i+8 N+16}_{\fj_k}(\tjA^N)\to \C.\]
\end{conjecture}

\section{Final comments}
In this paper, I outlined one of the approaches to local cohomology of ind scheme $\Maps_{poly}(\C^{\times},\cone)$. There were other attempts to define local cohomology along a different route. Let us consider pair of Lie algebras $M=L(-\infty,\infty), K=L((0)^0,\infty)$ form Section \ref{S:Kosdual}. One can try to define local cohomology $H_{\fa}^{i+\itwo}$ as the semi-infinite cohomology \cite{Feigin} of the pair $K\subset M$. Indeed, this approach demonstrated its utility in the case when $\tspace$ is a quadric \cite{AA},\cite{Movq}. The drawback of the method is the potential problem with the differential in the semi-infinite complex. Feigin's definition heavily uses the fact his Lie algebra has a finite dimension over $\C[z,z^{-1}]$. This condition is satisfied for the pair $K,M$ corresponding to a quadric or any complete intersection. The space $\cone$ is not a complete intersection. Because of that there are serious difficulties with convergence in the formulas that define the differential in the semi-infinite complex attached to $\cone$. It worthwhile to point out that the methods of this paper work for quadrics giving the answer consistent with \cite{AA},\cite{Movq}.

\bigskip
 \bigskip
 {\bf \Huge Appendix}
 \appendix
\section{Algebraic and sheaf-theoretic definition of local cohomology}\label{S:local}
\paragraph{Review of definitions of local cohomology}If $\hRo$ is a commutative algebra over $\C$, $\fa$ an ideal of $\hRo$, and $N$ an $\hRo$-module, then following \cite{ComEisenbud} \cite{Twenty} I define the zeroeth local cohomology $H^0_{\fa}(N)$ module of $M$ with
supports in $\fa$ to be the set of all elements of $N$ which are annihilated
by some power of $\fa$:
\[\Gamma_{\fa}(N)= \underset{n\rightarrow \infty}\lim \Hom({\hRo}/\fa^n, N).\]
The higher local cohomology $H_{\fa}^i(N)$ 
 is the $i$-th cohomology module of the complex obtained
by applying $\Gamma_{\fa}(N)$ to an injective resolution of $I^{\bullet}(N)$:
\[H_{\fa}^i(N):=H^i(\Gamma_{\fa}(I(N))).\]

There is a natural isomorphism:
\begin{equation}\label{E:loccohextdef}
H_{\fa}^i(N)=\underset{n\rightarrow \infty}\lim \Ext^i({\hRo}/\fa^n, N).
\end{equation}

According to \cite{Twenty} p.84 (see also \cite{ComEisenbud} Th A4.1,\cite{BrodmannSharp} Section 20.3), if $Y$ is an affine scheme of finite type with the algebra of functions $\O[Y]$ and $Z$ is a closed subscheme with the defining ideal $\fa$, then 
there is a natural exact sequence
\[0 \rightarrow H^{0}_{\fa}( \O[Y] ) \rightarrow \O[Y] \rightarrow \rH^{0}(Y\backslash Z,\O)
 \rightarrow H^1_{\fa}(\O[Y])\rightarrow 0\]
 and isomorphisms
\[H_{\fa}^i(\O[Y])=\rH^{i-1}(Y\backslash Z,\O), i\geq 2.\]
$\rH^{i}$ stands for the sheaf-theoretic cohomology.

\begin{remark}
By functoriality the symmetries of the pair $Z\subset Y$ act on $\rH^{i}(Y\backslash Z,\O)$. In addition to this, Lie algebra $\g$ of the connected component of the group of symmetries of $ Y$ act on $\rH^{i}(Y\backslash Z,\O)$ \cite{Kempf}. Note that exponents of some $l\in \g$ in this case don't have to preserve $Z$.
\end{remark}

 There is yet another equivalent definition of local cohomology that uses Koszul complexes. 
Before I give this definition I will describe two essentially isomorphic versions of the Koszul complex. They differ by multiplicative structure and gradation, which explains why I don't identify them.
The first, which will be denoted by $K$, is used in computation of local cohomology. The second, which will be denoted by $B$, is used constructing standard resolutions and computing $\Ext$ and $\Tor$ functors.

Let $N$ be a module over the ring $\hRo$. Let $\bx=\{x_1,\dots,x_s\}\subset \hRo$ be a sequence of elements. Koszul complex $K(N,\bx)$ is the tensor product 
\[K(N,\bx)=N\otimes \Lambda[\xi^1,\dots,\xi^s]=\bigoplus _{i=0}^sN\otimes \Lambda^i.\]
 The differential 
is defined as a multiplication on the element
\[\sum_{i=1}^sx_i \xi^i.\]

The other, the homological version of the Koszul complex is
\begin{equation}\label{E:kber}
(B(N,\bx),\sd)=(N\otimes \Lambda[\eta_1,\dots,\eta_s],\sd).
\end{equation} Its differential can be concisely written by using superalgebra notations
\[\sd=\sum_{i=1}^sx_i \frac{\sd}{\sd \eta_i}.\]
It is obvious that 
\[H^i(K)=H_{s-i}(B),\]
provided we made an identification of the linear spaces spanned by $\xi^i$ and $\eta_i$. 
\begin{convention}\label{C:simpl} {\rm
Let $M$ be a graded module over a graded algebra {\hRo}.
We will encounter a situation when a standard basis $\bx$ for ${\hRo}_1$ is labeled by a subset $\mathrm{X}\subset \sethE$. In this case, $B(M,\mathrm{X})$ and $K(M,\mathrm{X})$ will be abbreviations for $B(M,\{x_i|i\in \mathrm{X} \})$, $K(M,\{x_i|i\in \mathrm{X} \})$ respectively.
 To reduce clutter in notations , I will denote occasionally $B(M,\bx)$ and $K(M,\bx)$ by $B(M)$ and $K(M)$.
}\end{convention} 

There is a natural map of complexes
\begin{equation}\label{E:directlimitkoszul}
K(N,\bx^n):=K(N,x^n_1,\dots, x^n_s)\overset{s_n}{\longrightarrow} K(N,x^{n+1}_1,\dots, x^{n+1}_s)=:K(N,\bx^{n+1}).
\end{equation}
In degree $1$ it is given by the map $f: R^s \rightarrow R^s$ multiplying the $i$-th
component by $x_i$. In degree $d$ it is exterior power of $f$ $\wedge^df$, which acts by multiplying a
basis vector $\xi_{i_1}\wedge\cdots \wedge \xi_{i_d}$ by $x_{i_1}\cdots x_{i_d}$. Thus I may take the limit in each of the Koszul homology groups. 

Here is the description of the precise relation of local cohomology groups and Koszul cohomology.
\begin{proposition}\label{P:locKoszul} (\cite{ComEisenbud} Appendix A4 and in \cite{Twenty} p.81 )
Suppose $\fa\subset R$ is an ideal generated by $x_1,\dots,x_s$, then
\begin{equation}\label{E:locKoszul}
H_{\fa}^i(N)=\lim_{\overset{n}{\longrightarrow}}H^i(K(N,\bx^n)).
\end{equation}
\end{proposition}
\begin{remark}\label{E:costrdef}
Formula (\ref{E:locKoszul}) can be used as a constructive definition of local cohomology.
\end{remark}
\paragraph{Elementary examples of computations of local cohomology}
\begin{example}\label{E:simplestexample}
$K(\C[x],x^n)$ coincides with the two term complex $\C[x]\to\C[x], a\to x^na$. Obviously the only nontrivial cohomology
is in degree one: $H^1=\C[x]/(x^n).$ The explicit form of the maps in the direct limit (\ref{E:locKoszul}) is \[\C[x]/(x^n)\to \C[x]/(x^{n+1}),\quad a\to xa.\] The inclusion of $\C[x]/(x^n)$ into $x^{-1}\C[x^{-1}]:=\C[x,x^{-1}]/\C[x] ,$ which is given by the formula
\begin{equation}\label{E:quasi1}
a\mod x^n\to \frac{a}{x^n}\in\C[x,x^{-1}]/\C[x],
\end{equation}
identifies the later quotient with the direct limit $\underset{\to}{\lim}\C[x]/(x^n)$.
\end{example}
\begin{example}\label{E:polyloc}{\rm Let $P=\C[x_1,\dots,x_s]$, $N=P$ and $\fm=(x_1,\dots,x_s)$. By the K\"{u}nneth theorem, 
 the only nonzero local cohomology group 
is 
\begin{equation}\label{E:loccohaffine}
H_{\fm}^s(P)\cong \bigotimes_{i=1}^s x_i^{-1}\C[x_i^{-1}] = x_1^{-1}\cdots x_s^{-1} \C[x_1^{-1},\dots, x_s^{-1}].
\end{equation}
}
\end{example}

\paragraph{Koszul bicomplexes} One of the standard techniques of computing local cohomology with Koszul complexes relies on spectral sequences.
 Koszul complex can be broken down onto a bicomplex. Calculations with the spectral sequence of the bicomplex could simplify significantly computations of the cohomology of the original complex. Here are some notations used in the description of this method.
\begin{construction}\label{C:bicomplex}{\rm 
Fix a decomposition of the set 
\[\{x_i\}=\{x_{1},\dots,x_{k}\}\cup \{x_{k+1},\dots,x_{s}\}\]
of generators of an ideal $\fc\subset R$.
Define bicomplexes 
\[\begin{split}
&K^{i,j}=N\otimes \Lambda^i[\xi^1,\dots,\xi^k]\otimes \Lambda^j[\xi^{k+1},\dots,\xi^s],
\quad d_1(a)=\sum_{i=1}^kx_i \xi^ia, d_2(a)=\sum_{i=k+1}^sx_i \xi^ia
\end{split},\]
\[\begin{split}
&B_{i,j}=N\otimes \Lambda^i[\eta_1,\dots,\eta_k]\otimes \Lambda^j[\eta_{k+1},\dots,\eta_s], \quad \sd_1(a)=\sum_{i=1}^kx_i \frac{\sd a}{\sd \eta_i}, d_2(a)=\sum_{i=k+1}^sx_i \frac{\sd a}{\sd \eta_i}
\end{split}.\]
Their diagonal complexes coincide with $K^{\bullet}$ and $B_{\bullet}$ respectively.
}
\end{construction}

\paragraph{Some free $\C[x^1,\dots,x^n]$-modules}
 
In this paragraph will construct some free $\C[x^1,\dots,x^n]$-modules. 
A free $\C[x^1,\dots,x^n]$-modules $N$ has a property that the sequence $\{x^1,\dots,x^n\}$ is $N$-regular, which makes verification of regularity a trivial task.
We will use this observation in the main text. 
 
A polynomial algebra $H=\C[x^1,\dots,x^n]$ is a Hopf algebra. The diagonal $\Delta:H\to H\otimes H$ and the antipode $\iota:H\to H$ are
\[\Delta(x_i)=x^i\otimes 1+1\otimes x^i, \iota(x^i)=-x^i.\]
If $M,N$ are $H$-modules, then $M\otimes N$ is also an $H$-module with respect to three module structures: $T_0(a)(m\otimes n):=\sum a'_km\otimes a''_kn,\Delta(a)=\sum a'_k\otimes a''_k $ or $T_1(a)(m\otimes n):=\sum \iota(a'_k)m\otimes a''_kn$ or $T_2(a)(m\otimes n):=\sum a'_km\otimes \iota(a''_k)n$.
\begin{lemma}\label{L:stopiat}
Fix two $H$-modules $M$ and $N$.
\begin{enumerate}
\item If $M$ is free, then $M\otimes N$ is also free with respect to all the three structures.
\item $B(N\otimes H,(\underline{x}^i))$ is a free $H$ resolution of $N$. 
\begin{equation}\label{E:underlinenotation}
\begin{split}
&\underline{x}^i=x^i\otimes 1-1\otimes x^i,\\
&\ushortdw{x}^i=x^i\otimes 1+1\otimes x^i.\\
\end{split}
\end{equation}
\item Suppose $H_1=\C[x^1,\dots,x^n], H_2= \C[y_1,\dots,y_{n'}]$ and $H=H_1\otimes H_2$. If graded $H$-modules $M, N$ have the property
that $M=H_1\otimes M_2, M_2$ is an $H_2$-module and $\left.N\right|_{H_2}$ is free. Then $M\otimes N$ is free over $H$.
In particular, $({\underline{x}}_i,{\underline{y}}_j)$ is regular.
\end{enumerate}
\end{lemma}
\begin{proof}
\begin{enumerate}
\item I prove for the structure $T_0$ only. It suffice to prove the statement when $M=H$. The linear space $L=H\otimes N$ has two $H$-module structures. The diagonal structure $L_{\Delta}$ was described above. The other structure $L_{free}$ is $a(h\otimes n):=(ah)\otimes n$. $L_{free}$ is a free module. The isomorphism $\sff_H:L_{\Delta}\to L_{free}$ between $L_{\Delta}$ and $L_{free}$ maps $h\otimes n$ to
$\sum h_k\otimes \iota(h'_k) n$. The inverse is $\sff^{-1}_H(h\otimes n)=\sum h_k\otimes h'_k n$. 
\item By the first item $N\otimes H$ is $H$-free and $(\underline{x}^i)$ is regular. Thus $B(N\otimes H,(\underline{x}^i))$ has no higher cohomology and the zero cohomology is equal to $N$.
\item By the second item $B(N\otimes M,(\underline{x}^i,{\underline{y}}^j))$ computes $\Tor_i^H(M,N)$. Note that $H$ is flat over $H_2$ and 
$M=H\underset{H_2}{\otimes}M_1$. Thus
 \[\Tor_i^H(M,N)=\Tor_i^H(H\underset{H_2}{\otimes}M_2,N)\cong \Tor_i^{H_2}(M_2,N).\]
 $N$ is $H_2$-free. $\Tor_i^{H_2}(M_2,N)=0,i\geq 1$. This implies that $N\otimes M$ is $H$-free and $(\underline{x}^i,{\underline{y}}^j)$ is regular.
 \end{enumerate}
\end{proof}

\section{Generalities about Hibi rings}\label{A:hibi}
For more details on Hibi rings the reader can consult the survey \cite{VEne} or the original article \cite{THibi}.

Let $\setP$ be a poset. Element $y$ covers $x$ if $y > x$ and there is no
$z \in P$ with $y > z > x$. In this case we write $y\gtrdot x$.

The definition of Hasse diagram uses only $y\gtrdot x$ relation (see \cite{Birkhoff} p.5 for details). Here is an example of the Hasse diagram of a poset $\setP$ with $5$ elements, $x,y,z,t,u$ 
with $z\gtrdot x, z\gtrdot y, t\gtrdot y, t\gtrdot x, u\gtrdot t.$

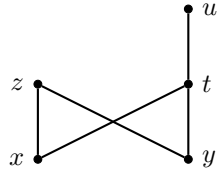
\begin{center}
\begin{tikzpicture}
	\begin{pgfonlayer}{nodelayer}
		\node [circle, fill=black, draw, style=none,label=180:$z$ ] (0) at (-3, 1) {.};
		\node [circle, fill=black, draw, style=none,label=180:$x$] (1) at (-3, 0) {.};
		\node [circle, fill=black, draw, style=none,label=0:$y$] (2) at (-1, 0) {.};
		\node [circle, fill=black, draw, style=none,label=0:$t$] (3) at (-1, 1) {.};
		\node [circle, fill=black, draw, style=none,label=0:$u$] (4) at (-1, 2) {.};
	\end{pgfonlayer}
	\begin{pgfonlayer}{edgelayer}
		\draw[thick] (0) to (1);
		\draw[thick] (0) to (2);
		\draw[thick] (3) to (2);
		\draw[thick] (4) to (3);
		\draw[thick] (1) to (3);
	\end{pgfonlayer}
\end{tikzpicture}
\end{center}

\begin{definition}\label{D:gradedlattice}{\em
A graded poset is a partially ordered set (poset) $\setP$ equipped with a rank function $\rho$ from $\setP$ to $\Z_{\geq0}$ satisfying the following two properties:
\begin{enumerate}
\item The rank function is compatible with the ordering, meaning that for every $x$ and y in the order with $x < y$, it must be the case that $\rho(x) < \rho(y)$, and
\item 
If $x \lessdot y$, then $\rho(y) = \rho(x) + 1$.
\end{enumerate}
}\end{definition}
\begin{example}\label{E:Egrposet}
$\sethE$ is a graded poset:

 the function $\rho:\sethE\rightarrow \Z$ in this case is defined by the rule
\begin{equation}\label{E:LdefE}
\begin{split}
&\rho((0)^r)=\rho((3)^{r-1})=8r\\
&\rho((12)^r)=\rho((2)^{r-1})=8r+1 \\
&\rho((13)^r)=\rho((1)^{r-1})=8r+2 \\
&\rho((14)^r)=\rho((23)^r)=8r+3 \\
&\rho((15)^r)=\rho((24)^r)=8r+4 \\
&\rho((25)^r)=\rho((34)^r)=8r+5\\
&\rho((35)^r)=\rho((5)^r)=8r+6 \\
&\rho((45)^r)=\rho((4)^r)=8r+7.
\end{split}
\end{equation}
\end{example}

\begin{example}\label{E:Bgraded}
$\rB[-\infty,\infty]$ (see page \pageref{E:goreneq}) is a graded poset:

the function $\rho_{\rB[-\infty,\infty]}:\rB[-\infty,\infty]\rightarrow \Z$ is 
\[\begin{split}
&\rho((14)^r)=\rho((1)^{r-1})=4r+1 \\
&\rho((0)^r)=\rho((2)^{r-1})=4r\\
&\rho((45)^{r-1})=\rho((5)^{r-1})=4r-1 \\
&\rho((15)^{r-1})=\rho((34)^{r-1})=4r-2.
\end{split}\]
\end{example}

 Let
\[\setC : x'=x_1 < \dots< x_t=x\]
be a chain in $\setP$ (i.e. a totally ordered subset of $\setP$). 
I say that $\setC$ is a chain descending from $x'$
. The length of $\setC$ will be 
\begin{equation}\label{E:lenght}
|\setC|-1.
\end{equation} The rank of a poset $\setP$, denoted by $\rk(\setP)$, is the supremum of the lengths of all chains
contained in $\setP$. A poset is called pure if all maximal chains have the same length. The {\it height}
of an element $\alpha\in \setP$, denoted $\rho_{\setP}(x)$ is:
\[\rho_{\setP}(x) = \sup\{ \mbox{ length of chains descending from } x\}.\]

 Let $\setP$ be a poset and $x,y\in \setP.$ An {\em upper bound} of $x,y$ is an element $z\in \setP$ such that 
$z\geq x$ and $z\geq y$. If the set $\{z\in \setP: z \text{ is an upper bound of } x \text{ and }y\}$ has the least element, is called the {\em join} of $x$ and $y$, and it is denoted $x\vee y.$ By duality, one defines 
the {\em meet} $x\wedge y$ of two elements $x,y$ in a poset.

\begin{definition}
{\em
Let $\setL$ be a lattice. $\setL$ is called {\em distributive} if satisfies one of the equivalent conditions:
\begin{itemize}
	\item [(i)] for any $x,y,z\in \setL,$ $x\vee(y\wedge z)=(x\vee y)\wedge(x\vee z)$;
	\item [(ii)] for any $x,y,z\in \setL,$ $x\wedge(y\vee z)=(x\wedge y)\vee(x\wedge z).$
\end{itemize}
}
\end{definition} 

A subset $\alpha$ of a poset $\setP$ is called an {\em order ideal} or {\em poset ideal} if it satisfies the following condition: 
for any $x\in \alpha$ and $y\in \setP,$ if $y\leq x,$, then $y\in \alpha.$ The set of all order ideals of $\setP$ is denoted $\MI(\setP).$
The union and intersection of two order ideals are obviously order ideals. Therefore, $\MI(\setP)$ is a distributive lattice with the union and intersection. 

Given a lattice $\setL,$ an element $x\in \setL$ is called {\em join-irreducible} if $x\neq \min \setL$ and whenever $x=y\vee z$ for some 
$y,z\in \setL$, we have either $x=y$ or $x=z.$

\begin{theorem}[Birkhoff]
Let $\setL$ be a distributive lattice and $\setP=\rB(\setL)$ its subposet of join-irreducible elements. Then $\setL$ is isomorphic to $\MI(\setP).$
\end{theorem}

Let $\setL$ be a lattice. The generators $\lambda^{\alpha}$ of the algebra $Hibi(\setL)$ are labeled by the vertices of the lattice $\setL$ and the straightening
laws are the so called Hibi relations:
\begin{equation}\label{E:Hibi}
\lambda^{\alpha}\lambda^{\beta} = \lambda^{\alpha\wedge \beta}\lambda^{\alpha \vee \beta}, \forall \alpha,\beta \in \setL\left| \alpha \nless \beta \text{ and } \alpha \ngtr \beta\right. .
\end{equation}

\begin{remark}\label{R:hibi}{\rm
Fix a distributed lattice $\setL=\MI(\setP)$.
\begin{enumerate}
\item (\cite{THibi}) Hibi rings are toric algebras with straightening laws, thus normal Cohen-Macaulay domains.
\item Krull dimension $\dim( Hibi(\setL)) $ is equal to $ \rk \setL + 1$.
\item (\cite{THibi}) $Hibi(\setL)=Hibi(\MI(\setP))$ is Gorenstein if and only if $\setP$ is pure.
\end{enumerate}
}\end{remark}

\section{General constructions with D-modules}\label{S:Dmodulenoteconstr}
\paragraph{Weyl algebras} Recall that the Weyl algebra $W(U, \omega)$ on a finite-dimensional symplectic linear space $(U,\omega)$ is the quotient 
\[\bigoplus_{i\geq 0}U^{\otimes i}/([u,u']-\omega(u,u')).\]
When an orthogonal sum decomposition of the symplectic space $U=U_1\oplus U_2$ is given, there is a canonical isomorphism of algebras
$W(U)\cong W(U_1)\otimes W(U_2).$

{\it Symplectization} of a finite-dimensional vector space $V$ is a direct sum $U(V)=V+V^*.$ It carries a symplectic form $\omega(v+v^*,v'+{v^*}')=<{v^*}',v>-<v^*,v'>.$
$D(V):=W(U(V))$ is the algebra of (polynomial) differential operators. The direct sum decomposition $V=V_1+V_2$ leads to orthogonal decomposition of the $U(V)$ and an isomorphism $D(V)\cong D(V_1)\otimes D(V_2).$ In particular, $D(V_1)$ becomes a subalgebra of $D(V)$. Let $V$ be a linear space equipped with a nondegenerate symmetric inner product $(.,.)$. If $V'\subset V$ is not degenerate, i.e. contains no isotropic subspaces, then $V=V'+{V'}^{\perp}$ and $D(V')$ is a subalgebra in $D(V)$. 

 The algebra $D(V)$ is isomorphic as a linear space to the tensor product of two symmetric algebras 
 \begin{equation}\label{E:decomposition1}
 \Sym[V]\otimes \Sym[V^*]\cong \Sym[V^*]\otimes \Sym[V].
 \end{equation} I interpret the factor $\Sym[V^*]=\C[V]= \O[V]$ as the space of polynomial function on $V$. The factor $\Sym[V]=\C[V^*]$ corresponds to differential operators with constant coefficients.
\paragraph{Modules over the Weyl algebra}
There are two standard $D(V)$-modules. The first is the tautological representations of $D(V)$ in the polynomial algebra $\C[V]$.
The unit element $\varpi_{V}=1\in \C[V]$ is a cyclic generator that satisfies 
\[
v\varpi_{V}=0, \quad v\in V.
\]
Because of the tensor decomposition (\ref{E:decomposition1}), these are defining relations for $\C[V]$.

The module which I denote by $\C[V]^{-1}$ has a cyclic generator $\varpi_{V^*}$ (it is usually called Dirac $\delta$-function) that satisfies 
\begin{equation}\label{E:inverse}
v^*\varpi_{V^*}=0,v^*\in V^*.
\end{equation}
$\C[V]^{-1}$ can be interpreted as a space of algebraic distributions on $V$ with support at $0$.
This opens a way for constructing numerous examples of $D(V)$-modules. 

Fix a pair of embedding $\sfa:V\to V'\cong \Imm\, \sfa \oplus \Imm\, \sfa^{\perp}$, $\sfb:U\to U'\cong \Imm\, \sfb \oplus \Imm\, \sfb^{\perp}$.
In the presence of dot product, they define inclusions
\begin{equation}\label{E:elemmaps}
\begin{split}
& \sfa:\C[V]\to \C[V'],\quad x\to \sfa(x)\otimes 1,\\
&\sfb:\C[U]^{-1}\to \C[U']^{-1},\quad y\to \sfb(y)\otimes \varpi_{\Imm\, \sfb^{\perp}} \text{ and }
\end{split}
\end{equation}
\begin{equation}\label{E:inclusion1}
\sfa\otimes \sfb: \C[V]\otimes \C[U]^{-1}\to \C[V']\otimes \C[U']^{-1}.
\end{equation}
 
\paragraph{A model for $\C[V]^{-1}$}
 Suppose that $V$ is a one-dimensional space with a coordinate $x$.
 \[0\to \C[x]\to\C[x^{\pm 1}]\to \C[x^{\pm 1}]/\C[x]\to 0\] is a short exact sequence of $D[V]$-modules. I denote $\C[x^{\pm 1}]/\C[x]$ by $x^{-1}\C[x^{-1}]$. It is easy to see that $x^{-1}\C[x^{-1}]$ is cyclicly generated over $D[V]$ by $\frac{1}{x}$ and as linear spaces \[x^{-1}\C[x^{-1}]\cong \C[V]^{-1}.\] 

 More generally, if $V$ is spanned by the set $\setA=\{x_1,\dots,x_n\}$, then the module 
 \begin{equation}\label{E:algdistrnot}
 \C[\setA^{-1}]:=\bigotimes_{x_i\in A} x_i^{-1}\C[x_i^{-1}]\cong \C[V]^{-1}
 \end{equation}
 is cyclicly generated over $D[V]$ by $\varpi_{V}=\prod _{x_i\in \setA} x_i^{-1}$. The module $ \C[\setA^{-1}]$ has a $\C$ basis \\$\{x_{1}^{-1-k_1}\cdots x_{n}^{-1-k_n}|k_i\geq 0\}$. There is a product 
 
\begin{equation}\label{E:prodelem}
\begin{split}
& \C[\setA]\otimes \C[\setA^{-1}]\to \C[\setA^{-1}] \text{ and more generally}\\
& \left(\C[\setA\cup \setB]\otimes \C[\setC^{-1}]\right)\otimes \left(\C[\setA\cup \setC]\otimes \C[\setB^{-1}]\right)\cong \\
&\cong \C[\setA]\otimes \C[\setA]\otimes \C[\setB]\otimes \C[\setB^{-1}]\otimes \C[\setC]\otimes \C[\setC^{-1}]\to\\
&\to \C[\setA]\otimes \C[\setB^{-1}]\otimes \C[\setC^{-1}]\cong \C[\setA]\otimes \C[(\setB\cup \setC)^{-1}].
\end{split}
\end{equation}
 
 The Lie algebra of differential operators $D[V]$ contains a copy of $\gl_n$ generated by $x_i\sd_j$. It acts on the generator $x_{1}^{-1}\cdots x_{n}^{-1}$ of $\C[V]^{-1}$ by 
 \begin{equation}\label{E:trace}
 \begin{split}
 &a x_{1}^{-1}\cdots x_{n}^{-1}=-\tr(a)x_{1}^{-1}\cdots x_{n}^{-1}\\
 &a=\sum_{i,j=1}^na^i_jx_i\sd_j, \tr(a)=\sum_{i=1}^na^i_i.
 \end{split}
 \end{equation}
 
\section{Pairing between  limits} 
\paragraph{Pairing between direct and inverse systems}

Let $\rA$ be a partially ordered set.  I will be interested in a pairing between the direct system of complex linear spaces $\iota^{\halpha',\halpha}:F^{\halpha}\to F^{\halpha'}, \halpha\leq \halpha'$ and the  inverse system $\pi_{\halpha,\halpha'}:G_{\halpha'}\to G_{\halpha}, \halpha\leq \halpha'$ (see \cite{CieliebakFrauenfelder}, \cite{Weibel} and \cite{Boardman} for terminology and definitions).
I set \[F:=\underrightarrow{\lim} F^{\halpha},\quad G:=\underleftarrow{\lim} G_{\halpha}.\]
I assume that the pairing
\begin{equation}\label{E:pairingindiv}
F^{\halpha}\otimes G_{\halpha}\overset{(\cdot,\cdot)}{\longrightarrow} \C
\end{equation}
is compatible with $\pi$ and $\iota$:
\begin{equation}\label{E:innerproductcomp}
(\iota^{\halpha',\halpha}(f^{\halpha}),g_{\halpha'})=(f^{\halpha},\pi_{\halpha,\halpha'}g_{\halpha'}).
\end{equation}
\begin{lemma}\label{L:pairing1}
The pairing (\ref{E:pairingindiv}) between the direct  system $\{F^{\halpha}\}$ and  the inverse system $\{G_{\halpha}\}$ induces a pairing between limits
\[F\otimes G\to \C.\] 
\end{lemma}
\begin{proof}
By definition of the inverse limit $g\in G$ is represented by a sequence $g_{\halpha}\in G_{\halpha}$  which satisfies coherence condition  $\pi_{\halpha,\halpha'}g_{\halpha'}=g_{\halpha}$. The direct limit (see e.g. \cite{CieliebakFrauenfelder} p.11 for details) is equipped with   maps $\psi^{\alpha} :F^{\halpha}\to F$. For any $f\in F$ there is $\halpha$ and $f^{\halpha}$ such that $\psi^{\alpha}f^{\halpha}=f$. 
I define \[(f,g):=(f^{\halpha},g_{\halpha}).\] 
By universality of the direct limit 
\[f=\psi^{\halpha}f^{\halpha}=\psi^{\halpha'}f^{\halpha'}, \halpha\leq \halpha'\Rightarrow f^{\halpha'}=\iota^{\halpha'\halpha}f^{\halpha}\]
This  together with (\ref{E:innerproductcomp}) verifies correctness of the definition.
\end{proof}
\begin{proposition}\label{P:nondegeneratepairing}
Let us suppose that in addition to  assumptions  of Lemma \ref{L:pairing1} the pairings (\ref{E:pairingindiv}) are nondegenerate and $G_{\halpha}$ satisfies Mittag-Leffler condition and  the set $\rA$ is totally ordered.
Then the pairing (\ref{L:pairing1}) has trivial right and left kernels. 
\end{proposition}
\begin{proof}
Fix $f\in F$ such that $(f,g)=0$ for all $g\in G$. 
Choose $f^{\alpha}\in F^{\halpha}, \psi^{\halpha}=f$. I use $f^{\alpha}$ to define a map $f:G_{\halpha'}\to \C_{\halpha'}, \halpha\leq \halpha'$, where $\C_{\halpha}=\C$ is a constant inverse system. The map is defined by the formula
\[g_{\halpha'}\to (f^{\alpha},\pi_{\halpha,\halpha'}g_{\halpha'})\]
As $\rA$ is totally ordered \[\lim_{\substack{\longleftarrow\\ \rA^{\geq \halpha}}}G_{\hbeta}=\lim_{\substack{\longleftarrow\\ \rA}}G_{\hbeta}\]
The short exact sequence of inverse systems $\{0\}\to \Ker_{\halpha}\to G_{\halpha}\to \C_{\halpha}\to \{0\}$ induces a long exact sequence of limits:
\begin{equation}\label{E:seqlimits}
\{0\}\to \underleftarrow{\lim} \Ker_{\halpha}\to \underleftarrow{\lim}G_{\halpha}\overset{f}{\to} \C\to \underleftarrow{\lim}^1\Ker_{\halpha}\to \cdots
\end{equation}
Fix $\beta\geq \halpha$. Let $\hbeta'\geq \beta$ be the index such that $\Imm \pi_{\beta,\hgamma}=\Imm \pi_{\beta,\hgamma'}$ for $\hgamma,\hgamma'\geq \hbeta'$. Mittag-Leffler condition postulates existence of such $\hbeta'$. Pick $g_{\hgamma}\in \pi^{-1}_{\beta,\hgamma}(\Imm \pi_{\beta,\hgamma}\cap \Ker_{\beta})$. 
Equality (\ref{E:innerproductcomp})
 implies that $g_{\hgamma}\in \Ker_{\hgamma}$. I conclude that $\{\Ker_{\hgamma}\}$ satisfies Mittag-Leffler condition. It implies (see e.g.\cite{Weibel} Proposition 3.5.7 ) that $ \underleftarrow{\lim}^1\Ker_{\halpha}=\{0\}$ and the map $f$ (\ref{E:seqlimits}) is nonzero.
\end{proof} 
\subsection{Pairing between bidirect systems}\label{S:pairingbidirect}
I  consider a poset 
$\rA$ and
a double indexed family of linear spaces $G_{\halpha}^\hbeta$ with
$\halpha,\hbeta \in \rA$ . I suppose that for every $\hbeta_0 \in \rA$ $G_{\halpha}^{\hbeta_0}$ is an inverse system. Symmetrically $G_{\halpha_0}^{\hbeta}$ is a direct system for every $\halpha_0 \in \rA$. Maps $\iota$ and $\pi$ satisfy compatibility relation which says that  for every
$\halpha' \leq \halpha$ and $\hbeta \leq \hbeta'$ the square
\begin{equation}\label{pii}
\begin{xy}
 \xymatrix{
  G_{\halpha}^{\hbeta} \ar[r]^{\pi^{\hbeta}_{\halpha',\halpha}}
  \ar[d]_{\iota^{\hbeta,\hbeta'}_{\halpha}} & G^{\hbeta}_{\halpha'}
  \ar[d]^{\iota_{\halpha'}^{\hbeta',\hbeta}}\\
  G^{\hbeta'}_{\halpha} \ar[r]_{\pi_{\halpha',\halpha}^{\hbeta'}} & G^{\hbeta'}_{\halpha'}
 }
\end{xy}
\end{equation}
is commutative. Such double indexed family will be called a {\it bidirect system}.
I will use the following abbreviations: 
\[G^{\hbeta}=\lim_{\substack{\longleftarrow\\ \halpha}} G^{\hbeta}_{\halpha},\quad 
G_{\halpha}=\lim_{\substack{\longrightarrow\\ \hbeta}} G^{\hbeta}_{\halpha}\]
$G^{\hbeta}$ has a structure of a direct system  and $G_{\halpha}$ of an inverse  system. I will refer to  $G^{\hbeta}$ ($G_{\halpha}$) as a contraction of $G^{\hbeta}_{\halpha}$ over lower(upper) index.
Occasionally it will be convenient to think about $G^{\hbeta}$ as of bidirect system whose $\pi$-maps are identities. Likewise $G_{\halpha}$ is a bidirect system whose $\iota$ maps are identities.

 A pairing between two bidirect system $\big(F,\pi,\iota\big)$, $\big(G,\pi,\iota\big)$ is a linear map 
 \begin{equation}\label{E:pairinfbidirect}
 F_{\halpha}^{\hbeta}\otimes G_{\hbeta}^{\halpha}\to \C
 \end{equation}
  which satisfies
\begin{equation}\label{E:pairingcomp}
\begin{split}
&(\pi^\hbeta_{\halpha',\halpha}f_\halpha^\hbeta,g_\hbeta^{\halpha'})=(f_\halpha^\hbeta,\iota_\hbeta^{\halpha,\halpha'}g_\hbeta^{\halpha'}),\quad \halpha'\leq \halpha\\
&(\iota_\halpha^{\hbeta',\hbeta}f_\halpha^\hbeta,g_\hbeta^{\halpha'})=(f_\halpha^\hbeta,\pi^{\halpha'}_{\hbeta',\hbeta}g_\hbeta^{\halpha'}),\quad \hbeta\leq \hbeta'
\end{split}
\end{equation}

Construction from Lemma \ref{L:pairing1}  defines  pairings  
\[
F_\halpha\otimes G^\halpha\to \C,\quad F^\halpha\otimes G_\halpha\to \C
\]
 between  pairs of  systems
 $\big(F_\halpha,\pi\big)$,$\big(G^\halpha,\iota\big)$ and $\big(F^\halpha,\iota\big)$,$\big(G_\halpha,\pi\big)$. Applying Lemma \ref{L:pairing1}  one more time I get  pairings
 
 \begin{equation}\label{E:bilimitpairings}
 \begin{split}
& \underleftarrow{\lim}\underrightarrow{\lim}F\otimes \underrightarrow{\lim}\underleftarrow{\lim}G\overset{(\cdot,\cdot)_L}{\longrightarrow}\C,\quad
 \underrightarrow{\lim}\underleftarrow{\lim}F\otimes 
\underleftarrow{\lim}\underrightarrow{\lim}G\overset{(\cdot,\cdot)_R}{\longrightarrow}\C\\
 \end{split}
 \end{equation}
 
 There is a four-term exact sequence
 \begin{equation}\label{E:fourterm}
 \{0\}\to \Ker_{\hbeta}^{\halpha}\to G_{\hbeta}^{\halpha}\overset{\kappa}{\to} G_{\beta}\to \Coker_{\hbeta}^{\halpha}\to \{0\}
 \end{equation}
 Applying $\underleftarrow{\lim}$ to $\kappa$ I get  a map $\kappa:G^{\alpha}\to \underleftarrow{\lim}\underrightarrow{\lim}G$. If now I use  $\underrightarrow{\lim}$ I get a map
 \[\kappa: \underrightarrow{\lim} \underleftarrow{\lim}G\to \underleftarrow{\lim}\underrightarrow{\lim}G\]
 \begin{remark}
 All the above considerations go through if we reduce the range of the bi-index  in $G_{\hbeta}^{\halpha}$ to a subset $\{(\halpha,\hbeta)|\halpha\geq \hbeta\}\subset \rA\times \rA$.
 \end{remark}
 \begin{example}\label{EX:dirinve}
 A bidirect system of finite-dimensional linear spaces  is called {\it generic} if the ranks of all the structure maps are maximal. Suppose $\rA=\Z$.
 \begin{enumerate}
 \item 
 Let 
 \[H_{\hbeta}^{\halpha}=
 \begin{cases}
 \C& \hbeta\geq \halpha\\
 \{0\}&\hbeta<\halpha
 \end{cases}\]
 be a generic bidirect system.
  In this case $\underleftarrow{\lim}\underrightarrow{\lim}H\cong\{0\}, \underrightarrow{\lim}\underleftarrow{\lim}H\cong\C$.
  
  \item 
 Let  
 \[F_{\hbeta}^{\halpha}=
 \begin{cases}
 \{0\}& \hbeta\geq \halpha\\
 \C&\hbeta<\halpha
 \end{cases}\]
 be a generic bidirect system. In this case $\underleftarrow{\lim}\underrightarrow{\lim}H\cong \C, \underrightarrow{\lim}\underleftarrow{\lim}F\cong\{0\}$.
  
 \item
Let  
 \[G_{\hbeta}^{\halpha}=
 \begin{cases}
 \C& \hbeta\neq \halpha\\
 \{0\}&\hbeta=\halpha
 \end{cases}\]
 be a generic bidirect system.
 In this case $G_{\hbeta}\cong G^{\halpha}\cong \underleftarrow{\lim}\underrightarrow{\lim}G\cong \underrightarrow{\lim}\underleftarrow{\lim}G\cong\C$ but $\kappa=0$.
 \end{enumerate}
 \end{example}
 In my application I am mostly interested in systems such that 
 \begin{equation}\label{E:bound}
 \dim_{\C} G_{\beta}^{\alpha}\leq C
 \end{equation}
  for some constant $C$.
  \begin{proposition}
  Suppose that the bidirect system $G_{\beta}^{\alpha}$ satisfies (\ref{E:bound}). Then there is a short exact sequence
   \begin{equation}\label{E:fourtermhom}
 \{0\}\to  \underrightarrow{\lim} \underleftarrow{\lim} \Ker_{\beta}^{\alpha}\to \underrightarrow{\lim} \underleftarrow{\lim}G\ \overset{\kappa}{\to} \underleftarrow{\lim}\underrightarrow{\lim}G\ \to  \underrightarrow{\lim} \underleftarrow{\lim} \Coker_{\beta}^{\alpha}\to \{0\}
 \end{equation}
  \end{proposition} 
  \begin{proof}
Fix $\halpha$ in (\ref{E:fourterm}).   Mittag-Leffler condition for all inverse systems in (\ref{E:fourterm}) follows trivially from (\ref{E:bound}). If we use now $\underleftarrow{\lim}$ we will get an exact four term sequence 
\[\{0\}\to \Ker^{\alpha}\to G^{\alpha}\overset{\kappa}{\to} \underleftarrow{\lim}\underrightarrow{\lim}G \to \Coker^{\alpha}\to \{0\} \]
$\underrightarrow{\lim}$ is an exact functor in our context (\cite{Weibel} Theorem  2.6.15 ). This  verifies (\ref{E:fourtermhom})
  \end{proof}
  \begin{proposition}\label{P:kappaisomorphism}
  Let $G_{\hbeta}^{\halpha},\hbeta,\halpha\in \Z$ is  a bidirect system.
  \begin{enumerate}
  \item If $\iota_{\hbeta}^{\halpha',\halpha}$ are injective, then the map $\kappa$ is also injective.
   \item If $\iota_{\hbeta}^{\halpha',\halpha}$ $\pi^{\halpha}_{\hbeta,\hbeta'}$ are bijections, then the map $\kappa$ is an isomorphism.
  \end{enumerate}
  \end{proposition}
\begin{proof}
Let $g_{\hbeta}^{\halpha}$ be a coherent system of element $\halpha\geq \halpha_0$ that represents $g\in \underrightarrow{\lim}\underleftarrow{\lim}G$. I extend it to a system that is defined for all $\hbeta,\halpha$ such that $\halpha\geq \halpha_0$. As $\iota_{\hbeta}^{\halpha',\halpha}$ are embeddings the system is determined by a  subsystem defined for  $\{\halpha,\hbeta|\halpha\geq A(\hbeta)\}$. The function $A:\Z\to\Z$ satisfies  $A(\hbeta')>A(\hbeta)$ for $\hbeta'>\hbeta$. Such subsystem coincides with $\kappa g$.  In particular if such a subsystem  is zero then $g=0$.

The second statement is a special case \cite{FreiMacdonald} Theorem 5.6.
\end{proof}
 \begin{proposition}
 Suppose $\rA=\Z$. Pairing $(\kappa a,b)_L$ and $( a,\kappa b)_R$ coincide. They define a bilinear map
 \[\underrightarrow{\lim}\underleftarrow{\lim}F\otimes \underrightarrow{\lim}\underleftarrow{\lim}G\to\C\]
 \end{proposition}
\begin{proof}
The element $f\in \underrightarrow{\lim}\underleftarrow{\lim}F$ and $f'\in \underleftarrow{\lim}\underrightarrow{\lim}F$ is represented by coherent collections $\{f^\halpha_\hbeta\}$ and $\{f^{'\halpha}_\hbeta\}$ such that the  range of indices $\{(\halpha,\hbeta)\}$ belong to subsets $\supp f, \supp f'\subset\Z\times\Z$. 
\[\supp f =\{{\halpha,\hbeta}|\hbeta\geq B_f(\halpha),\halpha\geq \halpha_0\},\quad  B_f(\halpha')\geq B_f(\halpha) \text{ if } \halpha'\geq \halpha\] 
\[\supp f' =\{{\halpha,\hbeta}|\halpha\geq A_{f'}(\hbeta),\hbeta\geq \hbeta_0\},\quad  A_{f'}(\hbeta+1)\geq A_{f'}(\hbeta)\text{ if } \hbeta'\geq \hbeta\]
$A:\Z\to\Z,B:\Z\to\Z$ are some functions. One can use structure maps of the bidirect system to unambiguously extend  support of the collection. In particular support of $f$ can be extended to $\{{\halpha,\hbeta}|\halpha\geq A(\hbeta)= \halpha_0\}$. The procedure of extension coincides with the map $\kappa$. The pairing $(\kappa f,g)$ $( f,\kappa g)$ are computed by performing  the extension procedure on $\{f^\halpha_\hbeta\}$ or $\{g^\halpha_\hbeta\}$ respectively followed by coupling  $(f^\halpha_\hbeta,g^\hbeta_\halpha)$. Equations (\ref{E:pairingcomp}) imply that 
$(f^\halpha_\hbeta,g^\hbeta_\halpha)=(f^{\halpha'}_{\hbeta'},g^{\hbeta'}_{\halpha'})$ for $\halpha>\hbeta$ and $\halpha'<\hbeta'$.This is  the assertion we are trying to prove.
\end{proof} 
\begin{proposition}\label{P:abstrpairing}
  Suppose that the bidirect system $F^{\beta}_{\alpha}, G_{\beta}^{\alpha}$ satisfies (\ref{E:bound}), the pairing (\ref{E:pairinfbidirect}) is not degenerate, and $\rA$ is totally ordered. Then the pairing (\ref{E:bilimitpairings}) are not degenerate.
  \end{proposition}
\begin{proof}
I contract  $F^{\beta}_{\alpha},G_{\beta}^{\halpha}$ over index $\hbeta$. The result 
is a pair of systems with a pairing. It follows from finite-dimensionality (\ref{E:bound})  that inverse systems $G_{\beta}^{\alpha}$ satisfy Mittag-Leffler condition.  By Proposition \ref{P:nondegeneratepairing} the induced pairing $ F_{\halpha}\otimes G^{\halpha}\to\C$  is not degenerate. It follows from  (\ref{E:bound}) that $\dim^{\C}G_{\halpha}, \dim_{\C}F_{\halpha}\leq C$. I repeat the previous arguments with contraction one more time but now  with the index $\halpha$. This way I get nondegeneracy of the first  pairing  (\ref{E:bilimitpairings}).

Nondegeneracy of the other pairing in (\ref{E:bilimitpairings}) is proved similarly.
\end{proof}
\section{Some technical lemmas}\label{S:techlemmas}
\begin{definition}
{\em
For $\halpha\in \sethE$ define
\[\begin{split}
&\setCL^{\pm}(\halpha):=\{\hbeta \in \sethE| (\halpha,\hbeta)\text{ is a clutter and } \pm\rho (\hbeta)\leq \pm\rho (\halpha)\}.
\end{split}\]
}\end{definition}
\begin{definition}\label{D:Mpmidef0}{\em
Define the subsets of $\sethE$
\[\rM_1^{-}:=\left\{ \left.\halpha \in \sethE\right| |\setCL^{-}(\halpha)|=1\right\}=
\left\{ (13)^{r},(14)^{r},(24)^r,(34)^r,(35)^r,(45)^r,(3)^r,(2)^r| r\in \Z\right\},
\]
\[ \rM_2^{-}:=\left\{\left.\halpha \in \sethE\right| |\setCL^{-}(\halpha)|=2\right\}=\left\{(4)^r,(12)^r,(23)^r,(25)^r\right\},\]
\[ \rM_3^{-}:=\left\{\left.\halpha \in \sethE\right| |\setCL^{-}(\halpha)|=3\right\}=\left\{(0)^r,(1)^r,(15)^r,(5)^r\right\},\]
\[ \rM_1^{+}:=\left\{ \left.\halpha \in \sethE\right| |\setCL^{+}(\halpha)|=1\right\}=
\left\{(12)^r, (13)^r,(23)^r,(24)^{r},(25)^r,(35)^{r},(4)^{r}, (3)^{r}| r\in \Z\right\},
\]
\[ \rM_2^{+}:=\left\{\left.\halpha \in \sethE\right| |\setCL^{+}(\halpha)|=2\right\}=\left\{(14)^r,(2)^r,(45)^r,(34)^r\right\},\]
\[\rM_3^{+}:=\left\{\left.\halpha \in \sethE\right| |\setCL^{+}(\halpha)|=3\right\}=\left\{(0)^r,(1)^r,(15)^r,(5)^r\right\}.\]
}\end{definition}

\begin{lemma}\label{L:multstatM}

\begin{enumerate}
\item \label{I:first} $1\leq |\setCL^{\pm}(\halpha)|\leq 3,\forall \halpha\in \sethE$. The sets $\setCL^{\pm}(\halpha)$ are totally ordered.

\item \label{I:second} Let $\halpha:=\hgamma\wedge \hdelta'$. Then $\halpha$ doesn't depend on $\hgamma\in \setCL^{-}(\hdelta')$. In addition, $\halpha$ satisfies
$\halpha\lessdot \hdelta'$.

\item \label{I:thirdw} $\rho(\hgamma)\leq \rho(\setCL^{-}(\hgamma))$, $\rho(\hgamma)\geq \rho(\setCL^{+}(\hgamma))$. Equality is achieved on some element of $\setCL^{\pm}(\hgamma)$.

\item\label{I:sixmultstatM}
\[\sethE=\rM_1^{-}\sqcup\rM_2^{-}\sqcup\rM_3^{-}=\rM_1^{+}\sqcup\rM_2^{+}\sqcup\rM_3^{+}, \]
\[\rM:=\rM^{\pm}_1\cap \rM_1^{\mp}=\left\{(3)^r,(13)^r,(24)^r,(35)^r\right\} \neq\emptyset,\]
\[\rM\sqcup \rM^{\pm}_2= \rM^{\mp}_1,\]
\[ \rM_2^{-}\cap \rM_2^{+}=\emptyset,\]
\[\rM_3^{+}= \rM_3^{-}.\]
\item \label{I:m3}
\[\begin{split}
&\forall \halpha \in \rM_2^{-}\cup \rM_3^{-}\quad \exists!\hbeta \in \rM_1^{+}\text{ such that } \halpha\lessdot\hbeta,\\
&\forall \halpha \in \rM_2^{+}\cup \rM_3^{+}\quad \exists!\hbeta \in \rM_1^{-}\text{ such that } \halpha\gtrdot\hbeta.\\
\end{split}\]
\item \label{I:nine} If $\rho(\hgamma)=\rho(\hgamma'), \hgamma\neq \hgamma'$, then one of the element belongs to $\rM_1^{-}$ and the other to $\rM_2^{-}$ or one of the element belongs to $\rM_1^{-}$ and the other to $\rM_3^{-}$.
\item 
If $\halpha,\hbeta\in \rM_1^{+}$ and $\rho(\halpha)\equiv \rho(\hbeta)\mod 2$, then $|\setCL^{-}(\halpha)|=|\setCL^{-}(\hbeta)|$. 
\end{enumerate}
\end{lemma}
\begin{proof}
The proof is a straightforward verification with the diagram (\ref{P:kxccvgw13}).
\end{proof}

The next lemma is very technical but straightforward.
\begin{lemma}\label{L:degrel}
Fix 
$ \hdelta<\hbeta$ such that $\hdelta\in \rM=\rM_1^{-}\cap \rM_1^{+}$. 
\begin{enumerate}
\item There is a unique element $\hdelta'\in \rM_1^{+}$ such that $\hdelta\lessdot \hdelta'$.
\item $|\setCL^{-}(\hdelta')|=2$.
\item $\lambda^{\hdelta'}$ is nonzero in ${\hA(\hdelta,\hbeta]}$.
\item \label{I:four} There are $\hgamma<\hgamma'$ such that $\lambda^{\hgamma},\lambda^{\hgamma'}$ are nonzero in ${\hA(\hdelta,\hbeta]}$ and 
\[\lambda^{\hgamma}\lambda^{\hdelta'}=\lambda^{\hgamma}\lambda^{\hdelta'}=0. \]
\item \label{I:five} Elements $\lambda^{\hgamma},\lambda^{\hgamma'}$ belong to the kernel of the projection ${\hA(\hdelta,\hbeta]}\to {\hA[\hdelta',\hbeta]}$ which corresponds to inclusion $[\hdelta',\hbeta]\subset (\hdelta,\hbeta]$.
\item \label{I:six}$(\hdelta,\hbeta]=\setCL^{-}(\hdelta')\sqcup [\hdelta',\hbeta]$.
\item \label{I:seven}$(\hdelta,\hbeta]=\setCL^{-}(\hgamma)\sqcup [\hgamma,\hbeta]$, $\setCL^{-}(\hgamma)=\{\hdelta'\}$.
\end{enumerate} 
\end{lemma}
\begin{proof}
\begin{enumerate}
\item $\hdelta'\in \rM^{+}_1= \rM\sqcup \rM^{-}_2 $ (item (\ref{I:sixmultstatM}) Lemma \ref{L:multstatM}) exists because $\left(\rM_1^{+}\right)_{\hdelta<}$ is a well ordered set (see (\ref{E:Mdef})). 
\item The set of all such $\hdelta'$ (see the diagram (\ref{P:kxccvgw13})) is $\{(25)^r,(4)^r,(12)^r,(23)^r\}= \rM_2^{-}$.
\item $\lambda^{\hdelta'}$ is one of the standard monomials in ${\hA(\hdelta,\hbeta]}$ (see Proposition \ref{R:stbasis}).
\item 
 I choose $\hgamma,\hgamma'\in \setCL^{-}(\hdelta')$. By the item (\ref{I:first}) of Lemma \ref{L:multstatM}, $\setCL^{-}(\hdelta')$ is totally ordered. I can assume $\hgamma<\hgamma'$. By item (\ref{I:second}) of the same lemma, $\hdelta'\wedge\hgamma=\hdelta'\wedge \hgamma'=\halpha, \hdelta\leq \halpha\lessdot \hdelta'$. Since $\hdelta\lessdot \hdelta'$
\begin{equation}\label{E:mindg}
\hdelta'\wedge\hgamma=\hdelta'\wedge \hgamma'=\hdelta.
\end{equation}

Relation (\ref{E:reluniv}) in ${\hA[\hdelta,\hbeta]}$ associated with the clutters $(\hdelta',\hgamma)$ $(\hdelta',\hgamma')$ contains no terms under summation sign because 
\[\{\halpha\in [\hdelta,\hbeta] |\halpha < \hdelta'\wedge\hgamma=\hdelta\}=\{\halpha\in [\hdelta,\hbeta] | \halpha < \hdelta'\wedge\hgamma'=\hdelta\}=\emptyset.\] In ${\hA(\hdelta,\hbeta]}$ (\ref{E:semiinterval})
\[\lambda^{\hgamma}\lambda^{\hdelta'}=\pm \lambda^{\hdelta}\lambda^{\hdelta'\vee\hgamma}=0.\]
By the same reasons, $\lambda^{\hgamma'}\lambda^{\hdelta'}=0$.
\item ${\hA[\hdelta',\hbeta]}$ has no zero divisors \cite{MovStr}.
\item From (\ref{E:mindg}), I conclude that $\setCL^{-}(\hdelta')$ and $[\hdelta',\hbeta]$ are subsets of $ (\hdelta,\hbeta]$. If $\halpha\in (\hdelta,\hbeta]\backslash [\hdelta',\hbeta]$, then, by definition of $\setCL$, $\halpha\in \setCL^{-}(\hdelta')$.
\item By construction $\hdelta< \hgamma$. By item (\ref{I:thirdw}) Lemma \ref{L:multstatM}, $\rho(\hdelta')=\rho(\hgamma)$. By item (\ref{I:nine} ) Lemma \ref{L:multstatM}, $|\setCL^{-} (\hgamma)|=1$ and therefore $\setCL^{-}(\hgamma)=\{\hdelta'\}$ and $(\hdelta,\hbeta]=\setCL^{-}(\hgamma)\sqcup[\hgamma,\hbeta]=\{\hdelta'\}\sqcup[\hgamma,\hbeta]$.
\end{enumerate}
\end{proof}
\paragraph{The list of intervals $[\hdelta,\hdelta']$ with $\rCap[\hdelta,\hdelta']=2$}
\begin{equation}\label{E:twointervalhigh}
\begin{split}
&[\hdelta,(12)^r], \hdelta\in \{(35)^{r-1},(25)^{r-1},(15)^{r-1}\},\\
&[\hdelta,(4)^r], \hdelta\in \{(24)^{r},(23)^{r},(1)^{r-1}\},\\
&[\hdelta,(25)^r], \hdelta\in \{(13)^{r},(12)^{r},(0)^{r}\},\\
&[\hdelta,(23)^r], \hdelta\in \{(3)^{r-1},(4)^{r-1},(5)^{r-1}\},\\
&\text{In this group of intervals }\{[\hdelta,\hdelta']\}\quad \hdelta'\in \rM_2^{-}.
\end{split}
\end{equation}
\begin{equation}\label{E:twointervalslow}
\begin{split}
&[(2)^r,\hdelta'], \hdelta'\in \{(24)^{r+1},(34)^{r+1},(5)^{r+1}\},\\
&[(45)^r,\hdelta'], \hdelta'\in \{(13)^{r+1},(14)^{r+1},(15)^{r+1}\},\\
&[(34)^r,\hdelta'], \hdelta'\in \{(3)^{r},(2)^{r},(1)^{r}\},\\
&[(14)^r,\hdelta'], \hdelta'\in \{(35)^{r},(45)^{r},(0)^{r+1}\},\\
&\text{In this group of intervals }\{[\hdelta,\hdelta']\}\quad \hdelta\in \rM_2^{+}.\\
\end{split}
\end{equation}
\paragraph{Dualizing modules for non Gorenstein $\hA[\hdelta,\delta']$}
In some rare cases I have to deal with Cohen-Macaulay algebras $\hA[\hdelta,\hbeta]$, which are not Gorenstein. In particular I am interested in the cases 
\label{I:leftirreg}
\begin{equation}\label{I:leftirreg}
\hdelta\in \rM_2^+, \hbeta\in \rM_1^-\sqcup\rM_3^-,
\end{equation} 
\begin{equation} \label{I:rightirreg} 
 \hdelta\in \rM_1^+\sqcup\rM_3^+, \hbeta \in \rM_2^-.
\end{equation}
 Here is an explicit description of the dualizing modules $\omega[\hdelta,\hbeta]$:
\begin{proposition}\label{P:dualizingmoduledescr}
\begin{enumerate}
\item In case (\ref{I:leftirreg}) $\omega[\hdelta,\hbeta]\cong (\lambda^{\hgamma},\lambda^{\hgamma'})\subset \hA[\hdelta,\hbeta]$, where $\hgamma=\hdelta$. 

$\hgamma'$ is characterized by the condition $ \hdelta \lessdot \hgamma'\in \rM_3^+$
\item In case (\ref{I:rightirreg}) $\omega[\hdelta,\hbeta]\cong (\lambda^{\hgamma},\lambda^{\hgamma'})\subset \hA[\hdelta,\hbeta]$, where $\hgamma=\hbeta$. $\hgamma'$ is characterized by the condition $ \hbeta \gtrdot  \hgamma'\in \rM_3^-$
\end{enumerate}
\end{proposition}
\begin{proof}
I will only the first statement. By Lemma \ref{L:multstatM} item \ref{I:nine}
there is a unique $\hdelta'\in \rM_1^-, \rho(\hdelta')=\rho(\hdelta)$. I define $\hdelta'':=\hdelta\wedge \hdelta'\in \rM_1^+$. By Proposition \ref{C:dimalg}, and (\ref{E:gorimplication}) $\hA[\hdelta,\hbeta]$ and $\hA(\hdelta'',\hbeta]$ are Gorenstein. Straightening relations in $\hA(\hdelta'',\hbeta]$ imply that there is a short exact sequence 
\begin{equation}\label{E:shortexactCM}
\{0\}\to \hA[\hdelta',\hbeta]\overset{\lambda^{\hdelta}\times?}{\longrightarrow} \hA(\hdelta'',\hbeta]\to \hA[\hdelta,\hbeta]\to \{0\}
\end{equation}
$\C(\hdelta'',\hbeta]$-modules. it is explained in the proof of the second part of Lemma \ref{L:saturate0} that $\lambda^{\hgamma},\lambda^{\hgamma'}$ generate $\Ann(\hdelta')$.
The assertion follows from  Proposition \ref{P:extmaintheorem}.
\end{proof}
\begin{proposition}\label{E:CMcohcomp}

Suppose that  $[\hdelta,\hbeta]$ satisfies conditions (\ref{I:leftirreg}) or  (\ref{I:rightirreg}). Then
\begin{enumerate}
\item \label{I:CMcohcomp1} $H^i_{\fm}[\hdelta,\hbeta]=\{0\},i\neq \dim=\rho(\hbeta)-\rho(\hdelta)+1$. If $i= \dim$, then there is a short exact sequence
\begin{equation}\label{E:exactseqcoh3}
\{0\}\to H^{\dim}_{\fm}\hA[\hdelta',\hbeta]\overset{\lambda^{\hdelta}\times?}{\longrightarrow} H^{\dim}_{\fm}\hA(\hdelta'',\hbeta]\to H^{\dim}_{\fm}\hA[\hdelta,\hbeta]\to \{0\}
\end{equation}
\item \label{I:CMcohcomp2} $H^{i}_{\fm}\omega[\hdelta,\hbeta]=\{0\}, i\neq 3,\dim$. If $i=\dim$ then $H^{\dim}_{\fm}\omega[\hdelta,\hbeta]=H^{\dim}_{\fm}[\hdelta,\hbeta]$. If $i=3$, then
$H^{3}_{\fm}\omega[\hdelta,\hbeta]=H^{2}_{\fm}\C[\hgamma,\hgamma']$.
\item \label{I:CMcohcomp3} If (\ref{I:rightirreg} ) is satisfied, then $H^{i}_{\fa}((\hgamma,\hgamma'))\cong H^{i}_{\fa}\omega[\hdelta,\hbeta]\cong H^{i}_{\fa}[\hdelta,\hbeta], i\neq 0$.$H^{3}_{\fa}\omega[\hdelta,\hbeta]\cong H^{2}_{\fm}\C[\hgamma,\hgamma']$
\item \label{E:idealloccohcomp} If (\ref{I:leftirreg}) is satisfied, then $H^{i}_{\fb}((\hgamma,\hgamma'))\cong  H^{i}_{\fb}\omega[\hdelta,\hbeta]\cong H^{i}_{\fb}[\hdelta,\hbeta], i\neq 0$.$H^{3}_{\fb}\omega[\hdelta,\hbeta]\cong H^{2}_{\fm}\C[\hgamma,\hgamma']$
\end{enumerate}
\end{proposition}
\begin{proof}

The vanishing result from item \ref{I:CMcohcomp1} follows from Lemma \ref{L:regularity} and Th 16.6 \cite{Matsumura}. With a help of Theorem 7.11 \cite{Twenty} I identify (\ref{E:exactseqcoh3}) with  a segment of the long exact sequence of local cohomology corresponding to (\ref{E:shortexactCM}). Its exactness  follows from the vanishing result.

Item \ref{I:CMcohcomp2} follows from consideration of the short exact sequence 
\[\{0\}\to (\lambda^{\hgamma},\lambda^{\hgamma'})\to \hA[\hdelta,\hbeta]\to \C[\hgamma,\hgamma']\to \{0\}\]
and associated long exact sequence of local cohomology.

Last two items are proved similarly.
\end{proof}
\begin{proposition}
In the assumptions of Proposition \ref{P:dualizingmoduledescr}
\begin{enumerate}
\item  $H^{\dim}_{\fm}\omega[\hdelta,(1)^{-1}]=H^{\dim}_{\fm}[\hdelta,(1)^{-1}]$ is a positive energy $\bT$-space with weights bounded from below by  $u[\hdelta'',(1)^{-1}]$ (\ref{E:sigmaalt})
\item $H^{\dim}_{\fm}\omega[(0)^0, \hbeta]=H^{\dim}_{\fm}[(0)^0, \hbeta ]$ is a negative energy $\bT$-space with weights bounded from above by  $u[(0)^{0}, \hbeta'']$
\end{enumerate}
\end{proposition}
\begin{proof}
The proof follows from Lemma \ref{L:weights} item \ref{I:weights3} and Proposition \ref{E:CMcohcomp}.
\end{proof}


\begin{thebibliography}{10}

\bibitem{AA}
Y.~Aisaka and E.~A. Arroyo.
\newblock Hilbert space of curved $\beta\gamma$ systems on quadric cones.
\newblock {\em JHEP}, 0808(052), 2008.

\bibitem{AABN}
Y.~Aisaka, E.~A. Arroyo, N.~Berkovits, and N.~Nekrasov.
\newblock Pure spinor partition function and the massive superstring spectrum,
  2008.
\newblock arXiv:0806.0584v1 [hep-th].

\bibitem{AvramovGolod}
L.~L. Avramov and E.~S. Golod.
\newblock Homology algebra of the \uppercase{K}oszul complex of a local
  \uppercase{G}orenstein ring.
\newblock {\em Mathematical notes of the Academy of Sciences of the USSR},
  9(1):30--32, January 1971.

\bibitem{Bass}
H.~Bass.
\newblock On the ubiquity of \uppercase{G}orenstein rings.
\newblock {\em Math. Zeitschr}, 82:8--28, 1963.

\bibitem{BerNek}
N.~Berkovits and N~Nekrasov.
\newblock The character of pure spinors.
\newblock {\em Lett.Math.Phys.}, 74:75--109, 2005.

\bibitem{BerkovitsNekrasovMultiloop}
N.~Berkovits and N.~Nekrasov.
\newblock Multiloop \uppercase{S}uperstring \uppercase{A}mplitudes from
  \uppercase{N}on-\uppercase{M}inimal \uppercase{P}ure \uppercase{S}pinor
  \uppercase{F}ormalism.
\newblock {\em JHEP}, 0612(029), 2006.

\bibitem{Birkhoff}
G.~Birkhoff.
\newblock {\em Lattice theory}.
\newblock AMS, 2 edition, 1948.

\bibitem{Boardman}
J.~M. Boardman.
\newblock Conditionally convergent spectral sequences.
\newblock In {\em Homotopy invariant algebraic structures(Baltimore, MD,
  1998)}, volume 239 of {\em Contemp. Math.}, pages 49--84. Amer. Math. Soc.,
  Providence, RI,, 1999.

\bibitem{BF}
A.~Braverman and M.~Finkelberg.
\newblock Semi-infinite \uppercase{S}chubert varieties and quantum
  \uppercase{K}-theory of flag manifolds.
\newblock {\em Journal Of The American Mathematical Society}, 27(4):1147--1168,
  October 2014.

\bibitem{BrodmannSharp}
M.~Brodmann and R.~Sharp.
\newblock {\em Local cohomology: an algebraic introduction with geometric
  applications}, volume~60 of {\em Cambridge Studies in Advanced Mathematics}.
\newblock Cambridge University Press, Cambridge, 1998.

\bibitem{BrunsandHerzog}
W.~Bruns and J.~Herzog.
\newblock {\em Cohen-Macaulay rings}.
\newblock Cambridge Studies in Advanced Mathematics. Cambridge University
  Press, 1993.

\bibitem{Cartan}
E.~Cartan.
\newblock {\em The theory of spinors}.
\newblock Dover Books on Advanced Mathematics. Dover Publications, Inc. New
  York, 1981.
\newblock With a foreword by Raymond Streater. A reprint of the 1966 English
  translation.

\bibitem{CE}
H.~Cartan and S.~Eilenberg.
\newblock {\em Homological algebra}.
\newblock Princeton University Press, Princeton, NJ, 1999.

\bibitem{CNT}
M.~Cederwall, B.E.W. Nilsson, and D.~Tsimpis.
\newblock Spinorial cohomology and maximally supersymmetric theories.
\newblock {\em J. High Energy Phys.}, 0202(009), 2002.

\bibitem{ChevalleySpinors}
C.~Chevalley.
\newblock {\em The algebraic theory of spinors and Clifford algebras}, volume~2
  of {\em collected works}.
\newblock Springer, 1997.

\bibitem{CieliebakFrauenfelder}
K.~Cieliebak and U.~Frauenfelder.
\newblock Morse homology on noncompact manifolds.
\newblock {\em J. Korean Math. Soc.}, 48(4):749--774, 2011.

\bibitem{CortiReid}
A.~Corti and M.~Reid.
\newblock Weighted grassmannians.
\newblock In {\em Algebraic geometry}, pages 141--163. Walter de Gruyter and
  Co., Berlin, 2002.

\bibitem{DenefLoeser}
J.~Denef and F.~Loeser.
\newblock Germs of arcs on singular algebraic varieties and motivic
  integration.
\newblock {\em Invent. Math.}, 135(1):201--232, 1999.

\bibitem{EckmannHilton}
B.~Eckmann and P.~Hilton.
\newblock Commuting limits with colimits.
\newblock {\em J. Algebra}, 11:116--144, 1969.

\bibitem{EisenbudStr}
D.~Eisenbud.
\newblock Introduction to algebras with straightening laws.
\newblock In {\em Ring theory and algebra III}, pages 243--268. Dekker, 1980.

\bibitem{ComEisenbud}
D.~Eisenbud.
\newblock {\em Commutative algebra, with a view toward algebraic geometry},
  volume 150 of {\em Graduate Texts in Mathematics}.
\newblock Springer, 1995.

\bibitem{VEne}
V.~Ene.
\newblock Syzygies of \uppercase{H}ibi rings.
\newblock {\em Acta Mathematica Vietnamica}, 40(3):403--446, September 2015.

\bibitem{Feigin}
B.~Feigin.
\newblock \uppercase{S}emi-infinite cohomology of
  \uppercase{K}ac-\uppercase{M}oody and \uppercase{V}irasoro \uppercase{L}ie
  algebras (in \uppercase{R}ussian).
\newblock {\em Usp. Mat. Nauk}, 39(2):195--196, 1984.

\bibitem{FreiMacdonald}
A.~Frei and J.~Macdonald.
\newblock Limits in categories of relations and limit-colimit commutation.
\newblock {\em J. Pure Appl. Algebra}, 1(2):179--197, 1971.

\bibitem{GMSch}
V.~Gorbounov, F.Malikov, and V.Schekhtman.
\newblock On chiral differential operators over homogeneous spaces.
\newblock {\em IJMMS}, 26(2):83---106, 2001.

\bibitem{GerbGMsch}
V.~Gorbounov, F.~Malikov, and V.~Schekhtman.
\newblock Gerbes of chiral differential operators, \uppercase{I-III}.
\newblock math.AG/0005201, math.AG/0003170, math.AG/9906117.

\bibitem{THibi}
T.~Hibi.
\newblock Distributive lattices, affine semigroup rings and algebras with
  straightening laws.
\newblock In M.~Nagata and H.~Matsumura, editors, {\em Commutative Algebra and
  Combinatorics}, volume~11 of {\em Adv. Stud. Pure Math.}, pages 93--109.
  North Holland, Amsterdam, 1987.

\bibitem{Iversen}
B.~Iversen.
\newblock {\em Cohomology of sheaves}.
\newblock Springer, 1986.

\bibitem{Twenty}
S.~B. Iyengar, G.~J. Leuschke, A.~Leykin, C.~Miller, E.~Miller, A.K. Singh, and
  U.~Walther.
\newblock {\em Twenty-four hours of local cohomology}, volume~87 of {\em
  Graduate Studies in Mathematics}.
\newblock American Mathematical Society, Providence, RI, 2007.

\bibitem{Kempf}
G.~Kempf.
\newblock The \uppercase{G}rothendieck-\uppercase{C}ousin complex of an induced
  representation.
\newblock {\em Adv. in Math.}, 29(3):310--396, 1978.

\bibitem{Hoermander}
Hoermander L.
\newblock {\em The Analysis of Linear PD Operators. III, Pseudo-Differential
  Operators}.
\newblock Springer, 3 edition, 2007.

\bibitem{MalikovSchechtman}
F.~Malikov and V.~Schechtman.
\newblock Chiral de \uppercase{R}ham complex \uppercase{I}\uppercase{I}.
\newblock In {\em Differential topology, infinite-dimensional Lie algebras, and
  applications}, volume 194 of {\em Amer. Math. Soc. Transl. Ser. 2}, pages
  149--188. Amer. Math. Soc., Providence, RI,, 1999.

\bibitem{MalikovSchechtmanVaintrob}
F.~Malikov, V.~Schechtman, and A.~Vaintrob.
\newblock Chiral de \uppercase{R}ham complex.
\newblock {\em Comm. Math. Phys.}, 204(2):439--473, 1999.

\bibitem{Matsumura}
H.~Matsumura.
\newblock {\em Commutative ring theory}, volume~8 of {\em Cambridge Studies in
  Advanced Mathematics}.
\newblock Cambridge University Press, Cambridge, 1986.

\bibitem{MillerandSturmfels}
E.~Miller and B.~Sturmfels.
\newblock {\em Combinatorial commutative algebra.}, volume 227 of {\em Graduate
  Texts in Mathematics}.
\newblock Springer-Verlag, New York, NY, 2005.

\bibitem{MovStr}
M.~V. Movshev.
\newblock Geometry of pure spinor formalism in superstring theory.
\newblock {\em Adv. Math.}, 268:201--240, 2015.

\bibitem{Movq}
M.V. Movshev.
\newblock On quasimaps to quadrics, 2010.
\newblock arXiv:1008.0804v1 [math.QA].

\bibitem{Nekrasovbetagamma}
N.~Nekrasov.
\newblock Lectures on curved beta-gamma systems, pure spinors, and anomalies,
  2005.
\newblock arXiv:hep-th/0511008v1.

\bibitem{PP}
A.~Polishchuk and L.~Positselski.
\newblock {\em Quadratic algebras}, volume~37 of {\em Univ. Lecture Ser.}
\newblock Amer. Math. Soc., Providence, RI, 2005.

\bibitem{RoccoSturmfels}
Sandra~Di Rocco and Bernd Sturmfels, editors.
\newblock {\em Combinatorial Algebraic Geometry: Levico Terme, Italy}.
\newblock Springer, 2013.

\bibitem{Ruffo}
J.~Ruffo.
\newblock Quasimaps, straightening laws, and quantum cohomology for the
  \uppercase{L}agrangian \uppercase{G}rassmannian, 2008.
\newblock arXiv:0806.0834v1 [math.AG].

\bibitem{Schapira}
P.~Schapira.
\newblock {\em Theorie des Hyperfonctions}.
\newblock Lecture Notes in Mathematics. Springer, 1970.

\bibitem{Serre2}
J.-P. Serre.
\newblock Sur les modules projectifs.
\newblock {\em Seminaire Dubreil. Algebre et Theorie des Nombres}, 14(1):1--16,
  1960.

\bibitem{SottileSturmfels}
F.~Sottile and B.~Sturmfels.
\newblock A sagbi basis for the quantum \uppercase{G}rassmannian.
\newblock {\em J. Pure and Appl. Alg.}, 158:347--366., 2001.

\bibitem{Stanley}
R.~Stanley.
\newblock Hilbert functions of graded algebras.
\newblock {\em Advances in Math.}, 28:57--83, 1978.

\bibitem{Strohmaier}
A.~Strohmaier.
\newblock {\em Quantum Field Theory on Curved Spacetimes: Concepts and
  Mathematical Foundations}, volume 786 of {\em Lecture Notes in Physics},
  chapter~4.
\newblock Springer-Verlag Berlin Heidelberg, 2009.

\bibitem{Thomason}
R.W Thomason.
\newblock Equivariant resolution, linearization, and \uppercase{H}ilbert's
  fourteenth problem over arbitrary base schemes.
\newblock {\em Advances in Mathematics}, 65(1):16--34, 1987.

\bibitem{BVinbergALOnishchik}
E.~B Vinberg and A.~L. Onishchik.
\newblock {\em Seminar on Lie groups and algebraic groups}.
\newblock Springer series in Soviet mathematics. Springer Verlag, Berlin, New
  York, 1990.

\bibitem{Weibel}
C.~Weibel.
\newblock {\em An introduction to homological algebra}, volume~38 of {\em
  Cambridge Studies in Advanced Mathematics}.
\newblock Cambridge University Press, 1994.

\end{thebibliography}

\end{document}